\documentclass[oneside,11pt]{article} 

\usepackage[utf8]{inputenc} 
\usepackage{geometry}
\usepackage{graphicx} 
\usepackage{booktabs} 
\usepackage{array} 
\usepackage{paralist} 
\usepackage{verbatim} 
\usepackage{subfig}
\usepackage{amsmath}
\usepackage{amsfonts}
\usepackage{amssymb}
\usepackage{amsthm}
\usepackage{authblk}
\usepackage{etoolbox}
\usepackage{fancyhdr}
\usepackage{sectsty}
\allsectionsfont{\sffamily\mdseries\upshape} 
\usepackage{url}
\usepackage[colorlinks=true, linkcolor=blue, citecolor=blue, urlcolor=blue]{hyperref}
\usepackage[nottoc,notlof,notlot]{tocbibind} 
\usepackage[titles,subfigure]{tocloft}

\theoremstyle{remark}
\newtheorem{remark}{Remark}
\theoremstyle{plain}
\newtheorem{theorem}{Theorem}
\theoremstyle{plain}
\newtheorem{prop}{Proposition}
\theoremstyle{plain}
\newtheorem{lemma}{Lemma}
\theoremstyle{plain}
\newtheorem{cor}{Corollary}
\theoremstyle{definition}
\newtheorem{defin}{Definition}
\theoremstyle{definition}
\newtheorem{example}{Example}

\makeatletter
\pretocmd{\proof}{\setlength{\parindent}{0pt}}{}{}
\makeatother

\AtBeginEnvironment{lemma}{\setlength{\parindent}{0pt}}

\numberwithin{equation}{section}
\begin{document}
\title{Wide neural networks with general weights:\\
convergence rate and explicit dependence on the hyper-parameters}
\author{Lucia Celli}

\date{} 
\affil[]{Department of Mathematics, University of Luxembourg}

\maketitle
\begin{abstract}
    Using Stein's method techniques introduced by Chatterjee (2008) and further extended by Kasprzak and Peccati (2022) and by Lachi\`eze-Rey and Peccati (2017), we derive novel quantitative bounds on the convergence in distribution of feed-forward fully connected neural networks (with Lipschitz activation functions) towards Gaussian processes, as the hidden layer width $n$ tends to infinity. We consider networks initialized with independent and identically distributed (i.i.d.) weights possessing sufficiently many finite moments, and i.i.d.\ Gaussian biases independent of the weights.

Specifically, when the network is evaluated at a single input, we obtain convergence rates of order $O(n^{-1/2})$ in both total variation and Wasserstein distances. When evaluated at a general finite collection of inputs, we establish bounds of the same order in terms of the convex distance. All bounds are given in explicit and computable form.

As a consequence of our estimates, we also deduce a novel convergence result in the regime where the depth of the neural network increases simultaneously with the width $n$, up to order $O\big((\log_2 n)^{1/3}\big)$. To the best of our knowledge, this is the first CLT in the infinite width/depth limit holding for general (nonlinear) Lipschitz activation functions and non-Gaussian weight distributions.

Our analysis yields several results of independent interest, including: (i) an explicit lower bound on the determinant of the limiting covariance matrix and (ii) new advances in Stein's method, both for the one-dimensional Stein's equation associated with the square of a Lipschitz function and for the multivariate Stein's equation associated with the tensor product of a Lipschitz function with itself.

\end{abstract}

\tableofcontents

\section{Introduction}
\subsection{Overview}
In recent years, many papers have studied the problem of assessing the discrepancy between a neural network with Gaussian initialization and its limit in law when the inner widths diverge to infinity (see \cite{BT24,FHMNP,Torr23,Trev,CP25,BFF24,CC23,CMSV24,BR25}). Such a direction of research is motivated by applications to Bayesian neural networks (see \cite{Trev,CP25}), and has deep connections with the analysis of the so-called NTK regime (see \cite{CC23, Ye}). The inability of infinitely wide feed-forward neural networks with Gaussian initialization to learn data-dependent features (see \cite{COB19,YH20}) has motivated researchers to investigate alternative models. These include networks with different hyper-parameters scaling regimes (see e.g. \cite{MFBV25,G20,GSJW20}), alternative limiting procedures (e.g., \cite{NR21,CKZ23,HZ23,PAPGGR23,LNR23,BLR24,HN18}), or non-Gaussian initializations (see \cite{SMCG13,HXP20,FGAOWTYWA21}).

In this work we focus on the latter area and study the discrepancy between the infinite-width limit and the distribution of a fully connected neural network with independent and identically distributed (i.i.d.) parameter initialization and finite hidden-layer widths.
{
The main results of this paper (Theorems \ref{finale_uno_1} and \ref{pres_sec_prob}) are obtained through a version of Stein's method for functionals of independent random elements: such an approach, first introduced in \cite{C08} and then extended in \cite{KG22, LG17} uses difference operators built from independent copies of the underlying random quantities. The method developed in \cite{C08, KG22, LG17} is natural in our setting, which deals with functionals of conditionally independent random variables or vectors. Moreover, in the classical case of a sum of $n\in\mathbb{N}$ independent random variables or vectors, the previously mentioned papers provide an optimal convergence rate in $n$ to a Gaussian distribution.  
}

We proceed with an informal formulation of our main theorem, which extends the results of \cite{BT24,Trev,FHMNP,CP25,Torr23,BFF24} to the case of non-Gaussian weights. All necessary definitions can be found in Section~\ref{prel_sec}.

\begin{theorem}[Informal version of Theorems \ref{finale_uno_1} and \ref{pres_sec_prob}]\label{inf_th_conv_n}
    Let \( z^{(L+1)}(\mathcal{X}) \) be the output of a fully connected feedforward neural network evaluated on a finite set of distinct inputs \( \mathcal{X} = \{x^{(1)}, \dots, x^{(d)}\} \). Assume the network has i.i.d.\ Gaussian biases with variance \( C_b \neq 0 \), i.i.d.\ weights with sufficiently many finite moments, inner widths equal to \( n \in \mathbb{N} \), and a Lipschitz activation function.

    Then:
    \begin{itemize}
        \item If \( d = 1 \), the distribution of the neural network evaluated in one input converges to a Gaussian law as \( n \to \infty \), with a rate of convergence in both the Kolmogorov and Wasserstein distances of order \( O\left(\frac{1}{\sqrt{n}}\right) \).
        \item If \( d > 1 \), then \( z^{(L+1)}(\mathcal{X}) \) converges in law to a $d$-dimensional Gaussian vector in the convex distance, with the same rate \( O\left(\frac{1}{\sqrt{n}}\right) \), provided that adequate additional assumptions are satisfied by the derivative of the activation function and by the covariance matrix of the limiting Gaussian distribution.
    \end{itemize}
   Furthermore, the constants involved in the above bounds are explicitly given in terms of the neural network's hyper-parameters.
\end{theorem}

\begin{remark}{\rm For neural networks evaluated on a finite number of inputs $d>1$, alternative approaches could improve the dependence on $d$. For example, \cite{FK21} gives finer bounds measured via hyper-rectangles, \cite{FK22} provides multidimensional Wasserstein distance results  (see Definition \ref{Wdist}) that are optimal in $n$ for independent entries, and \cite{FKLZ23} treats cases where the Gaussian limit has a singular covariance but satisfies non-degeneracy conditions on $2\times 2$ and $3\times 3$ sub-matrices, still yielding bounds optimal in $n$, but with a worse dependence on $d$ compared to the case with a non-singular covariance matrix.
These extensions are left for future work. Improving the dependence on $d$ is challenging, not only due to the initial distance bounds, but also because of lower bounds on the minimal eigenvalue of the limiting covariance matrix, when required.

}
\end{remark}

The availability of explicit bounds on the distances between the law of the neural network and that of the Gaussian limit enables us to analyze the joint limit in which both the depth \( L \) and the hidden-layer widths tend to infinity. The next corollary informally summarizes the content of Corollary~\ref{rem_L_inf_1} and~\ref{rem_L_inf_d}, as stated in Section \ref{main_results_sec}.

\begin{cor}\label{cor_intr_gen_Lninf}
Under the assumptions and notations of Theorem~\ref{inf_th_conv_n}, suppose that the \( p \)-th moments of the weights are bounded by a polynomial function of \( p \), and that the depth satisfies
\begin{equation}\label{cond_L}
L = O\left((\log_2 n)^{\frac{1}{3}}\right).
\end{equation}
Then:
\begin{itemize}
    \item If \( d = 1 \),  as \( n \to \infty \) the Kolmogorov and Wasserstein distances between the law of the output of the neural network evaluated at a single input and that of its infinite-width Gaussian limit converge to zero. 
    
    \item If \( d > 1 \), as \( n \to \infty \) the convex distance between the law of \( z^{(L+1)}(\mathcal{X}) \) and that of the infinite-width Gaussian limit converges to zero, provided that an additional assumption is satisfied by the expected value of the derivative of the activation function under the infinite-width Gaussian law. 
\end{itemize}
\end{cor}
{
The condition~\eqref{cond_L}, which restricts the growth of \(L\) to be at most logarithmic in \(n\), ensures that the depth increases slowly enough with respect to the width, for the network to remain close to its Gaussian limit. Intuitively, the constrain on $L$ controls the accumulation of non-Gaussian fluctuations across layers, allowing the explicit bounds on the distances to vanish as \(n \to \infty\).

The bounds of Corollary~\ref{cor_intr_gen_Lninf} can be further refined in the case of Gaussian initialization for several activation functions. These include the perceptron with $d=1$ (see Example~\ref{es_1_g_RL}), the identity activation (see Example~\ref{ex_id_G_1d} for $d=1$ and Section~\ref{id:mult} for $d>1$), and the ReLU activation with $d=1$ (see Example~\ref{ex_relu_1d}). 
Under suitable choices of the variances of the biases and weights, we obtain upper bounds for the distances under consideration of order $\frac{1}{n}$ for the perceptron activation and of order $\frac{L}{n}$ for the other activation functions.
}
These refinements are obtained by applying the same strategy as in the proof of Theorem~\ref{inf_th_conv_n}, while additionally exploiting specific properties of the Gaussian distribution and of the activation functions where applicable. The resulting bounds are consistent with those found in the existing literature (see~\cite{BLR24, Hanin2022CorrelationFI}) 
{
and show that Corollary~\ref{cor_intr_gen_Lninf} is not sharp in several specific regimes. This lack of sharpness should be understood as a reasonable trade-off for allowing non-Gaussian parameter initialization and general Lipschitz activation functions.

When $d=1$, we are also able to improve the result under non-Gaussian initialization, provided that the activation function is bounded and Lipschitz (see Section \ref{lip_bound_s_sec}). For a suitable range of values of the variances of the biases and weights, we obtain in this case bounds of order $\frac{L}{n}$ for the distances under consideration, under weaker moment assumptions on the weights than those required in Theorem~\ref{inf_th_conv_n} and Corollary~\ref{cor_intr_gen_Lninf}.

}
 To the best of our knowledge, our work is the first to establish the Gaussian fluctuations of the fully connected neural network, as both the depth and the hidden-layer widths tend to infinity, under general assumptions on the activation function and general i.i.d.\ initialization of the network parameters.

The proofs of our main results also lead to a number of intermediate insights and technical contributions of independent interest. 
\begin{enumerate}
    \item[(i)] The proof of Theorem~\ref{inf_th_conv_n} relies on a novel application of Stein's method (see Section~\ref{prel_sec}), in which we study the classical one-dimensional and multi-dimensional Stein's equations (see, respectively, \eqref{Stein_eq_1} and \eqref{stein_eq_mult_intr} below) in the situation where the test function is either the square of a Lipschitz function ($d=1$) or the tensor product of a real Lipschitz function \( \sigma \) with itself, i.e.,
\[
h(y) := \sigma(y_i)\sigma(y_j)
\]
for any \( y \in \mathbb{R}^d \) and fixed indices \( i, j \in \{1, \dots, d\} \)\footnote{ {Although Stein's method is effectively applied in a two-dimensional setting, the proof of Theorem~\ref{inf_th_conv_n} is conceptually simpler when the underlying function is regarded as defined on $\mathbb{R}^d$ rather than on $\mathbb{R}^2$.}
}. 
A precise characterization of the regularity properties of the solutions of the Stein's equations can be found in 
Lemmas~\ref{fs} and~\ref{lipf} in Appendix~A (one-dimensional case) and 
Lemmas~\ref{sol_st_mult}, \ref{norm_inf_hess}, \ref{lemma_hess}, \ref{mom_quart_der_sec}, and~\ref{der_ter_lemma} in Appendix~B (for $d\ge 2$).

    \item[(ii)] In order to prove the multi-dimensional case of Theorem~\ref{inf_th_conv_n}, we also establish a novel lower bound on the determinant of the covariance matrix of the Gaussian limit of the neural network as the inner widths tend to infinity (see Theorem~\ref{Hanin}). This result is stated in Theorem~\ref{bound_K_gauss} and is based on a classical inequality from~\cite{Ca82}, which we restate in Proposition~\ref{Caco}.

\end{enumerate}
\subsection{Structure of the paper}
In Section~\ref{prel_sec}, we introduce the neural network model in a rigorous manner, define the probability distances considered throughout the paper, and present some basic notations related to Stein's method. The main results, along with a number of remarks, are formally stated in Section~\ref{main_results_sec}. A discussion of related literature and connections to existing work is provided in Section~\ref{literature_review_sec}.

A high-level overview of the proof of Theorem~\ref{inf_th_conv_n} is given in Section~\ref{proof_scheme_sec}, mainly to guide the reader through the main ideas. The detailed analysis of the one-dimensional case is carried out in Section~\ref{uno_d}, where we also derive moment bounds for the neural network, and present examples involving Gaussian weights and activation functions that are Lipschitz continuous and bounded.

Section~\ref{multi_dimensional_problem_sec} contains the proof of the multi-dimensional case, including a lower bound on the determinant and eigenvalues of the limiting covariance matrix, as well as an illustrative example with a linear activation function.

Appendix~A collects the technical estimates required for the one-dimensional case, including results on the Stein's equation associated with the square of a Lipschitz function. Similarly, Appendix~B gathers the auxiliary lemmas used in the multi-dimensional case, and the analysis of the Stein's equation associated with the tensor product of a Lipschitz function with itself.

\section{Preliminaries}\label{prel_sec}

For any vector $x \in \mathbb{R}^d$, we denote by 
\[
\|x\| := \left( \sum_{i=1}^d x_i^2 \right)^{1/2}
\]the Euclidean norm, and by 
\[
\langle x, y \rangle := \sum_{i=1}^d x_i y_i
\]
the standard inner product in $\mathbb{R}^d$.
For any matrix $M \in \mathbb{R}^{d \times d}$, we denote by 
\[
\|M\|_{HS} := \left( \sum_{i,j=1}^{d} M_{i,j}^2 \right)^{1/2}
\]
the Hilbert--Schmidt norm, by \[
\|M\|_{op} := \max_{\|x\| = 1} x^\top M x
\]the operator norm, by 
\[
\operatorname{tr} M := \sum_{i=1}^d M_{i,i}
\]the trace of $M$, and by 
\[
\lambda(M) := \min_{\|x\| = 1} x^\top M x
\]
the smallest eigenvalue of $M$. Given $M, N \in \mathbb{R}^{d \times d}$, we denote their Hilbert--Schmidt inner product by 
\[
\langle M, N \rangle_{HS} := \sum_{i,j=1}^d M_{i,j} N_{i,j}.
\]

If $M \in \mathbb{R}^{d \times d}$ is positive semi-definite (risp. $M \ge 0$ when $d = 1$), we denote by $\mathcal{N}_d(\mu, M)$ the Gaussian distribution on $\mathbb{R}^d$ with mean $\mu \in \mathbb{R}^d$ and covariance matrix $M$ (resp. variance). We write $X \sim \mathcal{D}$ to indicate that a random variable $X$ has distribution $\mathcal{D}$.

\subsection{Random Fully Connected Neural Networks}\label{sec:def_nn}
Neural networks are widely used in supervised learning tasks, largely thanks to the universal approximation theorems~\cite{Cybenko,Hornik,Leshno,Pinkus,GRIPENBERG,Yarotsky,lu,HS18,kidger}, which guarantee their ability to approximate a broad class of functions. In a typical supervised learning setting, one is given a collection of input-output pairs \( \{(x^{(i)}, f(x^{(i)}))\} \), and the objective is to construct a neural network that approximates the unknown function \( f \), not only on the training inputs \( x^{(i)} \) but also on unseen data (see \cite{Goodfellow,shalev}).

To achieve this, one seeks network parameters that minimize a prescribed loss function, which quantifies the discrepancy between the network's predictions and the true values. This optimization is commonly carried out using algorithms such as gradient descent starting from a random initialization of the network parameters (see e.g. \cite{Agg,RYH22,Ye}).

\,

In the present paper, we focus on a specific class of neural networks: fully connected networks with independent and identically distributed (i.i.d.) random weights at initialization and a Lipschitz continuous activation function (see similar definitions in \cite{Han22,BT24,FHMNP,BR25,Trev,Han_Gas}).

\begin{defin}[Random Fully Connected Neural Network]\label{def_NN}
Let \( C_W > 0 \) and let 

\[
 L, n_0, n_1, \dots, n_L, n_{L+1} \ge 1 \]
 be positive integers. A \emph{random fully connected neural network} is specified by a depth \( L \) (i.e., \( L \) is the number of hidden layers), by an input dimension \( n_0 \), by an output dimension \( n_{L+1} \), and by the hidden layer widths \( n_1, \dots, n_L \). The network uses an activation function \( \sigma : \mathbb{R} \to \mathbb{R} \), which is assumed to be Lipschitz continuous, with \( \|\sigma\|_{\mathrm{Lip}} \) its Lipschitz norm.

Given an input \( x \in \mathbb{R}^{n_0} \), the network output is the vector
\[
z^{(L+1)}(x) := \left(z_1^{(L+1)}(x), \dots, z_{n_{L+1}}^{(L+1)}(x)\right) \in \mathbb{R}^{n_{L+1}},
\]
whose components are defined recursively by
\begin{equation}\label{NN_expr}
\begin{cases}
z_{j}^{(1)}(x) = b_j^{(1)} + \dfrac{\sqrt{C_W}}{\sqrt{n_0}} \displaystyle\sum_{k=1}^{n_0} W_{j,k}^{(1)} x_k & \text{if } \ell = 1, \\[1.5ex]
z_{j}^{(\ell)}(x) = b_j^{(\ell)} + \dfrac{\sqrt{C_W}}{\sqrt{n_{\ell-1}}} \displaystyle\sum_{k=1}^{n_{\ell-1}} W_{j,k}^{(\ell)} \sigma\big(z_{k}^{(\ell-1)}(x)\big) & \text{for } \ell = 2, \dots, L+1,
\end{cases}
\end{equation}
for all \( j = 1, \dots, n_\ell \) and \( \ell = 1, \dots, L+1 \).

The biases \( \{b_j^{(\ell)}\} \) are independent and identically distributed (i.i.d.) random variables with distribution \( \mathcal{N}_1(0, C_b) \), for some \( C_b \ge 0 \). The weights \( \{W_{j,k}^{(\ell)}\} \) are i.i.d.\ copies of a generic random variable \( W \), independent of the biases, and satisfying
\[
\mathbb{E}[W] = 0, \qquad \mathbb{E}[W^2] = 1.
\]
We stress that \( W \) is not assumed to be Gaussian.
\end{defin}

\begin{remark}
Definition~\ref{def_NN} is tailored to the specific setting considered in this paper. More general definitions of fully connected neural networks allow for different choices of activation functions and initialization schemes (see e.g.~\cite{Han22,Agg}).

\end{remark}

The following result, which holds under more general assumptions on the activation function \( \sigma \), describes the limiting behavior of a random neural network as the widths of the hidden layers tend to infinity.

\begin{theorem}[Theorem 1.2 in \cite{Han22}]\label{Hanin}
Fix \( n_0 \) and \( n_{L+1} \), and let \( T \subseteq \mathbb{R}^{n_0} \) be a compact set. Suppose the weight distribution \( W \) has finite moments of all orders. Then, as the hidden layer widths \( n_1, \dots, n_L \to \infty \), the sequence of stochastic processes
\[
x \in \mathbb{R}^{n_0} \mapsto z^{(L+1)}(x) \in \mathbb{R}^{n_{L+1}}
\]
converges weakly in \( C^0(T, \mathbb{R}^{n_{L+1}}) \) to a centered Gaussian process with independent and identically distributed coordinates. The coordinate-wise covariance function
\[
K_{j,k}^{(L+1)} := \lim_{n_1, \dots, n_L \to \infty} \operatorname{Cov}\left(z_{i}^{(L+1)}(x^{(j)}), z_{i}^{(L+1)}(x^{(k)})\right)
\]
satisfies the following layer-wise recursion for \( \ell \ge 2 \):
\[
K_{j,k}^{(\ell)} = C_b + C_W \, \mathbb{E} \left[ \sigma\left(G_{1}^{(\ell-1)}(x^{(j)})\right) \sigma\left(G_{1}^{(\ell-1)}(x^{(k)})\right) \right],
\]
where
\[
\left( G_{1}^{(\ell-1)}(x^{(j)}), G_{1}^{(\ell-1)}(x^{(k)}) \right) \sim \mathcal{N}\left(0,
\begin{pmatrix}
K_{j,j}^{(\ell-1)} & K_{j,k}^{(\ell-1)} \\
K_{j,k}^{(\ell-1)} & K_{k,k}^{(\ell-1)}
\end{pmatrix}
\right),
\]
with initial condition
\[
K_{j,k}^{(2)}=C_b+C_W\mathbb{E}\Big[\sigma(z_{1}^{(1)}(x^{(j)}))\sigma(z_{1}^{(1)}(x^{(k)}))\Big].
\]
\end{theorem}

The previous result is key to understanding the convergence of gradient-based training algorithms. When all hidden layer widths tend to infinity, the training dynamics of the neural network becomes indistinguishable from its linearization. As shown in~\cite{Jacot,arora_2020,CC23,EAT25}, the parameter evolution of such a linearizated network is entirely determined by the \emph{Neural Tangent Kernel} (NTK), defined as follows.
\begin{defin}
Let \( \Theta := \{W_{r,s}^{(\ell)},\, b_r^{(\ell)}\}_{\ell, r, s} \) be the set of all trainable parameters as in Definition \ref{def_NN}. Given inputs \( \{x^{(1)}, \dots, x^{(n)}\} \), the NTK matrix \( \Sigma \in \mathbb{R}^{n \times n} \) is defined as
\[
\Sigma_{i,j} := \left\langle \nabla_\Theta z^{(L+1)}(x^{(i)}; \Theta),\; \nabla_\Theta z^{(L+1)}(x^{(j)}; \Theta) \right\rangle,
\]
where \( \nabla_\Theta \) denotes the gradient with respect to all parameters.
\end{defin}
In the infinite-width regime, \( \Sigma \) converges to a deterministic kernel that remains constant during training, and the network output evolves accordingly along a linear trajectory defined by this kernel (see e.g. \cite{Jacot,arora_2020,EAT25}).

\subsection{Probability metrics}
Theorem~\ref{Hanin} provides a qualitative Central Limit Theorem (CLT) for the distribution of a fully connected neural network with i.i.d.\ initialization. However, it does not yield a quantitative estimate of the discrepancy between the law of the neural network and that of its Gaussian limit. 

In this paper, we address this gap by studying a quantitative version of the above Central Limit Theorem, focusing in particular on the following distances between probability measures:

{ 
\begin{defin}[Multi-dimensional Kolmogorov distance]\label{def_mK}
Let $X$ and $Y$ be random variables taking values in $\mathbb{R}^d$, where $d \in \mathbb{N}$.  
The multi-dimensional Kolmogorov distance between the laws of $X$ and $Y$ is defined by
\[
d_{mK}(X,Y)
:=\sup_{t_1,\dots,t_d\in\mathbb{R}}
\Big|
\mathbb{P}\big(X_1\le t_1,\dots,X_d\le t_d\big)
-
\mathbb{P}\big(Y_1\le t_1,\dots,Y_d\le t_d\big)
\Big|.
\]
For $d=1$, this reduces to the classical Kolmogorov distance, which we denote by $d_K(X,Y)$ (see e.g. \cite[Definition C.2.1]{NP12}).
\end{defin}
}
\begin{defin}[$p$-Wasserstein Distance {\cite{NP12,villani}}]\label{d_w_def}
Let $p \geq 1$ be an integer, and let $X$ and $Y$ be $L_p$-integrable random variables taking values in $\mathbb{R}^d$, where $d \in \mathbb{N}$. The $p$-Wasserstein distance between the laws of $X$ and $Y$ is defined as
\[
W_p(X, Y) = \inf_{(Z, W) \in \Pi} \left( \mathbb{E}\left[\|Z - W\|^p\right] \right)^{\frac{1}{p}},
\]
where $\Pi$ denotes the set of all couplings of $X$ and $Y$, that is,
\[
\Pi = \left\{ (Z, W) \in L^p(\Omega; \mathbb{R}^d) : Z \overset{\text{Law}}{=} X, \; W \overset{\text{Law}}{=} Y \right\}.
\]
\end{defin}
When $p=1$ the following dual definition for the Wasserstein distance holds:
\begin{prop}[Kantorovich–Rubinstein Duality \cite{villani}]\label{form_dual_wass}
The 1-Wasserstein distance between the distributions of $X$ and $Y$ verifies the identity
\[
W_1(X, Y) = \sup_{h \in \mathrm{Lip}_d(1)} \left| \mathbb{E}[h(X)] - \mathbb{E}[h(Y)] \right|,
\]
where
\[
\mathrm{Lip}_d(1) := \left\{ h : \mathbb{R}^d \to \mathbb{R} \;\middle|\; \sup_{x \neq y} \frac{|h(x) - h(y)|}{\|x - y\|} \leq 1 \right\}
\]
is the set of all 1-Lipschitz functions on $\mathbb{R}^d$.
\end{prop}
\begin{defin}[Convex distance \cite{Torr23, FHMNP, KG22}]\label{d_c_def}
If $X,Y$ are random variables with values in $\mathbb{R}^d$, with $d\ge 1$ denoting an integer, then the convex distance between the laws of $X$ and $Y$ is defined as
    \[
    d_C(X,Y):=\sup_{C \subseteq\mathbb{R}^d, convex}\Big|\mathbb{P}\left(X\in C\right)-\mathbb{P}\left(Y\in C\right)\Big|,
    \]
    where the supremum runs over all the convex sets $C$ of $\mathbb{R}^d$.
\end{defin}
\begin{defin}[Total variation distance \cite{NP12}]\label{d_tv_def}
If $X,Y$ are random variables with values in $\mathbb{R}^d$, with $d\ge 1$ denoting an integer, then the total variation distance between the laws of $X$ and $Y$ is defined as
    \[
    d_{TV}(X,Y):=\sup_{B\in\mathcal{B}(\mathbb{R}^d)}\Big|\mathbb{P}\left(X\in B\right)-\mathbb{P}\left(Y \in B\right)\Big|,
    \]
    where the supremum runs over all the Borel sets $B$ of $\mathbb{R}^d$.
\end{defin}
{
\begin{remark}\label{d_c_less_d_TV}
    Since left half-lines are convex sets and convex sets are also Borel sets, it follows that
    \[
   d_{mK}(X,Y)\le d_C(X,Y)\le d_{TV}(X,Y).
    \]
\end{remark}}
\begin{remark}
{
    Thanks to Proposition C.3.1 in \cite{NP12}, we know that if $\{F_n\}_{n},F$ are random variables such that the Kolmogorov, convex, total variation, or Wasserstein distance between $F_n$ and $F$ convergence to zero as $n\to\infty$, then $F_n\to F$ in distribution.
    }
\end{remark}

\subsection{Stein's method}
A widely used method to derive upper bounds on a distance (such as those defined in the previous section) between the law of a random element and a Gaussian distribution is the so-called \emph{Stein's method} (see e.g. \cite{stein,CGS10,CM08,KG22,C08, BR25,Torr23}). In the one-dimensional case, this method relies on the Gaussian integration by parts formula presented in Lemma~\ref{int_gauss_uno}. As explained in \cite{NP12, CGS10}, the crux of the method consists in studying the properties of a canonical solution $f_h$ to the differential equation (known as the \emph{Stein's equation})
\begin{equation}\label{Stein_eq_1}
\sigma^2 f{'}(x) - x f(x) = h(x) - \mathbb{E}[h(N)],\quad{x\in\mathbb{R},}
\end{equation}
where $N$ is a Gaussian random variable with law $\mathcal{N}_1(0,\sigma^2)$, and $h:\mathbb{R} \to \mathbb{R}$ is a Borel-measurable function satisfying $\mathbb{E}[h(N)] < \infty$. Typically, the function $h$ is chosen from the class $\mathcal{H}$ of test functions associated with the considered distance. This allows one to study the quantity
\begin{equation}\label{St_eq_1d_prima}
\sup_{h \in \mathcal{H}} \left| \mathbb{E}[\sigma^2 f_h'(F_n)] - \mathbb{E}[F_n f_h(F_n)] \right|
\end{equation}
instead of
\begin{equation}\label{St_eq_1d_dopo}
\sup_{h \in \mathcal{H}} \left| \mathbb{E}[h(F_n)] - \mathbb{E}[h(N)] \right|.
\end{equation}
Studying \eqref{St_eq_1d_prima} instead of \eqref{St_eq_1d_dopo} is typically more amenable to analysis since the Gaussian random variable $N$—which may be defined on a different probability space than $F_n$—no longer appears explicitly in the quantities defining the supremum and one can use the regularity properties of the solution $f_h$ to derive an upper bound to the supremum in \eqref{St_eq_1d_prima}.

\,

In the multi-dimensional case, the Stein's approach must be slightly adapted, since the relevant multi-dimensional Gaussian integration by parts formula involves the second derivatives of the test function as well (see Lemma~\ref{Gauss_integ_multdim}). More precisely (see e.g. \cite{CM08,CGS10,NP12}), the Stein's equation takes the form
\begin{equation}\label{stein_eq_mult_intr}
\langle C, \operatorname{Hess} f(x) \rangle_{HS} - \langle x, \nabla f(x) \rangle = h(x) - \mathbb{E}[h(N)],
\end{equation}
where $C \in \mathbb{R}^{d \times d}$ is a positive definite matrix, $\operatorname{Hess}f$ is the Hessian matrix of $f$, $N \sim \mathcal{N}_d(0, C)$, and $h:\mathbb{R}^d \to \mathbb{R}$ is a measurable function such that $\mathbb{E}[h(N)] < \infty$. As in the one-dimensional setting, the left-hand side of~\eqref{stein_eq_mult_intr} can be used to facilitate the analysis of the previously introduced distances.

\,

Throughout the paper, we combine Stein's method with the approach introduced in \cite{C08} and further developed in \cite{KG22} and \cite{LG17}, which involves the introduction of an independent copy of the random variable under consideration (i.e., $F_n$).

\section{Main results}\label{main_results_sec}
This paper addresses two main problems:
\begin{enumerate}
    \item[\textbf{Problem 1}] Deriving upper bounds on the Kolmogorov and Wasserstein distances between the random neural network~\eqref{NN_expr} evaluated at a single point \( x^{(1)} := x \in \mathbb{R}^{n_0} \), and a Gaussian distribution with covariance given by \( K_{1,1}^{(L)} \), as stated in Theorem~\ref{Hanin};
    
    \item[\textbf{Problem 2}] Establishing a quantitative central limit theorem using the convex distance between the law of the neural network~\eqref{NN_expr} evaluated at a finite set of distinct points \( \mathcal{X} := \{x^{(1)}, \dots, x^{(d)}\} \subset \mathbb{R}^{n_0} \), with \( d \geq 2 \), and a Gaussian distribution with the covariance structure described in Theorem~\ref{Hanin}.
\end{enumerate}

In both cases, we obtain explicit bounds on the respective distances, which in turn allows us to study the convergence in law as both the hidden layer widths \( n_1, \dots, n_L \) and the depth \( L \) tend to infinity under suitable assumptions.

Our main result addressing {\bf{Problem 1}} is the following statement, which is proved in Section~\ref{uno_d} using the strategy outlined in Section~\ref{strategy_uno_d}.

\begin{theorem}\label{finale_uno_1}
Let \( L \ge 1 \) be an integer, fix \( C_b, C_W > 0 \), as well as \( x \in \mathbb{R}^{n_0} \). Let \( z_1^{(L+1)}(x) \), \( \sigma \), and \( W \) be defined as in Definition~\ref{def_NN}. If \( \mathbb{E}[W^{5\cdot 2^{L+1}}] < \infty \), then the following bound holds:
\begin{multline*}
  \max\Big\{d_K(z_{1}^{(L+1)}(x),G_{1}^{(L+1)}(x)),W_1(z_{1}^{(L+1)}(x),G_{1}^{(L+1)}(x))\Big\}\\
    \le 7\cdot 2^{4L+16}\Big(1+\frac{1}{C_b}\Big)
\left({M^{(L)}_1(x)}\right)^{\frac{L+1}{2}} \Big(\sum_{j=1}^{L}\frac{1}{n_j}\Big)^{1/2} \Big(1+(2C_W\|\sigma\|_{Lip}^2)^{L}\Big)^{L+1}
      L^{3L+11}\cdot\\
\cdot\Big( 4\sqrt{5\cdot 2^{L} }+ K\frac{5\cdot 2^{L+1}}{\log(5\cdot 2^{L+1})}\Big)^{2L+10}
\Bigg(8+ \Big( 6K\frac{5\cdot 2^{L+1}}{\log(5\cdot 2^{L+1})}  \sqrt{C_W} \|\sigma\|_{\mathrm{Lip}}\cdot\\
        \cdot  \mathbb{E}[(W_{1,1}^{(1)})^{5\cdot 2^{L+1}}]^{1/(5\cdot 2^{L+1})} \Big)^{L}\Bigg)^{\frac{L}{2}+3} ,
\end{multline*}
where $K>0$ is a constant defined in Lemma \ref{mom_p_iid0} and the quantity $M^{(L)}_1(x)$  depends on the parameters \[
\left\{C_W,C_b,\|\sigma\|_{\text{Lip}},|\sigma(0)|, \|x\|, n_0, \mathbb{E}[(W_{1,1}^{(1)})^{5\cdot 2^{L+1}}]^{\frac{1}{2^{L}}}\right\}
\]
and it is defined in Lemma \ref{cost_schif}.
\end{theorem}
\begin{remark}
    Theorem \ref{finale_uno_1} is stated considering only the first component of the neural network and of the corresponding limit. However, in the case of the Wasserstein distance, it can be extended to the full vectors
    \[
    z^{(L+1)}(x) = \left(z_1^{(L+1)}(x), \dots, z_{n_{L+1}}^{(L+1)}(x)\right) \quad \text{and} \quad G^{(L+1)}(x) = \left(G_1^{(L+1)}(x), \dots, G_{n_{L+1}}^{(L+1)}(x)\right).
    \]
    Indeed, by Definition \ref{d_w_def},
    \[
    W_1\left(z^{(L+1)}(x), G^{(L+1)}(x)\right) \le \sum_{i=1}^{n_{L+1}} W_1\left(z_i^{(L+1)}(x), G_i^{(L+1)}(x)\right),
    \]
    and one can then apply the bound from Theorem \ref{finale_uno_1} to each term in the sum to obtain a more general bound.
Clearly, the same argument does not apply in the case of the Kolmogorov distance, as it is only defined for real-valued random variables, but one could consider the multi-dimensional Kolmogorov distance defined in Definition \ref{def_mK}.
   In particular, assuming $z^{(L+1)}(x)$ independent of $G^{(L+1)}(x)$ and observing that the components of $z^{(L+1)}(x)$ are independent conditional on the $\sigma$-field generated by the weights and biases up to layer $L$, denoted by $\mathcal{F}_L$, we can easily deduce that
\[
    d_{mK}\bigl(z^{(L+1)}(x),G^{(L+1)}(x)\bigr)
    \le \sum_{i=1}^{n_{L+1}}
    \mathbb{E}\Bigl[
        d_K\bigl(z^{(L+1)}_i(x),G^{(L+1)}_i(x)\bigr)
        \,\big|\,\mathcal{F}_L
    \Bigr].
\]
Since the proof of Theorem~\ref{Kdist} provides an upper bound on 
\(
\mathbb{E}\bigl[d_K(z^{(L+1)}_i(x),G^{(L+1)}_i(x))\mid\mathcal{F}_L\bigr],
\)
it follows directly that the same upper bound as in Theorem~\ref{finale_uno_1} for the Kolmogorov distance also applies to the multidimensional Kolmogorov distance, up to a multiplicative constant depending only on $n_{L+1}$.    
\end{remark}

{

\begin{remark}\label{rem_Kne0imp}
Assuming that the infinite-width variance defined in Theorem~\ref{Hanin} is positive, i.e.\ $K^{(L+1)}>0$, one can study the distance between the distributions of the normalized neural network
\[
\frac{z_{1}^{(L+1)}(x)}{\sqrt{K^{(L+1)}}}
\]
and its infinite-width Gaussian limit. In our case, since $C_b \neq 0$, it follows that $K^{(L+1)} \ge C_b > 0$. 
Observe that
\[
d_{K}\bigl(z_{1}^{(L+1)}(x),G_{1}^{(L+1)}(x)\bigr)
= d_K\left(\frac{z_{1}^{(L+1)}(x)}{\sqrt{K^{(L+1)}}},\frac{G_{1}^{(L+1)}(x)}{\sqrt{K^{(L+1)}}}\right),
\]
whereas for the Wasserstein distance we have
\[
W_1\bigl(z_{1}^{(L+1)}(x),G_{1}^{(L+1)}(x)\bigr)
= \sqrt{K^{(L+1)}}\,W_1\left(\frac{z_{1}^{(L+1)}(x)}{\sqrt{K^{(L+1)}}},\frac{G_{1}^{(L+1)}(x)}{\sqrt{K^{(L+1)}}}\right).
\]

Suppose now that both distances between the laws of the two normalized random variables converge to zero.  
In our setting, since $K^{(L+1)}$ does not converge to zero as $L \to \infty$, we obtain the same rate of convergence whether or not we normalize.  
On the other hand, if $K^{(L+1)} \to 0$, then the rate of convergence of the Wasserstein distance between the normalized random variables would be slower than the rate of convergence between the non-normalized variables.

\end{remark}
}

\begin{remark}\label{rate_final_1d}
From the previous Theorem it follows that, under the same assumptions on weights and biases, if $n_1, \dots, n_L \to \infty$, then $z_1^{(L+1)}(x)$ converges in distribution to $G_1^{(L+1)}(x)$ both in the Kolmogorov and Wasserstein distances, with a rate of convergence given by $\sqrt{\sum_{j=1}^{L} \frac{1}{n_j}}$. This rate was already established in \cite{Torr23,FHMNP,BT24,BFF24} for the case of Gaussian weights and later improved to the rate $\sum_{j=1}^{L} \frac{1}{n_j}$ in \cite{FHMNP,Trev,CP25}.
\end{remark}

The bound provided by the previous Theorem can also be used to study the distance between the law of the neural network and that of the infinite-widths Gaussian limit when both the depth and the widths tend to infinity. 
\begin{cor}\label{rem_L_inf_1}
    Assume that $n_1, \dots, n_L \asymp n$, and that there exist constants $C_1, C_2 \ge 0$, independent of $L$, such that
\begin{equation}\label{cond_W_conv}
\mathbb{E}[(W_{1,1}^{(1)})^{5 \cdot 2^{L+1}}]^{\frac{1}{5\cdot 2^{L+1}}} \le C_1 \cdot 2^{C_2 L}.
\end{equation}
Then, there exists a constant $R$ independent of $L$ and $n$ such that
\begin{eqnarray}
&& \max\left\{d_K(z_1^{(L+1)}(x), G_1^{(L+1)}(x)), W_1(z_1^{(L+1)}(x), G_1^{(L+1)}(x))\right\}\notag \\
&& =O\left( \sqrt{\frac{L}{n}} \cdot 2^{R L^2\log_2\left(\frac{5\cdot 2^{L+1}}{\log_2(5\cdot 2^{L+1})}\right)}\right).\label{cond_l_n_gen1}
\end{eqnarray}
For instance, if $L \le \left(\left(\frac{1}{2}-\varepsilon \right)\log_2 n\right)^{1/3}$ for some $\varepsilon \in \left(0,\frac{1}{2}\right)$, then as $L,n \to \infty$ one obtains
\begin{equation*}
\max\left\{d_K(z_1^{(L+1)}(x), G_1^{(L+1)}(x)), W_1(z_1^{(L+1)}(x), G_1^{(L+1)}(x))\right\} = 
O\left(\frac{1}{n^{\varepsilon}}\right).
\end{equation*}
\end{cor}
{
Condition~\eqref{cond_W_conv} is satisfied by a broad class of distributions, including sub-Gaussian distributions (which in particular cover all bounded random variables) and sub-exponential distributions (see \cite{W19}).
}

\begin{remark}
{Section~\ref{Gaus_example} presents examples in which the Kolmogorov and Wasserstein distances between the neural network and its infinite-width Gaussian limit are analyzed for different activation functions—specifically, the perceptron, identity, and ReLU functions—under Gaussian initialization.}
The results are derived by following the proof strategy of Theorem~\ref{finale_uno_1}, with optimizations applied to improve the bounds on the two distances whenever possible. Assuming the critical initialization (see Section \ref{subsec_crit}), we show that both distances decay at the rate
\begin{equation}\label{d_Gauss_lim_1d}
O\left(\sqrt{\sum_{\ell=1}^{L} \frac{1}{n_\ell}}\right)
\end{equation}
in the cases of the identity and ReLU functions. For the perceptron function, a sharper rate of
\begin{equation}\label{d_Gauss_lim_1d_perc}
O\left(\frac{1}{\sqrt{n_L}}\right)
\end{equation}
is obtained, even for arbitrary $C_b >0$ and $C_W \ne 0$.
We also note that the general bound in~\eqref{cond_l_n_gen1} yields a slower rate of convergence compared to what is observed in these specific examples under Gaussian weights. This discrepancy is primarily due to the use of Lemma~\ref{mom_p_iid0}, which is suboptimal in several settings—particularly when the weights are Gaussian random variables.
{ Furthermore, in \eqref{d_Gauss_lim_1d} and \eqref{d_Gauss_lim_1d_perc} the bound scales with the square root of the width, whereas in \cite{FHMNP} the authors obtain bounds of order $1/\text{width}$ for the Total Variation distance (and hence also for the Kolmogorov distance) and for the Wasserstein distance.
In this work, the suboptimal dependence on the network's width is compensated by the fact that the constants underlying \eqref{d_Gauss_lim_1d} and \eqref{d_Gauss_lim_1d_perc} can be made completely explicit in terms of the hyperparameters of the network (see, again, Section \ref{Gaus_example}).  
This choice enables a rigorous analysis of the joint asymptotic regime $L, n_1, \dots, n_L \to \infty$, but comes at the expense of the optimal convergence rates.

}

\end{remark}

\begin{remark}
When the activation function is both Lipschitz and bounded, one can derive an upper bound on the Kolmogorov and Wasserstein distances between the neural network's output and its Gaussian limit, depending on the Wasserstein distance between the output of the previous layer (conditioned on an appropriate $\sigma$-field) and its corresponding Gaussian approximation. This recursive structure allows for an improvement over the general bound appearing in Theorem~\ref{finale_uno_1}.
More precisely, as shown in Section~\ref{lip_bound_s_sec}, in this case the Kolmogorov and Wasserstein distances decay at the rate
\[
O\left(\sqrt{\sum_{\ell=1}^{L} \frac{1}{n_\ell} \left( \frac{6C_W}{\pi C_b} \right)^{L - \ell}}\right),
\]
which represents a significantly faster convergence compared to the one featured in Theorem~\ref{finale_uno_1}.
For a more detailed discussion of these improvements, see Remark~\ref{rem_impr_lip_bound_prop}.
\end{remark}

For {\bf Problem 2}, we establish the following result:

\begin{theorem}\label{pres_sec_prob}
Let \( L \geq 1 \) be an integer, and assume \( C_b,C_W > 0 \). Let the neural network output \( z_1^{(L+1)} \), the activation function \( \sigma \), and the weight distribution \( W \) be defined as in Definition~\ref{def_NN}. Recall the Gaussian limit \( G_1^{(L+1)} \) from Theorem~\ref{Hanin}.

Suppose that \( \mathbb{E}[W^{5 \cdot 2^{L+2}}] < \infty \), and let \( d \geq 2 \) be an integer. Consider a set of distinct input points \( \mathcal{X} := \{x^{(1)}, \dots, x^{(d)}\} \subset \mathbb{R}^{n_0} \), and define
\[
z_1^{(L+1)}(\mathcal{X}) := \big(z_1^{(L+1)}(x^{(1)}), \dots, z_1^{(L+1)}(x^{(d)})\big) \in \mathbb{R}^{d \times n_{L+1}},
\]
and similarly define the Gaussian limit
\[
G_1^{(L+1)}(\mathcal{X}) := \big(G_1^{(L+1)}(x^{(1)}), \dots, G_1^{(L+1)}(x^{(d)})\big) \in \mathbb{R}^{d \times n_{L+1}}.
\]

If the covariance matrix \( K^{(2)} \) from Theorem~\ref{Hanin} satisfies \( \operatorname{det}(K^{(2)}) \neq 0 \) and for every $\ell=2,\dots,L$ and $i=1,\dots d-1$ we have that $\mathbb{E}\big[\sigma'(G^{(\ell)}_1(x^{(i)}))\big]\ne 0$, then

\begin{multline*}
d_C(z_1^{(L+1)}(\mathcal{X}),G_1^{(L+1)}(\mathcal{X}))\le 
 Cd^{d+6}(L + 1)^{3L+d+20}
\Big(\sum_{j=1}^{L}\frac{1}{n_j}\Big)^{1/2}\left(1+C_W+\frac{1}{C_W}\right)^{4d+dL-4}\cdot\\
\cdot (1+C_b)^{2d-2}(1+|\sigma(0)|)^4\Bigg(9d+\frac{\sum_{j=1}^d\|x^{(j)}\|^4}{n_0^2}\Bigg)^{d-1}\Big(2+(2C_W\|\sigma\|_{Lip}^2)^{2L}\Big)^{L+d}\cdot\\
\cdot \Bigg(2+ \Big( K\frac{5\cdot 2^{L+1}}{\log(5\cdot 2^{L+1})}  \sqrt{C_W} \|\sigma\|_{\mathrm{Lip}} \mathbb{E}[(W_{1,1}^{(1)})^{5\cdot 2^{L+1}}]^{1/(5\cdot 2^{L+1})} \Big)^{L - 1}\Bigg)^{\frac{L}{2}+24}\cdot\\
\cdot
 \Bigg(\sum_{i=1}^d\Big(1+\left(M^{(L)}_2(x^{(i)})\right)^{\frac{L}{2}+2}\Big)\Bigg)
  2^{4L}   \Big(
 4\sqrt{5\cdot 2^L }
+ K\frac{5\cdot 2^{L+1}}{\log(5\cdot 2^{L+1})}
\Big)^{2L+8}\cdot\\
\cdot \left(1+d!\lambda(K^{(2)})^{-2d}\sum_{\ell=2}^{L+1}\prod_{i=1}^d\left(\sum_{\substack{k_i=1\\k_i\ne k_s\, \forall s=1,\dots,i-1}}^d\prod_{r_i=1}^{\ell-2}\mathbb{E}\big[\sigma'(G^{(\ell-r_i)}_1(x^{(k_i)}))\big]^2\right)^{-1}\right),
\end{multline*}
where $C>0$ is a universal constant that does not depend on any of the neural network parameters, $K>0$ is a constant defined in Lemma \ref{mom_p_iid0}, and for every $j=1,\dots,d$, $M_2^{(L)}(x^{(j)})$ is defined as in Lemma \ref{cost_schif} and it depends  on
\[
\left\{C_W,k,|\sigma(0)|,\|\sigma\|_{\text{Lip}},C_b,n_0,\mathbb{E}[W^{5\cdot 2^{L+2}}]^{\frac{1}{2^L}},\mathcal{X}\right\}.
\]

\end{theorem}
\begin{remark}
    As in the one-dimensional case for the Wasserstein distance, the upper bound on the convex distance in Theorem~\ref{pres_sec_prob} can also be reformulated for the vector
    \begin{equation}\label{vec_z_chi}
        z^{(L+1)}(\mathcal{X}) = \left(z_1^{(L+1)}(x_1), \dots, z_1^{(L+1)}(x_d), \dots, z_{n_{L+1}}^{(L+1)}(x_1), \dots, z_{n_{L+1}}^{(L+1)}(x_d)\right),
    \end{equation}
    and its corresponding Gaussian limit.
Indeed, the proof of the theorem can be repeated for the vector \eqref{vec_z_chi}, noting that, according to Theorem~\ref{Hanin}, the covariance of the Gaussian limit has a block-diagonal structure, where all blocks are equal to \( K^{(L+1)} \) (also defined in Theorem~\ref{Hanin}). The resulting upper bound coincides with that of the one-dimensional output case, up to a multiplicative factor equal to the output dimension $n_{L+1}$.

\end{remark}

\begin{remark}
    The assumption that $C_b\ne 0$ could be dropped if one develops a lower bound for $Q^{(\ell)}_4$ (defined in \eqref{Q_k}) using the same strategy we used to give a lower bound to the determinant of the limiting covariance matrix $K^{(\ell)}$ (see Section \ref{sec_low_bound_detK}). {
   The same reasoning can then be applied to the one-dimensional problem, but in this case the assumption $C_b \neq 0$ must be replaced by the condition that for every $\ell = 2, \dots, L$ we have
\[
\mathbb{E}\bigl[\sigma'(G^{(\ell)}_1(x))\bigr] \neq 0.
\]

    }
\end{remark}
As in the one-dimensional case, thanks to Theorem \ref{pres_sec_prob}, it is possible to study the distance between the law of the neural network and the law of its Gaussian counterpart as both \( L \) and \( n_1, \dots, n_L \cong n \) diverge to \( \infty \).

\begin{cor}\label{rem_L_inf_d}

Let the assumptions of Proposition~\ref{fin_relu} hold and assume that there exist constants \( C_1, C_2, C_3, C_4\ge 0 \), independent of \( L \), such that
\begin{equation}\label{cond_W_conv_d}
\mathbb{E}\left[(W_{1,1}^{(1)})^{5 \cdot 2^{L+2}}\right]^{\frac{1}{2^{L+1}}} \le C_1 e^{C_2 L}.
\end{equation}
and for every $\ell=2,\dots, L-2$ and $i=1,\dots,d$
\begin{equation}\label{cond_der_s_inG}
    \mathbb{E}\big[\sigma'(G^{(\ell)}_1(x^{(i)}))\big]\ge C_42^{-C_3L}.
\end{equation}
Under these assumptions, there exists a constant \( S > 0 \), independent of \( L, n_1, \dots, n_L, n \), such that
\[ 
d_C\left(z_{1}^{(L+1)}(\mathcal{X}), G_{1}^{(L+1)}(\mathcal{X})\right) =
O\left( \sqrt{\frac{L}{n}} \cdot 2^{S L^2\log_2\left(\frac{5\cdot 2^{L+1}}{\log_2(5\cdot 2^{L+1})}\right)}\right).
\]
It follows that, if
\[
L \le \left(\left( \frac{1}{2} - \varepsilon\right) \log_2 n \right)^{1/3}
\quad \text{for some } \varepsilon \in \left(0,\frac{1}{2}\right),
\]
then as \( n \to \infty \),
\begin{equation*}
d_C\left(z_{1}^{(L+1)}(\mathcal{X}), G_{1}^{(L+1)}(\mathcal{X})\right) = O\left(  \frac{1 }{ n^{\varepsilon} }  \right).
\end{equation*}

\end{cor}
Observe that this result yields the same rate of convergence as in the one-dimensional case (see Remark~\ref{rate_final_1d}). However, in the multi-dimensional setting, the lower bound on the minimum eigenvalue of the limiting covariance matrix tends to zero when \( C_3 \neq 0 \), and consequently the limit in law of the neural network (if it exists) as \( L, n \to \infty \) may be degenerate.
As previously noted, condition~\eqref{cond_W_conv_d} is satisfied by a broad class of distributions, including sub-Gaussian and sub-exponential ones (see~\cite{W19}). In contrast, condition~\eqref{cond_der_s_inG} is generally more difficult to verify; nonetheless, it is satisfied by both the linear activation function and the ReLU activation.
\begin{remark}\label{rem_d_c_ok_l0}
{
  Thanks to Proposition C.3.1 in \cite{NP12}, Theorem \ref{pres_sec_prob} implies the convergence in law of the neural network to a Gaussian distribution when the upper bound on the convex distance tends to zero as the inner widths $n_1,\dots,n_L$ diverge to infinite. However, it does not provide information about the Gaussian limit in the regime \( L \to \infty \). Nevertheless, the problem remains well-posed even when the limit for $L\to\infty$ of the covariance matrix $K^{(L)}$ is not invertible. Indeed the convex distance is invariant under normalization and Remark \ref{rem_Kne0imp} still applies.
}
   
\end{remark}

\begin{remark}
   In Section \ref{id:mult}, the computations are optimized for the case of a linear activation function ($\sigma(x) = x$ for all $x \in \mathbb{R}$) and Gaussian weight initialization. The result shows that, when $C_b = 0$ and $C_W = 1$, as $L,n_1,\dots,n_L\to \infty$
\begin{equation}\label{d_c_Gauss_lim_md}
d_C\left(z_1^{(L+1)}(\mathcal{X}), G_1^{(L+1)}(\mathcal{X})\right) = O\left(\sqrt{\sum_{i=1}^{L} \frac{1}{n_i}}\right).
\end{equation}
This bound is significantly tighter than the one obtained from Theorem~\ref{pres_sec_prob} (in the case $C_b \ne 0$), illustrating that the generality of that theorem comes at the cost of optimality. As in the one-dimensional case, this sub-optimality originates from the application of Lemma~\ref{mom_p_iid0} and Theorem~\ref{bound_K_gauss}, which in turn rely on Proposition~\ref{Caco}—a result that is not optimal for every choice of activation function. 
{ Furthermore, we note that in \eqref{d_c_Gauss_lim_md} the bound exhibits a square-root dependence, whereas in \cite{CP25} the authors obtain a bound on the Total Variation distance (and consequently on the Convex distance, by Remark~\ref{d_c_less_d_TV}) of order $1/\text{width}$ rather than its square root. As observed before in the one-dimensional case, the key difference here is that the constants in \eqref{d_c_Gauss_lim_md} are explicit, which allows for a rigorous analysis of the regime where $L, n_1, \dots, n_L \to \infty$.
}

\end{remark}

\section{Literature review}\label{literature_review_sec}
This paper investigates the discrepancy in distribution between the finite-dimensional output of a deep neural network and its Gaussian limit as identified in Theorem~\ref{Hanin}. We consider networks initialized with independent and identically distributed (i.i.d.) weights having sufficiently many finite moments, and with independent Gaussian biases.

While several works have addressed similar questions under the assumption of Gaussian-distributed weights—typically focusing on the role of the inner layer widths (see, e.g., \cite{BT24,FHMNP,Torr23,Trev,CP25,BFF24,CC23})—our approach allows for more general weight distributions. 

From Theorem~\ref{finale_uno_1}, we deduce that for any input $x \in \mathbb{R}^{n_0}$, there exists a constant $C > 0$, independent of the widths $n_1, \dots, n_L$, such that
\[
\max\left\{d_K\left(z_{1}^{(L+1)}(x), G_{1}^{(L+1)}(x)\right),\ 
           W_1\left(z_{1}^{(L+1)}(x), G_{1}^{(L+1)}(x)\right)\right\}
\le C\left(\sum_{j=1}^{L} \frac{1}{n_j}\right)^{1/2},
\]
where \( d_K \) and \( W_1 \) denote the Kolmogorov and Wasserstein distances, respectively defined in Definitions \ref{def_mK} and \ref{d_w_def}.

In the multidimensional case, Theorem~\ref{pres_sec_prob} guarantees that for any finite set of distinct inputs \( \mathcal{X} := \{x^{(1)}, \dots, x^{(d)}\} \subseteq \mathbb{R}^{n_0} \), there exists a constant \( D > 0 \), also independent of \( n_1, \dots, n_L \), such that
\[
d_C\left(z_1^{(L+1)}(\mathcal{X}), G_1^{(L+1)}(\mathcal{X})\right)
\le D\left(\sum_{i=1}^{L} \frac{1}{n_i}\right)^{1/2},
\]
where \( d_C \) denotes the convex distance as defined in Definition \ref{d_c_def}.

Notably, our results also apply when the weights are Gaussian, enabling a direct comparison with the bounds established in the aforementioned literature.

\subsection{One-dimensional case (Gaussian initialization)}
{
\begin{itemize}

\item In \cite{BFF24}, the authors consider a shallow neural network (i.e., with a single hidden layer, $L = 1$) and a univariate output at initialization, with Gaussian-distributed weights. They focus on a single input \( x \in \mathbb{R}^{n_0} \) and study the Kolmogorov, total variation, and Wasserstein distances (see respectively Definitions \ref{def_mK}, \ref{d_tv_def}, \ref{d_w_def}) between the network output and its Gaussian limit. They derive explicit bounds of order \( \frac{1}{\sqrt{n_1}} \), where the constants depend on the limiting covariance \( K^{(2)} \) (assumed to be non-zero), and on the growth properties of the activation function. In particular, the activation function and its first two derivatives are assumed to be polynomially bounded.

Our results recover the same rate of convergence, \( \frac{1}{\sqrt{n_1}} \), for both the Kolmogorov and Wasserstein distances, thus matching the bounds obtained in \cite{BFF24} in this setting.

It is worth noting that the approach in \cite{BFF24} heavily exploits the Gaussianity of the weights in order to apply a second-order Poincaré inequality from \cite{Vid}, which plays a central role in their analysis.

\medskip

\item 
In \cite{Torr23} the authors assumed \( L = 1 \), \( n_2 = 1 \), and Gaussian initialization, and under the assumption that \( \mathbb{E}[\sigma^2(z^{(1)}(x))] < \infty \) for a fixed input \( x \in \mathbb{R}^{n_0} \), they provide upper bounds for the Kolmogorov, total variation, and Wasserstein distances of order \( \frac{1}{\sqrt{n_1}} \). The constants involved are explicit and depend on \( \operatorname{Var}(\sigma(z^{(1)}(x))) \) and on  \( K^{(2)} \ne 0 \).

To analyze the convergence in law for a multidimensional output, the authors consider the convex and Wasserstein distances. Under general assumptions on the activation function, they establish a convergence rate of order
\[
O\left(\sqrt{\sum_{\ell=1}^{L} \frac{c_\ell}{n_\ell}}\right),
\]
where the constants \( c_\ell \) depend on \( \mathbb{E}[\sigma^4(z^{(\ell)}(x))] \). Therefore, the bound is not fully explicit, for example in the depth parameter \( L \), in contrast to our result in Theorem~\ref{finale_uno_1}.

The approach of the authors of \cite{Torr23} relies heavily on Stein's method and Gaussian integration by parts, which is feasible thanks to the Gaussian initialization. A crucial assumption in their analysis is the following: for any \( a_1, a_2 \ge 0 \) and constants \( C_b, C_W > 0 \), there exists a polynomial \( P \) with non-negative coefficients (depending only on \( \sigma, C_b, C_W \)), and with degree independent of \( \sigma, a_1, a_2, C_b, C_W \) such that for all \( x \in \mathbb{R} \),
\begin{equation}\label{hyp_T}
\left|\sigma^2\left(x\sqrt{C_b + C_W a_2}\right) - \sigma^2\left(x\sqrt{C_b + C_W a_1}\right)\right| 
\le P(|x|)\, |a_2 - a_1|.
\end{equation}
The authors verify that this condition holds for various activation functions, including the perceptron and Lipschitz continuous functions.

However, in our setting with general (non-Gaussian) weights, assumption \eqref{hyp_T} cannot be applied directly. To address this, we adapt Stein’s method to a new setting involving the square of a Lipschitz function. Specifically, for each \( \ell = 1, \dots, L \), we analyze the quantity
\[
\left| \mathbb{E}[\sigma^2(z_1^{(\ell)}(x)) \mid \mathcal{F}_{\ell-1}] 
- \mathbb{E}[\sigma^2(G_1^{(\ell)}(x))] \right|,
\]
where \( \mathcal{F}_{\ell-1} \) is the \( \sigma \)-field generated by the weights and biases up to layer \( \ell-1 \). 

In the Gaussian case, \( z_1^{(\ell)}(x) \) is conditionally Gaussian given \( \mathcal{F}_{\ell-1} \), and for \( N \sim \mathcal{N}(0,1) \), we have the identity in law
\[
\mathbb{E}[\sigma^2(z_1^{(\ell)}(x)) \mid \mathcal{F}_{\ell-1}] 
= \mathbb{E}\left[\sigma^2\left(N\sqrt{C_b + C_W \sigma^2(z_1^{(\ell-1)}(x))}\right) 
\mid \mathcal{F}_{\ell-1}\right],
\]
which makes assumption \eqref{hyp_T} easy to apply in that context.

Another key difference in our proof strategy is the replacement of Gaussian integration by parts (unavailable in the non-Gaussian setting) with a discrete version introduced in \cite{C08} and later used in \cite{KG22}. By combining this discretization with Stein’s method, we obtain the same main term as in the Gaussian setting, along with an additional remainder term. This remainder is tractable thanks to Taylor’s theorem and the regularity properties of the solution to Stein’s equation.

\medskip

\item In \cite{FHMNP}, it is shown that the true optimal rate of convergence of the neural network's law to that of its Gaussian counterpart is of order $1/n$, although our results recover the same rate as those in \cite{BFF24} and \cite{Torr23} when the inner widths satisfy $n_1, \dots, n_L \asymp n \to \infty$.
 In particular, the authors of \cite{FHMNP} establish both lower and upper bounds of this order for the Wasserstein and total variation distances in the case of deep neural networks (\( L \ge 1 \)), under the assumptions of a polynomially bounded activation function, evaluation at a fixed input, and a non-degeneracy condition on the limiting covariance matrices \( K^{(\ell)} \), which are required to be invertible for every \( \ell = 1, \dots, L+1 \).

The main ingredients in their proof are bounds from \cite{Han_Gas}, Stein’s method, and a refinement based on Lusin's Theorem, which allows for a tighter approximation and ultimately leads to the improved convergence rate. 

Moreover, \cite{FHMNP} actually provide a more general result: they study the convergence in law of the gradients of the neural network with respect to the input, comparing them to the corresponding Gaussian limit. They obtain upper bounds of the optimal order \( \frac{1}{n} \), without providing explicit constants. In contrast, although our results are suboptimal in the case of Gaussian initialization, they offer explicit dependence on the network’s hyper-parameters, including \( C_b, C_W, L, \sigma, x \), and the widths \( n_1,\dots,n_L \). This makes our bounds specifically suitable for understanding the behavior of the neural network output in the combined infinite width/depth limit.


\end{itemize}
}
\subsection{Multidimensional Case (Gaussian initialization)}
{
\begin{itemize}
\item The authors in \cite{BFF24} extend their results to the multidimensional setting, still under the assumption \( L = 1 \). In this case, they obtain a bound of order \( \frac{1}{\sqrt{n_1}} \), with an explicit constant depending on the minimum and maximum eigenvalues of the limiting covariance matrix \( K^{(2)} \), which is assumed to be invertible.

In Theorem~\ref{pres_sec_prob}, we consider a more general setting with \( L \ge 1 \) and multiple distinct inputs \( \mathcal{X} := \{x^{(1)}, \dots, x^{(d)}\} \subseteq \mathbb{R}^{n_0} \). In order to obtain a lower bound on the minimum eigenvalue of \( K^{(\ell)} \) for \( \ell \ge 3 \), we assume that
\[
\mathbb{E}\big[\sigma'(G^{(\ell)}_1(x^{(i)}))\big] \ne 0 \quad \text{for every } \ell = 2,\dots,L \text{ and } i = 1,\dots,d-1.
\]
This condition is not needed in the case \( L = 1 \), as the relevant term involving \( \mathbb{E}[\sigma'(G^{(\ell)}_1(x^{(i)}))] \) does not appear.

We thus conclude that, in the case of Gaussian weights, our result in Theorem~\ref{pres_sec_prob} is fully analogous to Theorem 3.3 in \cite{BFF24}, with the added benefit of not requiring any additional assumptions beyond those already stated.

\medskip

\item When considering the output of a neural network evaluated at multiple inputs, the authors of \cite{FHMNP} were not able to extend the approximation technique used in the one-dimensional case. Nevertheless, by again applying Stein’s method in conjunction with the bounds from \cite{Han_Gas}, they obtained a convergence rate of order \( O\left(\frac{1}{\sqrt{n}}\right) \) for the convex distance between the gradients of the neural network and their Gaussian limit, in the regime where \( L \ge 1 \) and \( n_1, \dots, n_L \asymp n \to \infty \). This rate is suboptimal, and highlights the increased technical challenges that arise when using Stein's method in the multidimensional setting.

\medskip

\item Without requiring the invertibility of the limiting covariance matrices \( K^{(\ell)} \), \( \ell = 1, \dots, L+1 \), the authors of \cite{BT24} presented a bound on the \( 2 \)-Wasserstein distance (see Definition \ref{d_w_def}). They assumed a Gaussian initialization, a Lipschitz activation function, \( L \ge 1 \), and a set of distinct inputs \( \mathcal{X} := \{x^{(1)}, \dots, x^{(d)}\} \) with \( d > 1 \). 

Their approach relied primarily on inductive arguments and key properties of the \( 2 \)-Wasserstein distance. They showed that
\[
W_2\left(z^{(L+1)}(\mathcal{X}), G^{(L+1)}(\mathcal{X})\right) \le \sqrt{n_{L+1}} \sum_{\ell=1}^L \frac{C^{(\ell+1)} \|\sigma\|_{\mathrm{Lip}}^{L-\ell} \sqrt{\prod_{i=\ell+1}^L C_W^i}}{\sqrt{n_\ell}},
\]
thus demonstrating, for the first time, an exponential dependence on the depth \( L \ge 1 \). However, the bound is still not fully explicit in the hyper-parameters, as the constants \( C^{(\ell)} \) are not specified.

\medskip
\item An optimal bound for the \( p \)-Wasserstein distance (see Definition \ref{d_w_def}) was first obtained in \cite{Trev} by applying results from optimal transportation theory. The authors proved a bound of order \( \sum_{\ell=1}^L \frac{1}{n_\ell} \) (with non explicit constants) for a neural network—more general than the fully connected architecture—initialized with Gaussian weights, equipped with a Lipschitz activation function, with \( L \ge 1 \), and evaluated on multiple inputs. Their analysis was conducted under a non-degeneracy condition similar to the one in \cite{FHMNP} (in the case without derivatives).

Our results in Theorem~\ref{pres_sec_prob}, as previously discussed, exhibit a slower rate of convergence. However, we focused on the convex distance, which is not directly comparable to the \( 2 \)-Wasserstein distance.

\bigskip

\item A result that can be directly compared with Theorem~\ref{pres_sec_prob} is Theorem 5 in \cite{CP25}. There, the authors show that the total variation distance between the gradients of the neural network (with respect to the inputs) and the corresponding Gaussian limit converges to zero at a rate of \( \frac{1}{n} \). This result also implies that the same rate of convergence holds for the convex distance (see Remark \ref{d_c_less_d_TV}). Consequently, Theorem~\ref{pres_sec_prob} is suboptimal in the case of Gaussian initialization, even in the multi-input setting.

This sub-optimality reflects the trade-off for addressing a more general setting (i.e., not necessarily Gaussian initialization) and for providing explicit constants in terms of the network's hyper-parameters, valid for many activation functions.

The proof of Theorem 5 in \cite{CP25} relies on entropic bounds in combination with the estimates from \cite{Han_Gas}. We also observe that one could apply Theorem 12 from \cite{CP25} together with Lemma~\ref{somma_schifa}, Lemma~\ref{cost_schif}, Proposition~\ref{mom_z_prop}, and Remark~\ref{HS_bondK} to obtain bounds for both the Kolmogorov and the \( 2 \)-Wasserstein distances. These bounds would be explicit in the hyper-parameters and optimal in the regime where \( L < \infty \) is fixed and \( n_1,\dots,n_L \asymp n \to \infty \).

\end{itemize}
}
\subsection{Non-Gaussian initialization}
{
The limiting distribution of a wide fully connected neural network with i.i.d.\ weights and i.i.d.\ biases (independent from the weights) has been proven to be Gaussian when the biases are Gaussian and the weights have all finite moments, as shown in \cite{Han22} (see Theorem~\ref{Hanin}). Observe that Theorems~\ref{finale_uno_1} and~\ref{pres_sec_prob} do not require the weights to have all finite moments, thus slightly improving the qualitative result of Theorem~\ref{Hanin} as well (using the fact that convergence in Kolmogorov, Wasserstein, and convex distance implies convergence in distribution, by Proposition~C.3.1 in \cite{NP12}).   
}

The situation is different when the weights and biases lack finite moments—for example, in the case of i.i.d.\ symmetric and centered $\alpha$-stable initializations: under a suitable scaling of the weights, the neural network converges to an $\alpha$-stable process (see \cite{PFF20,FFP23,BFF23,FFP22Kernel}).

In both settings, there are few works studying quantitative versions of these convergence theorems. For $\alpha$-stable initialization, the only result we are aware of is Theorem~11 in \cite{FFP23}, which provides rates for the sup-norm convergence under both ``joint growth'' and ``sequential growth'' of the inner layer widths.

As for the Gaussian limit case (i.i.d.\ weights with finite moments and Gaussian biases), the only quantitative central limit theorems we know are found in \cite{EMS21,Klu22,BGRS23,BR25}. In particular \cite{EMS21,Klu22,BGRS23} investigate the functional distance between a fully connected neural network and its Gaussian limit while instead \cite{BR25} the authors study the finite dimensional distributions of the neural network. Moreover, in \cite{EMS21,Klu22}, the authors focus on shallow networks (i.e., $L=1$) with zero bias, while in \cite{BGRS23,BR25}, the authors consider deep networks (i.e., $L \geq 1$).

{
\begin{itemize}
\item In \cite{EMS21}, the authors study the $p$-Wasserstein and $\infty$-Wasserstein distances, viewing the neural network as a random variable in $L^p(\mathbb{S}^{n_0-1})$, where 
\[
\mathbb{S}^{n_0-1} = \left\{x \in \mathbb{R}^{n_0} : \|x\| = 1\right\}.
\]
The network is defined with symmetric, sub-Gaussian random variables as weights in the first layer, and Bernoulli random variables taking values $1$ and $-1$ with equal probability as weights in the last layer. The authors consider various activation functions, including polynomials, ReLU, tanh, and more generally, functions that are square-integrable with respect to the standard Gaussian measure. The best convergence rate with respect to the inner width is of order $\frac{1}{n_1^{1/3}}$ for the $\infty$-Wasserstein distance.

Therefore, Theorems~\ref{finale_uno_1} and~\ref{pres_sec_prob} improve upon the results of \cite{EMS21}, achieving a convergence rate of order $\frac{1}{\sqrt{n}}$, though only for the finite-dimensional distributions of the neural network and under the assumption of Lipschitz activation functions.

\item A better bound for the $2$-Wasserstein distance, of order $\frac{1}{\sqrt{n_1}}$, is obtained in \cite{Klu22}, where it is also shown that, for general weight distributions, this rate is asymptotically sharp in $n_1$. In this setting, the neural network is initialized with i.i.d.\ weights uniformly distributed on the sphere of radius $n_0$ in the first layer, and i.i.d.\ weights with finite second moment in the second layer. The activation function is assumed to be a polynomial with zero mean under the uniform distribution on the sphere of radius $n_0$. The author also studies the case of ReLU activation, assuming additionally a Bernoulli distribution for the weights in the last layer. However, in this case the convergence rate is worse and depends explicitly on the input dimension $n_0$.

Hence Theorems~\ref{finale_uno_1} and~\ref{pres_sec_prob} extend the convergence rate obtained in \cite{Klu22} to the setting of deep fully connected neural networks ($L \ge 1$) but again only for the finite-dimensional distributions of the neural network and under the assumption of Lipschitz activation functions.


\item In the non-shallow case, the authors of \cite{BGRS23} consider general weight distributions, where the first-layer weight matrix $W^{(1)} := (W_{i,j}^{(1)})_{i,j}$ has i.i.d.\ rows, while the weight matrices in the subsequent layers, $W^{(\ell)} := (W_{i,j}^{(\ell)})_{i,j}$ for $\ell = 2, \dots, L+1$, have i.i.d.\ entries. Their moment assumption requires the existence of $p > n_0$, where $n_0$ is the input dimension, such that
\[
\mathbb{E}\left[\left(W_{i,j}^{(\ell)}\right)^{2p}\right] \le \left(\frac{C_W^{(\ell)}}{n_{\ell-1}}\right)^p \left(B^{(\ell)}\right)^{p/2},
\]
where $C_W^{(\ell)}$ and $B^{(\ell)}$ are constants. In contrast, our results do not rely on this assumption, since our definition of the neural network (see Definition~\ref{def_NN}) already includes a normalization by $\sqrt{C_W / n_{\ell-1}}$.

Assuming a Lipschitz activation function, the authors of \cite{BGRS23} establish an upper bound on the Wasserstein distance between the neural network and its Gaussian limit, both viewed as random fields on the sphere. However, the bound they obtain depends on the ratios between the widths of different layers, and as a result, their bound does not guarantee convergence in distribution when all hidden-layer widths diverge to infinity, as considered in \cite{Han22}. They provide an improved estimate in the case where the activation function has three bounded derivatives, yet the convergence remains suboptimal in the regime where all widths tend to infinity.

By contrast, Theorem~\ref{pres_sec_prob} in the present work provides a quantitative description of the infinite-width limit. Its main limitation, however, is that it applies to finite-dimensional distributions rather than to the full functional setting.

\item The recent paper \cite{BR25} investigates the Wasserstein distance (see Definition \ref{d_w_def} and Proposition \ref{form_dual_wass}) between the finite-dimensional distributions of a fully connected neural network and the corresponding Gaussian limit as described in Theorem \ref{Hanin}. The authors assume i.i.d.\ Gaussian biases with zero mean and i.i.d.\ weights with zero mean, under the following moment conditions: there exist constants $c_3^{(L)}$ and $c_{2p}^{(\ell)} \ge 1$ for $\ell = 1, \dots, L-1$ such that
\[
\mathbb{E}\left[\left(W_{i,j}^{(\ell)}\right)^{2p}\right] \le \frac{c_{2p}^{(\ell)}}{n_\ell^p}
\quad \text{and} \quad
\mathbb{E}\left[\left|W_{i,j}^{(L)}\right|^3\right] \le \frac{c_3^{(L)}}{n_L^{3/2}}.
\]
Under these assumptions, the authors prove the following bound:
\begin{equation}\label{W_1_bound_BR25}
W_1\left(z^{(L+1)}(\mathcal{X}), G^{(L+1)}(\mathcal{X})\right)
\le C n_{L+1}^{1/3} \sum_{\ell=1}^{L} n_\ell^{-\frac{1}{6}\left(\frac{p-2}{3(2p-1)}\right)^{L-\ell}},
\end{equation}
where $C>0$ is a constant that does not depend on the inner widths $n_1,\dots,n_L$ and grows exponentially in the depth $L$.
As a consequence, when $n_1 \asymp \dots \asymp n_L \asymp n$, the convergence rate is worse than $1/\sqrt{n}$, in contrast to the results obtained in Theorems \ref{finale_uno_1} and \ref{pres_sec_prob}. We note, however, that \cite{BR25} does not assume invertibility of the limiting covariance matrix $K^{(\ell)}$ for every $\ell = 1, \dots, L+1$, and that their analysis is based on the Wasserstein distance, whereas in our multidimensional setting we consider the convex distance.
\end{itemize}
}
\subsection{Infinite widths and depth}\label{subsec_crit}

The behavior of fully connected feedforward neural networks at random initialization has been largely studied in the infinite-width regime, mostly under Gaussian initialization. On the contrary, the joint limit where both depth and width tend to infinity remains largely open. The main contributions are given by \cite{Han_Gas, Hanin2022CorrelationFI, HN18, BLR24, LNR23,BR25} and by  \cite{BGIPR25, HZ23,HZ24} on the study of the law of the predictive posterior.

{
\begin{itemize}
\item In \cite{Hanin2022CorrelationFI}, the author studies the cumulants of a fully connected neural network and shows that, for a certain class of activation functions, the network exhibits non-Gaussian behavior as the \emph{effective depth} $\frac{L}{n}$ increases. Understanding the role of $\frac{L}{n}$ in the convergence in law of the neural network is closely related to the notion of \emph{critical initialization} (see also \cite{PLRSDG16, RPKGSD17, RYH22}). Indeed, when $L$ tends to infinity, the recurrence relation for the limiting covariance matrix described in Theorem~\ref{Hanin} may diverge. To prevent this, critical initialization refers to a specific choice of the parameters $C_b$ and $C_W$ (as defined in \eqref{NN_expr}) such that the triplet $(K^{(L+1)}_{i,i}, K^{(L+1)}_{i,j}, K^{(L+1)}_{j,j})$ converges to a fixed point $(K^*, K^*, K^*)$. 

For instance, for the ReLU activation function $\sigma(x) = x \mathbf{1}_{\{x \ge 0\}}$, the critical initialization is given by $C_b = 0$ and $C_W = 2$, while for the linear activation $\sigma(x) = x$, the critical initialization corresponds to $C_b = 0$ and $C_W = 1$.

\item  In \cite{HN18} it is studied the case when the activation function is linear, i.e. $\sigma(x) = x$. In this case the neural network \eqref{NN_expr} without biases reduces to the product of $L$ matrices with independent and identically distributed entries applied to an input vector. As $L, n_1, \dots, n_L \to \infty$, the limiting distribution depends on the value of the ratio $\frac{L}{n} := \alpha \ge 0$. In particular, for general weight distributions that are symmetric, atomless, and have all finite moments, the authors of \cite{HN18} show that the Kolmogorov distance
\[
d_K\left(\ln\left(\frac{n_0}{n_{L+1}}\|z^{(L+1)}(x)\|^2\right),\mathcal{N}_1\left(-\frac{1}{\beta},\beta\right)\right) = O\left(\left(\sum_{i=1}^{L+1}\frac{1}{n_i^2}\right)^{\frac{1}{5}}\right),
\]
where $x \in \mathbb{R}^{n_0}$ satisfies $\|x\| = 1$, and
\[
\beta := 2 \sum_{i=1}^{L+1} \frac{1}{n_i} + \frac{\mathbb{E}\left[\left(W_{1,1}^{(1)}\right)^4\right] - 3}{n_1} \|x\|^4.
\]
This result shows that, when $n_1 \asymp \dots \asymp n_{L+1} \asymp n$ and $\frac{L}{n} \to 0$ as $n, L \to \infty$, the logarithm of the $\ell_2$ norm of the neural network output applied to a unit-norm input converges in law to a Gaussian distribution.

\item In \cite{BLR24} the authors fully characterize the limiting behavior of a neural network as $\frac{L}{n} \to \alpha \ge 0$ when $L, n \to \infty$, considering arbitrary input sets and general output dimensions. In particular, they show that this limit is a Gaussian mixture (nontrivial when $\alpha > 0$) in the case of zero bias and i.i.d.\ weights of the form $W_{i,j}^{(\ell)} \sim \mathcal{N}\left(0,\frac{1}{C_W \lambda_\ell}\right)$ (in our setting, this corresponds to $\lambda_\ell = C_W^{-1}$ for every $\ell = 1, \dots, L+1$).

\begin{prop}[Proposition 3.1, \cite{BLR24}]\label{prop_id_B}
Assume $\alpha = 0$, $\min\{n_1, \dots, n_{L+1}\} \ge D \ge 1$, and $\lambda_1 \cdots \lambda_{L+1} \to \lambda^* > 0$ as $L \to \infty$. Then
\[
z^{(L+1)}(\mathcal{X}) \xrightarrow{d} \frac{Z}{\sqrt{n_0 \lambda^*}} X,
\]
where $\mathcal{X} = \{x^{(1)}, \dots, x^{(d)}\} \subseteq \mathbb{R}^{n_0}$ is a set of distinct input vectors, $Z$ is a $D \times n_0$ matrix with independent standard Gaussian entries, and $X = [x^{(1)}, \dots, x^{(d)}]$ is the matrix formed by placing the input vectors as columns.
\end{prop}

Note that in our setting ($\lambda_\ell = C_W^{-1}$), we have $C_W^{-(L+1)} \to \lambda^* > 0$ if and only if $C_W = 1$, which corresponds to the same parametrization used in the critical initialization for $\sigma(x) = x$.

Although the case $\alpha > 0$ is addressed with an explicit description of the limit (see Proposition 3.2), we do not analyze this setting in our paper.

In Example~\ref{ex_id_G_1d} and Section~\ref{id:mult}, we present quantitative results related to Proposition~\ref{prop_id_B}. Specifically, we study the Kolmogorov and Wasserstein distances (for the one-dimensional input case) and the convex distance (for the multi-input case) between the neural network with Gaussian initialization and the Gaussian limit described in Theorem~\ref{Hanin}. This allows us to evaluate the rate at which the neural network converges in law to the Gaussian limit (even though we do not consider the exact limiting distribution). Our results show a convergence rate of order
\[
O\left(\sqrt{\frac{L}{n}}\right)
\]
under the assumption $n_1 \asymp \dots \asymp n_L \asymp n$.

\item In the case of the ReLU activation function ($\sigma(x) = x\mathbf{1}_{\{x \ge 0\}}$), much less is known, especially in the multi-dimensional setting (i.e., for more than one input). In \cite{H18, HN18, ZVP21, Hanin2022CorrelationFI}, it is shown that the limiting distribution of a neural network with 1-homogeneous activation functions, evaluated at a single input and with a one-dimensional output, depends on the \emph{effective depth} $\lim_{L,n \to \infty} \frac{L}{n}$. 

For example, in \cite{Hanin2022CorrelationFI} it is proved that the following limit in distribution holds:
\begin{equation}\label{lim_relu}
\lim_{\substack{L,n\to\infty\\ \frac{L}{n} \to \alpha \in [0, \infty)}} z_1^{(L+1)}(x)
= \left( \frac{2\|x\|^2}{n_0} \right)^{1/2} Z_1 e^{ -\mu(\alpha) + \sqrt{\mu(\alpha)} Z_2 },
\end{equation}
where $x \in \mathbb{R}^{n_0}$, $\mu(\alpha) = \frac{5\alpha}{4}$, and $Z_1, Z_2 \sim \mathcal{N}(0,1)$ are independent standard Gaussian random variables.

The limit \eqref{lim_relu} clearly shows that when $\alpha = 0$, the limiting distribution is Gaussian, as we also demonstrate in Example~\ref{ex_relu_1d}.

\item The already mentioned paper~\cite{BR25} establishes the bound~\eqref{W_1_bound_BR25} on the Wasserstein distance between the output of a neural network with a Lipschitz activation function, evaluated at a finite set of input points, and its corresponding Gaussian limit. Notably, the bound is explicit in the network depth \( L \), except for the constant \( C > 0 \). However, based on the details of the proof, it may be possible to determine the precise dependence of this constant on \( L \). This information can then be used to analyze the conditions under which the finite-dimensional distributions of the neural network converge to a Gaussian law, in the regime where both the depth \( L \) and the hidden layer widths \( n_1, \dots, n_L \) tend to infinity.

\end{itemize}
}
 
\subsection{Minimum eigenvalue of the limiting covariance matrix}
The study of the limiting distribution of a neural network as the inner layer widths tend to infinity is closely related to the behavior of the Neural Tangent Kernel (NTK) matrix, which governs the dynamics of the gradient descent algorithm (see \cite{arora_2020, Jacot, RYH22, Agg, Ye, CC23}). In particular, it has been observed that if the minimum eigenvalue of the NTK matrix at initialization is bounded below by a positive constant, then gradient descent converges globally (see \cite{AZLS19, Duetall18, DLLWZ19, NM20, SY19, OS20, WDW19, ZCZG19, ZG19}).

Recently, in \cite{CC24}, the authors provided mild conditions on the activation function and the inputs under which the NTK matrix is positive definite: $\sigma$ must be continuous, almost everywhere differentiable, and non-polynomial, while the inputs must all be distinct. Moreover, if there is no bias term ($C_b = 0$), it is also required that the inputs be pairwise non-proportional. This result is based on the analysis of the positivity of the limiting covariance matrix $K^{(L+1)}$ introduced in Theorem~\ref{Hanin}. More specifically, it is shown that the above conditions are sufficient to guarantee the positive definiteness of $K^{(L+1)}$ for every $L \ge 1$.

Clearly, this result is significantly more general than what we obtain $L \ge 2$, where we estimate a lower bound for the determinant of $\hat{K}^{(L+1)}:=K^{(L+1)}-C_b\mathbf{1}_{d\times d}$, where $\mathbf{1}_{d\times d}$ is the $d\times d$ matrix of ones,
\begin{equation}\label{det_K_intr}
\operatorname{det}(\hat{K}^{(L+1)})\ge \frac{ C_W^{d(L-1)}}{d!}\lambda(K^{(2)})^{d}\prod_{i=1}^d\left(\sum_{\substack{k_i=1\\k_i\ne k_s\, \forall s=1,\dots,i-1}}^d\prod_{r_i=1}^{L-1}\mathbb{E}\big[\sigma'(G^{(\ell-r_i)}_1(x^{(k_i)}))\big]^2\right),
\end{equation}
denoting the minimum eigenvalue of \( K^{(2)} \) by \( \lambda(K^{(2)}) \). 

We recall from Remark~\ref{HS_bondK} that the right-hand side of inequality~\eqref{det_K_intr} provides a direct lower bound on the minimum eigenvalue of \( K^{(L+1)} \).

However, the bound~\eqref{det_K_intr} becomes uninformative in the case where \( \lambda(K^{(2)}) = 0 \). Moreover, this lower bound may also vanish when the derivative of the activation function, \( \sigma' \), is an odd function. Nevertheless, we provide an explicit bound that holds for any Lipschitz activation function and that, in many cases, can be made explicit in terms of the network depth \( L \).

Our bound is intended to be used in conjunction with the results of \cite{CC24}. Specifically, it is not used to prove the positive definiteness of the limiting covariance matrix, but rather to provide an explicit lower bound on its minimum eigenvalue under the assumption that the matrix is positive definite, as ensured by the aforementioned conditions on the inputs.

To the best of the our knowledge, the only other explicit lower bounds related to the positivity of \( K^{(L+1)} \) are those presented in the proof of Theorem 3.2 in \cite{NMM22} and in the proof of Theorem 8 in \cite{KMM25}. Both results concern the minimum eigenvalue of \( K^{(2)} \) when the activation function is the ReLU, \( \sigma(x) = x \mathbf{1}_{\{x \geq 0\}} \). 

In the first case, the bound is obtained under Gaussian initialization, while in the second case it is assumed that for every \( i = 1, \dots, n_1 \), the vector \( (W^{(1)}_{i,1}, \dots, W^{(1)}_{i,n_0}) \) is uniformly distributed on the sphere, and that the inputs \( \{x^{(1)}, \dots, x^{(d)}\} \subseteq \mathbb{S}^{d-1} \) are \( \delta \)-separated for some \( \delta \in (0, \sqrt{2}] \), meaning that
\[
\min_{i \ne j} \min\left\{ \|x^{(i)} - x^{(j)}\|, \|x^{(i)} + x^{(j)}\| \right\} \ge \delta.
\]
In the Gaussian case, the result is that for any \( r > 0 \),
\[
\lambda(K^{(2)}) \le \min_{i=1,\dots,d} \|x^{(i)}\|^{2r} - (d-1)\max_{i \ne j} |\langle x^{(i)}, x^{(j)} \rangle|^r,
\]
while in the spherical case,
\[
\lambda(K^{(2)}) \ge \frac{\delta^4}{n_0} \left(1 + n_0 \frac{\log\left(\frac{1}{\delta}\right)}{\log(n_0)}\right)^{-3}.
\]

These bounds can be incorporated into~\eqref{det_K_intr} to obtain a complete lower bound on the determinant of the limiting covariance matrix, and consequently on its minimum eigenvalue, following the approach outlined in Remark~\ref{HS_bondK}.

\section{Proof's scheme}\label{proof_scheme_sec}
\subsection{Strategy for the one-dimensional problem}\label{strategy_uno_d}
The main assumption on the random neural network, which is crucial for proving quantitative Central Limit Theorems, is that the weights $\{W_{i,j}^{(\ell)}\}$ are independent and identically distributed (i.i.d.) across different indices $i$, $j$, and $\ell$. Thanks to this assumption, when conditioning the output $z_1^{(L+1)}(x)$ on the $\sigma$-algebra $\mathcal{F}_{L}$ generated by the weights of the previous layers, the neural network's last layer can be expressed as a bias term plus a linear combination of i.i.d. random variables. 

In this setting, one aims to apply Theorem 4.2 from \cite{LG17} (for the Kolmogorov distance) and Theorem 2.2 from \cite{C08} (for the Wasserstein distance). These Theorems provide bounds on the respective distances between a centered Gaussian distribution and a real-valued measurable function $f$ of independent random variables $X = (X_1, \dots, X_n)$ defined on a Polish space, under the assumption that $f(X)$ is centered and has a finite sixth or fourth moment, respectively.

However, applying these theorems directly introduces a term that depends on the variance of the bias $b$, which does not vanish as the widths of the layers tend to infinity. This issue is circumvented by using the following Gaussian integration by parts formula:

\begin{lemma}[Lemma 3.1.2 in \cite{NP12}]\label{int_gauss_uno}
Let $N \sim \mathcal{N}_1(0, \sigma^2)$. Then, for every differentiable function $f: \mathbb{R} \to \mathbb{R}$ such that $\mathbb{E}[Nf(N)] < \infty$ and $\mathbb{E}[f'(N)] < \infty$, it holds that
\[
\mathbb{E}\big[ N f(N) \big] = \mathbb{E}\big[ \sigma^2 f'(N) \big].
\]
\end{lemma}

Using this lemma, the aforementioned theorems can be suitably modified to yield Theorem~\ref{Kdist} and Theorem~\ref{Wdist}. From these results, one obtains the existence of two positive constants—explicitly computable—given by
\[
C_1 = C_1\big(K^{(L+1)}, \mathbb{E}[(W_{1,1}^{(1)})^6], \|\sigma\|_{\text{Lip}}, |\sigma(0)|, \mathbb{E}[|z_{1,\alpha}^{(L)}|^6], C_W, C_b\big),
\]
and
\[
C_2 = C_2(K^{(L+1)}),
\]
such that the following bound holds:
\begin{equation}\label{bound_dist_Q_2_noC}
\max\Big\{ d_K\big(z_1^{(L+1)}(x), G_1^{(L+1)}(x)\big),\; W_1\big(z_1^{(L+1)}(x), G_1^{(L+1)}(x)\big) \Big\} \leq \frac{C_1}{\sqrt{n_L}} + C_2 \sqrt{Q_2^{(L)}(x)},
\end{equation}
where, for $p \in \mathbb{N}$ and each $\ell = 1, \dots, L$, we define
\begin{equation}\label{Q_k}
Q_p^{(\ell)}(x) := \mathbb{E}\left[ \left| \sum_{j=1}^{n_\ell} \frac{C_W}{n_\ell} \sigma^2(z_j^{(\ell)}(x)) - K^{(\ell+1)} + C_b \right|^p \right].
\end{equation}

To ensure that inequality~\eqref{bound_dist_Q_2_noC} holds, it is necessary to prove that the neural network has finite sixth moments. This is established via Proposition~\ref{mom_z_prop}, which provides a bound on the moments of $z_1^{(\ell)}(x)$ for every $\ell = 1, \dots, L$ and $x \in \mathbb{R}^{n_0}$. The bound is finite provided the corresponding moments of $W$ are finite, and it does not depend on the hidden layer widths $n_1, \dots, n_L$.

More specifically, the proof of Proposition~\ref{mom_z_prop} relies on Lemma~\ref{mom_p_iid0}, which correspond to Theorem 1.5.11 in \cite{decoupling}.

The second step consists in bounding the quantity $Q_2^{(\ell)}(x)$ for every $\ell = 1, \dots, L$ in such a way that the bound converges to zero as the widths $n_1, \dots, n_L$ tend to infinity. This is achieved using Stein's method (see \cite{CGS10}), which in this case involves studying the properties of the solution $f_{\sigma^2}$ to the Stein's equation:
\[
K^{(\ell)} f'_{\sigma^2}(y) - y f_{\sigma^2}(y) = \sigma^2(y) - \mathbb{E}[\sigma^2(G_1^{(\ell)}(x))],
\]
where $\sigma$ is a Lipschitz function. This approach allows us to follow part of the proof of Theorem 4.2 in~\cite{LG17} to derive a bound on the quantity
\[
\left| \mathbb{E}[\sigma^2(z_1^{(\ell)}(x)) \mid \mathcal{F}_{\ell-1}] - \mathbb{E}[\sigma^2(G_1^{(\ell)}(x))] \right|,
\]
where $\mathcal{F}_{\ell-1}$ is the $\sigma$-algebra generated by the weights and biases up to layer $\ell - 1$ of the neural network.

This term appears in inequality~\eqref{bound_Q_2_postlemma}, which provides the following bound:
\begin{multline*}
Q_2^{(\ell)}(x) \leq \frac{2C_W^2}{n_\ell} \mathbb{E}\left[\left(\sigma^2(z_1^{(\ell)}(x)) - \mathbb{E}[\sigma^2(z_1^{(\ell)}(x))]\right)^2\right] \\
+ 2C_W^2 \mathbb{E}\left[\left| \mathbb{E}[\sigma^2(z_1^{(\ell)}(x)) \mid \mathcal{F}_{\ell-1}] - \mathbb{E}[\sigma^2(G_1^{(\ell)}(x))] \right|^2\right].
\end{multline*}

Applying Stein’s method, we obtain a bound on the second term, as shown in inequality~\eqref{bound_cond}. Specifically, there exist functions $g_1$, $g_2$, and $g_3$ such that, letting $W'$ denote an independent copy of $W$, the following inequality holds:
\begin{multline*}
\left| \mathbb{E}[\sigma^2(z_1^{(\ell)}(x)) \mid \mathcal{F}_{\ell-1}] - \mathbb{E}[\sigma^2(G_1^{(\ell)}(x))] \right| \\
\leq \frac{1}{n_{\ell-1}^{3/2}} \sum_{j=1}^{n_{\ell-1}} \mathbb{E}\left[ g_1\left(C_W, W, W', \sigma(z_j^{(\ell-1)}(x)), z_1^{(\ell)}(x), K^{(\ell)}, |\sigma(0)|, \|\sigma\|_{\text{Lip}} \right) \mid \mathcal{F}_{\ell-1} \right] 
+ \frac{1}{n_{\ell-1}}\cdot\\
\cdot g_2\left( \mathbb{E}[(W_{1,1}^{(1)})^4], C_W, \left(\sum_{j=1}^{n_{\ell-1}} \sigma^2(z_j^{(\ell-1)}(x))\right)^{1/2}, \|\sigma\|_{\text{Lip}}, |\sigma(0)|, K^{(\ell)}, \mathbb{E}[(z_1^{(\ell)}(x))^4 \mid \mathcal{F}_{\ell-1}] \right) \\
+ \mathbb{E}\left[ g_3\left(K^{(\ell)}, |\sigma(0)|, \|\sigma\|_{\text{Lip}}, z_1^{(\ell)}(x)\right) \mid \mathcal{F}_{\ell-1} \right] \cdot \left| \frac{C_W}{n_{\ell-1}} \sum_{j=1}^{n_{\ell-1}} \sigma^2(z_j^{(\ell-1)}(x)) - K^{(\ell)} + C_b \right|.
\end{multline*}

Applying H\"older inequality to the above estimate, one obtains a bound on $Q_2^{(\ell)}(x)$ (given explicitly by \eqref{bound_iter}) that depends on $\sqrt{Q_4^{(\ell-1)}(x)}$. For this reason, it is more convenient to study the general term
\[
\mathbb{E}\left[\left|\mathbb{E}[\sigma^2(z_1^{(\ell)}(x)) \mid \mathcal{F}_{\ell-1}] - \mathbb{E}[\sigma^2(G_1^{(\ell)}(x))]\right|^{2p}\right]
\]
for every $p \in \mathbb{N}$. 

In the end, using the forthcoming inequality~\eqref{bound_iter} and applying an iterative procedure over $\ell$ and $p$, one obtains a bound on $Q_{2p}^{(\ell)}(x)$, stated in Lemma~\ref{iteration_Q2}. This bound is not fully explicit in terms of constants, as it still depends on the moments of the neural network. The final result is given in Lemma~\ref{cost_schif}, which makes use of Proposition~\ref{mom_z_prop} and holds under the additional assumption that the variance of the bias, namely $C_b$, is non-zero.

In particular, Lemma~\ref{cost_schif} implies that there exists an explicit constant
\[
C_3 := C_3\left(C_b, C_W, |\sigma(0)|, \|\sigma\|_{\text{Lip}}, \mathbb{E}[|W|^{5 \cdot 2^{L+1}}], n_0, x, L, (K^{(r)})_{r=1,\dots,L}\right)
\]
such that
\begin{equation}\label{Q_2_non_expl}
Q_2^{(L-1)}(x) \leq C_3 \sum_{j=1}^{L} \frac{1}{n_j}.
\end{equation}

Substituting inequality~\eqref{Q_2_non_expl} into the bound~\eqref{bound_dist_Q_2_noC}, we obtain
\[
\max\left\{d_K(z_1^{(L+1)}(x), G_1^{(L+1)}(x)), W_1(z_1^{(L+1)}(x), G_1^{(L+1)}(x))\right\} \leq C_4 \sqrt{\sum_{j=1}^{L} \frac{1}{n_j}},
\]
where $C_4$ is an explicit constant that does not depend on the internal widths $n_1, \dots, n_L$, but only on the parameters
\[
\left\{C_b, C_W, L, \{K^{(\ell)}\}_{\ell=1}^{L+1}, \mathbb{E}[|W|^{5 \cdot 2^{L+1}}], \|\sigma\|_{\text{Lip}}, |\sigma(0)|, \|x\|\right\}.
\]

\subsection{Strategy for the multi-dimensional problem}
In this section, we outline the main steps of the proof of Theorem~\ref{pres_sec_prob}, which concerns the convex distance (see Definition~\ref{d_c_def}) between a fully connected neural network (as defined in Definition~\ref{def_NN}) evaluated at a finite set of distinct inputs \( \mathcal{X} = \{x^{(1)}, \dots, x^{(d)}\} \), and the corresponding Gaussian limit (defined in Theorem~\ref{Hanin}).

Note that for $X=(X_1,\dots,X_d)\in\mathbb{R}^d$ we use the shorthand notation 
\[
\sigma(X):=(\sigma(X_1),\dots,\sigma(X_d)).
\]

\,

As in the one-dimensional case, the assumption that the weights are i.i.d.\ is crucial to the proof. The argument is based on Theorem~2.2 from \cite{KG22} (restated in this paper as Theorem~\ref{KG22th}), which generalizes techniques developed in the previously mentioned Theorems~4.2 of \cite{LG17} and~2.2 of \cite{C08}. As before, the presence of the Gaussian random variable \( b_1^{(L+1)} \) in the definition of the neural network requires a slight modification of the above mentioned theorems, presented in Theorem~\ref{th_mult_fin}, using the following multidimensional Gaussian integration by parts formula.

\begin{lemma}[Lemma 4.1.3 in \cite{NP12}]\label{Gauss_integ_multdim}
Let \( \Sigma \in \mathbb{R}^{d \times d} \) be a non-negative definite matrix, and let \( N \sim \mathcal{N}_d(0, \Sigma) \) be a Gaussian random vector in \( \mathbb{R}^d \). Then, for every function \( f : \mathbb{R}^d \to \mathbb{R} \) with bounded first and second derivatives,
\[
\mathbb{E}[\langle N, \nabla f(N) \rangle] = \mathbb{E}[\langle \Sigma, \operatorname{Hess} f(N) \rangle_{\mathrm{HS}}].
\]
\end{lemma}

By applying Lemma~\ref{Gauss_integ_multdim} to the definition of the convex distance (after conditioning with respect to the $\sigma$-field generated by the weights), and subsequently using Theorem~\ref{th_mult_fin}, we obtain Theorem~\ref{mod_KG}, which states that there exists an explicit constant
\[
C_5 := C_5\Big(d, \|(K^{(L+1)})^{-1}\|_{\mathrm{op}}, \mathbb{E}[|W|^6], C_W, \mathbb{E}[\|\sigma(z_1^{(L)}(\mathcal{X}))\|^6], \mathbb{E}[\|\sigma(G_1^{(L)}(\mathcal{X}))\|^4] \Big)
\]
such that
\begin{multline*}
d\big(z_1^{(L+1)}(\mathcal{X}), G_1^{(L+1)}(\mathcal{X})\big) \\
\le C_5 \Bigg( \frac{1}{\sqrt{n_{L-1}}} + \left( \sum_{i=1}^d Q_2^{(L)}(x^{(i)}) \right)^{1/2} 
+ \left( \sum_{\substack{i,j=1\\i\ne j}}^d \mathbb{E}\left[(B_L(x^{(i)}, x^{(j)}))^2\right] \right)^{1/2} \Bigg),
\end{multline*}
where \( Q_2^{(L)} \) is defined in~\eqref{Q_k} and, for every \( i \ne j \) and \( \ell \ge 2 \),
\[
B_\ell(x^{(i)}, x^{(j)}) := \mathbb{E}\left[ \sigma(z_1^{(\ell)}(x^{(i)})) \sigma(z_1^{(\ell)}(x^{(j)})) \,\big|\, \mathcal{F}_{\ell-1} \right]
- \mathbb{E}\left[ \sigma(G_1^{(\ell)}(x^{(i)})) \sigma(G_1^{(\ell)}(x^{(j)})) \right],
\]
with \( \mathcal{F}_{\ell-1} \) denoting the $\sigma$-field generated by the weights and biases up to layer \( \ell - 1 \).

A bound for
\[
\mathbb{E}\left[\left(B_\ell(x^{(i)}, x^{(j)})\right)^2\right]
\]
is provided by Lemma~\ref{somma_schifa}, which relies on Stein's method. In this setting, the method consists in analyzing the properties of the solution \( f^{(\ell)}_{x^{(i)}, x^{(j)}} \) to the differential equation
\begin{multline*}
\left\langle K^{(\ell)}, \operatorname{Hess} f^{(\ell)}_{x^{(i)}, x^{(j)}}(y) \right\rangle_{\mathrm{HS}} 
- \left\langle y, \nabla f^{(\ell)}_{x^{(i)}, x^{(j)}}(y) \right\rangle \\
= \sigma(y_i)\sigma(y_j) 
- \mathbb{E}\left[\sigma(G_1^{(\ell)}(x^{(i)}))\sigma(G_1^{(\ell)}(x^{(j)}))\right],
\end{multline*}
which is studied recursively, together with the following Lemma, with an approach similar to that used to bound \( Q_2^{(\ell)}(x) \) for any $x\in\mathbb{R}^{n_0}$.
\begin{lemma}[Lemma 2.3 in \cite{C08}]\label{cov_ch}
Let $f,g:\mathbb{R}^k \to \mathbb{R}$ be measurable functions such that $\mathbb{E}[f(X)^2] < \infty$ and $\mathbb{E}[g(X)^2] < \infty$. Then,
\[
\operatorname{Cov}(f(X), g(X)) = \frac{1}{2} \sum_{A \subsetneq [k]} \frac{1}{\binom{k}{|A|}(k - |A|)} \sum_{j \notin A} \mathbb{E}\left[\Delta_j g(X) \Delta_j f(X^A)\right],
\]
where 
\[
[k] := \{1, 2, \dots, k\},
\]
\[
\Delta_j f(X) := f(X) - f(X^j),
\]
with $X^j := (X_1, \dots, X_{j-1}, X'_j, X_{j+1}, \dots, X_k)$, defining $X'$ as an independent copy of $X$, and
\[
X^A_j := \begin{cases}
X_j & \text{if } j \notin A, \\
X'_j & \text{if } j \in A.
\end{cases}
\]

\end{lemma}

In particular, the result is that there exist two explicit positive constants:
\[
C_6 := C_6(C_W, \|\sigma\|_{\mathrm{Lip}})
\]
and
\[
\begin{aligned}
C_{7,u} := C_{7,u}\Big(&d, \|\sigma\|_{\mathrm{Lip}}, C_W, \| (K^{(\ell-u)})^{-1} \|_{\mathrm{HS}}, \operatorname{tr} K^{(\ell-u)}, \mathbb{E}\left[\|\sigma(G_1^{(\ell-u-1)}(\mathcal{X}))\|^4\right],\\
&
\mathbb{E}\left[\|\sigma(z_1^{(\ell-u-1)}(\mathcal{X}))\|^{12}\right],\;
\mathbb{E}\left[\|z_1^{(\ell-u-1)}(\mathcal{X})\|^4\right],\;
\mathbb{E}[W^6]
\Big),
\end{aligned}
\]
such that the following bound holds:
\begin{equation} \label{bound_bl_intr}
\mathbb{E}\left[\left(B_\ell(x_\alpha, x_\beta)\right)^2\right] 
\leq \sum_{u=2}^{\ell} C_6^{\ell - u} \left( C_{7,u} \left( \sum_{i=1}^d \sqrt{Q_4^{(\ell-1)}(x^{(i)})} + \frac{1}{n_{\ell-1}} \right) \right)^u.
\end{equation}

A bound on \( Q_4^{(\ell)}(x^{(i)}) \), for every \( \ell = 1, \dots, L \) and $i=1,\dots,d$, is provided by Lemma~\ref{cost_schif}, which was previously used to bound \( Q_2^{(\ell)}(x^{(i)}) \) in the one-dimensional case when \( C_b \neq 0 \).

Therefore, the only term that does not explicitly depend on the hyperparameters of the neural network is the minimum eigenvalue of the limiting covariance matrix \( K^{(L+1)} \), denoted by \( \lambda(K^{(L+1)}) \), which is assumed to be positive. A lower bound on this eigenvalue is then used to control the norm \( \| (K^{(\ell)})^{-1} \|_{\mathrm{HS}} \) appearing in the upper bound~\eqref{bound_bl_intr}.

As shown in more detail in Remark~\ref{HS_bondK}, letting \( \mathbf{1}_{d \times d} \) denote the \( d \times d \) matrix with all entries equal to one, and defining
\[
\hat{K}^{(L+1)} := K^{(L+1)} - C_b \mathbf{1}_{d \times d},
\]
one obtains the following inequality:
\[
\lambda(K^{(L+1)}) \geq \lambda(\hat{K}^{(L+1)}) \geq \frac{\det(\hat{K}^{(L+1)})}{\left(\operatorname{tr}(\hat{K}^{(L+1)})\right)^{d-1}}.
\]
For this reason, we focus on deriving a lower bound for the determinant of the matrix \( \hat{K}^{(L+1)} \) (see Theorem~\ref{bound_K_gauss}).

The strategy of the theorem is to analyze \( \det(\hat{K}^{(L+1)}) \) by introducing independent copies of the Gaussian limit \( G_{1}^{(L+1)}(\mathcal{X}) \). This allows one to reinterpret the determinant of \( \hat{K}^{(L+1)} \) as the variance of a certain function evaluated on a Gaussian vector. With this formulation, one can apply Proposition~3.7 in~\cite{Ca82} (stated as Proposition~\ref{Caco} in the present paper) together with an inductive argument to obtain
\[
\det(\hat{K}^{(L+1)}) \geq \frac{C_W^{d(L-1)}}{d!} \, \lambda(K^{(2)})^d \prod_{i=1}^d \left( \sum_{\substack{k_i = 1 \\ k_i \neq k_s\, \forall s = 1,\dots,i-1}}^d \prod_{r_i = 1}^{L-1} \mathbb{E}\left[ \sigma'\left(G_1^{(\ell - r_i)}(x^{(k_i)}) \right) \right]^2 \right),
\]
which in some cases has an explicit dependence on the constant \( L \), once the activation function \( \sigma \) is fixed.

By combining all the previous results, one obtains the final bound stated in Theorem~\ref{pres_sec_prob}. Specifically, there exists an explicit constant \( C_8 > 0 \), depending on
\[
\left\{ d, C_b, C_W, L, \left\{ \mathbb{E}[\sigma'(G_{1}^{(\ell)}(x^{(i)}))] \right\}_{\substack{i = 1, \dots, d-1 \\ \ell = 2, \dots, L}}, \det(K^{(2)}), \|\sigma\|_{\mathrm{Lip}}, \mathbb{E}\left[ W^{5 \cdot 2^L} \right], |\sigma(0)|, \mathcal{X}, n_0 \right\}
\]
such that
\[
d_C\left(z_1^{(L+1)}(\mathcal{X}), G_1^{(L+1)}(\mathcal{X})\right) \leq C_8 \sqrt{\sum_{j=1}^{L} \frac{1}{n_j}}.
\]

\section{One-Dimensional Problem}\label{uno_d}

In this section, we fix an input $x^{(1)} := x \in \mathbb{R}^{n_0}$ and denote the limiting covariance matrix of $G_{1}^{(L+1)}(x)$, as defined in Theorem~\ref{Hanin}, by
\begin{equation}\label{not_K_1dim}
K^{(L+1)} := K_{1,1}^{(L+1)},
\end{equation}
in order to simplify the notation.

\subsection{Kolmogorov Distance}

Recall Definition~\ref{def_mK} for the Kolmogorov distance, and Definition~\ref{def_NN} for the neural network output $z_{1}^{(L+1)}(x)$ evaluated at the input $x \in \mathbb{R}^{n_0}$.

\begin{theorem}\label{Kdist}
Fix $L \ge 1$.  
If $\mathbb{E}[(W_{1,1}^{(1)})^6] < \infty$, $\mathbb{E}\big[|\sigma(z_{1}^{(L)}(x))|^6\big] < \infty$, and $K^{(L+1)} \ne 0$, then:
\begin{multline*}
d_K(z_{1}^{(L+1)}(x), G_{1}^{(L+1)}(x)) \le \frac{1}{K^{(L+1)}}\sqrt{Q_2^{(L)}(x)} + \frac{C_W}{K^{(L+1)}\sqrt{n_L}}\mathbb{E}\big[(W_{1,1}^{(1)})^6\big]^{1/2} \cdot \\
\cdot \left(1 + \mathbb{E}\big[|\sigma(z_{1}^{(L)}(x))|^6\big]^{1/2}\right) \left(2\sqrt{\frac{C_W}{K^{(L+1)}}} + 2\sqrt{C_W} \left(\sqrt{\frac{C_b}{K^{(L+1)}}} + \frac{\sqrt{2\pi}}{4}\right) + \frac{5}{2}\right),
\end{multline*}
where $G_{1}^{(L+1)}(x) \sim \mathcal{N}_1(0, K^{(L+1)})$ and $Q_2^{(L)}(x)$ is defined in~\eqref{Q_k}.
\end{theorem}

The idea of the proof is to apply Theorem 4.2 from \cite{LG17} to the output of the random neural network. However, the presence of the Gaussian bias term $b_1^{(L+1)}$ prevents a direct application of that theorem. Therefore, a suitably modified version will be employed.
\smallskip
Let $W := b + f(X)$, where $b \sim \mathcal{N}_1(0, C_b)$ is independent of $X$, the random vector $X := (X_1, \dots, X_k)$ lies in a Polish space $E$ and has independent components, and $f: E^k \to \mathbb{R}$ is a measurable function. Let $N \sim \mathcal{N}_1(0, K)$ with $K \ne 0$. Using the notation from Lemma~\ref{cov_ch}, define:
\begin{equation}\label{T_def}
T := \frac{1}{2} \sum_{A \subsetneq [k]} \frac{1}{\binom{k}{|A|}(k - |A|)} \sum_{j \notin A} \Delta_j f(X, X') \Delta_j f(X^A, X'),
\end{equation}
\[
T' := \frac{1}{2} \sum_{A \subsetneq [k]} \frac{1}{\binom{k}{|A|}(k - |A|)} \sum_{j \notin A} \Delta_j f(X, X') \left|\Delta_j f(X^A, X')\right|.
\]

\begin{theorem}\label{Kol_1}
Under the above assumptions, if moreover $\mathbb{E}[f(X)] = 0$ and $\mathbb{E}[f(X)^6] < \infty$, then:
\begin{multline*}
d_K(W, N) \le \frac{1}{K} \mathbb{E}\Big[\mathbb{E}[K - C_b - T \mid X]^2\Big]^{1/2}
+ \frac{1}{K} \Big(\mathrm{\operatorname{Var}} \, \mathbb{E}[T' \mid X]\Big)^{1/2} \\
+ \frac{1}{4} \left(\sqrt{\frac{C_b}{K}} + \frac{\sqrt{2\pi}}{4}\right) \sum_{j=1}^{k} \frac{1}{K} \mathbb{E}\big[|\Delta_j f(X)|^3\big] 
+ \frac{1}{4K\sqrt{K}} \sum_{j=1}^k \mathbb{E}\big[|\Delta_j f(X)|^6\big]^{1/2}.
\end{multline*}
\end{theorem}

\begin{proof}
We begin by observing:
\begin{align*}
d_K(W, N) &= d_K\left(\frac{W}{\sqrt{K}}, \frac{N}{\sqrt{K}}\right)
= \sup_{t \in \mathbb{R}} \left|\mathbb{P}\left(\frac{W}{\sqrt{K}} \le t\right) - \mathbb{P}\left(\frac{N}{\sqrt{K}} \le t\right)\right| \\
&= \sup_{t \in \mathbb{R}} \left|\mathbb{E}\left[g_t'\left(\frac{W}{\sqrt{K}}\right) - \frac{W}{\sqrt{K}} g_t\left(\frac{W}{\sqrt{K}}\right)\right]\right|,
\end{align*}
where $g_t$ is the solution to Stein's equation:
\begin{equation*}
g_t'(w) - w g_t(w) = {\mathbf{1}}_{\{w \le t\}} - \mathbb{P}\left(\frac{N}{\sqrt{K}} \le t\right).
\end{equation*}

Now, using Gaussian integration by parts for the Gaussian random variable $b$ and the independence between $b$ and $X$, we obtain:
\begin{align*}
&\left|\mathbb{E}\left[g_t'\left(\frac{W}{\sqrt{K}}\right) - \frac{W}{\sqrt{K}} g_t\left(\frac{W}{\sqrt{K}}\right)\right]\right| \\
&= \left|\mathbb{E}\left[g_t'\left(\frac{W}{\sqrt{K}}\right) - \frac{C_b}{K} g_t'\left(\frac{W}{\sqrt{K}}\right) - \frac{f(X)}{\sqrt{K}} g_t\left(\frac{W}{\sqrt{K}}\right)\right]\right| \\
&= \left|\mathbb{E}\left[\frac{K - C_b}{K} g_t'\left(\frac{W}{\sqrt{K}}\right) - \frac{T}{K} g_t'\left(\frac{W}{\sqrt{K}}\right)\right] 
+ \mathbb{E}\left[\frac{T}{K} g_t'\left(\frac{W}{\sqrt{K}}\right) - \frac{f(X)}{\sqrt{K}} g_t\left(\frac{W}{\sqrt{K}}\right)\right]\right| \\
&\le \left|\mathbb{E}\left[\frac{K - C_b - T}{K} g_t'\left(\frac{W}{\sqrt{K}}\right)\right]\right| 
+ \left|\mathbb{E}\left[\frac{T}{K} g_t'\left(\frac{W}{\sqrt{K}}\right) - \frac{f(X)}{\sqrt{K}} g_t\left(\frac{W}{\sqrt{K}}\right)\right]\right|.
\end{align*}

Each of the two terms is then bounded following the approach developed in \cite{LG17}.
\end{proof}

\begin{proof}[Proof of Theorem \ref{Kdist}]
Conditioning with respect to the $\sigma$-field $\mathcal{F}_{L}$ generated by the weights and biases $\{W_{i,j}^{(r)}, b_i^{(r)}\}^{r=1,\dots,L}_{i=1,\dots,n_r;\, j=1,\dots,n_{r-1}}$, we have:
\[
d_K(z_{1}^{(L+1)}(x), G_{1}^{(L+1)}(x)) = \sup_{t\in\mathbb{R}} \Big| \mathbb{P}(z_{1}^{(L+1)}(x) \le t) - \mathbb{P}(G_{1}^{(L+1)}(x) \le t) \Big|
\]
\begin{equation}\label{d_K_exp_cond}
\le \mathbb{E}\left[ \sup_{t\in\mathbb{R}} \left| \mathbb{E}[1_{(-\infty, t)}(z_{1}^{(L+1)}(x)) \mid \mathcal{F}_L] - \mathbb{E}[1_{(-\infty, t)}(G_{1}^{(L+1)}(x))] \right| \right].
\end{equation}

{
After conditioning with respect to the $\sigma$-field $\mathcal{F}_{L}$, the first component of the neural network output, $ z_1^{(L+1)}(x)$, can be written as a linear combination of independent and identically distributed random variables,
\[
z_1^{(L+1)}(x)
= \sum_{j=1}^{n_L} W^{(L+1)}_{1,j}\, a^{(L)}_j(x),
\]
where the coefficients $\{a^{(L)}_j(x)\}_{j=1}^{n_L}$ are $\mathcal{F}_L$--measurable and the weights
$\{W^{(L+1)}_{1,j}\}_{j=1}^{n_L}$ are i.i.d. and independent of $\mathcal{F}_L$.
This follows from the initialization assumption in Definition~\ref{def_NN} together with the freezing lemma
(Lemma~4.1 in \cite{Baldi17}).

Consequently, the random variable $z_1^{(L+1)}(x)$, conditionally on $\mathcal{F}_L$, is a sum of independent centered random variables.
Moreover, the quantity inside the expectation in \eqref{d_K_exp_cond} is the Kolmogorov distance between this conditional distribution  and the Gaussian random variable $G_1^{(L+1)}(x)$.
Hence, we apply Theorem~\ref{Kol_1} to this distance and obtain

}
\begin{multline*}
d_K(z_{1}^{(L+1)}(x), G_{1}^{(L+1)}(x)) \le \frac{1}{K^{(L+1)}} \mathbb{E}\left[ \left(\mathbb{E}[T - K^{(L+1)} + C_b \mid \mathcal{F}_{L}]\right)^2 \right]^{1/2} \\
+ \frac{1}{K^{(L+1)}} \mathbb{E}\left[ \left(\mathbb{E}[T' \mid \mathcal{F}_{L}]\right)^2 \right]^{1/2} \\
+ \frac{1}{4K^{(L+1)}\sqrt{K^{(L+1)}}} \sum_{j=1}^{n_L} \mathbb{E}\left[ \frac{C_W^3}{n_L^3} (W_{1,j}^{(L+1)} - {W'}_{1,j}^{(L+1)})^6 \sigma^6(z_j^{(L)}(x)) \right]^{1/2} \\
+ \frac{1}{4K^{(L+1)}}\left( \sqrt{\frac{C_b}{K^{(L+1)}}} + \frac{\sqrt{2\pi}}{4} \right) \sum_{j=1}^{n_L} \mathbb{E}\left[ \frac{C_W^{3/2}}{n_L^{3/2}} |W_{1,j}^{(L+1)} - {W'}_{1,j}^{(L+1)}|^3 |\sigma^3(z_j^{(L)}(x))| \right],
\end{multline*}
where
\begin{equation}\label{T}
T := \frac{1}{2} \sum_{j=1}^{n_L} \frac{C_W}{n_L} (W_{1,j}^{(L+1)} - {W'}_{1,j}^{(L+1)})^2 \sigma^2(z_j^{(L)}(x)),
\end{equation}
and
\[
T' := \frac{1}{2} \sum_{j=1}^{n_L} \frac{C_W}{n_L} (W_{1,j}^{(L+1)} - {W'}_{1,j}^{(L+1)}) \left| W_{1,j}^{(L+1)} - {W'}_{1,j}^{(L+1)} \right| \cdot |\sigma(z_j^{(L)}(x))|^2.
\]
In particular,
\begin{multline*}
\mathbb{E}\left[\left(\mathbb{E}[T - K^{(L+1)} + C_b \mid \mathcal{F}_{L}]\right)^2\right] \\
= \mathbb{E}\Bigg[\Bigg( \frac{1}{2} \sum_{j=1}^{n_L} \frac{C_W}{n_L} \big((W_{1,j}^{(L+1)})^2 - 1\big) \sigma^2(z_j^{(L)}(x)) + \sum_{j=1}^{n_L} \frac{C_W}{n_L} \big( \sigma^2(z_j^{(L)}(x)) - \mathbb{E}[\sigma^2(G_j^{(L)}(x))] \big) \Bigg)^2 \Bigg] \\
= \frac{C_W^2}{4n_L} \big( \mathbb{E}[(W_{1,1}^{(1)})^4] - 1 \big) \mathbb{E}[\sigma^4(z_1^{(L)}(x))] + \mathbb{E}\left[ \left( \sum_{j=1}^{n_L} \frac{C_W}{n_L} \big( \sigma^2(z_j^{(L)}(x)) - \mathbb{E}[\sigma^2(G_j^{(L)}(x))] \big) \right)^2 \right]
\end{multline*}
\begin{equation}\label{stima_T}
= \frac{C_W^2}{4n_L} \big( \mathbb{E}[(W_{1,1}^{(1)})^4] - 1 \big) \mathbb{E}[\sigma^4(z_1^{(L)}(x))] + Q_2^{(L)}(x),
\end{equation}
recalling the definition in (\ref{Q_k}).

Instead,
\begin{multline*}
\mathbb{E}\left[\left(\mathbb{E}[T' \mid \mathcal{F}_{L}]\right)^2\right] \le \mathbb{E}[(T')^2] = \frac{C_W^2}{4n_L^2} \sum_{j=1}^{n_L} \mathbb{E}\left[(W_{1,j}^{(1)} - {W'}_{1,j}^{(1)})^4\right] \mathbb{E}[\sigma^4(z_j^{(L)}(x))] \\
+ \frac{C_W^2}{4n_L^2} \sum_{j=1}^{n_L} \sum_{\substack{k=1 \\ k \ne j}}^{n_L} \mathbb{E}\left[(W_{1,j}^{(1)} - {W'}_{1,j}^{(1)}) |W_{1,j}^{(1)} - {W'}_{1,j}^{(1)}|\right] \mathbb{E}\left[(W_{1,k}^{(1)} - {W'}_{1,k}^{(1)}) |W_{1,k}^{(1)} - {W'}_{1,k}^{(1)}|\right] \cdot \\
\cdot \mathbb{E}\left[ \sigma(z_j^{(L)}(x)) |\sigma(z_j^{(L)}(x))| \cdot \sigma(z_k^{(L)}(x)) |\sigma(z_k^{(L)}(x))| \right]
\end{multline*}
\[
= \frac{C_W^2}{4n_L^2} \sum_{j=1}^{n_L} \mathbb{E}\left[(W_{1,j}^{(1)} - {W'}_{1,j}^{(1)})^4\right] \mathbb{E}[\sigma^4(z_j^{(L)}(x))],
\]
since \( \mathbb{E}[(W_{1,j}^{(1)} - {W'}_{1,j}^{(1)}) |W_{1,j}^{(1)} - {W'}_{1,j}^{(1)}|] = 0 \) for every \( j = 1, \dots, n_L \).

\medskip

In conclusion,
\begin{multline}\label{tv_1dim_primaTH}
d_K(z_1^{(L+1)}(x), G_1^{(L+1)}(x)) \le
\frac{1}{K^{(L+1)}} \cdot \frac{C_W}{2\sqrt{n_L}} \left( \mathbb{E}[(W_{1,1}^{(1)})^4] - 1 \right)^{1/2} \mathbb{E}[\sigma^4(z_1^{(L)}(x))]^{1/2} \\
+ \frac{1}{K^{(L+1)}} \sqrt{Q_2^{(L)}(x)}
+ \frac{1}{K^{(L+1)}} \cdot \frac{C_W}{2\sqrt{n_L}} \mathbb{E}[(W_{1,1}^{(1)} - {W'}_{1,1}^{(1)})^4]^{1/2} \mathbb{E}[\sigma^4(z_1^{(L)}(x))]^{1/2} \\
+ \frac{1}{4K^{(L+1)} \sqrt{K^{(L+1)}}} \cdot \frac{C_W^{3/2}}{\sqrt{n_L}} \mathbb{E}[(W_{1,1}^{(L+1)} - {W'}_{1,1}^{(L+1)})^6]^{1/2} \mathbb{E}[\sigma^6(z_1^{(L)}(x))]^{1/2} \\
+ \frac{1}{4K^{(L+1)}} \cdot \frac{C_W^{3/2}}{\sqrt{n_L}} \left( \sqrt{\frac{C_b}{K^{(L+1)}}} + \frac{\sqrt{2\pi}}{4} \right) \mathbb{E}\left[ |W_{1,1}^{(1)} - {W'}_{1,1}^{(1)}|^3 \right] \mathbb{E}\left[ |\sigma^3(z_1^{(L)}(x))| \right]
\end{multline}

\begin{multline*}
\le \frac{1}{K^{(L+1)}} \sqrt{Q_2^{(L)}(x)} + \frac{1}{K^{(L+1)}} \cdot \frac{C_W}{\sqrt{n_L}} \mathbb{E}[(W_{1,1}^{(1)})^6]^{1/2} \left( 1 + \mathbb{E}[|\sigma(z_1^{(L)}(x))|^6]^{1/2} \right) \cdot \\
\cdot \left( 2\sqrt{\frac{C_W}{K^{(L+1)}}} + 2\sqrt{C_W} \left( \sqrt{\frac{C_b}{K^{(L+1)}}} + \frac{\sqrt{2\pi}}{4} \right) + \frac{5}{2} \right).
\end{multline*}
\end{proof}

\subsection{Wasserstein Distance}

Recall Definition~\ref{d_w_def} for the definition of the Wasserstein distance, and Definition~\ref{def_NN} for the definition of the neural network output $z_{1}^{(L+1)}(x)$ evaluated at the input $x \in \mathbb{R}^{n_0}$.

\begin{theorem}\label{Wdist}
Let $L \ge 1$. Suppose that $\mathbb{E}[(W_{1,1}^{(1)})^4] < \infty$, $\mathbb{E}[|\sigma(z_{1}^{(L)}(x))|^4] < \infty$, and $K^{(L+1)} \ne 0$. Then the Wasserstein distance between $z_{1}^{(L+1)}(x)$ and the Gaussian random variable $G_{1}^{(L+1)}(x) \sim \mathcal{N}(0, K^{(L+1)})$, as defined in Theorem~\ref{Hanin}, satisfies
\begin{multline*}
W_1\big(z_{1}^{(L+1)}(x), G_{1}^{(L+1)}(x)\big)
\le \frac{C_W}{\sqrt{n_L} \sqrt{K^{(L+1)}}} \, \mathbb{E}[(W_{1,1}^{(1)})^4]^{3/4} \left( \mathbb{E}[|\sigma(z_{1}^{(L)}(x))|^4]^{3/4} + 1 \right) \cdot \\
\cdot \left( \frac{4\sqrt{C_W}}{\sqrt{K^{(L+1)}}} + \frac{1}{\sqrt{2\pi}} \right)
+ \sqrt{ \frac{2}{K^{(L+1)} \pi} } \, \sqrt{Q_2^{(L)}(x)},
\end{multline*}
where $Q_2^{(L)}(x)$ is defined in equation~\eqref{Q_k}.
\end{theorem}

As in the previous section, it is not possible to directly apply Theorem~2.2 from \cite{C08}. Instead, we make use of the following adapted version:

\begin{theorem}\label{th_wass_orig}
Under the assumptions of Theorem~\ref{Kol_1}, and with $T$ defined as in equation~\eqref{T}, the following inequality holds:
\[
W_1(b + f(X), N) \le \sqrt{\frac{2}{K\pi}} \, \mathbb{E}\left[ \mathbb{E}[K - C_b - T \mid W]^2 \right]^{1/2}
+ \frac{1}{2K} \sum_{j=1}^{k} \mathbb{E}\left[ |\Delta_j f(X)|^3 \right],
\]
where $N \sim \mathcal{N}(0,K)$ and $f(X)$ is a function of independent random variables as in Theorem~\ref{Kol_1}.
\end{theorem}
\begin{proof}
    The argument follows similarly to the proof of Theorem~\ref{Kol_1}.
\end{proof}

\begin{proof}[Proof of Theorem~\ref{Wdist}]
Condition on the $\sigma$-field $\mathcal{F}_{L}$ generated by the weights and biases up to layer $L$, that is,
\[
\mathcal{F}_{L} := \sigma\left( \left\{ W_{i,j}^{(r)},\, b_i^{(r)} \right\}_{i=1,\dots,n_r;\, j=1,\dots,n_{r-1}}^{r=1,\dots,L} \right).
\]
Using Definition~\ref{d_w_def}, Proposition \ref{form_dual_wass} and Theorem~\ref{th_wass_orig}, we obtain:
\begin{multline*}
W_1\big(z_{1}^{(L+1)}(x), G_{1}^{(L+1)}(x)\big)
\le \mathbb{E}\left[\sup_{h \in \mathrm{Lip}_1(1)} \left| \mathbb{E}\left[h\big(z_{1}^{(L+1)}(x)\big) \mid \mathcal{F}_{L} \right] - \mathbb{E}\left[ h\big(G_{1}^{(L+1)}(x)\big) \right] \right| \right] \\
\le \frac{1}{2K^{(L+1)}} \sum_{j=1}^{n_L} \mathbb{E}\left[ \left| \frac{\sqrt{C_W}}{\sqrt{n_L}} \left(W_{1,j}^{(L+1)} - {W'}_{1,j}^{(L+1)}\right) \sigma\left(z_{j}^{(L)}(x)\right) \right|^3 \right] \\
+ \sqrt{ \frac{2}{K^{(L+1)} \pi} } \cdot \sqrt{ \mathbb{E} \left[ \mathbb{E}\left[ T - K^{(L+1)} + C_b \mid \mathcal{F}_{L+1} \right]^2 \right] }.
\end{multline*}

Now, using the estimate in equation~\eqref{stima_T} and the definition of $Q_2^{(L)}(x)$ from equation~\eqref{Q_k}, we get:
\begin{multline}\label{W1_1dim_primaTH}
W_1\big(z_{1}^{(L+1)}(x), G_{1}^{(L+1)}(x)\big)
\le \frac{1}{2K^{(L+1)}} \cdot \frac{C_W^{3/2}}{\sqrt{n_L}} \, \mathbb{E}\left[ |W_{1,1}^{(1)} - {W'}_{1,1}^{(1)}|^3 \right] \, \mathbb{E}\left[ |\sigma(z_{j}^{(L)}(x))|^3 \right] \\
+ \frac{C_W}{\sqrt{2K^{(L+1)} \pi n_L}} \left( \mathbb{E}\left[(W_{1,1}^{(1)})^4\right] - 1 \right)^{1/2} \mathbb{E}\left[ |\sigma(z_{1}^{(L)}(x))|^4 \right]^{1/2} 
+ \sqrt{ \frac{2}{K^{(L+1)} \pi} } \cdot \sqrt{ Q_2^{(L)}(x) }.
\end{multline}

Finally, applying Hölder's inequality and simplifying:
\begin{multline*}
W_1\big(z_{1}^{(L+1)}(x), G_{1}^{(L+1)}(x)\big)
\le \frac{C_W}{\sqrt{n_L} \sqrt{K^{(L+1)}}} \, \mathbb{E}[(W_{1,1}^{(1)})^4]^{3/4} 
\left( \mathbb{E}[|\sigma(z_{1}^{(L)}(x))|^4]^{3/4} + 1 \right)\cdot  \\
\cdot\left( \frac{4\sqrt{C_W}}{\sqrt{K^{(L+1)}}} + \frac{1}{\sqrt{2\pi}} \right)
+ \sqrt{ \frac{2}{K^{(L+1)} \pi} } \cdot \sqrt{ Q_2^{(L)}(x) }.
\end{multline*}
\end{proof}

\subsection{Estimate of the Moments of the Random Neural Network}\label{mom_z_sec}

To apply Theorems~\ref{Kdist} and~\ref{Wdist}, we require that $\mathbb{E}[\sigma^6(z_1^{(L)}(x))] < \infty$ and $\mathbb{E}[\sigma^4(z_1^{(L)}(x))] < \infty$, respectively. Since the activation function $\sigma$ is Lipschitz continuous, it suffices to study the conditions under which the moments of the network's output evaluated at a given input $x \in \mathbb{R}^{n_0}$ are finite.

\begin{prop}\label{mom_z_prop}
Let $L+1 \ge \ell \ge 2$ and $p \ge 2$ even integer. Then
\begin{multline}\label{moment_meglio}
\mathbb{E}[(z_{1}^{(\ell)}(x))^p]^{1/p} \le 
\Bigg(\Big(
 \sqrt{C_b} \big((p-1)!! \big)^{1/p}
+ K\sqrt{C_W}\frac{p}{\log(p)} |\sigma(0)| \mathbb{E}[(W_{1,1}^{(1)})^p]^{1/p}
\Big)(\ell - 1)\\
+ \mathbb{E}[(z_{1}^{(1)}(x))^p]^{1/p}
\Bigg)  \Bigg(
1 + 
\Big( K\frac{p}{\log(p)}  \sqrt{C_W} \|\sigma\|_{\mathrm{Lip}} \mathbb{E}[(W_{1,1}^{(1)})^p]^{1/p} \Big)^{\ell - 1}
\Bigg),
\end{multline}
where $K>0$ is a numerical constant and the term $\mathbb{E}[(z_{1}^{(1)}(x))^p]$ does not depend on the layer widths $n_i$ for $i = 1, \dots, L+1$.

More precisely, we have:
\[
\mathbb{E}[(z_{1}^{(1)}(x))^p]^{1/p} \le 
 K\sqrt{C_b} \big((p-1)!! \big)^{1/p}
+ \sqrt{\frac{C_W}{n_0}} \frac{p}{\log(p)} \max\{\|x\|,\|x\|_p\} \mathbb{E}[(W_{1,1}^{(1)})^p]^{1/p}.
\]
\end{prop}

The proof of Proposition~\ref{mom_z_prop} follows from the next Lemma, which is a particular case of Theorem 1.5.11 in \cite{decoupling}.

\begin{lemma}\label{mom_p_iid0}
Let $X_1,\dots,X_n$ be independent and identically distributed random variables with $\mathbb{E}[X_i]=0$ for every $i=1,\dots,n$, and let $p \ge 2$ be an even integer. Then, for every $a_1,\dots,a_n\in\mathbb{R}$,
\[ 
\mathbb{E}\left[\left(\sum_{j=1}^{n} a_j X_j\right)^p\right]^{1/p} \le K\frac{p}{\log(p)}\mathbb{E}[X_1^p]^{1/p}\max\left\{\left(\sum_{i=1}^n a_i^p\right)^{1/p},\left(\sum_{i=1}^n a_i^2\right)^{1/2}\right\},
\]
where $K>0$ is a numeric constant that does not depend on $p$.
\end{lemma}

\noindent

\begin{proof}[Proof of Proposition~\ref{mom_z_prop}]
From Lemma~\ref{mom_p_iid0} and the tower property, it follows that if $p \ge 2$ is an even integer and $\ell \ge 2$, then:
\begin{multline*}
\mathbb{E}[(z_{1}^{(\ell)}(x))^p]^{1/p} \le \mathbb{E}[(b_1^{(\ell)})^p]^{1/p}
+  \mathbb{E}\left[\left(\sum_{j=1}^{n_{\ell-1}} \sqrt{\frac{C_W}{n_{\ell-1}}} W_{1,j}^{(\ell)} \sigma(z_{j}^{(\ell-1)}(x))\right)^p\right]^{1/p} \\
\le \sqrt{C_b} \big((p-1)!!\big)^{1/p}
+ K \frac{p}{\log(p)} \sqrt{{C_W}}
\mathbb{E}[\sigma^p(z_{1}^{(\ell-1)}(x))]^{1/p} \mathbb{E}[(W_{1,1}^{(\ell)})^p]^{1/p}.
\end{multline*}

Using the Lipschitz continuity of $\sigma$, we further estimate:
\begin{multline*}
\mathbb{E}[(z_{1}^{(\ell)}(x))^p] \le 
 \sqrt{C_b}\big((p-1)!! \big)^{1/p}
+  K\|\sigma\|_{\text{Lip}} \frac{p}{\log(p)} \sqrt{{C_W}}
\mathbb{E}[(z_{1}^{(\ell-1)}(x))^p]^{1/p} \mathbb{E}[(W_{1,1}^{(\ell)})^p]^{1/p} \\
+ K \frac{p}{\log(p)} \sqrt{{C_W}}
|\sigma(0)| \mathbb{E}[(W_{1,1}^{(\ell)})^p]^{1/p}.
\end{multline*}

Iterating the inequality over $\ell$ then yields the desired bound stated in Proposition~\ref{mom_z_prop}.
\end{proof}

\begin{remark}\label{fin_mom_z}
Proposition~\ref{mom_z_prop} shows, in particular, that if $\mathbb{E}[W^p] < \infty$ for some even integer $p \ge 2$, then:
\[
\sup_{n_1, \dots, n_L \in \mathbb{N}} \mathbb{E}[(z_{1}^{(L)}(x))^p] < \infty.
\]
\end{remark}

Therefore, the moment assumptions on $\sigma(z_{1}^{(L)}(x))$ in Theorems~\ref{Kdist} and~\ref{Wdist} can be dropped if the activation function $\sigma$ is Lipschitz continuous.

\subsection{A bound on $Q_{2p}^{(\ell)}(x)$ for $\ell,p \ge 1$}

In Theorems~\ref{Kdist} and~\ref{Wdist}, the only non-explicit term is $Q_2^{(L)}(x)$, so in this section we present an upper bound that resolves this issue.

Recall that, for $\ell \ge 1$ and $x \in \mathbb{R}^{n_0}$, the quantity $Q_2^{(\ell)}(x)$ is defined as:
\[
Q_{2}^{(\ell)}(x) := \mathbb{E}\left[\left|\sum_{j=1}^{n_{\ell}}\frac{C_W}{n_{\ell}}\sigma^2(z_{j}^{(\ell)}(x)) - K^{(\ell+1)} + C_b\right|^{2}\right],
\]
where
\[
K^{(\ell+1)} = C_b + C_W \, \mathbb{E}[\sigma^2(G_{1}^{(\ell)}(x))],
\]
and $G_{1}^{(\ell)}(x) \sim \mathcal{N}(0, K^{(\ell)})$, as defined in Theorem~\ref{Hanin}. By assumption, we have $K^{(\ell)} \ne 0$ for every $\ell = 1, \dots, L+1$.

Also recall that $\mathcal{F}_{\ell-1}$ denotes the $\sigma$-field generated by
\[
\left\{ W_{i,j}^{(r)},\, b_i^{(r)} \right\}_{\substack{r=1,\dots,\ell-1; \\ i=1,\dots,n_r; \\ j=1,\dots,n_{r-1}}}.
\]

\vspace{0.3em}

Consider now the quantity, for $\ell \ge 2$,
\[
\left| \mathbb{E}\left[\sigma^2(z_{1}^{(\ell)}(x)) \mid \mathcal{F}_{\ell-1} \right] - \mathbb{E}\left[ \sigma^2(G_{1}^{(\ell)}(x)) \right] \right|.
\]

Following the strategy used in~\cite{KG22} to prove their Theorem~2.2 (stated here as Theorem~\ref{KG22th}), we consider the solution $f_{\sigma^2}$ to the Stein's equation associated with the function $\sigma^2$. This solution satisfies:
\begin{equation}\label{stein_uno}
K^{(\ell)} f'_{\sigma^2}(z) - z f_{\sigma^2}(z) = \sigma^2(z) - \mathbb{E}[\sigma^2(G_{1}^{(\ell)}(x))].
\end{equation}

An explicit expression for the solution $f_{\sigma^2}$ is given by:
\begin{equation}\label{expr_f_s2_uni}
f_{\sigma^2}(z) = \frac{e^{\frac{z^2}{2K^{(\ell)}}}}{K^{(\ell)}} \int_{-\infty}^{z} e^{-\frac{y^2}{2K^{(\ell)}}} \left( \sigma^2(y) - \mathbb{E}[\sigma^2(G_{1}^{(\ell)}(x))] \right) \, dy.
\end{equation}

Then, using equation~\eqref{stein_uno} and the Gaussian integration by parts formula from Lemma~\ref{int_gauss_uno}, we obtain:
\[
\left|\mathbb{E}[\sigma^2(z_{1}^{(\ell)}(x)) \mid \mathcal{F}_{\ell-1}] - \mathbb{E}[\sigma^2(G_{1}^{(\ell)}(x))]\right|
= \left| \mathbb{E}\left[ K^{(\ell)} f'_{\sigma^2}(z_{1}^{(\ell)}(x)) - z_{1}^{(\ell)}(x) f_{\sigma^2}(z_{1}^{(\ell)}(x)) \mid \mathcal{F}_{\ell-1} \right] \right|.
\]
\begin{multline*}
\le \left| \mathbb{E}\left[ C_b f'_{\sigma^2}(z_{1}^{(\ell)}(x)) - b_1^{(\ell)} f_{\sigma^2}(z_{1}^{(\ell)}(x)) \mid \mathcal{F}_{\ell-1} \right] \right| \\
+ \left| \mathbb{E}\left[ (K^{(\ell)} - C_b) f'_{\sigma^2}(z_{1}^{(\ell)}(x)) - (z_{1}^{(\ell)}(x) - b_1^{(\ell)}) f_{\sigma^2}(z_{1}^{(\ell)}(x)) \mid \mathcal{F}_{\ell-1} \right] \right|.
\end{multline*}
\begin{equation}\label{eq_in_Q2}
= \left| \mathbb{E}\left[ (K^{(\ell)} - C_b) f'_{\sigma^2}(z_{1}^{(\ell)}(x)) - (z_{1}^{(\ell)}(x) - b_1^{(\ell)}) f_{\sigma^2}(z_{1}^{(\ell)}(x)) \mid \mathcal{F}_{\ell-1} \right] \right|,
\end{equation}
since $b_1^{(\ell)} \sim \mathcal{N}(0, C_b)$.

\smallskip

Now define a copy in law of $(W_{1,1}^{(\ell)}, \dots, W_{1,n_{\ell-1}}^{(\ell)})$ that is independent of $\mathcal{F}_{\ell}$ and denote it by $({W'}_{1,1}^{(\ell)}, \dots, {W'}_{1,n_{\ell-1}}^{(\ell)})$. Define also $z_{1}^{(\ell)(j)}(x)$ as the version of $z_{1}^{(\ell)}(x)$ in which $W_{1,j}^{(\ell)}$ is replaced by ${W'}_{1,j}^{(\ell)}$.

Using Lemma~\ref{cov_ch}, we obtain:
\begin{multline*}
\mathbb{E}\left[(z_{1}^{(\ell)}(x) - b_1^{(\ell)}) f_{\sigma^2}(z_{1}^{(\ell)}(x)) \mid \mathcal{F}_{\ell-1} \right] \\
= \frac{1}{2} \sum_{A \subset [n_{\ell-1}]} \frac{1}{\binom{n_{\ell-1}}{|A|}(n_{\ell-1} - |A|)} \sum_{j \notin A} \mathbb{E}\Bigg[ \frac{\sqrt{C_W}}{\sqrt{n_{\ell-1}}}(W_{1,j}^{(\ell)} - {W'}_{1,j}^{(\ell)}) \sigma(z_{j}^{(\ell-1)}(x)) \cdot \\
\cdot \left( f_{\sigma^2}(z_{1}^{(\ell)}(x)) - f_{\sigma^2}(z_{1}^{(\ell)(j)}(x)) \right) \mid \mathcal{F}_{\ell-1} \Bigg].
\end{multline*}
By applying the second-order expansion, this becomes:
\begin{multline}\label{eq_sec_Q2}
= \frac{1}{2} \sum_{j=1}^{n_{\ell-1}} \mathbb{E}\left[ \frac{C_W}{n_{\ell-1}} (W_{1,j}^{(\ell)} - {W'}_{1,j}^{(\ell)})^2 \sigma^2(z_{j}^{(\ell-1)}(x)) f'_{\sigma^2}(z_{1}^{(\ell)}(x)) \mid \mathcal{F}_{\ell-1} \right] \\
+ \frac{1}{2} \sum_{j=1}^{n_{\ell-1}} \mathbb{E}\left[ \frac{\sqrt{C_W}}{\sqrt{n_{\ell-1}}}(W_{1,j}^{(\ell)} - {W'}_{1,j}^{(\ell)}) \sigma(z_{j}^{(\ell-1)}(x)) R \mid \mathcal{F}_{\ell-1} \right],
\end{multline}
where
\[
R = f_{\sigma^2}(z_{1}^{(\ell)}(x)) - f_{\sigma^2}(z_{1}^{(\ell)(j)}(x)) + f'_{\sigma^2}(z_{1}^{(\ell)}(x)) \cdot \frac{\sqrt{C_W}}{\sqrt{n_{\ell-1}}}({W'}_{1,j}^{(\ell)} - W_{1,j}^{(\ell)}) \sigma(z_{j}^{(\ell-1)}(x)).
\]

Let us define \( h := z_{1}^{(\ell)(j)}(x) - z_{1}^{(\ell)}(x) \). Using the identity~\eqref{stein_uno}, we find that
\[
|R| = \left| \int_{0}^{h} \left( f'_{\sigma^2}(z_{1}^{(\ell)}(x) + t) - f'_{\sigma^2}(z_{1}^{(\ell)}(x)) \right) \, dt \right|.
\]
\begin{multline*}
= \frac{1}{K^{(\ell)}} \Bigg| \int_0^h \Big( (z_{1}^{(\ell)}(x)+t) f_{\sigma^2}(z_{1}^{(\ell)}(x)+t)
+ \sigma^2(z_{1}^{(\ell)}(x)+t) \\
- z_{1}^{(\ell)}(x) f_{\sigma^2}(z_{1}^{(\ell)}(x)) - \sigma^2(z_{1}^{(\ell)}(x)) \Big) dt \Bigg|
\end{multline*}
\begin{multline}\label{bound_R_I123}
\le \frac{1}{K^{(\ell)}} \left| \int_0^h z_{1}^{(\ell)}(x) \left( f_{\sigma^2}(z_{1}^{(\ell)}(x)+t) - f_{\sigma^2}(z_{1}^{(\ell)}(x)) \right) dt \right| \\
+ \frac{1}{K^{(\ell)}} \left| \int_0^h t f_{\sigma^2}(z_{1}^{(\ell)}(x)+t) \, dt \right|
+ \frac{1}{K^{(\ell)}} \left| \int_0^h \left( \sigma^2(z_{1}^{(\ell)}(x)+t) - \sigma^2(z_{1}^{(\ell)}(x)) \right) dt \right| \\
=: I_1 + I_2 + I_3.
\end{multline}

The following lemma provides bounds on the terms \( I_1 \), \( I_2 \), and \( I_3 \), based on the regularity properties of \( f_{\sigma^2} \).

\begin{lemma}\label{I_1,2,3}
Under the assumptions of the present section, we have the following bounds:
\begin{multline*}
I_1 \le \frac{|z_{1}^{(\ell)}(x)|}{K^{(\ell)}} \Bigg[
\left( D_\ell(x) + \frac{4 \|\sigma\|_{\mathrm{Lip}}^2}{K^{(\ell)}} |z_{1}^{(\ell)}(x)|^2 + 2 \|\sigma\|_{\mathrm{Lip}}^2 |z_{1}^{(\ell)}(x)| \right) \frac{h^2}{2} \\
+ \left( D_\ell(x) + 6 \|\sigma\|_{\mathrm{Lip}}^2 |z_{1}^{(\ell)}(x)| \right) \frac{|h|^3}{3}
+ \|\sigma\|_{\mathrm{Lip}}^2 \left( \frac{4}{K^{(\ell)}} + 2 \right) \frac{h^4}{4}
\Bigg],
\end{multline*}
where
\begin{equation}\label{const_D}
D_\ell(x) := 2\left( 2\sqrt{\frac{\pi K^{(\ell)}}{2}} + 1 \right)
\left( \frac{2 |\sigma(0)|^2}{K^{(\ell)}} + \|\sigma\|_{\mathrm{Lip}}^2 \right).
\end{equation}

Moreover,
\[
I_2 \le \sqrt{\frac{\pi}{2K^{(\ell)}}} \left( \frac{|\sigma(0)|^2}{K^{(\ell)}} + \|\sigma\|_{\mathrm{Lip}}^2 + \frac{2\|\sigma\|_{\mathrm{Lip}}^2}{K^{(\ell)}} \right) \frac{h^2}{2}
+ \frac{2 \|\sigma\|_{\mathrm{Lip}}^2}{K^{(\ell)}} \frac{|h|^3}{3},
\]
\[
I_3 \le \frac{\|\sigma\|_{\mathrm{Lip}}^2}{K^{(\ell)}} \frac{|h|^3}{3}
+ \frac{\|\sigma\|_{\mathrm{Lip}}}{K^{(\ell)}} |\sigma(z_{1}^{(\ell)}(x))| h^2.
\]
\end{lemma}

\begin{proof}
See Appendix~A.
\end{proof}

Hence, applying Lemma~\ref{I_1,2,3} to inequality~\eqref{bound_R_I123}, we obtain that
\begin{multline*}
|R| \le 
\Bigg[
\frac{|z_{1}^{(\ell)}(x)|}{K^{(\ell)}}\left(D_\ell(x) + \frac{4\|\sigma\|_{\text{Lip}}^2}{K^{(\ell)}} |z_{1}^{(\ell)}(x)|^2 + 2\|\sigma\|_{\text{Lip}}^2 |z_{1}^{(\ell)}(x)|\right) \\
+ \sqrt{\frac{\pi}{2K^{(\ell)}}} \left( \frac{|\sigma(0)|^2}{K^{(\ell)}} + \|\sigma\|_{\text{Lip}}^2 + \frac{2\|\sigma\|_{\text{Lip}}^2}{K^{(\ell)}} \right) 
+ \frac{2\|\sigma\|_{\text{Lip}}}{K^{(\ell)}} \left( |\sigma(0)| + \|\sigma\|_{\text{Lip}} |z_{1}^{(\ell)}(x)| \right)
\Bigg] \frac{h^2}{2} \\
+ \Bigg[
\frac{|z_{1}^{(\ell)}(x)|}{K^{(\ell)}}\left( D_\ell(x) + 6\|\sigma\|_{\text{Lip}}^2 |z_{1}^{(\ell)}(x)| \right) + \frac{3\|\sigma\|_{\text{Lip}}^2}{K^{(\ell)}}
\Bigg] \frac{|h|^3}{3} \\
+ \frac{|z_{1}^{(\ell)}(x)|}{K^{(\ell)}} \|\sigma\|_{\text{Lip}}^2 \left( \frac{4}{K^{(\ell)}} + 2 \right) \frac{h^4}{4}.
\end{multline*}

Recalling that
\[
h = \sum_{j=1}^{n_{\ell-1}} \frac{\sqrt{C_W}}{\sqrt{n_{\ell-1}}}(W_{1,j}^{(\ell)} - {W'}_{1,j}^{(\ell)})\sigma(z_{j}^{(\ell-1)}(x)),
\]
we derive the subsequent bound:
\begin{multline}\label{term_R_Q2}
\left| 
\sum_{j=1}^{n_{\ell-1}} \mathbb{E} \left[
\frac{\sqrt{C_W}}{\sqrt{n_{\ell-1}}}(W_{1,j}^{(\ell)} - {W'}_{1,j}^{(\ell)}) \sigma(z_{j}^{(\ell-1)}(x)) R \bigg| \mathcal{F}_{\ell-1}
\right]
\right| \\
\le \frac{C_W^{3/2}}{2n_{\ell-1}^{3/2}} \sum_{j=1}^{n_{\ell-1}} \mathbb{E} \Bigg[
|W_{1,j}^{(\ell)} - {W'}_{1,j}^{(\ell)}|^3 |\sigma(z_{j}^{(\ell-1)}(x))|^3 \cdot\\
\cdot \Bigg(
\frac{|z_{1}^{(\ell)}(x)|}{K^{(\ell)}}\left(D_\ell(x) + \frac{4\|\sigma\|_{\text{Lip}}^2}{K^{(\ell)}} |z_{1}^{(\ell)}(x)|^2 + 2\|\sigma\|_{\text{Lip}}^2 |z_{1}^{(\ell)}(x)|\right) \\
+ \sqrt{\frac{\pi}{2K^{(\ell)}}} \left( \frac{|\sigma(0)|^2}{K^{(\ell)}} + \|\sigma\|_{\text{Lip}}^2 + \frac{2\|\sigma\|_{\text{Lip}}^2}{K^{(\ell)}} \right)
+ \frac{2\|\sigma\|_{\text{Lip}}}{K^{(\ell)}}\left( |\sigma(0)| + \|\sigma\|_{\text{Lip}} |z_{1}^{(\ell)}(x)| \right)
\Bigg) \Bigg| \mathcal{F}_{\ell-1} \Bigg] \\
+ \frac{C_W^2}{3n_{\ell-1}^2} \sum_{j=1}^{n_{\ell-1}} \mathbb{E} \Bigg[
|W_{1,j}^{(\ell)} - {W'}_{1,j}^{(\ell)}|^4 \sigma^4(z_{j}^{(\ell-1)}(x)) \cdot \\
\cdot  \left( \frac{|z_{1}^{(\ell)}(x)|}{K^{(\ell)}} \left( D_\ell(x) + 6\|\sigma\|_{\text{Lip}}^2 |z_{1}^{(\ell)}(x)| \right) + \frac{3\|\sigma\|_{\text{Lip}}^2}{K^{(\ell)}} \right)
\Bigg| \mathcal{F}_{\ell-1} \Bigg] \\
+ \frac{C_W^{5/2}}{4n_{\ell-1}^{5/2}} \|\sigma\|_{\text{Lip}}^2 \left( \frac{4}{K^{(\ell)}} + 2 \right)
\sum_{j=1}^{n_{\ell-1}} \mathbb{E} \left[
|W_{1,j}^{(\ell)} - {W'}_{1,j}^{(\ell)}|^5 |\sigma(z_{j}^{(\ell-1)}(x))|^5 \frac{|z_{1}^{(\ell)}(x)|}{K^{(\ell)}}
\bigg| \mathcal{F}_{\ell-1}
\right].
\end{multline}

We now consider the term
\begin{multline*}
\left|\mathbb{E}\left[\left(\sum_{j=1}^{n_{\ell-1}} \frac{C_W}{2n_{\ell-1}} (W_{1,j}^{(\ell)} - {W'}_{1,j}^{(\ell)})^2 \sigma^2(z_{j}^{(\ell-1)}(x)) - K^{(\ell)} + C_b\right) 
f'_{\sigma^2}(z_{1}^{(\ell)}(x)) \,\Big|\, \mathcal{F}_{\ell-1} \right]\right|.
\end{multline*}

We split it as follows:
\begin{multline*}
\left|\mathbb{E}\left[\left(\sum_{j=1}^{n_{\ell-1}} \frac{C_W}{2n_{\ell-1}} (W_{1,j}^{(\ell)} - {W'}_{1,j}^{(\ell)})^2 \sigma^2(z_{j}^{(\ell-1)}(x)) - K^{(\ell)} + C_b\right) 
f'_{\sigma^2}(z_{1}^{(\ell)}(x)) \,\Big|\, \mathcal{F}_{\ell-1} \right]\right|\\
\le 
\left|\mathbb{E}\left[\left(\sum_{j=1}^{n_{\ell-1}} \frac{C_W}{n_{\ell-1}} \sigma^2(z_{j}^{(\ell-1)}(x)) - K^{(\ell)} + C_b\right) 
f'_{\sigma^2}(z_{1}^{(\ell)}(x)) \,\Big|\, \mathcal{F}_{\ell-1} \right]\right| \\
+ \left|\mathbb{E}\left[\left(\sum_{j=1}^{n_{\ell-1}} \frac{C_W}{2n_{\ell-1}} \left((W_{1,j}^{(\ell)} - {W'}_{1,j}^{(\ell)})^2 - 2\right) \sigma^2(z_{j}^{(\ell-1)}(x))\right) 
f'_{\sigma^2}(z_{1}^{(\ell)}(x)) \,\Big|\, \mathcal{F}_{\ell-1} \right]\right|.
\end{multline*}

By integrating with respect to $W'$ the second term,
\begin{multline}\label{boundbello}
\left|\mathbb{E}\left[\left(\sum_{j=1}^{n_{\ell-1}} \frac{C_W}{2n_{\ell-1}} (W_{1,j}^{(\ell)} - {W'}_{1,j}^{(\ell)})^2 \sigma^2(z_{j}^{(\ell-1)}(x)) - K^{(\ell)} + C_b\right) 
f'_{\sigma^2}(z_{1}^{(\ell)}(x)) \,\Big|\, \mathcal{F}_{\ell-1} \right]\right|\\
\le
\left|\mathbb{E}\left[\left(\sum_{j=1}^{n_{\ell-1}} \frac{C_W}{n_{\ell-1}} \sigma^2(z_{j}^{(\ell-1)}(x)) - K^{(\ell)} + C_b\right) 
f'_{\sigma^2}(z_{1}^{(\ell)}(x)) \,\Big|\, \mathcal{F}_{\ell-1} \right]\right| \\
+ \left|\mathbb{E}\left[\left(\sum_{j=1}^{n_{\ell-1}} \frac{C_W}{2n_{\ell-1}} \left((W_{1,j}^{(\ell)})^2 - 1\right) \sigma^2(z_{j}^{(\ell-1)}(x))\right) 
f'_{\sigma^2}(z_{1}^{(\ell)}(x)) \,\Big|\, \mathcal{F}_{\ell-1} \right]\right|.
\end{multline}

We will now use the following lemma, proved in Appendix~A:
\begin{lemma}\label{derivative}
For every \( z \in \mathbb{R} \), with $f_{\sigma^2}$ defined in \eqref{expr_f_s2_uni}, we have:
\begin{equation}\label{deriv_f}
|f'_{\sigma^2}(z)| \le  
\left(
\frac{4\|\sigma\|_{\mathrm{Lip}}^2}{K^{(\ell)}} + 4\sqrt{\frac{\pi}{2K^{(\ell)}}} \left(\frac{|\sigma(0)|^2}{K^{(\ell)}} + \|\sigma\|_{\mathrm{Lip}}^2\right)
\right) |z|^2 + D_\ell(x),
\end{equation}
where \( D_\ell(x) \) is defined in~\eqref{const_D}.
\end{lemma}

Define
\[
A_{\ell}(x) := \sum_{j=1}^{n_{\ell-1}} \frac{C_W}{n_{\ell-1}} \sigma^2(z_{j}^{(\ell-1)}(x)) - K^{(\ell)} + C_b.
\]
Then, we can bound the following term as
\begin{multline*}
\Bigg|\mathbb{E}\Bigg[\Bigg(\sum_{j=1}^{n_{\ell-1}} \frac{C_W}{2n_{\ell-1}} (W_{1,j}^{(\ell)} - {W'}_{1,j}^{(\ell)})^2 \sigma^2(z_{j}^{(\ell-1)}(x)) - K^{(\ell)} + C_b\Bigg) f'_{\sigma^2}(z_{1}^{(\ell)}(x)) \,\Big|\, \mathcal{F}_{\ell-1} \Bigg]\Bigg| \\
\leq \mathbb{E}\Bigg[\Bigg(\sum_{j=1}^{n_{\ell-1}} \frac{C_W}{2n_{\ell-1}} \big((W_{1,j}^{(\ell)})^2 - 1\big) \sigma^2(z_{j}^{(\ell-1)}(x))\Bigg)^2 \,\Big|\, \mathcal{F}_{\ell-1} \Bigg]^{1/2} 
\cdot \mathbb{E}\Big[ \big(f'_{\sigma^2}(z_{1}^{(\ell)}(x))\big)^2 \,\big|\, \mathcal{F}_{\ell-1} \Big]^{1/2} \\
+ |A_{\ell}(x)| \cdot \mathbb{E}\Big[ \big|f'_{\sigma^2}(z_{1}^{(\ell)}(x))\big| \,\big|\, \mathcal{F}_{\ell-1} \Big].
\end{multline*}

Using Lemma~\ref{derivative}, we further bound the expression as
\begin{multline*}
\Bigg|\mathbb{E}\Bigg[\Bigg(\sum_{j=1}^{n_{\ell-1}} \frac{C_W}{2n_{\ell-1}} (W_{1,j}^{(\ell)} - {W'}_{1,j}^{(\ell)})^2 \sigma^2(z_{j}^{(\ell-1)}(x)) - K^{(\ell)} + C_b\Bigg) f'_{\sigma^2}(z_{1}^{(\ell)}(x)) \,\Big|\, \mathcal{F}_{\ell-1} \Bigg]\Bigg| \\
\leq |A_{\ell}(x)| \cdot \Bigg( 
\left( \frac{4\|\sigma\|_{\text{Lip}}^2}{K^{(\ell)}} 
+ 4\sqrt{\frac{\pi}{2K^{(\ell)}}} \left( \frac{|\sigma(0)|^2}{K^{(\ell)}} + \|\sigma\|_{\text{Lip}}^2 \right) \right) 
\mathbb{E}\Big[ |z_{1}^{(\ell)}(x)|^2 \,\big|\, \mathcal{F}_{\ell-1} \Big] + D_\ell(x) \Bigg) \\
+ \mathbb{E}\Big[ (W_{1,1}^{(\ell)})^4 \Big]^{1/2} \cdot \frac{C_W}{2n_{\ell-1}} 
\Bigg( \sum_{j=1}^{n_{\ell-1}} \sigma^4(z_{j}^{(\ell-1)}(x)) \Bigg)^{1/2} \cdot \\
\cdot \Bigg(
\left( \frac{4\|\sigma\|_{\text{Lip}}^2}{K^{(\ell)}} 
+ 4\sqrt{\frac{\pi}{2K^{(\ell)}}} \left( \frac{|\sigma(0)|^2}{K^{(\ell)}} + \|\sigma\|_{\text{Lip}}^2 \right) \right) 
\mathbb{E}\Big[ |z_{1}^{(\ell)}(x)|^4 \,\big|\, \mathcal{F}_{\ell-1} \Big]^{1/2} + D_\ell(x)
\Bigg).
\end{multline*}

\begin{multline*}
\le  |A_{\ell}(x)|\Bigg(\Big(\frac{4\|\sigma\|_{\text{Lip}}^2}{K^{(\ell)}}+4\sqrt{\frac{\pi  }{2K^{(\ell)}}}\Big(\frac{|\sigma(0)|^2}{K^{(\ell)}}+\|\sigma\|_{\text{Lip}}^2\Big)\Big)\mathbb{E}\Big[|z_{1}^{(\ell)}(x)|^2\Big|\mathcal{F}_{\ell-1}\Big]+D_\ell(x)\Bigg)\\
+\mathbb{E}\Big[(W_{1,1}^{(\ell)})^4\Big]^{1/2}\frac{{C_W}}{2{n_{\ell-1}}} \Big(\sum_{j=1}^{n_{\ell-1}}\sigma^4(z_{j}^{(\ell-1)}(x))\Big)^{1/2}\cdot\\
\cdot \Bigg(\Big(\frac{4\|\sigma\|_{\text{Lip}}^2}{K^{(\ell)}}+4\sqrt{\frac{\pi  }{2K^{(\ell)}}}\Big(\frac{|\sigma(0)|^2}{K^{(\ell)}}+\|\sigma\|_{\text{Lip}}^2\Big)\Big)\mathbb{E}\Big[|z_{1}^{(\ell)}(x)|^4|\mathcal{F}_{\ell-1}\Big]^{1/2}+D_\ell(x)\Bigg)
\end{multline*}
\begin{multline}\label{bound_der_Q2}
\le  |A_{\ell}(x)|\Bigg(\Big(\frac{4\|\sigma\|_{\text{Lip}}^2}{K^{(\ell)}}+4\sqrt{\frac{\pi  }{2K^{(\ell)}}}\Big(\frac{|\sigma(0)|^2}{K^{(\ell)}}+\|\sigma\|_{\text{Lip}}^2\Big)\Big)\mathbb{E}\Big[|z_{1}^{(\ell)}(x)|^2\Big|\mathcal{F}_{\ell-1}\Big]+D_\ell(x)\Bigg)\\
+\mathbb{E}\Big[(W_{1,1}^{(\ell)})^4\Big]^{1/2}\frac{{C_W}}{2{n_{\ell-1}}} \Big(\sum_{j=1}^{n_{\ell-1}}\sigma^4(z_{j}^{(\ell-1)}(x))\Big)^{1/2}\cdot\\
\cdot \Bigg(\Big(\frac{4\|\sigma\|_{\text{Lip}}^2}{K^{(\ell)}}+4\sqrt{\frac{\pi  }{2K^{(\ell)}}}\Big(\frac{|\sigma(0)|^2}{K^{(\ell)}}+\|\sigma\|_{\text{Lip}}^2\Big)\Big)\mathbb{E}\Big[|z_{1}^{(\ell)}(x)|^4|\mathcal{F}_{\ell-1}\Big]^{1/2}+D_\ell(x)\Bigg)
\end{multline}

Hence, applying (\ref{eq_in_Q2}), (\ref{eq_sec_Q2}), (\ref{term_R_Q2}) and (\ref{bound_der_Q2}) it can be shown the following Lemma, which is proved in Appendix A.
\begin{lemma}\label{lemma_prim_Q2}
    \[
    \mathbb{E}\Big[|\mathbb{E}[\sigma^2(z_{1}^{(\ell)}(x))|\mathcal{F}_{\ell-1}]-\mathbb{E}[\sigma^2 (G_{1}^{(\ell)}(x))]|^{2p}\Big]
\le\frac{L_1^{(\ell)(p)}(x)}{n_{\ell-1}^p}+L_2^{(\ell)(p)}(x)\sqrt{Q_{4p}^{(\ell-1)}(x)},
\]
where
\begin{multline}\label{L_1_def}
L_1^{(\ell)(p)}(x)\le 
    {5^{2p-1}2^{10p+2}} C_W^{2p}(C_W+1)^{3p}\Big(1+\mathbb{E}[\sigma^{20p}(z_1^{(\ell-1)})]^{1/2}\Big)\Big(1+\mathbb{E}[|z_{1}^{(\ell)}(x)|^{12p}]^{1/2}\Big)\cdot\\
\cdot \Big(1+\mathbb{E}[(W_{1,1}^{(\ell)})^{20p}]^{1/2}\Big)\Big(|\sigma(0)|^2+\|\sigma\|_{\text{Lip}}^2\Big)^{2p}\Bigg({12}+{2K^{(\ell)}}+\frac{20}{K^{(\ell)}}+\frac{15}{(K^{(\ell)})^2}
\Bigg)^{2p},
\end{multline}
and 
\begin{multline}\label{L_2_def}
L_2^{(\ell)(p)}(x)\le5^{2p-1}2^{6p-1}\Big(1+\mathbb{E}\Big[|z_{1}^{(\ell)}(x)|^{8p}\Big]^{1/2}\Big)\Big(\|\sigma\|_{\text{Lip}}^{2}+|\sigma(0)|^2\Big)^{2p}\cdot\\
\cdot \Big(6+K^{(\ell)}+\frac{9}{K^{(\ell)}}+\frac{3}{(K^{(\ell)})^2}\Big)^{2p}.
\end{multline}
\end{lemma}

We are now going to show how the expression 
\[
\mathbb{E}\Big[|\mathbb{E}[\sigma^2(z_{1}^{(\ell)}(x))|\mathcal{F}_{\ell-1}]-\mathbb{E}[\sigma^2 (G_{1}^{(\ell)}(x))]|^{2p}\Big]
\]
arises from the definition  (\ref{Q_k}) of $Q_2^{(\ell)}(x)$.

It follows from  Lemma \ref{mom_p_iid0} that if $\ell\ge 2$
\begin{multline}\label{bound_Q_2_postlemma}
Q_{2p}^{(\ell)}(x):=\mathbb{E}\Bigg[\Big|\sum_{j=1}^{n_{\ell}}\frac{C_W}{n_{\ell}}\sigma^2(z_{j}^{(\ell)}(x))-K^{(\ell+1)}+C_b\Big|^{2p}\Bigg]\\
\le 2^{2p-1}\mathbb{E}\Bigg[\Big|\sum_{j=1}^{n_{\ell}}\frac{C_W}{n_{\ell}}\Big(\sigma^2(z_{j}^{(\ell)}(x))-\mathbb{E}[\sigma^2(z_{j}^{(\ell)}(x))]\Big)\Big|^{2p}\Bigg]\\
+2^{2p-1}\Big|\sum_{j=1}^{n_{\ell}}\frac{C_W}{n_{\ell}}\Big(\mathbb{E}[\sigma^2(z_{j}^{(\ell)}(x))]-\mathbb{E}[\sigma^2(G_{j}^{(\ell)}(x))]\Big)\Big|^{2p}\\
\le 2^{2p-1}p^{2p} \frac{C_W^{2p}}{n_{\ell}^{p}}\mathbb{E}\Big[\Big(\sigma^2(z_{1}^{(\ell)}(x))-\mathbb{E}[\sigma^2(z_{1}^{(\ell)}(x))]\Big)^{2p}\Big]\\
+2^{2p-1}{C_W}^{2p}\mathbb{E}\Bigg[\Big|\mathbb{E}\Big[\sigma^2(z_{1}^{(\ell)}(x))|\mathcal{F}_{\ell-1}\Big]-\mathbb{E}[\sigma^2(G_{1}^{(\ell)}(x))]\Big|^{2p}\Bigg]
\end{multline}
\begin{multline*}
\le 
2^{4p-1}p^{2p} \frac{C_W^{2p}}{n_{\ell}^{p}}\mathbb{E}\Big[\sigma^{4p}(z_{1}^{(\ell)}(x))\Big]
+2^{2p-1}{C_W}^{2p}\cdot\\
\cdot\mathbb{E}\Bigg[\Big|\mathbb{E}\Big[\sigma^2(z_{1}^{(\ell)}(x))|\mathcal{F}_{\ell-1}\Big]-\mathbb{E}[\sigma^2(G_{1}^{(\ell)}(x))]\Big|^{2p}\Bigg]
\end{multline*}

\[
=\frac{L_4^{(\ell)(p)}(x)}{n_{\ell}^p}+L_5^{(p)}\mathbb{E}\Big[\Big(\mathbb{E}\Big[\sigma^2(z_{1}^{(\ell)}(x))\Big|\mathcal{F}_{\ell-1}\Big]-\mathbb{E}\Big[\sigma^2(G_{1}^{(\ell)}(x))\Big]\Big)^{2p}\Big]
\]
\[
\le  \frac{L_4^{(\ell)(p)}(x)}{n_{\ell}^p}+L_5^{(p)}\Big(\frac{L_1^{(\ell)(p)}(x)}{n_{\ell-1}^p}+L_2^{(\ell)(p)}(x)\sqrt{Q_{4p}^{(\ell-1)}(x)}\Big),
\]
using Lemma \ref{lemma_prim_Q2} in the last inequality and defining

\begin{equation}\label{L_4}
{L_4^{(\ell)(p)}}(x):=2^{4p-1}p^{2p} {C_W^{2p}}\mathbb{E}\Big[\sigma^{4p}(z_{1}^{(\ell)}(x))\Big]
\end{equation}
and
\begin{equation}\label{L_5}
L_5^{(p)}:=2^{2p-1}C_W^{2p}.
\end{equation}

Define now
\[
L_6^{(\ell)(p)(n)}(x):=\frac{L_4^{(\ell)(p)}(x)}{n_{\ell}^p}+L_5^{(p)}\frac{L_1^{(\ell)(p)}(x)}{n_{\ell-1}^p}
\]
and
\[
L_7^{(\ell)(p)}(x):=L_5^{(p)}L_2^{(\ell)(p)}(x).
\]

One has that for $p\ge 2$ even integer
\begin{equation}\label{bound_iter}
Q_{2p}^{(\ell)}(x)\le 
\begin{cases}
L_6^{(\ell)(p)(n)}(x)+L_7^{(\ell)(p)}(x) \sqrt{Q_{4p}^{(\ell-1)}(x)} &\text{if $\ell\ge 2$}\\
\frac{L_4^{(1)(p)}(x)}{n_{1}^p}& \text{if $\ell=1$}.
\end{cases}
\end{equation}

The following Lemma can be proved by induction on $\ell\ge 1$ thanks to inequality (\ref{bound_iter}).
\begin{lemma}\label{iteration_Q2}
For $\ell\ge 1$ and $p\ge 1$ integers,
\begin{multline*}
        Q_{2p}^{(\ell)}(x)\le 
        \Big\{\max_{k=2,\dots,\ell;j=0,\dots,\ell-1}\Big(L_{7}^{(k)(2^jp)}(x)\Big)^{\frac{1}{2^j}}\Big\}^{\ell-1}\Bigg[ \Bigg(\frac{L_{4}^{(1)(2^\ell p)}(x)}{n_1^{2^{\ell-1} p}}\Bigg)^{\frac{1}{2^{\ell-1}}}\\
        + \sum_{k=0}^{\ell-2}\Big(L_{6}^{(\ell-k)(2^k p)(n)}(x)\Big)^{\frac{1}{2^k}}\Bigg]
    \end{multline*}
\end{lemma}
Using Lemma \ref{iteration_Q2} and the bound (\ref{moment_meglio}) { together with some routine estimates (omitted for the sake of brevity)}, one obtains the following Lemma.
\begin{lemma}\label{cost_schif}
Fix $p\ge 1$ integer and suppose that $C_b\ne 0$. Then
\begin{multline}\label{Q_2}
 Q_{2p}^{(\ell)}(x)
    \le M^{(\ell)}_p(x)^{\ell}\Big(\sum_{j=1}^{\ell}\frac{1}{n_j^p}\Big) \Big(1+(2C_W\|\sigma\|_{Lip}^2)^{\ell}\Big)^{2\ell p}
      (\ell + 1)^{6\ell p+16p}\cdot\\
       \Big( 4\sqrt{5\cdot 2^\ell p}+ K\frac{5\cdot 2^{\ell+1}p}{\log(5\cdot 2^{\ell+1}p)}
\Big)^{4\ell p+16p} \cdot\Bigg(9+\Big( 3\cdot 2^{4{(\ell-1)}p}p^{4p} 
+2^{12p+3}\Big)^2\Bigg)
 \cdot\\
\cdot\Bigg(41+ \Big( K\frac{5\cdot 2^{\ell+1}p}{\log(5\cdot 2^{\ell+1}p)}  \sqrt{C_W} \|\sigma\|_{\mathrm{Lip}} \mathbb{E}[(W_{1,1}^{(1)})^{5\cdot 2^{\ell+1}p}]^{1/(5\cdot 2^{\ell+1}p)} \Big)^{2(\ell - 1)}\Bigg)^{\ell}
    \end{multline}
    where
    \begin{multline*}
        M^{(\ell)}_p(x):={5^{2p}}2^{26p}(1+C_W)^{19p}\Big(\|\sigma\|_{Lip}+|\sigma(0)|+1\Big)^{36p} \left(1+\frac{\|x\|}{\sqrt{n_0}}\right)^{24p}\cdot\\
    \cdot \Big(2+C_b+\frac{1}{C_b}+\frac{1}{C_b^2}\Big)^{12p}\Big(1+E[(W_{1,1}^{(1)})^{5p\cdot 2^{\ell+1}}]^{\frac{1}{2^\ell}}\Big)^3
    \end{multline*} .
\end{lemma}
\begin{remark}
In the bound of the previous lemma one could obtain much sharper constants but with longer calculations. The purpose of this result is just to show how the parameters  $\{\ell,n_0,n_1,\dots,n_\ell\}$ and the choice of the data $x$ can affect the convergence to the Gaussian law.
\end{remark}

\subsection{Final result}
Thanks to Lemma \ref{cost_schif} and to the bound (\ref{moment_meglio}), Theorem \ref{Kdist} and Theorem \ref{Wdist} read respectively as follows.
\begin{theorem}\label{finale_uno}
Let \( L \ge 1 \) be an integer, and suppose \( C_b, C_W \neq 0 \), and \( x \in \mathbb{R}^{n_0} \). Let \( z_1^{(L+1)}(x) \), \( \sigma \), and \( W \) be defined as in Definition~\ref{def_NN}. If \( \mathbb{E}[W^{5\cdot 2^{L+1}}] < \infty \), then the following bound holds:
\begin{multline*}
    \max\Big\{d_K(z_{1}^{(L+1)}(x),G_{1}^{(L+1)}(x)),W_1(z_{1}^{(L+1)}(x),G_{1}^{(L+1)}(x))\Big\}
    \le 7\cdot 2^{4L+16}\Big(1+\frac{1}{C_b}\Big)\cdot\\
\cdot
\left({M^{(L)}_1(x)}\right)^{\frac{L+1}{2}} \Big(\sum_{j=1}^{L}\frac{1}{n_j}\Big)^{1/2} \Big(1+(2C_W\|\sigma\|_{Lip}^2)^{L}\Big)^{L+1}
      \Big( 4\sqrt{5\cdot 2^{L} }+ K\frac{5\cdot 2^{L+1}}{\log(5\cdot 2^{L+1})}\Big)^{2L+10}\cdot\\
        \cdot 
\ell^{3L+11}\Bigg(8+ \Big( 6K\frac{5\cdot 2^{L+1}}{\log(5\cdot 2^{L+1})}  \sqrt{C_W} \|\sigma\|_{\mathrm{Lip}} \mathbb{E}[(W_{1,1}^{(1)})^{5\cdot 2^{L+1}}]^{1/(5\cdot 2^{L+1})} \Big)^{L}\Bigg)^{\frac{L}{2}+3} ,
\end{multline*}
where the constant $K>0$ is defined in Lemma \ref{mom_p_iid0} and the quantity $M^{(L)}_1(x)$ is defined in Lemma \ref{cost_schif} and depends on \[
C_W,C_b,\|\sigma\|_{\text{Lip}},|\sigma(0)|,\|x\|, n_0, \mathbb{E}[(W_{1,1}^{(1)})^{5\cdot 2^{L+1}}]^{\frac{1}{2^{L}}}.
\]

\end{theorem}

\subsection{Examples with Gaussian initialization}\label{Gaus_example}
{
In this section we study the special case of {Gaussian initialization} under specific activation functions.  
In this setting the computations simplify substantially, which allows us to obtain sharper bounds on the dependence on the depth $L$ and on the inner widths $n_1,\dots,n_L$.  
Moreover, the Gaussian assumption on the weights allows one to deduce a direct comparison with the optimal results of \cite{FHMNP} for the Wasserstein and total variation distances in the single–input case, and with \cite{CP25} for the total variation distance in the multi–input case.  
As will be shown in the following examples, although we lose optimality with respect to the inner width $n$, we gain fully explicit constants and the ability to analyze the joint asymptotic regime
\[
L,n \to \infty, \qquad \frac{L}{n} \to 0 .
\]

}
\begin{example}{\bf (Perceptron activation function)}\label{es_1_g_RL}
Assume that the random variable $W$ introduced in Definition~\ref{def_NN} is distributed as $W \sim \mathcal{N}_1(0,1)$ and that the activation function is the indicator $\sigma(x) := \mathbf{1}_{\{x \ge 0\}}$. 
In this case the limiting covariance matrix in $K^{(L+1)}=C_b+\frac{C_W}{2}\ne 0$ for every $L\ge 1$ since $C_W$ must be always not zero. Because of this observation, we know that the neural network is not going to converge to a degenerate Gaussian law (i.e. $G^{(L+1)}=0$) and hence we do not have to worry about normalizing the neural network to compute the Wasserstein distance (see Remark \ref{rem_Kne0imp}).
For every $\ell=1,\dots,L$ and $x\in\mathbb{R}^{n_0}$ one obtains the bound
\[
Q_2^{(\ell)}(x) \le 11\left(\frac{C_W}{2}\right)^2 \frac{1}{n_\ell}.
\]
This can be seen by starting from the bound~\eqref{bound_Q_2_postlemma} with $p=1$ and observing that
\[
\mathbb{E}\left[\left(\sigma^2(z_1^{(\ell)}(x)) - \mathbb{E}[\sigma^2(G_1^{(\ell)}(x))]\right)^2\right] = \frac{1}{4},
\]
and
\[
\mathbb{E}\left[\prod_{j=1}^{2}\left(\mathbb{E}[\sigma^2(z_j^{(\ell)}(x)) \mid \mathcal{F}_{\ell-1}] - \mathbb{E}[\sigma^2(G_j^{(\ell)}(x))]\right)\right] = 0,
\]
since $z_1^{(\ell)}(x)$ is conditionally Gaussian with zero mean and independent components (see, e.g., Lemma 7.1 in~\cite{Han_Gas}).
Thanks to Theorems~\ref{Kdist} and~\ref{Wdist}, we then obtain the following explicit bounds:
\begin{multline*}
d_K(z_1^{(L+1)}(x), G_1^{(L+1)}(x)) \\
\le \frac{8C_W}{\sqrt{n_{L}}(C_b + \frac{C_W}{2})} \left( \frac{2\sqrt{C_W}}{\sqrt{C_b + \frac{C_W}{2}}} + 2\sqrt{C_W}\left(\frac{\sqrt{C_W}}{\sqrt{C_b + \frac{C_W}{2}}} + \frac{\sqrt{2\pi}}{4}\right) + \frac{7}{2} \right),
\end{multline*}
and
\[
W_1(z_1^{(L+1)}(x), G_1^{(L+1)}(x)) \le \frac{6C_W}{\sqrt{n_{L}}\sqrt{C_b + \frac{C_W}{2}}} \left( \frac{3\sqrt{C_W}}{\sqrt{C_b + \frac{C_W}{2}}} + \frac{1}{\sqrt{2\pi}} \right).
\]
It follows that the convergence in law of the neural network output to a Gaussian limit holds as $n_{L} \to \infty$, even when $L \to \infty$.
\end{example}

\begin{example}{\bf (Linear activation function)}\label{ex_id_G_1d}
Assume that the random variable $W$ introduced in Definition~\ref{def_NN} is distributed as $W \sim \mathcal{N}_1(0,1)$ and that the activation function is the identity, i.e., $\sigma(x) := x$.
In this case for any $\ell\ge 1$
\[
K^{(\ell)}=C_b\sum_{k=0}^{\ell-1}C_W^k+\frac{C_W^{\ell}}{n_0}\|x\|^2
\]
and therefore we have that if $C_W<1$ then $K^{(\ell)}\to 0$ when $C_b=0$.
Using the definition~\eqref{Q_k} with $p=2$ and a recursive argument, one can show that for every $\ell = 1, \dots, L$ and $x \in \mathbb{R}^{n_0}$,
\[
Q_2^{(\ell)}(x) \le \sum_{k=2}^{\ell} \frac{C^{(k)}}{n_k} C_W^{2(\ell+1-k)} + C_W^{2(\ell-1)} Q_2^{(1)},
\]
where
\begin{multline*}
C^{(k)} = \mathbb{E}\left[\left((z_1^{(k)}(x))^2 - \mathbb{E}[(G_1^{(k)}(x))^2]\right)^2\right]
\le 2\mathbb{E}[(z_1^{(k)}(x))^4] + 2\mathbb{E}[(G_1^{(k)}(x))^2]^2 \\
= 6\,\mathbb{E}[(z_1^{(k)}(x))^2]^2 + 2\,\mathbb{E}[(G_1^{(k)}(x))^2]^2 
= 8\left(C_b \sum_{r=0}^{k-2} C_W^r + C_W^{k-1} \mathbb{E}[(z_1^{(1)}(x))^2]\right)^2,
\end{multline*}
where we used the conditional Gaussianity of the neural network as in Example~\ref{es_1_g_RL}, along with another inductive argument.
Assuming now the critical initialization $C_b = 0$ and $C_W = 1$  (see Section \ref{subsec_crit}), we do not have a degenerate limit for $L\to\infty$ and hence we do not need to normalize by the standard deviation of the Gaussian limit as in the previous example (see Remark \ref{rem_Kne0imp}). Applying Theorems~\ref{Kdist} and~\ref{Wdist}, we obtain the following bounds:
\begin{equation}\label{dKid}
d_K(z_1^{(L+1)}(x), G_1^{(L+1)}(x)) \le 2\sqrt{15} \left( \frac{2n_0}{\|x\|^2} + \frac{5\sqrt{n_0}}{\|x\|} + \sqrt{15}\|x\|^2 + 1 \right)^2 \sqrt{\sum_{\ell=1}^{L} \frac{1}{n_\ell}},
\end{equation}
and
\begin{equation}\label{dWid}
W_1(z_1^{(L+1)}(x), G_1^{(L+1)}(x)) \le 4\left(1 + \frac{\sqrt{n_0}}{\|x\|} + 3\|x\|^2\right) \left(4\frac{\sqrt{n_0}}{\|x\|} + \frac{\|x\|}{\sqrt{n_0}} + 1\right) \sqrt{\sum_{\ell=1}^{L} \frac{1}{n_\ell}}.
\end{equation}
These follow from the estimates:
\[
Q_2^{(1)}(x) \le \frac{3\|x\|^4}{n_1 n_0^2}, \qquad \mathbb{E}[(G_1^{(1)}(x))^2] = \mathbb{E}[(z_1^{(1)}(x))^2] = \frac{\|x\|^2}{n_0}.
\]
From inequalities~\eqref{dKid} and~\eqref{dWid}, it follows that if $n_1 , \dots , n_L \asymp n$ and $\frac{L}{n} \to 0$, then $z_1^{(L+1)}(x)$ converges in law to $G_1^{(L+1)}(x)$, partially recovering Proposition~3.1 in~\cite{BLR24} as a special case.
\end{example}

\begin{example}{\bf (ReLU activation function)}\label{ex_relu_1d}
Assume the random variable $W$ introduced in Definition~\ref{def_NN} is distributed as $W \sim \mathcal{N}_1(0,1)$ and the activation function is the ReLU, i.e., $\sigma(x) := x\mathbf{1}_{\{x \ge 0\}}$. Then, for every integer $p \ge 1$, $\ell=1,\dots,L$ and $x\in\mathbb{R}^{n_0}$ one has
\[
\mathbb{E}[(G_1^{(\ell)}(x))^{2p}] = \frac{1}{2} \mathbb{E}[(G_1^{(\ell)}(x))^{2p}], \qquad 
\mathbb{E}[(z_1^{(\ell)}(x))^{2p}] = \frac{1}{2} \mathbb{E}[(z_1^{(\ell)}(x))^{2p}].
\]
Moreover for every $\ell\ge 2$
\[
K^{(\ell)}=C_b\sum_{k=0}^{\ell-1}\left(\frac{C_W}{2}\right)^k+\left(\frac{C_W}{2}\right)^{\ell}\frac{\|x\|^2}{n_0},
\]
which means that if $C_b=0$ and $C_W<2$ then $K^{(\ell)}\to 0$ as $\ell\to\infty$.
Applying this observation to inequality~\eqref{bound_Q_2_postlemma}, it is possible to show that
\[
Q_2^{(\ell)}(x) \le \frac{96 C_W^2}{n_\ell} \left( \mathbb{E}[(z_1^{(\ell)}(x))^2]^2 + \mathbb{E}[(G_1^{(\ell)}(x))^2]^2 \right) + \frac{C_W^2}{4} Q_2^{(\ell-1)}(x),
\]
and hence, iteratively,
\[
Q_2^{(\ell)}(x) \le 96 C_W^2 \sum_{k=2}^{\ell} \frac{C_W^{2(\ell-k)}}{2^{2(\ell-k)} n_k} \left( \mathbb{E}[(z_1^{(k)}(x))^2]^2 + \mathbb{E}[(G_1^{(k)}(x))^2]^2 \right) + \left( \frac{C_W}{2} \right)^{2(\ell-1)} Q_2^{(1)}(x).
\]
Moreover, by induction, one has
\[
\mathbb{E}[(z_1^{(\ell)}(x))^2],\ \mathbb{E}[(G_1^{(\ell)}(x))^2] = C_b \sum_{k=0}^{\ell-2} \left( \frac{C_W}{2} \right)^k + \left( \frac{C_W}{2} \right)^{\ell-1} \left( C_b + C_W \frac{\|x\|^2}{n_0} \right).
\]
To simplify computations, let us assume $C_b = 0$. Then
\[
Q_2^{(\ell)}(x) \le 192 \left( \frac{C_W}{2} \right)^{2(\ell+1)} \left( 1 + \left( \frac{C_W}{2} \right)^{2(\ell-1)} \right) \frac{\|x\|^4}{n_0^2} \sum_{k=1}^{\ell} \frac{1}{n_k},
\]
using that in general
\[
Q_2^{(1)}(x) \le \frac{3C_W^2}{2n_1} \left( C_b + C_W \frac{\|x\|^2}{n_0} \right)^2.
\]
Now, using inequalities~\eqref{tv_1dim_primaTH} and~\eqref{W1_1dim_primaTH} (since Theorems~\ref{Kdist} and~\ref{Wdist} are suboptimal in this case), and normalizing the random vectors because of Remark \ref{rem_Kne0imp}, we obtain the following bounds for the Kolmogorov and Wasserstein distances:
\begin{multline*}
d_{K}\left(\frac{z_1^{(L+1)}(x)}{\sqrt{K^{(L+1)}}}, \frac{G_1^{(L+1)}(x)}{\sqrt{K^{(L+1)}}}\right)=d_{K}\left(z_1^{(L+1)}(x),G_1^{(L+1)}(x)\right)\\
\le 3\sqrt{5} \sqrt{ \sum_{\ell=1}^{L} \frac{1}{n_\ell} } \left( 7 + 2 \left( \frac{C_W}{2} \right)^{L-1} + \sqrt{5\pi} \frac{\|x\|}{\sqrt{n_0}} \left( \frac{C_W}{2} \right)^{\frac{L+1}{2}} \right),
\end{multline*}
and
\begin{multline*}
W_1\left(\frac{z_1^{(L+1)}(x)}{\sqrt{K^{(L+1)}}}, \frac{G_1^{(L+1)}(x)}{\sqrt{K^{(L+1)}}}\right) =\frac{1}{\sqrt{K^{(L+1)}}}d_{K}\left(z_1^{(L+1)}(x),G_1^{(L+1)}(x)\right)\\
\le 8  \left( 1 + \left( \frac{C_W}{2} \right)^{L-1} \right) \sqrt{ \sum_{\ell=1}^{L} \frac{1}{n_\ell} }.
\end{multline*}
These results allow us to study the convergence in distribution of the neural network output to a Gaussian distribution as both the width and depth of the network increase. For instance, assuming the critical initialization \( C_b=0, C_W = 2 \) (see Section \ref{subsec_crit}) and assuming the hidden layer widths satisfy \( n_1, \dots, n_{L} \asymp n \), convergence holds provided that
\[
\frac{L}{n} \to 0.
\]
In particular, the following bounds hold:
\[
d_{K}\left(\frac{z_1^{(L+1)}(x)}{\sqrt{K^{(L+1)}}}, \frac{G_1^{(L+1)}(x)}{\sqrt{K^{(L+1)}}}\right), W_1\left(\frac{z_1^{(L+1)}(x)}{\sqrt{K^{(L+1)}}}, \frac{G_1^{(L+1)}(x)}{\sqrt{K^{(L+1)}}}\right)= O\left( \sqrt{ \frac{L}{n} } \right).
\]

\end{example}

\subsection{A special case: Lipschitz and bounded activation function}\label{lip_bound_s_sec}
In this section, we assume that the activation function $\sigma$ satisfies the following regularity conditions: for all $x, y \in \mathbb{R}$, there exist constants $\|\sigma\|_{\infty}, \|\sigma\|_{\text{Lip}} < \infty$ such that
\begin{equation} \label{lip_bound_s}
|\sigma(x)| \le \|\sigma\|_{\infty} \quad \text{and} \quad |\sigma(x) - \sigma(y)| \le \|\sigma\|_{\text{Lip}} |x - y|.
\end{equation}
Examples of activation functions satisfying this assumption include the logistic sigmoid and the hyperbolic tangent, which are both commonly used in practice.

\smallskip

Using inequality \eqref{bound_Q_2_postlemma}, we obtain the following bound for every $\ell=1,\dots,L$ and $x\in\mathbb{R}^{n_0}$:
\begin{multline} \label{bound_Q2_lip_bound}
Q_2^{(\ell)}(x) \le \frac{C_W^2}{n_\ell} \mathbb{E} \left[ \left( \sigma^2(z_1^{(\ell)}(x)) - \mathbb{E}[\sigma^2(G_1^{(\ell)}(x))] \right)^2 \right] \\
+ C_W^2 \mathbb{E} \left[ \left( \mathbb{E}[\sigma^2(z_1^{(\ell)}(x)) \mid \mathcal{F}_{\ell-1}] - \mathbb{E}[\sigma^2(G_1^{(\ell)}(x))] \right)^2 \right] \\
\le \frac{4C_W^2 \|\sigma\|_{\infty}^4}{n_\ell} + C_W^2 \mathbb{E} \left[ W_1(z_1^{(\ell)}(x) \mid \mathcal{F}_{\ell-1}, G_1^{(\ell)}(x))^2 \right],
\end{multline}
where $\mathcal{F}_{\ell-1}$ denotes the $\sigma$-algebra generated by the weights and biases up to layer $\ell-1$ and $z_1^{(\ell)}(x) \mid \mathcal{F}_{\ell-1}$ denotes the conditional law of the neuron output. Under assumption \eqref{lip_bound_s}, the function $\sigma^2$ is Lipschitz, which justifies the use of the Wasserstein distance in the last term.

Applying Theorem~\ref{th_wass_orig}, we obtain the bound:
\begin{multline*}
W_1(z_1^{(\ell)}(x) \mid \mathcal{F}_{L}, G_1^{(\ell)}(x)) \\
\le \sqrt{\frac{2}{K^{(\ell)} \pi}} \mathbb{E} \left[ \left( K^{(\ell)} - C_b - \frac{1}{2} \sum_{j=1}^{n_{\ell-1}} \frac{C_W}{n_{\ell-1}} \mathbb{E} \left[ (W_{1,j}^{(\ell)} - W_{1,j}^{\prime(\ell)})^2 \mid W_{1,j}^{(\ell)} \right] \sigma^2(z_j^{(\ell-1)}(x)) \right)^2 \Bigg| \mathcal{F}_{\ell-1} \right]^{1/2} \\
+ \frac{1}{2K^{(\ell)}} \sum_{j=1}^{n_{\ell-1}} \mathbb{E} \left[ \left| \frac{\sqrt{C_W}}{\sqrt{n_{\ell-1}}} (W_{1,j}^{(\ell)} - W_{1,j}^{\prime(\ell)}) \sigma^2(z_j^{(\ell-1)}(x)) \right|^3 \Big| \mathcal{F}_{\ell-1} \right].
\end{multline*}
This can be bounded further using uniform bounds on $\sigma$:
\begin{multline*}
W_1(z_1^{(\ell)}(x) \mid \mathcal{F}_{L}, G_1^{(\ell)}(x))\le \sqrt{\frac{2}{K^{(\ell)} \pi}} \left( K^{(\ell)} - C_b - \frac{1}{2} \sum_{j=1}^{n_{\ell-1}} \frac{C_W}{n_{\ell-1}} \sigma^2(z_j^{(\ell-1)}(x)) \right)\\
+ \frac{4 \|\sigma\|_{\infty}^6 C_W^{3/2}}{K^{(\ell)} \sqrt{n_{\ell-1}}} \mathbb{E}[|W_{1,j}^{(\ell)}|^3] 
+ \sqrt{\frac{2}{K^{(\ell)} \pi}} \cdot \frac{\|\sigma\|_{\infty}^2}{2} \cdot \frac{C_W}{\sqrt{n_{\ell-1}}} \mathbb{E}[(W_{1,1}^{(\ell)})^4]^{1/2}.
\end{multline*}
Taking the square and expectation of both sides, we obtain:
\begin{multline*}
\mathbb{E} \left[ W_1(z_1^{(\ell)}(x) \mid \mathcal{F}_{\ell-1}, G_1^{(\ell)}(x))^2 \right] 
\le \frac{6}{K^{(\ell)} \pi} \mathbb{E} \left[ \left( K^{(\ell)} - C_b - \frac{1}{2} \sum_{j=1}^{n_{\ell-1}} \frac{C_W}{n_{\ell-1}} \sigma^2(z_j^{(\ell-1)}(x)) \right)^2 \right] \\
+ \frac{48 \|\sigma\|_{\infty}^{12} C_W^3}{(K^{(\ell)})^2 n_{\ell-1}} \mathbb{E}[|W_{1,j}^{(\ell)}|^3]^2
+ \frac{3}{K^{(\ell)} \pi} \cdot \frac{\|\sigma\|_{\infty}^4 C_W^2}{2 n_{\ell-1}} \mathbb{E}[(W_{1,1}^{(\ell)})^4].
\end{multline*}
Note that the first expectation on the right-hand side corresponds to $Q_2^{(\ell-1)}(x)$. Hence,
\begin{multline*}
\mathbb{E} \left[ W_1(z_1^{(\ell)}(x) \mid \mathcal{F}_{\ell-1}, G_1^{(\ell)}(x))^2 \right]
\le \frac{6}{K^{(\ell)} \pi} Q_2^{(\ell-1)}(x)
+ \frac{48 \|\sigma\|_{\infty}^{12} C_W^3}{(K^{(\ell)})^2 n_{\ell-1}} \mathbb{E}[|W_{1,j}^{(\ell)}|^3]^2\\
+ \frac{3}{K^{(\ell)} \pi} \cdot \frac{\|\sigma\|_{\infty}^4 C_W^2}{2 n_{\ell-1}} \mathbb{E}[(W_{1,1}^{(\ell)})^4].
\end{multline*}

\smallskip

Inserting this into inequality \eqref{bound_Q2_lip_bound} and iterating, we obtain the recursive bound
\begin{multline}\label{Q_2_lip_bound}
Q_2^{(\ell)}(x) \le \sum_{k=2}^{\ell} \frac{C_W^2 \|\sigma\|_{\infty}^4}{n_{k-1}} \left( 4 + \frac{48 \|\sigma\|_{\infty}^8 C_W^3 \mathbb{E}[(W_{1,1}^{(1)})^3]^2}{(K^{(k)})^2} + \frac{3 C_W^2 \mathbb{E}[(W_{1,1}^{(1)})^4]}{2\pi K^{(k)}} \right) \prod_{j=k+1}^{\ell} \frac{6C_W}{\pi K^{(j)}} \\
+ \frac{4C_W^2 \|\sigma\|_{\infty}^4}{n_1} \prod_{k=2}^{\ell} \frac{6C_W}{\pi K^{(k)}}.
\end{multline}

Finally, applying Theorems~\ref{Kdist} and~\ref{Wdist}, we derive the following bound for the Kolmogorov and Wasserstein distances:
\begin{multline*}
\max \left\{ d_K(z_1^{(L+1)}(x), G_1^{(L+1)}(x)), \, W_1(z_1^{(L+1)}(x), G_1^{(L+1)}(x)) \right\} \le
\frac{C_W}{\sqrt{n_{L}}} \left(1 + \frac{1}{K^{(L+1)}} \right) \cdot \\
\cdot \left( 4 \sqrt{\frac{C_W}{K^{(L+1)}}}
+ 2 \sqrt{C_W} \left( \sqrt{\frac{C_b}{K^{(L+1)}}} + \frac{\sqrt{2\pi}}{4} \right) + \frac{5}{2} \right)
\left( \mathbb{E}[\sigma^6(z_1^{(L)}(x))]^{1/2} + 1 \right)
\mathbb{E}[(W_{1,1}^{(1)})^6]^{1/2} \\
+ \left(1 + \frac{1}{K^{(L+1)}} \right) \sqrt{Q_2^{(L)}(x)}.
\end{multline*}

Using the boundedness of the activation function and inequality \eqref{Q_2_lip_bound}, we further estimate the above as:
\begin{multline*}
\max \left\{ d_K(z_1^{(L+1)}(x), G_1^{(L+1)}(x)), \, W_1(z_1^{(L+1)}(x), G_1^{(L+1)}(x)) \right\}\\
\le
\frac{C_W(1+\sqrt{C_W})}{\sqrt{n_{L}}}\left(2+\frac{1}{K^{(L+1)}}\right)^2\left(8+2\sqrt{{C_b}}\right)\left(\|\sigma\|_{\infty}^3+1\right)\mathbb{E}[(W_{1,1}^{(1)})^6]^{1/2}\\
+\left(1+\frac{1}{K^{(L+1)}}\right)\left(\sum_{k=2}^{L+1}\frac{C_W^2\|\sigma\|_{\infty}^4}{{n_{k-1}}}\left(4+\frac{48\|\sigma\|_{\infty}^{8}C_W^{3}\mathbb{E}[(W_{1,1}^{(1)})^3]^2}{(K^{(k)})^2}+\frac{6\pi C_W^2 \mathbb{E}[(W_{1,1}^{(1)})^4]}{4\pi^2 K^{(k)}}\right)\cdot\right.\\
\left.\cdot \prod_{j=k+1}^{L+1}\frac{6C_W}{\pi K^{(j)}}\right)^{1/2}
+\left(1+\frac{1}{K^{(L+1)}}\right)\frac{2C_W\|\sigma\|_{\infty}^2}{\sqrt{n_1}}\left(\prod_{k=2}^{L+1}\frac{6C_W}{\pi K^{(k)}}\right)^{1/2}.
\end{multline*}

We summarize this estimate in the following proposition:

\begin{prop}
Assume that \( C_b \neq 0 \), \( \mathbb{E}[(W_{1,1}^{(1)})^6] < \infty \), and that the activation function \( \sigma \) satisfies the boundedness condition in~\eqref{lip_bound_s}. Then the following bound holds:
\begin{multline*}
\max\left\{d_K(z_{1}^{(L+1)}(x),G_{1}^{(L+1)}(x)),\, W_1(z_{1}^{(L+1)}(x),G_{1}^{(L+1)}(x))\right\}\\
\le
\frac{C_W(1+\sqrt{C_W})}{\sqrt{n_{L}}}\left(2+\frac{1}{C_b}\right)^2\left(8+2\sqrt{C_b}\right)\left(\|\sigma\|_{\infty}^3+1\right)\mathbb{E}[(W_{1,1}^{(1)})^6]^{1/2}\\
+\left(1+\frac{1}{C_b}\right)\left(2+\frac{4\sqrt{3}\|\sigma\|_{\infty}^{4}C_W^{3/2}\mathbb{E}[(W_{1,1}^{(1)})^3]}{C_b}+\frac{\sqrt{6\pi}C_W\mathbb{E}[(W_{1,1}^{(1)})^4]^{1/2}}{2\pi\sqrt{C_b}}\right) \cdot\\
\cdot 
\left(\sum_{k=2}^{L+1}\frac{C_W^2\|\sigma\|_{\infty}^4}{{n_{k-1}}}\left(\frac{6C_W}{\pi C_b}\right)^{L+1-k}\right)^{1/2}
+\left(1+\frac{1}{C_b}\right)\frac{2C_W\|\sigma\|_{\infty}^2}{\sqrt{n_1}}\left(\frac{6C_W}{\pi C_b}\right)^{\frac{L}{2}}.
\end{multline*}
\end{prop}

\begin{remark}\label{rem_impr_lip_bound_prop}
This Proposition highlights how the assumption of a bounded activation function significantly improves the estimates for the Kolmogorov and Wasserstein distances between the neural network output and the Gaussian limit, compared to the bounds obtained in Theorem~\ref{finale_uno}. The key improvements are:
\begin{itemize}
    \item In Theorem~\ref{finale_uno}, we require the finiteness of moments up to order \(5 \cdot 2^{L+1}\), whereas here it suffices to assume that \( \mathbb{E}[(W_{1,1}^{(1)})^6] < \infty \).
    
    \item The bound in the Proposition grows exponentially with the depth \( L \), while the one in the theorem grows exponentially with a polynomial of $L$. This implies that convergence in distribution for both width and depth tending to infinity can be achieved under weaker constraints on the growth of \( L \).
    
    \item The bound in the Proposition is more tractable when analyzing convergence regimes as \( L, n_1, \dots, n_{L} \to \infty \). For instance, if \( n_1 , \dots , n_{L} \asymp n \), then:
    \begin{enumerate}
    \item[(I)] If $\frac{6C_W}{\pi C_b} \le 1$ then
    \[
    \max\left\{d_K(z_{1}^{(L+1)}(x),G_{1}^{(L+1)}(x)),\, W_1(z_{1}^{(L+1)}(x),G_{1}^{(L+1)}(x))\right\} \to 0
    \]
    when $ \frac{L}{n} \to 0$;
    \item[(II)] If $\frac{6C_W}{\pi C_b} > 1$ then
    \[
    \max\left\{d_K(z_{1}^{(L+1)}(x),G_{1}^{(L+1)}(x)),\, W_1(z_{1}^{(L+1)}(x),G_{1}^{(L+1)}(x))\right\} \to 0
    \]
    when $ \frac{1}{L} \ln\left(\frac{L}{n}\right) \to -\infty  $.
    \end{enumerate}
\end{itemize}
\end{remark}

\section{Multi-dimensional problem}\label{multi_dimensional_problem_sec}
\subsection{Convex distance}
Let the notation of Subsection \ref{sec:def_nn} prevail,
fix $d \in \mathbb{N}$ and consider $d$ distinct inputs 
\[
\mathcal{X} := \{x^{(1)}, \dots, x^{(d)}\} \subseteq \mathbb{R}^{n_0}.
\]
We are interested in studying the convex distance (see Definition~\ref{d_c_def}) between the law of
\[
z_1^{(L+1)}(\mathcal{X}) := \big(z_1^{(L+1)}(x^{(1)}), \dots, z_1^{(L+1)}(x^{(d)})\big)
\]
and that of its Gaussian limit
\[
G_1^{(L+1)}(\mathcal{X}) := \big(G_1^{(L+1)}(x^{(1)}), \dots, G_1^{(L+1)}(x^{(d)})\big).
\]
For any vector $X = (X_1, \dots, X_d) \in \mathbb{R}^d$ and given an activation function $\sigma:\mathbb{R}\to\mathbb{R}$, we use the shorthand notation
\[
\sigma(X) := \big(\sigma(X_1), \dots, \sigma(X_d)\big).
\]

\smallskip

Assuming that the limiting covariance matrix $K^{(L)} := K^{(L)}(\mathcal{X})$ (as defined in Theorem \ref{Hanin}) is invertible, our goal is to apply the forthcoming Theorem~\ref{KG22th}, for which some further notation is needed.

Let $X_1, \dots, X_n$ be independent random variables taking values in $\mathbb{R}$, and define the random vector
\[
X := (X_1, \dots, X_n).
\]
Let $\tilde{X} := (\tilde{X}_1, \dots, \tilde{X}_n)$ and $X' := (X_1', \dots, X_n')$ be independent copies of $X$, such that the random vectors $X$, $\tilde{X}$, and $X'$ are mutually independent.
Let $f : \mathbb{R}^n \to \mathbb{R}^d$ be a measurable function such that the random variable
\[
W := f(X_1, \dots, X_n)
\]
satisfies the conditions $\mathbb{E}[W] = 0$ and $\mathbb{E}[\|W\|] < \infty$. Let $\Sigma \in \mathbb{R}^{d \times d}$ be a non-negative definite matrix, and let $N_{\Sigma} \sim \mathcal{N}_d(0, \Sigma)$.

For every $n \in \mathbb{N}$, $j \in \{1, \dots, n\}$, and $A \subseteq \{1, \dots, n\}$ (with $|A|$ denoting the cardinality of the set $A$), define the following quantities:
\[
[n] := \{1, \dots, n\},
\]
\[
k_{n,A} := \frac{1}{\binom{n}{|A|}(n - |A|)},
\]
\[
\Delta_j f(X) := f(X_1, \dots, X_n) - f(X_1, \dots, X_j', \dots, X_n),
\]
\[
\tilde{\Delta}_j g(X, X') := g(X, X') - g((X_1, \dots, \tilde{X}_j, \dots, X_n), X').
\]

\[
(X^A)_j:=\begin{cases}
    X_j\quad\text{if}\quad j\notin A,\\
    X_j'\quad\text{if}\quad j\in A.
\end{cases}
\]
\[
\gamma_1:=\sum_{j=1}^{n}\mathbb{E}\Big[\|\Delta_j f(X)\|^3\Big],
\]
\[
\gamma_2:=\Big(\sum_{j=1}^{n}\mathbb{E}\Big[\|\Delta_j f(X)\|^4\Big]\Big)^{1/2},
\]
\begin{multline*}
\gamma_{p+2}:=\Bigg\{\frac{3}{2}\sum_{i=1}^n \mathbb{E}\Big[\Big(\sum_{A\subsetneq [n]}k_{n,A}\sum_{j\notin A}1_{[\tilde{\Delta}_i\Delta_j f(X)\ne 0]}\sqrt{\|\Delta_j f(X)\|^p+\|\tilde{\Delta}_i\Delta_j f(X)\|^p}\cdot\\
\cdot \|\Delta_j f(X^A)\|\|\Delta_j f(X)\|\Big)^2\Big]
+\Big(\frac{9}{2}+\frac{9}{2p}\Big)\sum_{i=1}^n \mathbb{E}\Big[\Big(\sum_{A\subsetneq [n]}k_{n,A}\sum_{j\notin A}\sqrt{\|\Delta_j f(X)\|^p+\|\tilde{\Delta}_i\Delta_j f(X)\|^p}\cdot\\
\cdot \|\tilde{\Delta}_i\Delta_j f(X^A)\|\|\Delta_j f(X)\|\Big)^2\Big]\\
+\Big(\frac{9}{2}+\frac{9}{2p}\Big)\sum_{i=1}^n \mathbb{E}\Big[\Big(\sum_{A\subsetneq [n]}k_{n,A}\sum_{j\notin A}\sqrt{\|\Delta_j f(X)\|^p+\|\tilde{\Delta}_i\Delta_j f(X)\|^p}\|\tilde{\Delta}_i\Delta_j f(X)\|\|\Delta_j f(X^A)\|\Big)^2\Big]\\
+\Big(\frac{9}{2}+\frac{9}{2p}\Big)\sum_{i=1}^n \mathbb{E}\Big[\Big(\sum_{A\subsetneq [n]}k_{n,A}\sum_{j\notin A}\sqrt{\|\Delta_j f(X)\|^p+\|\tilde{\Delta}_i\Delta_j f(X)\|^p}\cdot\\
\cdot \|\tilde{\Delta}_i\Delta_j f(X^A)\|\|\tilde{\Delta}_i\Delta_j f(X)\|\Big)^2\Big]\Bigg\}^{1/(2+p)}
\end{multline*}
for $p=1,2$.
\begin{theorem}[\cite{KG22}, Theorem 2.2]\label{KG22th}
Suppose that $\Sigma\in\mathbb{R}^{d\times d}$ is positive definite and that $\mathbb{E}[\|\Delta_j f(X)\|^6]<\infty$ for every $j=1,\dots,n$. Define
\[
\gamma:=\max\Big\{\sqrt{\mathbb{E}\Big[\|\mathbb{E}[T-\Sigma|X]\|_{\text{HS}}^2\Big]},\gamma_1,\gamma_2,\gamma_3,\gamma_4\Big\},
\]
with 
\begin{equation}\label{def_T}
T:=\frac{1}{2}\sum_{A\subsetneq [n]}k_{n,A}\sum_{j\notin A}[\Delta_j f(X)][\Delta_j f(X^A)]^{T}\in \mathbb{R}^{d\times d}.
\end{equation}
Then
\[
d_C(W,N_{\Sigma})\le 541 d^4\max\{1,\|\Sigma^{-1}\|_{\mathrm{op}}^2\}\gamma.
\]
\end{theorem}
In particular, we will use the following generalization of the previous theorem.

\begin{theorem}\label{th_mult_fin}
Let $f : \mathbb{R}^n \to \mathbb{R}^d$ be a measurable function such that $
\mathbb{E}[f(X)] = 0$ and $
\mathbb{E}\big[\|f(X)\|\big] < \infty .
$
Let $X := (X_1,\dots,X_n)$ be a vector of independent and identically distributed random variables satisfying
$
\mathbb{E}[X_i] = 0$ and
$ 
\mathbb{E}[X_i^2] = 1,$ for $i=1,\dots,n .
$
Let $B := (b,\dots,b) \in \mathbb{R}^d$ be independent of $X$, where $b \sim \mathcal{N}(0,C_b)$.
Define
$
Y := B + f(X).
$

Denote by $\mathbf{1}_{d\times d}$ the $d\times d$ matrix whose entries are all equal to one, and set
\[
\tilde{T} := C_b\, \mathbf{1}_{d\times d} + T,
\]
where $T$ is given in~\eqref{def_T}.  
Moreover define
\[
\tilde{\gamma} 
:= 
\max \Bigg\{
\sqrt{ \mathbb{E}\!\left[ 
\left\| \mathbb{E}\!\left[ \tilde{T} - \Sigma \mid X \right] \right\|_{\mathrm{HS}}^2 
\right] },
\; \gamma_1,\; \gamma_2,\; \gamma_3,\; \gamma_4
\Bigg\}.
\]

Then, if $\Sigma \in \mathbb{R}^{d \times d}$ is positive definite and $\mathbb{E}[\|\Delta_j f(X)\|^6] < \infty$ for every $j = 1, \dots, n$, we have
\[
d_C(Y, N_{\Sigma}) \leq 541\, d^4\, \max\left\{1, \|\Sigma^{-1}\|_{\mathrm{op}}^2 \right\} \tilde{\gamma}.
\]
\end{theorem}

\begin{proof}
Following the proof of Theorem 2.2 in \cite{KG22}, and using the same notation, we obtain:
\begin{multline*}
d_C(Y, N_{\Sigma}) \leq \frac{4}{3} \sup_{h = 1_C,\, C\ \text{convex}} \left| 
\mathbb{E}\left[\langle \nabla f_{t,h,\Sigma}(Y), Y \rangle \right] 
- \mathbb{E}\left[\langle \operatorname{Hess}(f_{t,h,\Sigma})(Y), \tilde{T} \rangle_{\mathrm{HS}} \right] \right| \\
+ \frac{4}{3} \sup_{h = 1_C,\, C\ \text{convex}} \left| 
\mathbb{E}\left[\langle \operatorname{Hess}(f_{t,h,\Sigma})(Y), \tilde{T} - \Sigma \rangle_{\mathrm{HS}} \right] 
\right| 
+ \frac{20}{\sqrt{2}}\, k \frac{\sqrt{t}}{1 - t}.
\end{multline*}

Using the Gaussian integration by parts formula from Theorem~\ref{Gauss_integ_multdim}, the first term becomes
\begin{multline*}
\frac{4}{3} \sup_{h = 1_C,\, C\ \text{convex}} \left| 
\mathbb{E}\left[\langle \nabla f_{t,h,\Sigma}(Y), f(X) \rangle \right] 
- \mathbb{E}\left[\langle \operatorname{Hess}(f_{t,h,\Sigma})(Y), T \rangle_{\mathrm{HS}} \right] \right|.
\end{multline*}

Therefore,
\begin{multline*}
d_C(Y, N_{\Sigma}) \leq \frac{4}{3} \sup_{h = 1_C,\, C\ \text{convex}} \left| 
\mathbb{E}\left[\langle \nabla f_{t,h,\Sigma}(Y), f(X) \rangle \right] 
- \mathbb{E}\left[\langle \operatorname{Hess}(f_{t,h,\Sigma})(Y), T \rangle_{\mathrm{HS}} \right] \right| \\
+ \frac{4}{3} \sup_{h = 1_C,\, C\ \text{convex}} \left| 
\mathbb{E}\left[\langle \operatorname{Hess}(f_{t,h,\Sigma})(Y), \tilde{T} - \Sigma \rangle_{\mathrm{HS}} \right] \right| 
+ \frac{20}{\sqrt{2}}\, k \frac{\sqrt{t}}{1 - t}.
\end{multline*}

The rest of the proof follows exactly as in~\cite{KG22}.
\end{proof}

Applying  Remark \ref{fin_mom_z}, Theorem \ref{th_mult_fin} and Example 2.5 from \cite{KG22} to the neural network defined in \eqref{NN_expr} and to its Gaussian limit defined in Theorem \ref{Hanin}, next Theorem follows.
\begin{theorem}\label{mod_KG}
If $L \geq 1$ is an integer, $\mathbb{E}\left[\,|W_{1,1}^{(1)}|^6\,\right] < \infty$, and the matrix $K^{(L+1)}$ is invertible, then

\begin{multline*}
d_C(z_1^{(L+1)}(\mathcal{X}),G_1^{(L+1)}(\mathcal{X}))\le 541\cdot d^4 \sqrt{C_W}(1+C_W)\max\{1,\|(K^{(L+1)})^{-1}\|_{\mathrm{op}}^2\}\mathbb{E}\Big[|W_{1,1}^{(1)}|^6\Big]^{1/2}\cdot\\
\cdot\frac{1}{\sqrt{n_{L}}} \Bigg\{43\Big(1+\mathbb{E}\Big[\|\sigma(z_1^{(L)}(\mathcal{X})\|^6\Big]^{1/2}+\mathbb{E}\Big[\|\sigma(G_1^{(L)}(\mathcal{X})\|^4\Big]^{1/2}\Big)\\
+\sqrt{2}\sqrt{n_{L}}\Bigg(\sum_{j,k=1}^{d}\mathbb{E}\Big[(B_{L}(x^{(j)},x^{(k)}))^2\Big]\Bigg)^{1/2}\Bigg\}
\end{multline*}
where for every $i,j\in\{1,\dots,d\}$ and $\ell\ge 2$
\begin{equation}\label{B_l}
B_{\ell}(x^{(i)},x^{(j)}):=\mathbb{E}\Big[\sigma(z_{1}^{(\ell)}(x^{(i)}))\sigma(z_1^{(\ell)}(x^{(j)}))\big|\mathcal{F}_{\ell-1}\Big]-\mathbb{E}\Big[\sigma(G_{1}^{(\ell)}(x^{(i)}))\sigma(G_1^{(\ell)}(x^{(j)}))\Big].
\end{equation}
\end{theorem}

\begin{proof}
Denoting by $\mathcal{F}_{L}$ the $\sigma$-field generated by the weights and biases up to layer $L$, it follows that

\[
d_C(z_1^{(L+1)}(\mathcal{X}),G_1^{(L+1)}(\mathcal{X}))=\sup_{C\subseteq \mathbb{R}^d\, \text{convex}}\Big|\mathbb{E}\Big[1_{\{z_1^{(L+1)}(\mathcal{X})\in C\}}\Big]-\mathbb{E}\Big[1_{\{G_1^{(L+1)}(\mathcal{X})\in C\}}\Big]\Big|
\]
\[
=\sup_{C\subseteq \mathbb{R}^d\, \text{convex}}\Big|\mathbb{E}\Big[\mathbb{E}\Big[1_{\{{z}_1^{(L+1)}(\mathcal{X})\in C\}}\Big|\mathcal{F}_{L}\Big]\Big]-\mathbb{E}\Big[1_{\{G_1^{(L+1)}(\mathcal{X})\in C\}}\Big]\Big|
\]
\[
\le \sup_{C\subseteq \mathbb{R}^d\, \text{convex}}\mathbb{E}\Big[\Big|\mathbb{E}\Big[1_{\{{z}_1^{(L+1)}(\mathcal{X})\in C\}}\Big|\mathcal{F}_{L}\Big]-\mathbb{E}\Big[1_{\{G_1^{(L+1)}(\mathcal{X})\in C\}}\Big]\Big|\Big].
\]
{
Hence, by applying Theorem~\ref{KG22th} in the same way as Theorem~\ref{Kol_1} is used in the proof of Theorem~\ref{Kdist}, we obtain

}
\[
d_C\big(z_1^{(L+1)}(\mathcal{X}), G_1^{(L+1)}(\mathcal{X})\big) 
\leq 541\, d^4 \max\left\{1, \left\|(K^{(L+1)})^{-1}\right\|_{\mathrm{op}}^2\right\} \mathbb{E}[\tilde{\gamma}],
\]
where in this case
\[
Y_i := b_1 + f_i\big(W_{1,1}^{(L+1)}, \dots, W_{1,n_L}^{(L+1)}\big) 
:= b_1 + \frac{\sqrt{C_W}}{\sqrt{n_L}} \sum_{j=1}^{n_L} W_{1,j}^{(L+1)} \sigma\big(z_j^{(L)}(x^{(i)})\big),
\]
for each $i \in \{1, \dots, d\}$. In particular, denoting
\[
\sigma\big(z_1^{(L)}(\mathcal{X})\big) := 
\Big( \sigma\big(z_1^{(L)}(x^{(1)})\big), \dots, \sigma\big(z_1^{(L)}(x^{(d)})\big) \Big),
\]
one has that

\begin{multline*}
\gamma_1=\sum_{j=1}^{n_{L}}\mathbb{E}\Big[\Big\|\frac{\sqrt{C_W}}{\sqrt{n_{L}}}\Big(W_{1,j}^{(L+1)}-{W'}_{1,j}^{(L+1)}\Big)\sigma(z_j^{(L)}(\mathcal{X}))\Big\|^3\Big|\mathcal{F}_{L}\Big]\\
\le \Big(\frac{C_W}{n_{L}}\Big)^{3/2}\sum_{j=1}^{n_{L}}\mathbb{E}\Big[\Big(|W_{1,j}^{(L+1)}|+|{W'}_{1,j}^{(L+1)}|\Big)^3\Big]\Big\|\sigma(z_j^{(L)}(\mathcal{X}))\Big\|^3\\
\le 8\Big(\frac{C_W}{n_{L}}\Big)^{3/2}\sum_{j=1}^{n_{L}}\mathbb{E}\Big[|W_{1,j}^{(L+1)}|^3\Big]\Big\|\sigma(z_j^{(L)}(\mathcal{X}))\Big\|^3,
\end{multline*}
\begin{multline*}
\gamma_2=\Bigg(\sum_{j=1}^{n_{L}}\mathbb{E}\Big[\Big\|\frac{\sqrt{C_W}}{\sqrt{n_{L}}}\Big(W_{1,j}^{(L+1)}-{W'}_{1,j}^{(L+1)}\Big)\sigma(z_j^{(L)}(\mathcal{X}))\Big\|^4\Big|\mathcal{F}_{L}\Big]\Bigg)^{1/2}\\
\le \frac{4C_W}{n_{L}}\Bigg(\sum_{j=1}^{n_{L}}\mathbb{E}\Big[\Big(W_{1,j}^{(L+1)}\Big)^4\Big]\Big\|\sigma(z_j^{(L)}(\mathcal{X}))\Big\|^4\Bigg)^{1/2}.
\end{multline*}
Using the observations in Example 2.5 of~\cite{KG22}, we obtain that for $p = 1, 2$:

\begin{multline*}
\gamma_{p+2}=\Bigg\{\sum_{i=1}^{n_{L}}\mathbb{E}\Big[\Big(\|\Delta_i f(X)\|^p+\|\tilde{\Delta}_i f(X)\|^p\Big)\Big(\frac{3}{2}\|\Delta_i f(X)\|^4\\
+\Big(9+\frac{9}{p}\Big)\|\Delta_i f(X)\|^2 \|\tilde{\Delta}_i f(X)\|^2
+\Big(\frac{9}{2}+\frac{9}{2p}\Big)\|\tilde{\Delta}_i f(X)\|^4\Big)\Big|\mathcal{F}_{L}\Big]\Bigg\}^{\frac{1}{p+2}}
\end{multline*}
\begin{multline*}
=\Bigg\{\sum_{i=1}^{n_{L}}\mathbb{E}\Bigg[\Bigg(\Big\|\frac{\sqrt{C_W}}{\sqrt{n_{L}}}\Big(W_{1,i}^{(L+1)}-{W'}_{1,i}^{(L+1)}\Big)\sigma(z_i^{(L)}(\mathcal{X}))\Big\|^p\\
+\Big\|\frac{\sqrt{C_W}}{\sqrt{n_{L}}}\Big(W_{1,i}^{(L+1)}-{\tilde{W}}_{1,i}^{(L+1)}\Big)\sigma(z_i^{(L)}(\mathcal{X}))\Big\|^p\Bigg)\Bigg(\frac{3}{2}\Big\|\frac{\sqrt{C_W}}{\sqrt{n_{L}}}\Big(W_{1,i}^{(L+1)}-{W'}_{1,i}^{(L+1)}\Big)\sigma(z_i^{(L)}(\mathcal{X}))\Big\|^4\\
+\Big(9+\frac{9}{p}\Big)\Big\|\frac{\sqrt{C_W}}{\sqrt{n_{L}}}\Big(W_{1,i}^{(L+1)}-{W'}_{1,i}^{(L+1)}\Big)\sigma(z_i^{(L)}(\mathcal{X}))\Big\|^2 \Big\|\frac{\sqrt{C_W}}{\sqrt{n_{L}}}\Big(W_{1,i}^{(L+1)}-{\tilde{W}}_{1,i}^{(L+1)}\Big)\sigma(z_i^{(L)}(\mathcal{X}))\Big\|^2\\
+\Big(\frac{9}{2}+\frac{9}{2p}\Big)\Big\|\frac{\sqrt{C_W}}{\sqrt{n_{L}}}\Big(W_{1,i}^{(L+1)}-{\tilde{W}}_{1,i}^{(L+1)}\Big)\sigma(z_i^{(L)}(\mathcal{X}))\Big\|^4\Bigg)\Bigg|\mathcal{F}_{L}\Bigg]\Bigg\}^{\frac{1}{p+2}}
\end{multline*}
\[
\le \Big(\frac{4C_W}{n_{L}}\Big)^{\frac{p+4}{2(p+2)}}\Big(30+\frac{27}{p}\Big)^{\frac{1}{p+2}}\Bigg\{\sum_{i=1}^{n_{L}}\|\sigma(z_i^{(L)}(\mathcal{X}))\|^{p+4}\mathbb{E}\Big[|W_{1,i}^{(L+1)}|^{p+4}\Big]\Bigg\}^{\frac{1}{p+2}}.
\]
On the other hand,
\begin{multline*}
\mathbb{E}\Big[\|\mathbb{E}[\tilde{T}-K^{(L+1)}|W^{(L+1)}\lor \mathcal{F}_{L}]\|_{\text{HS}}^2\Big|\mathcal{F}_{L}\Big]\\
=\mathbb{E}\Big[\Big\|\mathbb{E}\Big[C_b\mathbf{1}_{d\times d}+\frac{1}{2}\sum_{i=1}^{n_{L}}\frac{C_W}{n_{L}}(W_{1,i}^{(L+1)}-{W'}_{1,i}^{(L+1)})^2\sigma(z_i^{(L)}(\mathcal{X}))\sigma(z_i^{(L)}(\mathcal{X}))^T\\
-K^{(L+1)}\Big|\mathcal{F}_{L+1}\Big]\Big\|_{\text{HS}}^2\Big|\mathcal{F}_{L}\Big]
\end{multline*}
\begin{multline*}
=\mathbb{E}\Bigg[\sum_{j,k=1}^{d}\Bigg(\mathbb{E}\Bigg[\frac{1}{2}\sum_{i=1}^{n_{L-1}}\frac{C_W}{n_{L}}\Big((W_{1,i}^{(L+1)}-{W'}_{1,i}^{(L+1)})^2\sigma(z_i^{(L)}(x^{(j)}))\sigma(z_i^{(L)}(x^{(k)}))\\
-2\mathbb{E}\Big[\sigma({G}_i^{(L)}(x^{(j)}))\sigma({G}_i^{(L)}(x^{(k)}))\Big]\Big)\Bigg|\mathcal{F}_{L+1}\Bigg]\Bigg)^2\Bigg|\mathcal{F}_{L}\Bigg]
\end{multline*}
\begin{multline*}
=\sum_{j,k=1}^{d}\mathbb{E}\Bigg[\Bigg(\frac{1}{2}\sum_{i=1}^{n_{L}}\frac{C_W}{n_{L}}\mathbb{E}\Big[\Big((W_{1,i}^{(L+1)}-{W'}_{1,i}^{(L+1)})^2-2\Big)\Big|W^{(L+1)}\Big]\sigma(z_i^{(L)}(x^{(j)}))\sigma(z_i^{(L)}(x^{(k)}))\\
+\sum_{i=1}^{n_{L}}\frac{C_W}{n_{L}}\Big(\sigma(z_i^{(L)}(x^{(j)}))\sigma(z_i^{(L)}(x^{(k)}))
-\mathbb{E}\Big[\sigma({G}_i^{(L)}(x^{(j)}))\sigma({G}_i^{(L)}(x^{(k)}))\Big]\Big)
\Bigg)^2\Bigg|\mathcal{F}_{L}\Bigg]
\end{multline*}
\begin{multline*}
\le 2\sum_{j,k=1}^{d}\mathbb{E}\Bigg[\Bigg(\frac{1}{2}\sum_{i=1}^{n_{L}}\frac{C_W}{n_{L}}\Big((W_{1,i}^{(L+1)})^2-1\Big)\sigma(z_i^{(L)}(x^{(j)}))\sigma(z_i^{(L)}(x^{(k)}))
\Bigg)^2\Bigg|\mathcal{F}_{L}\Bigg]\\
+2\sum_{j,k=1}^{d}\mathbb{E}\Bigg[\Bigg(\sum_{i=1}^{n_{L}}\frac{C_W}{n_{L}}\Big(\sigma(z_i^{(L)}(x^{(j)}))\sigma(z_i^{(L)}(x^{(k)}))
-\mathbb{E}\Big[\sigma({G}_i^{(L)}(x^{(j)}))\sigma({G}_i^{(L)}(x^{(k)}))\Big]\Big)\Bigg)^2\Bigg|\mathcal{F}_{L}\Bigg]\\
\end{multline*}

\begin{multline*}
= \frac{1}{2}\sum_{j,k=1}^{d}\sum_{i=1}^{n_{L}}\frac{C_W^2}{n_{L}^2}\mathbb{E}\Big[\Big((W_{1,i}^{(L+1)})^2-1\Big)^2\Big]\sigma^2(z_i^{(L)}(x^{(j)}))\sigma^2(z_i^{(L)}(x^{(k)}))\\
+2\sum_{j,k=1}^{d}\Bigg(\sum_{i=1}^{n_{L}}\frac{C_W}{n_{L}}\Big(\sigma(z_i^{(L)}(x^{(j)}))\sigma(z_i^{(L)}(x^{(k)}))
-\mathbb{E}\Big[\sigma({G}_i^{(L)}(x^{(j)}))\sigma({G}_i^{(L)}(x^{(k)}))\Big]\Big)\Bigg)^2\\
\end{multline*}
Therefore,
\begin{multline*}
d_C(z_1^{(L+1)}(\mathcal{X}),G_1^{(L+1)}(\mathcal{X}))\le 541 d^4 \max\{1,\|(K^{(L+1)})^{-1}\|_{\mathrm{op}}^2\}\Bigg(\mathbb{E}[\gamma_1]+\mathbb{E}[\gamma_2]\\
+\mathbb{E}[\gamma_3]+\mathbb{E}[\gamma_4]
+{\mathbb{E}\bigg[\mathbb{E}\Big[\|\mathbb{E}[\tilde{T}-K^{(L+1)}|W^{(L+1)}\lor \mathcal{F}_{L}]\|_{\text{HS}}^2\Big|\mathcal{F}_{L}\Big]^{1/2}\bigg]}\Bigg)
\end{multline*}
\begin{multline*}
\le 541 d^4 \max\{1,\|(K^{(L+1)})^{-1}\|_{\mathrm{op}}^2\}\Bigg\{8\Big(\frac{C_W}{n_{L}}\Big)^{3/2}\sum_{i=1}^{n_{L}}\mathbb{E}\Big[|W_{1,i}^{(L+1)}|^3\Big]\mathbb{E}\Big[\|\sigma(z_i^{(L)}(\mathcal{X}))\|^3\Big]\\
+ \frac{4C_W}{n_{L}}{\mathbb{E}\Big[\Big(W_{1,1}^{(L+1)}\Big)^4\Big]}^{1/2}\Bigg(\sum_{i=1}^{n_{L}}\mathbb{E}\Big[\|\sigma(z_i^{(L)}(\mathcal{X}))\|^4\Big]\Bigg)^{1/2}\\
+\Big(\frac{4C_W}{n_{L}}\Big)^{5/6}(57)^{1/3}\mathbb{E}\Big[|W_{1,1}^{(L+1)}|^{5}\Big]^{1/3}\Bigg(\sum_{i=1}^{n_{L}}\mathbb{E}\Big[\|\sigma(z_i^{(L)}(\mathcal{X}))\|^{5}\Big]\Bigg)^{1/3}\\
+\Big(\frac{4C_W}{n_{L}}\Big)^{3/4}\Big(\frac{87}{2}\Big)^{1/4}\mathbb{E}\Big[|W_{1,1}^{(L+1)}|^{6}\Big]^{1/4}\Bigg(\sum_{i=1}^{n_{L}}\mathbb{E}\Big[\|\sigma(z_i^{(L)}(\mathcal{X}))\|^{6}\Big]\Bigg)^{1/4}\\
+{\frac{1}{\sqrt{2}}}\frac{C_W}{n_{L}}\mathbb{E}\Big[(W_{1,1}^{(L+1)})^4\Big]^{1/2}\Bigg(\sum_{j,k=1}^{d}\sum_{i=1}^{n_{L}}\mathbb{E}\Big[\sigma^2(z_i^{(L)}(x^{(j)}))\sigma^2(z_i^{(L)}(x^{(k)}))\Big]\Bigg)^{1/2}\\
+\sqrt{2}\Bigg(\sum_{j,k=1}^{d}\mathbb{E}\Bigg[\Bigg(\sum_{i=1}^{n_{L}}\frac{C_W}{n_{L}}\Big(\sigma(z_i^{(L)}(x^{(j)}))\sigma(z_i^{(L)}(x^{(k)}))\\
-\mathbb{E}\Big[\sigma({G}_i^{(L)}(x^{(j)}))\sigma({G}_i^{(L)}(x^{(k)}))\Big]\Big)\Bigg)^2\Bigg]\Bigg)^{1/2}
\Bigg\}
\end{multline*}
\begin{multline*}
\le 541 d^4 \max\{1,\|(K^{(L+1)})^{-1}\|_{\mathrm{op}}^2\}\mathbb{E}\Big[|W_{1,1}^{(L+1)}|^6\Big]^{1/2}\frac{1}{\sqrt{n_{L}}}\Bigg\{8\sqrt{C_W^3}\mathbb{E}\Big[\|\sigma(z_1^{(L)}(\mathcal{X}))\|^3\Big]\\
+ {4C_W}\mathbb{E}\Big[\|\sigma(z_1^{(L)}(\mathcal{X}))\|^4\Big]^{1/2}
+({4C_W})^{5/6}(57)^{1/3}\mathbb{E}\Big[\|\sigma(z_1^{(L)}(\mathcal{X}))\|^{5}\Big]^{1/3}\\
+({4C_W})^{3/4}\Big(\frac{87}{2}\Big)^{1/4}\mathbb{E}\Big[\|\sigma(z_1^{(L)}(\mathcal{X}))\|^{6}\Big]^{1/4}
+{{\frac{1}{\sqrt{2}}}}{C_W}\mathbb{E}\Big[\|\sigma^2(z_1^{(L)}(\mathcal{X}))\|^4\Big]^{1/2}\\
+\sqrt{2}C_W\Bigg(\sum_{j,k=1}^{d}\mathbb{E}\Bigg[\Big(\sigma(z_1^{(L)}(x^{(j)}))\sigma(z_1^{(L)}(x^{(k)})
-\mathbb{E}\Big[\sigma({G}_1^{(L)}(x^{(k)}))\sigma(G_1^{(L)}(x^{(j)})\Big]\Big)^2\Bigg]\Bigg)^{1/2}\\
+\sqrt{2}\sqrt{n_{L}}C_W\Bigg(\sum_{j,k=1}^{d}\mathbb{E}\Bigg[\Big(\mathbb{E}\Big[\sigma(z_1^{(L)}(x^{(j)}))\sigma(z_1^{(L)}(x^{(k)}))\big|\mathcal{F}_{L-1}\Big]\\
-\mathbb{E}\Big[\sigma({G}_1^{(L)}(x^{(j)}))\sigma({G}_1^{(L)}(x^{(k)}))\Big]\Big)^2\Bigg]\Bigg)^{1/2}
\Bigg\}
\end{multline*}
where we used the fact that, conditionally on the $\sigma$-field $\mathcal{F}_{L-1}$ generated by the weights up to layer $L - 1$, the components of the neural network vector $z^{(L)}$ are independent and identically distributed.

Hence
\begin{multline*}
d_C(z_1^{(L+1)}(\mathcal{X}),G_1^{(L+1)}(\mathcal{X}))\le 541\cdot d^4 \sqrt{C_W}(1+C_W)\max\{1,\|(K^{(L+1)})^{-1}\|_{\mathrm{op}}^2\}\cdot\\
\cdot\mathbb{E}\Big[|W_{1,1}^{(L+1)}|^6\Big]^{1/2}\frac{1}{\sqrt{n_{L}}} \Bigg\{43\Big(1+\mathbb{E}\Big[\|\sigma(z_1^{(L)}(\mathcal{X})\|^6\Big]^{1/2}+\mathbb{E}\Big[\|\sigma(G_1^{(L)}(\mathcal{X})\|^4\Big]^{1/2}\Big)\\
+\sqrt{2}\sqrt{n_{L}}\Bigg(\sum_{j,k=1}^{d}\mathbb{E}\Big[(B_{L}(x^{(j)},x^{(k)}))^2\Big]\Bigg)^{1/2}\Bigg\}
\end{multline*}
\end{proof}
\subsection{A bound on $B_\ell(x^{(i)},x^{(j)})$}
{
Let $x^{(1)},\dots,x^{(d)}\in\mathbb{R}^{n_0}$ such that $x^{(i)}\ne x^{(j)}$ for $x\ne j$. Then,
for every $i, j \in \{1, \dots, d\}$ such that $i\ne j$, a bound on the term 

\[
\mathbb{E}\left[\big(B_{\ell}(x^{(i)}, x^{(j)})\big)^2\right]
\]
is now provided by the following lemma, whose proof is given in Appendix~B.
}
\begin{lemma}\label{somma_schifa}
\begin{multline*}
\sum_{1\le i\ne j\le d}\mathbb{E}\Big[\Big(B_\ell(x^{(i)},x^{(j)})\Big)^2\Big]\\
:=\sum_{1\le i\ne j\le d}\mathbb{E}\Big[\Big|\mathbb{E}\Big[\sigma(z_{1}^{(\ell)}(x^{(i)}))\sigma(z_1^{(\ell)}(x^{(j)}))\big|\mathcal{F}_{\ell-1}\Big]-\mathbb{E}\Big[\sigma(G_{1}^{(\ell)}(x_i))\sigma(G_1^{(\ell)}(x_j))\Big]\Big|^2\Big]\\
\le\sum_{u=0}^{\ell-2}V_1^{(\ell-u)(n)}(\mathcal{X})V_2^u
\end{multline*}
where
\[
V_2:=4{C_W^2}{\|\sigma\|_{\text{Lip}}^4},
\]
\begin{multline*}
V_1^{(\ell-u)}:=
2^9\Big(1+\|\sigma\|_{\text{Lip}}^2\Big)^2(1+C_W)^4\mathbb{E}\Big[ \big(W_{1,1}^{(1)}\big)^6\Big]^2\Bigg(\operatorname{tr} K^{(\ell)}+2\mathbb{E}\Big[\|\sigma(z^{(\ell)}_1(\mathcal{X}))\|^4\Big]^{1/2}\\
   +2(\|\sigma\|_{\text{Lip}}^2+1)\mathbb{E}\Big[\|z^{(\ell)}_1(\mathcal{X})\|^4\Big]^{1/2}
   +\mathbb{E}\Big[\|\sigma(G_1^{(\ell-1)}(\mathcal{X})\|^4\Big]^{1/2}
+2\mathbb{E}\Big[\|\sigma(z^{(\ell-1)}_1(\mathcal{X}))\|^4\Big]^{1/2}+\|\sigma\|_{\text{Lip}}^2\Bigg)\cdot\\
\Bigg(1+\mathbb{E}\Big[\|\sigma(G_1^{(\ell-1)}(\mathcal{X})\|^4\Big]^{1/2}+\mathbb{E}\Big[\|\sigma(z_1^{(\ell-1)}(\mathcal{X}))\|^{12}\Big]^{1/2}\Bigg)\cdot\\
\cdot \Big(1+\|(K^{(\ell)})^{-1}\|^2_{\text{HS}}\Big)\Big(1+ {(\operatorname{tr}K^{(\ell)})^2}\Big)\Bigg\{\sum_{i=1}^d\sqrt{Q_4^{(\ell-1)}(x^{(i)})}
+\frac{2}{n_{\ell-1}}
\\
+\frac{1}{{n_{\ell-1}}}\Big(1+\sqrt{d}\Big)
\Big(4d+5\|(K^{(\ell)})^{-1}\|^2_{op}+12\Big)\Bigg\},
\end{multline*}
recalling the notation $Q_4^{(r)}(x_i)$ from equation \eqref{Q_k}.
\end{lemma}

\begin{remark}\label{bound_V_1}
Suppose now that $C_b \neq 0$. Performing some algebraic manipulations, and applying the Lipschitz property of $\sigma$, inequality~\eqref{moment_meglio}, and Lemma~\ref{Q_2}, one obtains that for every $\ell \geq k \geq 2$,
\begin{multline*}
V_1^{(k)(n)}(\mathcal{X})\le C \ell^{6k+44} 2^{4k} \Big(\sum_{j=1}^{k-1}\frac{1}{n_j}\Big)d^8(2d+1)2^9\Big(1+\|\sigma\|_{Lip}^2\Big)^8(1+C_W)^{23}(1+C_b)^{15}\cdot\\
\cdot E\Big[ \big(W_{1,1}^{(1)}\big)^{12}\Big]^3(1+|\sigma(0)|^2)^{22}
\Big(d+\frac{\sum_{j=1}^d\|x^{(j)}\|^8}{n_0^4}+\frac{\sum_{j=1}^d\|x^{(j)}\|^{24}}{{n_0}^{12}}\Big)^2 \cdot\\
\cdot  \Bigg(1+ \Big( K\frac{5\cdot 2^{k+1}}{\log(5\cdot 2^{k+1})}  \sqrt{C_W} \|\sigma\|_{\mathrm{Lip}} \mathbb{E}[(W_{1,1}^{(1)})^{5\cdot 2^{k+1}}]^{1/(5\cdot 2^{k+1})} \Big)^{(k - 1)}\Bigg)^{\frac{k}{2}+47}\cdot\\
\cdot \Big(1+(2C_W\|\sigma\|_{Lip}^2)^{4k}\Big)^{k}\Big(1+\frac{1}{\lambda(K^{(k)})^2}\Big)^2\left(\sum_{i=1}^d\left(1+\left(M^{(k-1)}_2(x^{(i)})\right)^{k/2}\right)\right)\cdot\\
\cdot \Bigg( 4\Big(
  \sqrt{ 2^{k+1} }
+ K\frac{ 2^{k}}{\log( 2^{k+2})}
\Big)^{4}\Bigg)^{k+4}
\end{multline*}
where $C > 0$ is an absolute constant independent of any parameter of the neural network, 
$\lambda(K^{(k)})$ denotes the smallest eigenvalue of the matrix $K^{(k)}$, $M_2^{(k- 1)}(x^{(j)})$ is defined in Lemma~\ref{Q_2}, and $K$
 is the constant defined in Lemma \ref{mom_p_iid0}.
 
Therefore, by Lemma \ref{somma_schifa} and Remark \ref{bound_V_1}, it follows that
\begin{multline*}
\sum_{1\le i\ne j\le d}\mathbb{E}[(B_\ell(x^{(i)},x^{(j)}))^2]\le\Big(4{C_W^2}{\|\sigma\|_{\text{Lip}}^4}\Big)^{\ell-2}C \ell^{6\ell+44} 2^{4\ell} \Big(\sum_{j=1}^{\ell-1}\frac{1}{n_j}\Big)d^8(2d+1)\cdot  \\
\cdot 2^9\Big(1+\|\sigma\|_{Lip}^2\Big)^8(1+C_W)^{23}(1+C_b)^{15}\Bigg( 4\Big(
  \sqrt{ 2^{\ell+1} }
+ K\frac{ 2^{\ell}}{\log( 2^{\ell+2})}
\Big)^{4}\Bigg)^{\ell+4}\cdot\\
\cdot E\Big[ \big(W_{1,1}^{(1)}\big)^{12}\Big]^3(1+|\sigma(0)|^2)^{22}
\Big(d+\frac{\sum_{j=1}^d\|x^{(j)}\|^8}{n_0^4}+\frac{\sum_{j=1}^d\|x^{(j)}\|^{24}}{{n_0}^{12}}\Big)^2 \cdot\\
\cdot  \Bigg(2+ \Big( K\frac{5\cdot 2^{\ell+1}}{\log(5\cdot 2^{\ell+1})}  \sqrt{C_W} \|\sigma\|_{\mathrm{Lip}} \mathbb{E}[(W_{1,1}^{(1)})^{5\cdot 2^{\ell+1}}]^{1/(5\cdot 2^{\ell+1})} \Big)^{(\ell - 1)}\Bigg)^{\frac{\ell}{2}+47}\cdot\\
\cdot \Big(2+(2C_W\|\sigma\|_{Lip}^2)^{4\ell}\Big)^{\ell}\left(\sum_{k=2}^{\ell}\Big(1+\frac{1}{\lambda(K^{(k)})^2}\Big)^2\right)\left(\sum_{i=1}^d\left(1+\left(M^{(\ell-1)}_2(x^{(i)})\right)^{\ell/2}\right)\right).
\end{multline*}
\end{remark}

\subsection{Lower bound for the determinant of $K^{(\ell)}$}\label{sec_low_bound_detK}
\begin{theorem}\label{bound_K_gauss}
For every $\ell \geq 3$, define 
\[
\hat{K}^{(\ell)} := K^{(\ell)} - C_b \mathbf{1}_{d \times d},
\]
where $K^{(\ell)}$ is the limiting covariance matrix defined in Theorem~\ref{Hanin} and $\mathbf{1}_{d \times d}$ is the $d \times d$ matrix with all entries equal to one. Then

\[
\operatorname{det}(\hat{K}^{(\ell)})\ge \frac{ C_W^{d(\ell-2)}}{d!}\lambda(K^{(2)})^{d}\prod_{i=1}^d\left(\sum_{\substack{k_i=1\\k_i\ne k_s\, \forall s=1,\dots,i-1}}^d\prod_{r_i=1}^{\ell-2}\mathbb{E}\big[\sigma'(G^{(\ell-r_i)}_1(x^{(k_i)}))\big]^2\right),
\]
where $\lambda(K^{(2)})$ denotes the minimum eigenvalue of $K^{(2)}$, as defined in Theorem~\ref{Hanin}.

\end{theorem}

\begin{remark}
The proof of Theorem~\ref{bound_K_gauss} relies on Proposition~\ref{Caco}, which is strongly based on the properties of the Gaussian distribution. In particular, the proof's strategy cannot be directly adapted to yield a lower bound for $\operatorname{det}(K^{(2)})$. 

Recall that $K^{(2)}$ is defined as
\[
K^{(2)}_{i,j} := C_b + C_W \, \mathbb{E}\Big[\sigma(z_1^{(1)}(x^{(i)})) \sigma(z_1^{(1)}(x^{(j)}))\Big] \quad \text{for every } i,j=1,\dots,d,
\]
where, for each $i=1,\dots,d$, 
\[
z_1^{(1)}(x^{(i)}) := b_1 + \frac{\sqrt{C_W}}{\sqrt{n_0}} \sum_{r=1}^{n_0} W_{1,r}^{(1)} x_r^{(i)}
\]
is the first component of the neural network at the first layer, and the random variables $\{W_{1,r}^{(1)}\}_{r=1,\dots,n_0}$ are i.i.d., but not necessarily Gaussian.
\end{remark}

The proof of Theorem~\ref{bound_K_gauss} will follow after having established several technical lemmas.

\begin{lemma}\label{agg_quad}
 For $\ell \geq 3$, consider $d$ independent copies of ${G}_1^{(\ell-1)}(\mathcal{X})$ (as defined in Theorem \ref{Hanin}), denoted by 
\[
\{\hat{G}_j^{(\ell-1)}(\mathcal{X})\}_{j=1,\dots,d},
\]
and define 
\[
\hat{K}^{(\ell)} := K^{(\ell)} - C_b \mathbf{1}_{d \times d},
\]
where $\mathbf{1}_{d \times d}$ is the $d \times d$ matrix with all entries equal to one.
Then, we have
\[
\operatorname{det}(\hat{K}^{(\ell+1)}) = \frac{C_W^d}{d!} \mathbb{E} \left[
\left( \sum_{\rho \in \Sigma_d} \operatorname{sgn}(\rho) \prod_{i=1}^d \sigma\big(\hat{G}_i^{(\ell-1)}(x^{(\rho(i))})\big) \right)^2
\right],
\]
where $\Sigma_d$ denotes the set of all permutations of $d$ elements.

\end{lemma}
\begin{remark}\label{oss_det}
For every matrix $A\in\mathbb{R}^{d\times d}$ one has that
\[
\operatorname{det}A=\sum_{\eta\in\Sigma_d}\operatorname{sgn}(\eta)\prod_{i=1}^{d}A_{i,\eta(i)}=\frac{1}{d!}\sum_{\rho\in\Sigma_d}\sum_{\eta\in\Sigma_d}\operatorname{sgn}(\eta)\operatorname{sgn}(\rho)\prod_{i=1}^{d}A_{\rho(i),\eta(i)}.
\]
In fact
\[
\sum_{\rho\in\Sigma_d}\sum_{\eta\in\Sigma_d}\operatorname{sgn}(\eta)\operatorname{sgn}(\rho)\prod_{i=1}^{d}A_{\rho(i),\eta(i)}=\sum_{\rho\in\Sigma_d}\sum_{\eta\in\Sigma_d}\operatorname{sgn}(\eta\circ\rho^{-1})\prod_{i=1}^{d}A_{\rho(i),(\eta\circ\rho^{-1})(\rho(i))}
\]
\[
=\sum_{\rho\in\Sigma_d}\sum_{\eta\in\Sigma_d}\operatorname{sgn}(\eta)\prod_{i=1}^{d}A_{\rho(i),\eta(\rho(i))}=d!\sum_{\eta\in\Sigma_d}\operatorname{sgn}(\eta)\prod_{i=1}^{d}A_{i,\eta(i)}.
\]
\end{remark}
\begin{proof}
Applying Remark \ref{oss_det} and using the notation introduced in the statement, one has that
\[
 \operatorname{det}(\hat{K}^{(\ell)})=\frac{C_W^d}{d!}\sum_{\rho\in\Sigma_d}\sum_{\eta\in\Sigma_d}\operatorname{sgn}(\rho)\operatorname{sgn}(\eta)\prod_{i=1}^{d}\mathbb{E}\left[\sigma(G_1^{(\ell-1)}(x^{(\rho(i))}))\sigma(G_1^{(\ell-1)}(x^{(\eta(i))}))\right]
\]
\[
=\frac{C_W^d}{d!}\sum_{\rho\in\Sigma_d}\sum_{\eta\in\Sigma_d}\operatorname{sgn}(\rho)\operatorname{sgn}(\eta)\mathbb{E}\left[\prod_{i=1}^{d}\sigma(\hat{G}_i^{(\ell-1)}(x^{(\rho(i))}))\sigma(\hat{G}_i^{(\ell-1)}(x^{(\eta(i))}))\right]
\]

\[
=\frac{C_W^d}{d!}\mathbb{E}\left[\left(\sum_{\rho\in\Sigma_d}\operatorname{sgn}(\rho)\prod_{i=1}^{d}\sigma(\hat{G}_i^{(\ell-1)}(x^{(\rho(i))}))\right)^2\right],
\]
thus concluding the proof.
\end{proof}

In the proof of Theorem \ref{bound_K_gauss} the following proposition, proved by T. Cacoullos in \cite{Ca82}, will be used.

\begin{prop}[Proposition 3.7 in \cite{Ca82}]\label{Caco}
Let $N \sim \mathcal{N}_d(\mu, K)$ with $K \in \mathbb{R}^{d \times d}$, and let $g: \mathbb{R}^d \to \mathbb{R}$ be a differentiable function with gradient $\nabla g$ such that $\mathbb{E}[|\nabla g(N)|]<\infty$. Then

\[
\operatorname{Var}\Big[g(N)\Big]\ge \mathbb{E}\Big[\Big(\nabla g(N)\Big)^T\Big]K \mathbb{E}\Big[\nabla g(N)\Big].
\]
\end{prop}
{
\begin{remark}
   In the nontrivial case where $\mathbb{E}[g(N)^2]<\infty$ (otherwise the statement of Proposition~\ref{Caco} is immediate),  
Proposition~\ref{Caco} admits a natural interpretation in terms of the Wiener chaos decomposition of the random variable $g(N)$; see \cite{NP12}. By the Stroock formula (Corollary~2.7.8 in \cite{NP12}) together with Proposition~2.7.5 in \cite{NP12}, the square--integrable random variable
$g(N) \in L^2(\Omega)$ admits the orthogonal decomposition
\[
g(N) = \mathbb{E}[g(N)] + J_1(g(N)) + J_{\ge 2}(g(N)),
\]
where $J_1(g(N))$ denotes the projection of $g(N)$ onto the first Wiener chaos and
$J_{\ge 2}(g(N))$ collects all higher--order chaoses. Since Wiener chaoses are orthogonal in $L^2(\Omega)$, it follows that
\begin{equation}\label{proof_ca_chaos}
\operatorname{Var}\big(g(N)\big)
= \operatorname{Var}\!\left(J_1(g(N))\right)
+ \operatorname{Var}\!\left(J_{\ge 2}(g(N))\right)
\ge \operatorname{Var}\!\left(J_1(g(N))\right).
\end{equation}
Assume in addition that $g$ satisfies mild regularity conditions, for instance that its partial derivatives are bounded or that $g$ is Lipschitz continuous.
Then, by Proposition~2.3.7 or Proposition~2.3.8 in \cite{NP12}, combined with Corollary~2.7.8 and Proposition~2.7.5 in \cite{NP12}, one obtains
\[
\operatorname{Var}\!\left(J_1(g(N))\right)
= \mathbb{E}\!\left[\nabla g(N)\right]^T K\, \mathbb{E}\!\left[\nabla g(N)\right].
\]
Together with \eqref{proof_ca_chaos}, this yields exactly the lower bound stated in Proposition~\ref{Caco}.
\end{remark}
}
\begin{proof}[Proof of Theorem \ref{bound_K_gauss}]
Thanks to Lemma \ref{agg_quad} one has that
\[
\operatorname{det}(\hat{K}^{(L+1)})=\frac{C_W^d}{d!}\mathbb{E}\left[\left(\sum_{\rho\in\Sigma_d}\operatorname{sgn}(\rho)\prod_{i=1}^{d}\sigma(\hat{G}_i^{(\ell-1)}(x^{(\rho(i))}))\right)^2\right]
\]

\[
=\frac{C_W^d}{d!}\mathbb{E}\left[\mathbb{E}\left[\left(\sum_{\rho\in\Sigma_d}\operatorname{sgn}(\rho)\prod_{i=1}^{d}\sigma(\hat{G}_i^{(\ell-1)}(x^{(\rho(i))}))\right)^2\Bigg|\Big(\hat{G}_s^{(\ell-1)}(\mathcal{X})\Big)_{s=2,\dots,d}\right]\right]
\]

\[
\ge\frac{C_W^d}{d!}\mathbb{E}\left[\operatorname{Var}\left[\sum_{\rho\in\Sigma_d}\operatorname{sgn}(\rho)\prod_{i=1}^{d}\sigma(\hat{G}_i^{(\ell-1)}(x^{(\rho(i))}))\Bigg|\Big(\hat{G}_s^{(\ell-1)}(\mathcal{X})\Big)_{s=2,\dots,d}\right]\right]
\]

\begin{multline}\label{uso_caco}
\ge \frac{C_W^{d+1}}{d!}\mathbb{E}\Bigg[\Bigg(\sum_{k=1}^d\mathbb{E}\big[\sigma'\big({G}^{(\ell-1)}_{1}(x^{(k)})\big)\big]\sigma\big(\hat{G}^{(\ell-2)}_{1}(x^{(k)})\big)\cdot\\
\cdot \sum_{\substack{\rho\in\Sigma_d\\\rho(1)=k}}\operatorname{sgn}(\rho)\prod_{i=2}^{d}\sigma(\hat{G}_i^{(\ell-1)}(x^{(\rho(i))}))\Bigg)^2\Bigg]
\end{multline}

\begin{multline}\label{iter_caco}
\ge \frac{C_W^{d}C_W^{\ell-3}}{d!}\mathbb{E}\Bigg[\Bigg(\sum_{k=1}^d\left(\prod_{r=1}^{\ell-3}\mathbb{E}\big[\sigma'\big({G}^{(\ell-r)}_{1}(x^{(k)})\big)\big]\right)\sigma\big(\hat{G}^{(2)}_{1}(x^{(k)})\big)\cdot\\
\cdot \sum_{\substack{\rho\in\Sigma_d\\\rho(1)=k}}\operatorname{sgn}(\rho)\prod_{i=2}^{d}\sigma(\hat{G}_i^{(\ell-1)}(x^{(\rho(i))}))\Bigg)^2\Bigg]
\end{multline}

\begin{multline}\label{final_caco}
\ge \frac{C_W^{d}C_W^{\ell-3}}{d!}\lambda(K^{(2)})\sum_{k=1}^d\left(\prod_{r=1}^{\ell-2}\mathbb{E}\big[\sigma'\big({G}^{(\ell-r)}_{1}(x^{(k)})\big)\big]^2\right)\cdot\\
\cdot \mathbb{E}\left[\left(
 \sum_{\substack{\rho\in\Sigma_d\\\rho(1)=k}}\operatorname{sgn}(\rho)\prod_{i=2}^{d}\sigma(\hat{G}_i^{(\ell-1)}(x^{(\rho(i))}))\right)^2\right]
\end{multline}

where we have used Proposition \ref{Caco} in the inequalities \eqref{uso_caco} and \eqref{final_caco}, and iterated it for $\ell\ge 3$ to deduce inequality \eqref{iter_caco}.
We notice that for every $d\ge 2$ we have
\begin{multline*}
\mathbb{E}\left[\left(\sum_{\rho\in\Sigma_d}\operatorname{sgn}(\rho)\prod_{i=1}^{d}\sigma(\hat{G}_i^{(\ell-1)}(x^{(\rho(i))}))\right)^2\right]\\\ge C_W^{\ell-3}\lambda(K^{(2)})\sum_{k=1}^d\left(\prod_{r=1}^{\ell-2}\mathbb{E}\big[\sigma'\big({G}^{(\ell-r)}_{1}(x^{(k)})\big)\big]^2\right) \mathbb{E}\left[\left(
 \sum_{\substack{\rho\in\Sigma_d\\\rho(1)=k}}\operatorname{sgn}(\rho)\prod_{i=2}^{d}\sigma(\hat{G}_i^{(\ell-1)}(x^{(\rho(i))}))\right)^2\right]
\end{multline*}
and therefore one also has that for every $\ell\ge 3$

\[ 
\operatorname{det}(\hat{K}^{(\ell)})\ge \frac{ C_W^{d(\ell-2)}}{d!}\lambda(K^{(2)})^{d}\prod_{i=1}^d\left(\sum_{\substack{k_i=1\\k_i\ne k_s\, \forall s=1,\dots,i-1}}^d\prod_{r_i=1}^{\ell-2}\mathbb{E}\big[\sigma'(G^{(\ell-r_i)}_1(x^{(k_i)}))\big]^2\right).
\]
\end{proof}

\begin{remark}\label{HS_bondK}
Thanks to Theorem \ref{bound_K_gauss}, when $C_W \neq 0$, it is possible to provide an upper bound on $(\lambda(K^{(\ell)}))^{-1}$ for every $\ell = 3, \dots, L+1$. 
Observe first that, denoting $\hat{K}^{(\ell)} := K^{(\ell)} - C_b \mathbf{1}_{d \times d}$, where $\mathbf{1}_{d \times d}$ is the matrix of ones, we have
\[
\lambda(K^{(\ell)}) = \min_{\|y\|=1} y^T K^{(\ell)} y = C_b \sum_{i,j=1}^d y_i y_j + y^T \hat{K}^{(\ell)} y \geq \lambda(\hat{K}^{(\ell)}),
\]
where $\lambda(\hat{K}^{(\ell)})$ denotes the minimum eigenvalue of $\hat{K}^{(\ell)}$. Therefore

\begin{multline*}
\frac{1}{\lambda(K^{(\ell)})}\le \frac{1}{\lambda(\hat{K}^{(\ell)})}\le \frac{\Big(\operatorname{tr}(\hat{K}^{(\ell)})\Big)^{d-1}}{\operatorname{det}(\hat{K}^{(\ell)})}
\le \lambda(K^{(2)})^{-d} \left(C_W\sum_{j=1}^d\mathbb{E}\Big[\sigma^2(G_1^{(\ell-1)}(x^{(j)}))\Big]\right)^{d-1}\cdot\\
\cdot\frac{d!}{C_W^{d(\ell-2)}}\prod_{i=1}^d\left(\sum_{\substack{k_i=1\\k_i\ne k_s\, \forall s=1,\dots,i-1}}^d\prod_{r_i=1}^{\ell-2}\mathbb{E}\big[\sigma'(G^{(\ell-r_i)}_1(x^{(k_i)}))\big]^2\right)^{-1}
\end{multline*}
\begin{multline*}
   \le \lambda(K^{(2)})^{-d}\Bigg((1+C_W)^2(1+(2C_W\|\sigma\|_{Lip}^2)^{\ell-1})\Big(d\ell( C_b+2|\sigma(0)|^2)+\frac{\sum_{j=1}^d\|x^{(j)}\|^2}{n_0}\Big)\Bigg)^{d-1}\cdot\\
   \cdot \frac{d!}{C_W^{d(\ell-2)}}\prod_{i=1}^d\left(\sum_{\substack{k_i=1\\k_i\ne k_s\, \forall s=1,\dots,i-1}}^d\prod_{r_i=1}^{\ell-2}\mathbb{E}\big[\sigma'(G^{(\ell-r_i)}_1(x^{(k_i)}))\big]^2\right)^{-1}.
\end{multline*}
\end{remark}

\subsection{Final result}
The following theorem, proved in Appendix B, gives an explicit bound on the convex distance between the neural network evaluated at multiple points and its Gaussian limit.

\begin{theorem}\label{th_fin_mult_dim}
Let \( L \geq 1 \) be an integer, and assume \( C_b,C_W \neq 0 \). Let the neural network output \( z_1^{(L+1)} \), the activation function \( \sigma \), and the weight distribution \( W \) be defined as in Definition~\ref{def_NN}. Recall the Gaussian limit \( G_1^{(L+1)} \) from Theorem~\ref{Hanin}.
Suppose that \( \mathbb{E}[W^{5 \cdot 2^{L+2}}] < \infty \), and let \( d \geq 2 \) be an integer. Consider a set of distinct input points \( \mathcal{X} := \{x^{(1)}, \dots, x^{(d)}\} \subset \mathbb{R}^{n_0} \), and define
\[
z_1^{(L+1)}(\mathcal{X}) := \big(z_1^{(L+1)}(x^{(1)}), \dots, z_1^{(L+1)}(x^{(d)})\big) \in \mathbb{R}^{d \times n_{L+1}},
\]
and similarly define the Gaussian limit
\[
G_1^{(L+1)}(\mathcal{X}) := \big(G_1^{(L+1)}(x^{(1)}), \dots, G_1^{(L+1)}(x^{(d)})\big) \in \mathbb{R}^{d \times n_{L+1}}.
\]

If the covariance matrix \( K^{(2)} \) from Theorem~\ref{Hanin} satisfies \( \operatorname{det}(K^{(2)}) \neq 0 \) and for every $\ell=2,\dots,L$ and $i=1,\dots d-1$ we have that $\mathbb{E}\big[\sigma'(G^{(\ell)}_1(x^{(i)}))\big]\ne 0$, then
\begin{multline*}
d_C(z_1^{(L+1)}(\mathcal{X}),G_1^{(L+1)}(\mathcal{X}))\le 
 Cd^{d+6}(L + 1)^{3L+d+20}
\Big(\sum_{j=1}^{L}\frac{1}{n_j}\Big)^{1/2}\left(1+C_W+\frac{1}{C_W}\right)^{4d+dL-4}\cdot\\
\cdot (1+C_b)^{2d-2}(1+|\sigma(0)|)^4\Bigg(9d+\frac{\sum_{j=1}^d\|x^{(j)}\|^4}{n_0^2}\Bigg)^{d-1}\Big(2+(2C_W\|\sigma\|_{Lip}^2)^{2L}\Big)^{L+d}\cdot\\
\cdot \Bigg(2+ \Big( K\frac{5\cdot 2^{L+1}}{\log(5\cdot 2^{L+1})}  \sqrt{C_W} \|\sigma\|_{\mathrm{Lip}} \mathbb{E}[(W_{1,1}^{(1)})^{5\cdot 2^{L+1}}]^{1/(5\cdot 2^{L+1})} \Big)^{L - 1}\Bigg)^{\frac{L}{2}+24}\cdot\\
\cdot
 \Bigg(\sum_{i=1}^d\Big(1+\left(M^{(L)}_2(x^{(i)})\right)^{\frac{L}{2}+2}\Big)\Bigg)
  2^{4L}   \Big(
 4\sqrt{5\cdot 2^L }
+ K\frac{5\cdot 2^{L+1}}{\log(5\cdot 2^{L+1})}
\Big)^{2L+8}\cdot\\
\cdot \left(1+d!\lambda(K^{(2)})^{-2d}\sum_{\ell=2}^{L+1}\prod_{i=1}^d\left(\sum_{\substack{k_i=1\\k_i\ne k_s\, \forall s=1,\dots,i-1}}^d\prod_{r_i=1}^{\ell-2}\mathbb{E}\big[\sigma'(G^{(\ell-r_i)}_1(x^{(k_i)}))\big]^2\right)^{-1}\right),
\end{multline*}
where $C,K>0$ are numeric constants that do not depend on any parameter of the neural network, $K$ is defined as in Lemma \ref{mom_p_iid0}, and for every $j=1,\dots,d$, $M_2^{(L)}(x^{(j)})$ is defined as in Lemma \ref{Q_2}.

\end{theorem}

An immediate consequence of the previous theorem is the next Proposition \ref{fin_relu}, where the activation function $\sigma$ is the ReLU activation, defined as $\sigma(x):=x\,\mathbf{1}_{\{x\ge 0\}}$ for every $x \in \mathbb{R}$. In this case, it is interesting to notice that the upper bound on the convex distance between the neural network and its Gaussian limit is completely explicit in terms of $L, n_1, \dots, n_L$.

\begin{prop}\label{fin_relu}
Assume the hypotheses of Theorem~\ref{th_fin_mult_dim} and let the activation function be the ReLU function, defined by $\sigma(x) := x \mathbf{1}_{\{x \geq 0\}}$ for all $x \in \mathbb{R}$. Then the following bound holds:

\begin{multline}\label{d_c_relu}
d_C(z_1^{(L+1)}(\mathcal{X}),G_1^{(L+1)}(\mathcal{X}))\le 
 Cd^{d+6}(L + 1)^{3L+d+20}
\Big(\sum_{j=1}^{L}\frac{1}{n_j}\Big)^{1/2}\left(1+C_W+\frac{1}{C_W}\right)^{4d+dL-4}\cdot\\
\cdot (1+C_b)^{2d-2}\Bigg(9d+\frac{\sum_{j=1}^d\|x^{(j)}\|^4}{n_0^2}\Bigg)^{d-1}\Big(2+(2C_W)^{2L}\Big)^{L+d}\cdot\\
\cdot \Bigg(2+ \Big( K\frac{5\cdot 2^{L+1}}{\log(5\cdot 2^{L+1})}  \sqrt{C_W} \mathbb{E}[(W_{1,1}^{(1)})^{5\cdot 2^{L+1}}]^{1/(5\cdot 2^{L+1})} \Big)^{L - 1}\Bigg)^{\frac{L}{2}+24}\cdot\\
\cdot
 \Bigg(\sum_{i=1}^d\Big(1+\left(M^{(L)}_2(x^{(i)})\right)^{\frac{L}{2}+2}\Big)\Bigg)
  2^{4L}   \Big(
 4\sqrt{5\cdot 2^L }
+ K\frac{5\cdot 2^{L+1}}{\log(5\cdot 2^{L+1})}
\Big)^{2L+8} \left(1+\frac{4^{L}}{\lambda(K^{(2)})^{2d}}\right)
\end{multline}
where for every $x\in\mathbb{R}^{n_0}$
\begin{multline*}
    M_2^{(L)}(x)={5^{4}}2^{52}(1+C_W)^{38}\Big(1+\frac{\|x\|}{\sqrt{n_0}}\Big)^{48} \Big(2+C_b+\frac{1}{C_b}+\frac{1}{C_b^2}\Big)^{24}\Big(1+\mathbb{E}[(W_{1,1}^{(1)})^{5\cdot 2^{L+2}}]^{\frac{1}{2^L}}\Big)^3.
\end{multline*}
\end{prop}
{
\begin{proof}
To obtain the bound in \eqref{d_c_relu}, it suffices to apply Theorem~\ref{th_fin_mult_dim} and note that for the ReLU activation function we have $\|\sigma\|_{\mathrm{Lip}} = 1$, $\sigma(0) = 0$, and $\sigma'(x) = 1_{\{x > 0\}}$ for $x \neq 0$.  
Consequently, 
\[
\mathbb{E}\big[\sigma'(G^{(\ell)}_1(x^{(i)}))\big] = \frac{1}{2}
\]
for every $\ell = 1,\dots,L+1$ and every $i = 1,\dots,d$.

\end{proof}

In the next subsection, we study another special case, namely when the activation function is linear.

}

\subsection{A special case: the identity activation function}\label{id:mult}

In this section, we consider the case where the activation function introduced in Definition~\ref{def_NN} is the identity:
\[
\sigma(x) := x \quad \text{for all } x \in \mathbb{R}.
\]
As in the proof of Proposition~\ref{mom_z_prop}, it follows that for every $\ell = 1, \dots, L$

\begin{multline*}
\mathbb{E}\Big[\|z_1^{(\ell)}(\mathcal{X})\|^6\Big]^{1/2}\le 8\sqrt{15}\ell C_b^{3/2}d\sqrt{d}\Big(108\max\{C_W,1\}^{3/2}\mathbb{E}[(W_{1,1}^{(1)})^6]^{1/2}\Big)^\ell\\
+\sqrt{d}\Big(108C_W^{3/2}\mathbb{E}[(W_{1,1}^{(1)})^6]^{1/2}\Big)^\ell\sum_{i=1}^d\sqrt{\sum_{j=1}^{n_0}\frac{|x_j^{(i)}|^6}{{n_0}}}
\end{multline*}
and
\begin{multline}\label{mom_4_nn_id}
\mathbb{E}\Big[\|z_1^{(\ell)}(\mathcal{X})\|^4\Big]^{1/2}\le 2\sqrt{3}(\ell+1)C_bd\Big(8\max\{C_W,1\}\mathbb{E}[(W_{1,1}^{(1)})^4]^{1/2}\Big)^\ell\\
+\Big(8C_W\mathbb{E}[(W_{1,1}^{(1)})^4]^{1/2}\Big)^\ell\sum_{i=1}^d\sqrt{\sum_{j=1}^{n_0}\frac{|x_j^{(i)}|^4}{{n_0}}}.
\end{multline}
On the other hand, by exploiting the properties of Gaussian distributions, we obtain that for every $\ell = 1, \dots, L$,

\begin{equation}\label{mom_4_G_id}
\mathbb{E}\Big[\|G_1^{(\ell)}(\mathcal{X})\|^4\Big]^{1/2} \le \sqrt{3}d(\ell+2)C_b\big(\max\{C_W,1\}\big)^{\ell+1} + \sqrt{3}C_W^{\ell+1}\sum_{i=1}^d \frac{\|x^{(i)}\|^2}{n_0}.
\end{equation}

We now observe that, recalling definition~\eqref{B_l},

\begin{multline*}
\sum_{i,j=1}^d\mathbb{E}\Big[B_L(x^{(i)},x^{(j)})^2\Big]\\
=\sum_{i,j=1}^d\mathbb{E}\Big[\Big|\mathbb{E}\Big[z_1^{(L)}(x^{(i)})z_1^{(L)}(x^{(j)})|\mathcal{F}_{L-1}\Big]-\mathbb{E}\Big[G_1^{(L-1)}(x^{(i)})G_1^{(L-1)}(x^{(j)})\Big]\Big|^2\Big]
\end{multline*}
\[
=\sum_{i,j=1}^d\mathbb{E}\Big[\Big|\frac{C_W}{n_{L-1}}\sum_{k=1}^{n_{L-1}}\Big(z_k^{(L-1)}(x^{(i)})z_k^{(L-1)}(x^{(j)})-\mathbb{E}\Big[G_k^{(L-1)}(x^{(i)})G_k^{(L-1)}(x^{(j)})\Big]\Big)\Big|^2\Big]
\]
\begin{equation}\label{iter_B_id}
\le \frac{2C_W^2}{n_{L-1}}\Big(\mathbb{E}\Big[\|z_1^{(L-1)}(\mathcal{X})\|^4\Big]+\mathbb{E}\Big[\|G_1^{(L-1)}(\mathcal{X})\|^4\Big]\Big)+C_W^2\sum_{i,j=1}^d\mathbb{E}\Big[B_{L-1}(x^{(i)},x^{(j)})^2\Big],
\end{equation}
using the recursive definition of the covariance matrix of $G^{(L)}_1(\mathcal{X})$ (see Theorem~\ref{Hanin}), together with the independence of the components of the neural network $z^{(L-1)}(\mathcal{X})$ conditional on the $\sigma$-field generated by the weights and biases up to layer $L-2$.

Iterating inequality~\eqref{iter_B_id}, we obtain

\[
\sum_{i,j=1}^d\mathbb{E}\Big[B_L(x^{(i)},x^{(j)})^2\Big]\le 2\sum_{k=1}^{L-1} \frac{C_W^{2k}}{n_{L-k}}\Big(\mathbb{E}\Big[\|z_1^{(L-k)}(\mathcal{X})\|^4\Big]+\mathbb{E}\Big[\|G_1^{(L-k)}(\mathcal{X})\|^4\Big]\Big)
\]
and therefore, by applying inequalities~\eqref{mom_4_nn_id} and~\eqref{mom_4_G_id},

\begin{multline*}
    \Bigg(\sum_{i,j=1}^d\mathbb{E}\Big[B_L(x^{(i)},x^{(j)})^2\Big]\Bigg)^{1/2}\le \sqrt{2}\Bigg(4\sqrt{3}dLC_b\Big(8\max\{C_W,1\}^2\mathbb{E}[(W_{1,1}^{(1)})^4]^{1/2}\Big)^L\\
    +\Big(8\max\{C_W,1\}^2\mathbb{E}[(W_{1,1}^{(1)})^4]^{1/2}\Big)^L\sum_{i=1}^d\Big(\frac{\|x^{(i)}\|^2}{n_0}+\sqrt{\sum_{j=1}^{n_0}\frac{|x^{(i)}_j|^4}{n_0}}\Big)\Bigg)\sqrt{\sum_{k=1}^{L-1}\frac{1}{{n_{L-k}}}}.
\end{multline*}
To obtain a bound on the convex distance $d_C(z_1^{(L+1)}(\mathcal{X}),G_1^{(L+1)}(\mathcal{X}))$ in this specific case, it remains to estimate from below the smallest eigenvalue of $K^{(L+1)}$, denoted by $\lambda(K^{(L+1)})$. For every $y \in \mathbb{R}^d$ with $\|y\| = 1$, we have
\[
y^T K^{(L+1)} y \ge C_W y^T K^{(L)} y \ge \dots \ge C_W^{L-1} y^T K^{(2)} y,
\]
and therefore,
\[
\lambda(K^{(L+1)}) \ge C_W^{L-1} \lambda(K^{(2)}).
\]

\vspace{0.5em}

If $\mathbb{E}[(W_{1,1}^{(1)})^6] < \infty$ and $\det(K^{(2)}) \ne 0$, then, by Theorem~\ref{mod_KG},

\begin{multline}\label{d_c_id_mult_noG}
   d_C(z_1^{(L+1)}(\mathcal{X}),G_1^{(L+1)}(\mathcal{X}))\le Cd^4\sqrt{d}\Big(d+\sum_{i=1}^d\frac{\|x^{(i)}\|^2}{n_0}+2\sum_{i=1}^d\sqrt{\sum_{j=1}^{n_0}\frac{|x_j^{(i)}|^6+|x_j^{(i)}|^4}{n_0}}\Big)\cdot\\
    \cdot \sqrt{C_W}(1+C_W)(1+C_b+LC_b(1+C_b))\mathbb{E}[(W_{1,1}^{(1)})^6]^{1/2}\max\Bigg\{1,\frac{1}{C_W^{2(L-1)}\lambda(K^{(2)})^2}\Bigg\}\Big(\sum_{k=0}^{L-1}\frac{1}{n_{L-k}}\Big)^{1/2}\cdot\\
    \cdot \Big(108\max\{1,C_W\}^2\mathbb{E}[(W_{1,1}^{(1)})^6]^{1/2}\Big)^{L+1},
\end{multline}
where $C > 0$ is a universal constant independent of all parameters of the neural network.

\begin{remark}
    The bound in \eqref{d_c_id_mult_noG} does not make explicit the dependence on the parameters $d$ and $n_0$, but it shows, under the sole assumption that $\mathbb{E}[(W_{1,1}^{(1)})^6] < \infty$ and that $\det(K^{(2)}) \ne 0$, how the parameters $L, n_1, \dots, n_L$ affect the convergence in relation to the values of $C_b$, $C_W$, and $\mathbb{E}[(W_{1,1}^{(1)})^6]$. 

{
   It is worth noting that when $C_b = 0$ and $C_W = 1$, the condition $\frac{L}{n} \to 0$ as $n, L \to \infty$ (with $n \approx n_1, \dots, n_L$) is not sufficient to guarantee that the upper bound in \eqref{d_c_id_mult_noG} converges to zero.  
This contrasts with Example~\ref{ex_id_G_1d}, where with Gaussian initialization both the Kolmogorov and Wasserstein distances converge to zero under the same scaling of $L$ and $n$.  
In particular, in Example~\ref{ex_id_G_1d} the choice of Gaussian weights ensures that all moments of the network depend only on the second moment, which remains bounded with $L$ when $C_W = 1$ and $C_b = 0$.  
We now show that the same phenomenon occurs in the multidimensional case.
} In fact, if we now assume that $W_{i,j}^{(\ell)} \sim \mathcal{N}_1(0,1)$ for all $i,j,k$, and we repeat the computations for the general linear case, we obtain:
\begin{multline*}
    d_C(z_1^{(L+1)}(\mathcal{X}), G_1^{(L+1)}(\mathcal{X})) \le 541\, d^4 \sqrt{C_W}(1 + C_W) \max\left\{1, \frac{1}{C_W^{2(L-1)} \lambda(K^{(2)})^2} \right\} \cdot \\
    \cdot \mathbb{E}[(W_{1,1}^{(1)})^6]^{1/2} \max\{1, C_W\}^{2(L+1)} \left(d + \sum_{i=1}^d \frac{\|x^{(i)}\|^2}{n_0} \right)^{3/2} (C_b L + 1)^{3/2} \sqrt{\sum_{k=0}^{L-1} \frac{1}{n_{L-k}}},
\end{multline*}
where $C > 0$ is a universal constant.

It follows that, when $C_b = 0$ and $C_W = 1$, 
\[
d_C(z_1^{(L+1)}(\mathcal{X}), G_1^{(L+1)}(\mathcal{X})) = O\left(\sqrt{\sum_{k=0}^{L-1} \frac{1}{n_{L-k}}} \right),
\]
which coincides with the one-dimensional case described in Example~\ref{ex_id_G_1d}.
\end{remark}

{
\section*{Acknowledgments}
The author was supported by the Luxembourg National Research Fund via the grant
PRIDE/21/16747448/MATHCODA. The author is grateful to Giovanni Peccati for guidance and to Larry Goldstein, Boris Hanin, Domenico Marinucci, Ivan Nourdin and Dario Trevisan for several useful discussions.

}

\bibliographystyle{plain}
\bibliography{references}

@book{Agg,
  author = {C. C. Aggarwal},
  year = {2023},
  title = {Neural Networks and Deep Learning: A Textbook},
  publisher = {Springer International Publishing},
  isbn = {9783031296420}
}

@InProceedings{AZLS19,
  title = 	 {A Convergence Theory for Deep Learning via Over-Parameterization},
  author =       {Allen-Zhu, Z. and Li, Y. and Song, Z.},
  booktitle = 	 {Proceedings of the 36th International Conference on Machine Learning},
  pages = 	 {242--252},
  year = 	 {2019},
  editor = 	 {Chaudhuri, Kamalika and Salakhutdinov, Ruslan},
  volume = 	 {97},
  series = 	 {Proceedings of Machine Learning Research},
  month = 	 {09--15 Jun},
  publisher =    {PMLR},
  url = 	 {https://proceedings.mlr.press/v97/allen-zhu19a.html}
}

@article{Torr23,
  title={Normal approximation of random gaussian neural networks},
  author={Apollonio, N. and De Canditiis, D. and Franzina, G. and Stolfi, P. and Torrisi, G. L.},
  journal={Stochastic Systems},
  volume={15},
  number={1},
  pages={88--110},
  year={2025},
  publisher={INFORMS}
}

@book{Baldi17,
  title={Stochastic Calculus: An Introduction Through Theory and Exercises},
  author={Baldi, P.},
  isbn={9783319622262},
  series={Universitext},
  url={https://books.google.lu/books?id=fO09DwAAQBAJ},
  year={2017},
  publisher={Springer International Publishing}
}

@article{PAPGGR23,
  title={A statistical mechanics framework for Bayesian deep neural networks beyond the infinite-width limit},
  author={Pacelli, R. and Ariosto, S. and Pastore, M. and Ginelli, F. and Gherardi, M. and Rotondo, P.},
  journal={Nature Machine Intelligence},
  volume={5},
  number={12},
  pages={1497--1507},
  year={2023},
  publisher={Nature Publishing Group UK London}
}

@inbook{arora_2020,
author = {Arora, S. and Du, S. S. and Hu, W. and Li, Z. and Salakhutdinov, R. and Wang, R.},
title = {On exact computation with an infinitely wide neural net},
year = {2019},
publisher = {Curran Associates Inc.},
address = {Red Hook, NY, USA},
booktitle = {Proceedings of the 33rd International Conference on Neural Information Processing Systems},
articleno = {731},
numpages = {10},
}

@article{BGIPR25,
  author  = {F. Bassetti and M. Gherardi and A. Ingrosso and M. Pastore and P. Rotondo},
  title   = {Feature Learning in Finite-Width Bayesian Deep Linear Networks with Multiple Outputs and Convolutional Layers},
  journal = {Journal of Machine Learning Research},
  year    = {2025},
  volume  = {26},
  number  = {88},
  pages   = {1--35},
  url     = {http://jmlr.org/papers/v26/24-1158.html}
}

@misc{BLR24,
      title={Proportional infinite-width infinite-depth limit for deep linear neural networks}, 
      author={F. Bassetti and L. Ladelli and Pietro Rotondo},
      year={2024},
      eprint={2411.15267},
      archivePrefix={arXiv},
      primaryClass={stat.ML},
      url={https://arxiv.org/abs/2411.15267}, 
}

@article{Ca82,
  title={On Upper and Lower Bounds for the Variance of a Function of a Random Variable},
  author={T. Cacoullos},
  journal={Annals of Probability},
  year={1982},
  volume={10},
  pages={799-809},
  url={https://api.semanticscholar.org/CorpusID:122132681}
}

@misc{CP25,
      title={Entropic bounds for conditionally Gaussian vectors and applications to neural networks}, 
      author={L. Celli and G. Peccati},
      year={2025},
      eprint={2504.08335},
      archivePrefix={arXiv},
      primaryClass={math.PR},
      url={https://arxiv.org/abs/2504.08335}, 
}

@article{C08,
  title={A NEW METHOD OF NORMAL APPROXIMATION},
  author={S. Chatterjee},
  journal={Annals of Probability},
  year={2006},
  volume={36},
  pages={1584-1610},
  url={https://api.semanticscholar.org/CorpusID:16326972}
}

@misc{CM08,
      title={Multivariate normal approximation using exchangeable pairs}, 
      author={S. Chatterjee and E. Meckes},
      year={2008},
      eprint={math/0701464},
      archivePrefix={arXiv},
      primaryClass={math.PR},
      url={https://arxiv.org/abs/math/0701464}, 
}

@book{CGS10,
  title={Normal Approximation by Stein’s Method},
  author={Chen, L.H.Y. and Goldstein, L. and Shao, Q.M.},
  isbn={9783642150074},
  lccn={2010938379},
  series={Probability and Its Applications},
  url={https://books.google.lu/books?id=5jMVpAbs9UkC},
  year={2010},
  publisher={Springer Berlin Heidelberg}
}

@article{COB19,
  title={On lazy training in differentiable programming},
  author={Chizat, L. and Oyallon, E. and Bach, F.},
  journal={Advances in neural information processing systems},
  volume={32},
  year={2019}
}

@Article{Cybenko,
 Author = {Cybenko, G.},
 Title = {Approximation by superpositions of a sigmoidal function},
 FJournal = {MCSS. Mathematics of Control, Signals, and Systems},
 Journal = {Math. Control Signals Syst.},
 ISSN = {0932-4194},
 Volume = {2},
 Number = {4},
 Pages = {303--314},
 Year = {1989},
 Language = {English},
 DOI = {10.1007/BF02551274},
 zbMATH = {4114557},
 Zbl = {0679.94019}
}

@article{FK21,
  title={High-dimensional central limit theorems by Stein’s method},
  author={Fang, X. and Koike, Y.},
  journal={The Annals of Applied Probability},
  volume={31},
  number={4},
  pages={1660--1686},
  year={2021},
  publisher={JSTOR}
}

@article{FK22,
  title={New error bounds in multivariate normal approximations via exchangeable pairs with applications to Wishart matrices and fourth moment theorems},
  author={Fang, X. and Koike, Y.},
  journal={The Annals of Applied Probability},
  volume={32},
  number={1},
  pages={602--631},
  year={2022},
  publisher={Institute of Mathematical Statistics}
}

@article{FKLZ23,
  title={High-dimensional Central Limit Theorems by Stein's Method in the Degenerate Case},
  author={Fang, X. and Koike, Y. and Liu, S.-H. and Zhao, Y.-K.},
  journal={arXiv preprint arXiv:2305.17365},
  year={2023}
}

@article{BGRS23,
  author = {K. Balasubramanian and L. Goldstein and N. Ross and A. Salim},
  year = {2024},
  title = {Gaussian random field approximation via Stein's method with applications to wide random neural networks},
  journal = {Applied and Computational Harmonic Analysis},
  volume = {72},
  pages = {101668},
  issn = {1063-5203}
}

@misc{BR25,
      title={Finite-Dimensional Gaussian Approximation for Deep Neural Networks: Universality in Random Weights}, 
      author={K. Balasubramanian and N. Ross},
      year={2025},
      eprint={2507.12686},
      archivePrefix={arXiv},
      primaryClass={stat.ML},
      url={https://arxiv.org/abs/2507.12686}, 
}

@article{BT24,
  author = {A. Basteri and D. Trevisan},
  year = {2024},
  title = {Quantitative Gaussian approximation of randomly initialized deep neural networks},
  journal = {Mach. Learn.},
  volume = {113},
  pages = {6373--6393}
}

@article{BFF23,
  title={Infinitely wide limits for deep Stable neural networks: sub-linear, linear and super-linear activation functions},
  author={A. Bordino and S. Favaro and S. Fortini},
  journal={Trans. Mach. Learn. Res.},
  year={2023},
  volume={2022},
  url={https://api.semanticscholar.org/CorpusID:254587104}
}

@inproceedings{BFF24,
  author = {A. Bordino and S. Favaro and S. Fortini},
  year = {2024},
  title = {Non-asymptotic approximations of Gaussian neural networks via second-order Poincaré inequalities},
  booktitle = {Proceedings of Machine Learning Research (AABI24)}
}

@article{CMSV24,
  author = {V. Cammarota and D. Marinucci and M. Salvi and S. Vigogna},
  year = {2024},
  title = {A quantitative functional central limit theorem for shallow neural networks},
  journal = {Modern Stochastics: Theory and Applications},
  volume = {11},
  number = {1},
  pages = {85--108}
}

@article{CC24,
title = "The Positivity of the Neural Tangent Kernel",
author = "L. Carvalho and Costa, \{J. L.\} and J. Mourao and G. Oliveira",
note = "Publisher Copyright: {\textcopyright} 2025 Society for Industrial and Applied Mathematics.",
year = "2025",
doi = "10.1137/24M1659534",
language = "English",
volume = "7",
pages = "495--515",
journal = "SIAM Journal on Mathematics of Data Science",
issn = "2577-0187",
publisher = "Society for Industrial and Applied Mathematics Publications",
number = "2",
}

@misc{CC23,
  author = {L. Carvalho and J. L. Costa and J. Mourão and G. Oliveira},
  year = {2023},
  title = {Wide neural networks: From non-gaussian random fields at initialization to the NTK geometry of training},
  note = {arXiv:{2304.03385}}
}

@inproceedings{CKZ23,
  title={Bayes-optimal learning of deep random networks of extensive-width},
  author={Cui, H. and Krzakala, F. and Zdeborov{\'a}, L.},
  booktitle={International Conference on Machine Learning},
  pages={6468--6521},
  year={2023},
  organization={PMLR}
}

@book{decoupling,
  title={Decoupling: From Dependence to Independence},
  author={de la Pe{\~n}a, V. and Gin{\'e}, E.},
  isbn={9781461268086},
  lccn={98030322},
  series={Probability and Its Applications},
  url={https://books.google.lu/books?id=GfmXngEACAAJ},
  year={2012},
  publisher={Springer New York}
}

@article{Duetall18,
  author       = {S. S. Du and
                  X. Zhai and
                  B. P{\'{o}}czos and
                  A. Singh},
  title        = {Gradient Descent Provably Optimizes Over-parameterized Neural Networks},
  journal      = {CoRR},
  volume       = {abs/1810.02054},
  year         = {2018},
  url          = {http://arxiv.org/abs/1810.02054},
  eprinttype    = {arXiv},
  eprint       = {1810.02054},
  bibsource    = {dblp computer science bibliography, https://dblp.org}
}

@InProceedings{DLLWZ19,
  title = 	 {Gradient Descent Finds Global Minima of Deep Neural Networks},
  author =       {Du, S. and Lee, J. and Li, H. and Wang, L. and Zhai, X.},
  booktitle = 	 {Proceedings of the 36th International Conference on Machine Learning},
  pages = 	 {1675--1685},
  year = 	 {2019},
  editor = 	 {Chaudhuri, Kamalika and Salakhutdinov, Ruslan},
  volume = 	 {97},
  series = 	 {Proceedings of Machine Learning Research},
  month = 	 {09--15 Jun},
  publisher =    {PMLR},
  url = 	 {https://proceedings.mlr.press/v97/du19c.html}
}

@inproceedings{EMS21,
  author = {R. Eldan and D. Mikulincer and T. Schramm},
  year = {2021},
  title = {Non-asymptotic approximations of neural networks by Gaussian processes},
  booktitle = {Conference on Learning Theory, Proceedings of Machine Learning Research},
  pages = {1754--1775}
}

@article{EAT25,
  title={Quantitative convergence of trained single layer neural networks to Gaussian processes},
  author={Mosig, E. and Agazzi, A. and Trevisan, D.},
  journal={arXiv preprint arXiv:2509.24544},
  year={2025}
}

@article{FFP22Kernel,
  title={Large-width asymptotics for ReLU neural networks with $\alpha$-Stable initializations},
  author={Favaro, S. and Fortini, S. and Peluchetti, S.},
  journal={arXiv preprint arXiv:2206.08065},
  year={2022}
}

@article{FFP23,
  author = {S. Favaro and S. Fortini and S. Peluchetti},
  year = {2023},
  title = {Deep stable neural networks: large-width asymptotics and convergence rates},
  journal = {Bernoulli},
  volume = {29},
  number = {3},
  pages = {2574--2597}
}

@article{FHMNP,
  title={Quantitative CLTs in deep neural networks},
  author={S. Favaro and B. Hanin and D. Marinucci and I. Nourdin and G. Peccati},
  journal={Probability Theory and Related Fields},
  year={2023},
  volume={191},
  pages={933 - 977},
  url={https://api.semanticscholar.org/CorpusID:260636947}
}

@article{FGAOWTYWA21,
  title={Bayesian neural network priors revisited},
  author={Fortuin, V. and Garriga-Alonso, A. and Ober, S. W. and Wenzel, F. and R{\"a}tsch, G. and Turner, R. E. and van der Wilk, M. and Aitchison, L.},
  journal={arXiv preprint arXiv:2102.06571},
  year={2021}
}

@article{GSJW20,
  title={Disentangling feature and lazy training in deep neural networks},
  author={Geiger, M. and Spigler, S. and Jacot, A. and Wyart, M.},
  journal={Journal of Statistical Mechanics: Theory and Experiment},
  volume={2020},
  number={11},
  pages={113301},
  year={2020},
  publisher={IOP Publishing}
}

@inproceedings{G20,
  title={Towards a general theory of infinite-width limits of neural classifiers},
  author={Golikov, E.},
  booktitle={International Conference on Machine Learning},
  pages={3617--3626},
  year={2020},
  organization={PMLR}
}

@book{Goodfellow,
    title={Deep Learning},
    author={I. Goodfellow and Y. Bengio and A. Courville},
    publisher={MIT Press},
    note={\url{http://www.deeplearningbook.org}},
    year={2016}
}

@article{GRIPENBERG,
title = {Approximation by neural networks with a bounded number of nodes at each level},
journal = {Journal of Approximation Theory},
volume = {122},
number = {2},
pages = {260-266},
year = {2003},
issn = {0021-9045},
doi = {https://doi.org/10.1016/S0021-9045(03)00078-9},
url = {https://www.sciencedirect.com/science/article/pii/S0021904503000789},
author = {G. Gripenberg}
}

@article{Han22,
  author = {B. Hanin},
  year = {2023},
  title = {Random neural networks in the infinite width limit as Gaussian processes},
  journal = {Ann. Appl. Probab.},
  volume = {33},
  number = {6A},
  pages = {4798--4819}
}

@article{Hanin2022CorrelationFI,
  title={Correlation Functions in Random Fully Connected Neural Networks at Finite Width},
  author={B. Hanin},
  journal={ArXiv},
  year={2022},
  volume={abs/2204.01058},
  url={https://api.semanticscholar.org/CorpusID:247939951}
}

@article{Han_Gas,
  author = {B. Hanin},
  year = {2024},
  title = {Random Fully Connected Neural Networks as Perturbatively Solvable Hierarchies},
  journal = {Journal of Machine Learning Research}
}

@inproceedings{H18,
 author = {Hanin, B.},
 booktitle = {Advances in Neural Information Processing Systems},
 editor = {S. Bengio and H. Wallach and H. Larochelle and K. Grauman and N. Cesa-Bianchi and R. Garnett},
 pages = {},
 publisher = {Curran Associates, Inc.},
 title = {Which Neural Net Architectures Give Rise to Exploding and Vanishing Gradients?},
 url = {https://proceedings.neurips.cc/paper_files/paper/2018/file/13f9896df61279c928f19721878fac41-Paper.pdf},
 volume = {31},
 year = {2018}
}

@article{HN18,
  author = {B. Hanin and M. Nica},
  year = {2020},
  title = {Products of Many Large Random Matrices and Gradients in Deep Neural Networks},
  journal = {Commun. Math. Phys.},
  volume = {376},
  pages = {287--322}
}

@misc{HS18,
      title={Approximating Continuous Functions by ReLU Nets of Minimal Width}, 
      author={B. Hanin and M. Sellke},
      year={2018},
      eprint={1710.11278},
      archivePrefix={arXiv},
      primaryClass={stat.ML},
      url={https://arxiv.org/abs/1710.11278}, 
}

@article{HZ23,
  title={Bayesian interpolation with deep linear networks},
  author={Hanin, B. and Zlokapa, A.},
  journal={Proceedings of the National Academy of Sciences},
  volume={120},
  number={23},
  pages={e2301345120},
  year={2023},
  publisher={National Academy of Sciences}
}

@misc{HZ24,
      title={Bayesian Inference with Deep Weakly Nonlinear Networks}, 
      author={B. Hanin and A. Zlokapa},
      year={2024},
      eprint={2405.16630},
      archivePrefix={arXiv},
      primaryClass={stat.ML},
      url={https://arxiv.org/abs/2405.16630}, 
}

@article{Hornik,
  title={Approximation capabilities of multilayer feedforward networks},
  author={K. Hornik},
  journal={Neural Networks},
  year={1991},
  volume={4},
  pages={251-257},
  url={https://api.semanticscholar.org/CorpusID:7343126}
}

@article{HXP20,
  title={Provable benefit of orthogonal initialization in optimizing deep linear networks},
  author={Hu, W. and Xiao, L. and Pennington, J.},
  journal={arXiv preprint arXiv:2001.05992},
  year={2020}
}

@inproceedings{Jacot,
  author = {A. Jacot and F. Gabriel and C. Hongler},
  year = {2018},
  title = {Neural tangent kernel: Convergence and generalization in neural networks},
  booktitle = {Advances in neural information processing systems},
  volume = {31}
}

@inproceedings{KMM25,
author = {Karhadkar, K. and Murray, M. and Mont\'{u}far, G.},
title = {Bounds for the smallest eigenvalue of the NTK for arbitrary spherical data of arbitrary dimension},
year = {2025},
isbn = {9798331314385},
publisher = {Curran Associates Inc.},
address = {Red Hook, NY, USA},
booktitle = {Proceedings of the 38th International Conference on Neural Information Processing Systems},
articleno = {4387},
numpages = {53},
location = {Vancouver, BC, Canada},
series = {NIPS '24}
}

@article{KG22,
  author = {M. J. Kasprzak and G. Peccati},
  year = {2023},
  title = {Vector-valued statistics of binomial processes: Berry--Esseen bounds in the convex distance},
  journal = {Ann. Appl. Probab.},
  volume = {33},
  number = {5},
  pages = {3449--3492},
  doi = {10.1214/22-AAP1897}
}

@inproceedings{kidger,
  title={Universal approximation with deep narrow networks},
  author={K., Patrick and L., Terry},
  booktitle={Conference on learning theory},
  pages={2306--2327},
  year={2020},
  organization={PMLR}
}

@inproceedings{Klu22,
  author = {A. Klukowski},
  year = {2022},
  title = {Rate of convergence of polynomial networks to Gaussian processes},
  booktitle = {Conference on Learning Theory, Proceedings of Machine Learning Research},
  pages = {701--722}
}

@article{LG17,
  title={New Berry–Esseen bounds for functionals of binomial point processes.},
  author={R. Lachi{\`e}ze-Rey and G. Peccati},
  journal={Annals of Applied Probability},
  year={2017},
  volume={27},
  pages={1992-2031},
  url={https://api.semanticscholar.org/CorpusID:125289273}
}

@article{Leshno,
  title={Multilayer Feedforward Networks with a Non-Polynomial Activation Function Can Approximate Any Function},
  author={M. Leshno and V. Ya. Lin and A. Pinkus and S. Schocken},
  journal={New York University Stern School of Business Research Paper Series},
  year={1992},
  url={https://api.semanticscholar.org/CorpusID:206089312}
}

@article{lu,
  title={The expressive power of neural networks: A view from the width},
  author={Lu, Z. and Pu, H. and Wang, F. and Hu, Z. and Wang, L.},
  journal={Advances in neural information processing systems},
  volume={30},
  year={2017}
}

@article{MFBV25,
  title={Scaling resnets in the large-depth regime},
  author={Marion, P. and Fermanian, A. and Biau, G. and Vert, J. P.},
  journal={Journal of Machine Learning Research},
  volume={26},
  number={56},
  pages={1--48},
  year={2025}
}

@article{LNR23,
  title={The neural covariance SDE: Shaped infinite depth-and-width networks at initialization},
  author={Li, M. and Nica, M. and Roy, D.},
  journal={Advances in Neural Information Processing Systems},
  volume={35},
  pages={10795--10808},
  year={2022}
}

@article{NR21,
  title={A self consistent theory of gaussian processes captures feature learning effects in finite cnns},
  author={Naveh, G. and Ringel, Z.},
  journal={Advances in Neural Information Processing Systems},
  volume={34},
  pages={21352--21364},
  year={2021}
}

@inproceedings{NM20,
author = {Nguyen, Q. and Mondelli, M.},
title = {Global convergence of deep networks with one wide layer followed by pyramidal topology},
year = {2020},
isbn = {9781713829546},
publisher = {Curran Associates Inc.},
address = {Red Hook, NY, USA},
booktitle = {Proceedings of the 34th International Conference on Neural Information Processing Systems},
articleno = {1003},
numpages = {12},
location = {Vancouver, BC, Canada},
series = {NIPS '20}
}

@inproceedings{NMM22,
  title={Tight bounds on the smallest eigenvalue of the neural tangent kernel for deep relu networks},
  author={Nguyen, Q. and Mondelli, M. and Montufar, G. F.},
  booktitle={International Conference on Machine Learning},
  pages={8119--8129},
  year={2021},
  organization={PMLR}
}

@book{NP12,
place={Cambridge},
series={Cambridge Tracts in Mathematics},
title={Normal Approximations with Malliavin Calculus: From Stein’s Method to Universality},
publisher={Cambridge University Press},
author={Nourdin, I. and Peccati, G.},
year={2012},
collection={Cambridge Tracts in Mathematics}}

@ARTICLE{OS20,
  author={Oymak, S. and Soltanolkotabi, M.},
  journal={IEEE Journal on Selected Areas in Information Theory}, 
  title={Toward Moderate Overparameterization: Global Convergence Guarantees for Training Shallow Neural Networks}, 
  year={2020},
  volume={1},
  number={1},
  pages={84-105},
  doi={10.1109/JSAIT.2020.2991332}}

@InProceedings{PFF20,
  title = 	 {Stable behaviour of infinitely wide deep neural networks},
  author =       {Peluchetti, S. and Favaro, S. and Fortini, S.},
  booktitle = 	 {Proceedings of the Twenty Third International Conference on Artificial Intelligence and Statistics},
  pages = 	 {1137--1146},
  year = 	 {2020},
  editor = 	 {Chiappa, Silvia and Calandra, Roberto},
  volume = 	 {108},
  series = 	 {Proceedings of Machine Learning Research},
  month = 	 {26--28 Aug},
  publisher =    {PMLR},
  url = 	 {https://proceedings.mlr.press/v108/peluchetti20b.html}
}

@article{Pinkus,
  title={Approximation theory of the MLP model in neural networks},
  author={A. Pinkus},
  journal={Acta Numerica},
  year={1999},
  volume={8},
  pages={143 - 195},
  url={https://api.semanticscholar.org/CorpusID:16800260}
}

@inproceedings{PLRSDG16,
 author = {Poole, B. and Lahiri, S. and Raghu, M. and Sohl-Dickstein, J. and Ganguli, S.},
 booktitle = {Advances in Neural Information Processing Systems},
 editor = {D. Lee and M. Sugiyama and U. Luxburg and I. Guyon and R. Garnett},
 pages = {},
 publisher = {Curran Associates, Inc.},
 title = {Exponential expressivity in deep neural networks through transient chaos},
 url = {https://proceedings.neurips.cc/paper_files/paper/2016/file/148510031349642de5ca0c544f31b2ef-Paper.pdf},
 volume = {29},
 year = {2016}
}

@InProceedings{RPKGSD17,
  title = 	 {On the Expressive Power of Deep Neural Networks},
  author =       {M. Raghu and B. Poole and J. Kleinberg and S. Ganguli and J. Sohl-Dickstein},
  booktitle = 	 {Proceedings of the 34th International Conference on Machine Learning},
  pages = 	 {2847--2854},
  year = 	 {2017},
  editor = 	 {Precup, Doina and Teh, Yee Whye},
  volume = 	 {70},
  series = 	 {Proceedings of Machine Learning Research},
  month = 	 {06--11 Aug},
  publisher =    {PMLR},
  url = 	 {https://proceedings.mlr.press/v70/raghu17a.html}
}

@book{RYH22,
  author = {D. A. Roberts and S. Yaida and B. Hanin},
  year = {2022},
  title = {The Principles of Deep Learning Theory: An Effective Theory Approach to Understanding Neural Networks},
  publisher = {Cambridge University Press}
}

@article{SMCG13,
  title={Exact solutions to the nonlinear dynamics of learning in deep linear neural networks},
  author={A. M. Saxe and J. L. McClelland and S. Ganguli},
  journal={CoRR},
  year={2013},
  volume={abs/1312.6120},
  url={https://api.semanticscholar.org/CorpusID:17272965}
}

@book{shalev,
  title={Understanding Machine Learning: From Theory to Algorithms},
  author={Shalev-Shwartz, S. and Ben-David, S.},
  isbn={9781107057135},
  lccn={2014001779},
  series={Understanding Machine Learning: From Theory to Algorithms},
  url={https://books.google.lu/books?id=ttJkAwAAQBAJ},
  year={2014},
  publisher={Cambridge University Press}
}

@article{SY19,
  author       = {Z. Song and
                  X. Yang},
  title        = {Quadratic Suffices for Over-parametrization via Matrix Chernoff Bound},
  journal      = {CoRR},
  volume       = {abs/1906.03593},
  year         = {2019},
  url          = {http://arxiv.org/abs/1906.03593},
  eprinttype    = {arXiv},
  eprint       = {1906.03593},
  bibsource    = {dblp computer science bibliography, https://dblp.org}
}

@book{stein,
  title={Approximate Computation of Expectations},
  author={Stein, C. and Institute of Mathematical Statistics},
  isbn={9780940600089},
  lccn={86383372},
  series={IMS Lecture Notes},
  url={https://books.google.lu/books?id=6C2g7_-KnhsC},
  year={1986},
  publisher={Institute of Mathematical Statistics}
}

@misc{Trev,
  author = {D. Trevisan},
  year = {2023},
  title = {Wide deep neural networks with Gaussian weights are very close to Gaussian processes},
  note = {arXiv:2312.06092 [math.ST]}
}

@article{Vid,
  author = {A. Vidotto},
  year = {2020},
  title = {An Improved Second-Order Poincaré Inequality for Functionals of Gaussian Fields},
  journal = {Journal of Theoretical Probability},
  volume = {33},
  number = {1},
  pages = {396--427},
  month = {March},
  publisher = {Springer}
}

@book{villani,
  author = {C. Villani},
  year = {2009},
  title = {Optimal Transport, Old and New},
  series = {Grundlehren der mathematischen Wissenschaften},
  volume = {338},
  publisher = {Springer-Verlag Berlin Heidelberg}
}

@book{W19, place={Cambridge}, series={Cambridge Series in Statistical and Probabilistic Mathematics}, title={High-Dimensional Statistics: A Non-Asymptotic Viewpoint}, publisher={Cambridge University Press}, author={Wainwright, Martin J.}, year={2019}, collection={Cambridge Series in Statistical and Probabilistic Mathematics}}

@misc{WDW19,
      title={Global Convergence of Adaptive Gradient Methods for An Over-parameterized Neural Network}, 
      author={X. Wu and S. S. Du and R. Ward},
      year={2019},
      eprint={1902.07111},
      archivePrefix={arXiv},
      primaryClass={cs.LG},
      url={https://arxiv.org/abs/1902.07111}, 
}

@article{YH20,
  title={Feature learning in infinite-width neural networks},
  author={Yang, G. and Hu, E. J.},
  journal={arXiv preprint arXiv:2011.14522},
  year={2020}
}

@article{Yarotsky,
  title={Error bounds for approximations with deep ReLU networks},
  author={D. Yarotsky},
  journal={Neural networks : the official journal of the International Neural Network Society},
  year={2016},
  volume={94},
  pages={
          103-114
        },
  url={https://api.semanticscholar.org/CorpusID:426133}
}

@book{Ye,
  author = {J. C. Ye},
  year = {2022},
  title = {Geometry of Deep Learning: A Signal Processing Perspective},
  publisher = {Springer},
  address = {Singapore},
  pages = {195--226}
}

@article{ZVP21,
  title={Exact marginal prior distributions of finite Bayesian neural networks},
  author={Zavatone-Veth, J. and Pehlevan, C.},
  journal={Advances in Neural Information Processing Systems},
  volume={34},
  pages={3364--3375},
  year={2021}
}

@article{ZCZG19,
  title={Gradient descent optimizes over-parameterized deep ReLU networks},
  author={Zou, D. and Cao, Y. and Zhou, D. and Gu, Q.},
  journal={Machine Learning},
  volume={109},
  number={3},
  pages={467--492},
  year={2019},
  publisher={Springer},
  doi={10.1007/s10994-019-05839-6}
}

@inproceedings{ZG19,
 author = {Zou, D. and Gu, Q.},
 booktitle = {Advances in Neural Information Processing Systems},
 editor = {H. Wallach and H. Larochelle and A. Beygelzimer and F. d\textquotesingle Alch\'{e}-Buc and E. Fox and R. Garnett},
 pages = {},
 publisher = {Curran Associates, Inc.},
 title = {An Improved Analysis of Training Over-parameterized Deep Neural Networks},
 url = {https://proceedings.neurips.cc/paper_files/paper/2019/file/6a61d423d02a1c56250dc23ae7ff12f3-Paper.pdf},
 volume = {32},
 year = {2019}
}
\newpage

\section{Appendix A}\label{Appendix_A}
In this Appendix, we collect all results that are necessary for the proofs of the main theorems, as presented in Section~\ref{uno_d}.

\subsection{Proof of Lemma \ref{I_1,2,3}}
Prior to proving Lemma~\ref{I_1,2,3}, we present some auxiliary results.  
The next lemma follows from of a simple computation.

\begin{lemma}\label{lips}
For every $y\in\mathbb{R}$ and $x\in\mathbb{R}^{n_0}$
\begin{equation}\label{lips2}
|\sigma^2(y)-\mathbb{E}[\sigma^2(G_{1}^{(\ell)}(x))]|\le 2\|\sigma\|_{\text{Lip}}^2y^2+4|\sigma(0)|^2
+2\|\sigma\|_{\text{Lip}}^2K^{(\ell)}.
\end{equation}
\end{lemma}

\begin{lemma}\label{fs}
For every $z\in\mathbb{R}$
\begin{equation}\label{bound_f_2}
|f_{\sigma^2}(z)|\le  2{\|\sigma\|_{\text{Lip}}^2}|z|+4\sqrt{\frac{\pi K^{(\ell)} }{2}}\Big(\frac{|\sigma(0)|^2}{K^{(\ell)}}+\|\sigma\|_{\text{Lip}}^2\Big).
\end{equation}
\end{lemma}
\begin{proof}
Using the definition~\eqref{expr_f_s2_uni} of \( f_{\sigma^2} \), Lemma \ref{lips}, and Remark~3.2.4 on \( f_{\sigma^2} \) in~\cite{NP12}, we obtain that

\begin{multline}\label{bound_mod_f_s}
|f_{\sigma^2}(z)|\le \frac{e^{\frac{z^2}{2K^{(\ell)}}}}{K^{(\ell)}}\int_{|z|}^{\infty}e^{-\frac{y^2}{2K^{(\ell)}}}|\sigma^2(y)-\mathbb{E}[\sigma^2(G_{1}^{(\ell)}(x))]|dy\\
\le 2\|\sigma\|_{\text{Lip}}^2 \frac{e^{\frac{z^2}{2K^{(\ell)}}}}{K^{(\ell)}}\int_{|z|}^{\infty}|y|^2e^{-\frac{y^2}{2K^{(\ell)}}}dy+4|\sigma(0)|^2\frac{e^{\frac{z^2}{2K^{(\ell)}}}}{K^{(
\ell)}}\int_{|z|}^{\infty}e^{-\frac{y^2}{2K^{(\ell)}}}dy\\
+2\|\sigma\|_{\text{Lip}}^2e^{\frac{z^2}{2K^{(\ell)}}}\int_{|z|}^{\infty}e^{-\frac{y^2}{2K^{(\ell)}}}dy
\end{multline}
\begin{multline*}
\le 2 {\|\sigma\|_{\text{Lip}}^2}e^{\frac{z^2}{2K^{(\ell)}}}\Big(|z|e^{-\frac{z^2}{2K^{(\ell)}}}+\int_{|z|}^{\infty}e^{-\frac{y^2}{2K^{(\ell)}}}dy\Big)+4\frac{|\sigma(0)|^2}{\sqrt{K^{(\ell)}}}\sqrt{\frac{\pi}{2}}
+2\|\sigma\|_{\text{Lip}}^2 \sqrt{\frac{\pi K^{(\ell)}}{2}}
\end{multline*}
\[
\le 2{\|\sigma\|_{\text{Lip}}^2}\Big(|z|+\sqrt{\frac{\pi K^{(\ell)}}{2}}\Big)+4\frac{|\sigma(0)|^2}{\sqrt{K^{(\ell)}}}\sqrt{\frac{\pi}{2}}
+2\|\sigma\|_{\text{Lip}}^2 \sqrt{\frac{\pi K^{(\ell)}}{2}},
\]
where we have exploited the bound
\[
e^{\frac{z^2}{2K^{(\ell)}}} \int_{|z|}^{\infty} e^{-\frac{y^2}{2K^{(\ell)}}} \, dy \le \sqrt{\frac{\pi K^{(\ell)}}{2}},
\]
as shown in the proof of Theorem~3.3.1 in~\cite{NP12}.

\end{proof}

\begin{lemma}\label{lipf}
For every $z,t\in\mathbb{R}$, one has that
\begin{multline*}
|f_{\sigma^2}(z+t)-f_{\sigma^2}(z)|\le 
\Big\{{2\|\sigma\|_{\text{Lip}}^2}|z+t|+4\sqrt{\frac{\pi K^{(\ell)}}{2}}\Big(\frac{|\sigma(0)|^2}{K^{(\ell)}}+\|\sigma\|_{\text{Lip}}^2\Big)\Big\}\Big|\frac{(z+t)^2-z^2}{2K^{(\ell)}}\Big|\\
+\frac{ 2\|\sigma\|_{\text{Lip}}^2}{K^{(\ell)}}\max\{z^2, (z+t)^2\}|t|
+2\Big(\frac{2|\sigma(0)|^2}{K^{(\ell)}}+\|\sigma\|_{\text{Lip}}^2\Big)|t|.
\end{multline*}

\end{lemma}
\begin{proof}
Using the inequality \( 1 - e^{-z} \le z \) for all \( z \ge 0 \), one infers that 

\begin{multline*}
|f_{\sigma^2}(z+t)-f_{\sigma^2}(z)|
=\Big|\frac{e^{\frac{(z+t)^2}{2K^{(\ell)}}}}{K^{(\ell)}}\int_{-\infty}^{z+t}e^{-\frac{y^2}{2K^{(\ell)}}}\Big(\sigma^2( y)-\mathbb{E}[\sigma^2(G_1^{(\ell)}(x))]\Big)dy\\
-\frac{e^{\frac{z^2}{2K^{(\ell)}}}}{K^{(\ell)}}\int_{-\infty}^{z}e^{-\frac{y^2}{2K^{(\ell)}}}\Big(\sigma^2( y)-\mathbb{E}[\sigma^2(G_1^{(\ell)}(x)]\Big)dy\Big|
\end{multline*}
\begin{multline*}
\le
\begin{cases}
\Big|\Big(1-e^{\frac{(z^2-(z+t)^2)}{2K^{(\ell)}}}\Big)\frac{e^{\frac{(z+t)^2}{2K^{(\ell)}}}}{K^{(\ell)}}\int_{-\infty}^{z+t}e^{-\frac{y^2}{2K^{(\ell)}}}\Big(\sigma^2( y)-\mathbb{E}[\sigma^2(G_1^{(\ell)}(x))]\Big)dy\Big|\\
\quad\quad\quad\quad\quad\quad\quad+\Big|\frac{e^{\frac{z^2}{2K^{(\ell)}}}}{K^{(
\ell)}}\int_{z}^{z+t}e^{-\frac{y^2}{2K^{(\ell)}}}\Big(\sigma^2( y)-\mathbb{E}[\sigma^2(G_1^{(\ell)}(x))]\Big)dy\Big| & \text{if $t\ge 0$}\\
\quad\\
\Big|\Big(1-e^{\frac{((z+t)^2-z^2)}{2K^{(\ell)}}}\Big)\frac{e^{\frac{z^2}{2K^{(\ell)}}}}{K^{(\ell)}}\int_{-\infty}^{z}e^{-\frac{y^2}{2K^{(\ell)}}}\Big(\sigma^2(y)-\mathbb{E}[\sigma^2(G_1^{(\ell)}(x))]\Big)dy\Big|\\
\quad\quad\quad\quad\quad\quad\quad+\Big|\frac{e^{\frac{(z+t)^2}{2K^{(\ell)}}}}{K^{(\ell)}}\int_{z+t}^{z}e^{-\frac{y^2}{2K^{(\ell)}}}\Big(\sigma^2( y)-\mathbb{E}[\sigma^2(G_1^{(\ell)}(x))]\Big)dy\Big| & \text{if $t< 0$}\\
\end{cases}
\end{multline*}
\begin{multline}\label{dis_casi}
\le \Big\{{2\|\sigma\|_{\text{Lip}}^2}|z+t|+4\sqrt{\frac{\pi K^{(\ell)}}{2}}\Big(\frac{|\sigma(0)|^2}{K^{(\ell)}}+\|\sigma\|_{\text{Lip}}^2\Big)\Big\}\Big|\frac{(z+t)^2-z^2}{2K^{(\ell)}}\Big|\\
+\frac{ 2\|\sigma\|_{\text{Lip}}^2}{K^{(\ell)}}\max\{z^2, (z+t)^2\}|t|
+2\Big(\frac{2|\sigma(0)|^2}{K^{(\ell)}}+\|\sigma\|_{\text{Lip}}^2\Big)|t|.
\end{multline}
To verify that inequality~\eqref{dis_casi} holds, we consider separately the cases \( t \ge 0 \) and \( t < 0 \).

If \( t \ge 0 \), note that \( e^{\frac{z^2}{2K^{(\ell)}}} \le e^{\frac{y^2}{2K^{(\ell)}}} \) for every \( y \in [z, z+t] \). Using this fact together with the bound~\eqref{lips2}, we obtain

\begin{multline*}
\Big|\frac{e^{\frac{z^2}{2K^{(\ell)}}}}{K^{(\ell)}}\int_{z}^{z+t}e^{-\frac{y^2}{2K^{(\ell)}}}\Big(\sigma^2(y)-\mathbb{E}[\sigma^2(G_1^{(\ell)}(x))]\Big)dy\Big|\le 2\|\sigma\|_{\text{Lip}}^2\frac{e^{\frac{z^2}{2K^{(\ell)}}}}{K^{(\ell)}}\int_z^{z+t}|y|^2e^{-\frac{y^2}{2K^{(\ell)}}}dy\\
+4|\sigma(0)|^2\frac{e^{\frac{z^2}{2K^{(\ell)}}}}{K^{(\ell)}}\int_z^{z+t}e^{-\frac{y^2}{2K^{(\ell)}}}dy+2\|\sigma\|_{\text{Lip}}^2e^{\frac{z^2}{2K^{(\ell)}}}\int_z^{z+t}e^{-\frac{y^2}{2K^{(\ell)}}}dy
\end{multline*}
\[
\le \frac{2\|\sigma\|_{\text{Lip}}^2}{K^{(\ell)}}\max\{z^2, (z+t)^2\}t
+2\Big(\frac{2|\sigma(0)|^2}{K^{(\ell)}}+\|\sigma\|_{\text{Lip}}^2\Big)t.
\]

If $t<0$, one obtains the estimate
\begin{multline*}
\Big|\frac{e^{\frac{(z+t)^2}{2K^{(\ell)}}}}{K^{(\ell)}}\int_{z+t}^{x}e^{-\frac{y^2}{2K^{(\ell)}}}\Big(\sigma^2( y)-\mathbb{E}[\sigma^2(G_1^{(\ell)}(x))]\Big)dy\Big|\le  2\|\sigma\|_{\text{Lip}}^2\frac{e^{\frac{(z+t)^2}{2K^{(\ell)}}}}{K^{(\ell)}}\int_{z+t}^{z}|y|^2e^{-\frac{y^2}{2K^{(\ell)}}}dy\\
+\frac{4|\sigma(0)|^2}{K^{(\ell)}}e^{\frac{(z+t)^2}{2K^{(\ell)}}}\int_{z+t}^{z}e^{-\frac{y^2}{2K^{(\ell)}}}dy+2\|\sigma\|_{\text{Lip}}^2e^{\frac{(z+t)^2}{2K^{(\ell)}}}\int_{z+t}^{x}e^{-\frac{y^2}{2K^{(\ell)}}}dy
\end{multline*}
\[
\le\frac{ 2\|\sigma\|_{\text{Lip}}^2}{K^{(\ell)}}\max\{z^2, (z+t)^2\}|t|
+2\Big(\frac{2|\sigma(0)|^2}{K^{(\ell)}}+\|\sigma\|_{\text{Lip}}^2\Big)|t|.
\]
The rest of the proof follows directly from the bound~\eqref{bound_f_2}.
\end{proof}

\begin{proof}[Proof of Lemma \ref{I_1,2,3}]
We organize the proof into three steps:
\begin{itemize}
\item 
\[
I_3\le \frac{1}{K^{(\ell)}}\Big|\int_0^h \Big|\Big(\sigma(z_{1}^{(\ell)}(x)+t)-\sigma(z_{1}^{(\ell)}(x))\Big)\Big(\sigma(z_{1}^{(\ell)}(x)+t)+\sigma(z_{1}^{(\ell)}(x))\Big)\Big|dt\Big|
\]
\[
\le\frac{1}{K^{(\ell)}} \Big|\int_0^h \|\sigma\|_{\text{Lip}} |t|\Big(\|\sigma\|_{\text{Lip}} |t|+2|\sigma(z_{1}^{(\ell)}(x))|\Big)dt\Big|
\]
\[
=\frac{\|\sigma\|_{\text{Lip}}^2}{K^{(\ell)}} \frac{|h|^3}{3}+\frac{\|\sigma\|_{\text{Lip}}}{K^{(\ell)}} |\sigma(z_{1}^{(\ell)}(x))|{h^2}.
\]
\item To estimate $I_2$, we apply Lemma~\ref{fs}, which yields

\[
I_2\le \sqrt{\frac{\pi}{2K^{(\ell)}}}\Big(\frac{|\sigma(0)|^2}{K^{(\ell)}}+\|\sigma\|_{\text{Lip}}^2\Big)\frac{h^2}{2}
+\frac{2\|\sigma\|_{\text{Lip}}^2}{K^{(\ell)}}\Big|\int_0^h \Big|{z_{1}^{(\ell)}}(x)+t\Big||t| dt\Big|
\]
\[
\le \sqrt{\frac{\pi}{2K^{(\ell)}}}\Big(\frac{|\sigma(0)|^2}{K^{(\ell)}}+\|\sigma\|_{\text{Lip}}^2+\frac{2\|\sigma\|_{\text{Lip}}^2}{K^{(\ell)}}\Big)\frac{h^2}{2}
+\frac{2\|\sigma\|_{\text{Lip}}^2}{K^{(\ell)}}\frac{|h|^3}{3}.
\]

\item 
Applying Lemma~\ref{lipf}, we obtain

\begin{multline*}
|f_{\sigma^2}(z+t)-f_{\sigma^2}(z)|\le\Big\{{2\|\sigma\|_{\text{Lip}}^2}|z+t|+4\sqrt{\frac{\pi K^{(\ell)}}{2}}\Big(\frac{|\sigma(0)|^2}{K^{(\ell)}}+\|\sigma\|_{\text{Lip}}^2\Big)\Big\}\cdot\\
\cdot \Big|\frac{(z+t)^2-z^2}{2K^{(\ell)}}\Big|
+\frac{ 2\|\sigma\|_{\text{Lip}}^2}{K^{(\ell)}}\max\{z^2, (z+t)^2\}|t|
+2\Big(\frac{2|\sigma(0)|^2}{K^{(\ell)}}+\|\sigma\|_{\text{Lip}}^2\Big)|t|
\end{multline*}
\begin{multline*}
\le\Big\{{2\|\sigma\|_{\text{Lip}}^2}|z|+2\Big(2\sqrt{\frac{\pi K^{(\ell)}}{2}}+1\Big)\Big(\frac{2|\sigma(0)|^2}{K^{(\ell)}}+\|\sigma\|_{\text{Lip}}^2\Big)+4\frac{\|\sigma\|_{\text{Lip}}^2}{K^{(\ell)}}z^2\Big\}|t|\\
+\Big\{{6\|\sigma\|_{\text{Lip}}^2}|z|+4\sqrt{\frac{\pi K^{(\ell)}}{2}}\Big(\frac{|\sigma(0)|^2}{K^{(\ell)}}+\|\sigma\|_{\text{Lip}}^2\Big)\Big\}|t|^2
+\|\sigma\|_{\text{Lip}}^2\Big(\frac{4}{K^{(\ell)}}+2\Big)|t|^3.
\end{multline*}

Therefore
\begin{multline*}
I_1\le \frac{|z_{1}^{(\ell)}(x)|}{K^{(\ell)}}\Big[\Big(D_\ell(x)+\frac{4\|\sigma\|_{\text{Lip}}^2}{K^{(\ell)}} |z_{1}^{(\ell)}(x)|^2+2\|\sigma\|_{\text{Lip}}^2|z_{1}^{(\ell)}(x)|\Big)\frac{h^2}{2}\\
+\Big(D_{\ell}(x)+6{\|\sigma\|_{\text{Lip}}^2}|z_{1}^{(\ell)}(x)|\Big)\frac{|h|^3}{3}
+\|\sigma\|_{\text{Lip}}^2\Big(\frac{4}{K^{(\ell)}}+2\Big)\frac{h^4}{4}\Big]
\end{multline*}
where
\[
D_\ell(x)=2\Big(2\sqrt{\frac{\pi K^{(\ell)}}{2}}+1\Big)\Big(\frac{2|\sigma(0)|^2}{K^{(\ell)}}+\|\sigma\|_{\text{Lip}}^2\Big).
\]
\end{itemize}
\end{proof}
\subsection{Proof of Lemma \ref{derivative}}
Lemma~\ref{derivative} follows from inequalities~\eqref{lips2} and~\eqref{bound_f_2}.  
Indeed, using the identity~\eqref{stein_uno}, we obtain

\[
|{f'}_{\sigma^2}(z)|\le\frac{1}{K^{(\ell)}}|z{f}_{\sigma^2}(z)|+\frac{1}{K^{(\ell)}}|\sigma^2( z)-\mathbb{E}[\sigma^2(G_{1}^{(\ell)}(x)]|
\]
\[
\le 
 \frac{4\|\sigma\|_{\text{Lip}}^2}{K^{(\ell)}}|z|^2+4\sqrt{\frac{\pi  }{2K^{(\ell)}}}\Big(\frac{|\sigma(0)|^2}{K^{(\ell)}}+\|\sigma\|_{\text{Lip}}^2\Big)|z|+\frac{4|\sigma(0)|^2}{K^{(\ell)}}
+2\|\sigma\|_{\text{Lip}}^2,
\]
yielding to the desired conclusion.

\subsection{Proof of Lemma \ref{lemma_prim_Q2}}
Using (\ref{eq_in_Q2}), (\ref{eq_sec_Q2}), (\ref{term_R_Q2}) and (\ref{bound_der_Q2}),
\begin{multline}\label{bound_cond}
|\mathbb{E}[\sigma^2(z_{1}^{(\ell)}(x))|\mathcal{F}_{\ell-1}]-\mathbb{E}[\sigma^2 (G_{1}^{(\ell)}(x))]|\le \frac{C_W^{3/2}}{2{n_{\ell-1}^{3/2}}}\sum_{j=1}^{n_{\ell-1}}\mathbb{E}\Bigg[|W_{1,j}^{(\ell)}-{W'}_{1,j}^{(\ell)}|^3\cdot\\
\cdot|\sigma^3(z_{j}^{(\ell-1)}(x))|\Bigg(\frac{|z_{1}^{(\ell)}(x)|}{K^{(\ell)}}\Big(D_\ell(x)+\frac{4\|\sigma\|_{\text{Lip}}^2}{K^{(\ell)}} |z_{1}^{(\ell)}(x)|^2+2\|\sigma\|_{\text{Lip}}^2|z_{1}^{(\ell)}(x)|\Big)\\
+\sqrt{\frac{\pi}{2K^{(\ell)}}}\Big(\frac{|\sigma(0)|^2}{K^{(\ell)}}+\|\sigma\|_{\text{Lip}}^2+\frac{2\|\sigma\|_{\text{Lip}}^2}{K^{(\ell)}}\Big)
 +\frac{2\|\sigma\|_{\text{Lip}}}{K^{(\ell)}}\Big( |\sigma(0)|+\|\sigma\|_{\text{Lip}}|z_{1}^{(\ell)}(x)|\Big)\Bigg)\Big|\mathcal{F}_{\ell-1}\Bigg]\\
+\frac{C_W^2}{3n_{\ell-1}^2}\sum_{j=1}^{n_{\ell-1}}\mathbb{E}\Bigg[|W_{1,j}^{(\ell)}-{W'}_{1,j}^{(\ell)}|^4\sigma^4(z_{j}^{(\ell-1)}(x))\Bigg( \frac{|z_{1}^{(\ell)}(x)|}{K^{(\ell)}}\Big(D_\ell(x)+6{\|\sigma\|_{\text{Lip}}^2}|z_{1}^{(\ell)}(x)|\Big)\\
+\frac{3\|\sigma\|_{\text{Lip}}^2}{K^{(\ell)}}\Bigg)\Big|\mathcal{F}_{\ell-1}\Bigg]
+\frac{C_W^{5/2}}{4n_{\ell-1}^{5/2}}\|\sigma\|_{\text{Lip}}^2\Big(\frac{4}{K^{(\ell)}}+2\Big)\sum_{j=1}^{n_{\ell-1}}\mathbb{E}\Big[|W_{1,j}^{(\ell)}-{W'}_{1,j}^{(\ell)}|^5|\sigma^5(z_{j}^{(\ell-1)}(x))|\cdot\\
\cdot \frac{|z_{1}^{(\ell)}(x)|}{K^{(\ell)}}\Big|\mathcal{F}_{\ell-1}\Big]
+\Big(\tilde{D}_\ell(x)\mathbb{E}\Big[|z_{1}^{(\ell)}(x)|^2\big|\mathcal{F}_{\ell-1}\Big]+D_\ell(x)\Big)|A_{\ell}(x)|
+\mathbb{E}\Big[(W_{1,1}^{(\ell)})^4\Big]^{1/2}\frac{{C_W}}{2{n_{\ell-1}}} \cdot\\
\cdot\Big(\sum_{j=1}^{n_{\ell-1}}\sigma^4(z_{j}^{(\ell-1)}(x))\Big)^{1/2} \Bigg(\tilde{D}_\ell(x)\mathbb{E}\Big[|z_{1}^{(\ell)}(x)|^4|\mathcal{F}_{\ell-1}\Big]^{1/2}+D_\ell(x)\Bigg),
\end{multline}
where
\[
\tilde{D}_\ell(x):=\frac{4\|\sigma\|_{\text{Lip}}^2}{K^{(\ell)}}+4\sqrt{\frac{\pi  }{2K^{(\ell)}}}\Big(\frac{|\sigma(0)|^2}{K^{(\ell)}}+\|\sigma\|_{\text{Lip}}^2\Big),
\]
 \[
D_\ell(x):=2\Big(2\sqrt{\frac{\pi K^{(\ell)}}{2}}+1\Big)\Big(\frac{2|\sigma(0)|^2}{K^{(\ell)}}+\|\sigma\|_{\text{Lip}}^2\Big),
\]
and \[
A_{\ell}(x):=\sum_{j=1}^{n_{\ell-1}}\frac{{C_W}}{{n_{\ell-1}}}\sigma^2(z_{j}^{(\ell-1)}(x))-K^{(\ell)}+C_b.
\]
Considering an arbitrary $p \in \mathbb{N}$, one obtains

\begin{multline*}
\mathbb{E}\Big[|\mathbb{E}[\sigma^2(z_{1}^{(\ell)}(x))|\mathcal{F}_{\ell-1}]-\mathbb{E}[\sigma^2 (G_{1}^{(\ell)}(x))]|^{2p}\Big]\le
\frac{5^{2p-1}C_W^{3p}}{2^{2p}{n_{\ell-1}^{p+1}}}\sum_{j=1}^{n_{\ell-1}}\mathbb{E}\Bigg[|W_{1,j}^{(\ell)}-{W'}_{1,j}^{(\ell)}|^{6p}\cdot\\
\cdot|\sigma^{6p}(z_{j}^{(\ell-1)}(x))|\Bigg(\frac{|z_{1}^{(\ell)}(x)|}{K^{(\ell)}}\Big(D_\ell(x)+\frac{4\|\sigma\|_{\text{Lip}}^2}{K^{(\ell)}} |z_{1}^{(\ell)}(x)|^2+2\|\sigma\|_{\text{Lip}}^2|z_{1}^{(\ell)}(x)|\Big)\\
+\sqrt{\frac{\pi}{2K^{(\ell)}}}\Big(\frac{|\sigma(0)|^2}{K^{(\ell)}}+\|\sigma\|_{\text{Lip}}^2+\frac{2\|\sigma\|_{\text{Lip}}^2}{K^{(\ell)}}\Big)
 +\frac{2\|\sigma\|_{\text{Lip}}}{K^{(\ell)}}\Big( |\sigma(0)|+\|\sigma\|_{\text{Lip}}|z_{1}^{(\ell)}(x)|\Big)\Bigg)^{2p}\Bigg]\\
+\frac{5^{2p-1}C_W^{4p}}{3^{2p}n_{\ell-1}^{2p+1}}\sum_{j=1}^{n_{\ell-1}}\mathbb{E}\Bigg[|W_{1,j}^{(\ell)}-{W'}_{1,j}^{(\ell)}|^{8p}\sigma^{8p}(z_{j}^{(\ell-1)}(x))\cdot\\
\cdot \Bigg( \frac{|z_{1}^{(\ell)}(x)|}{K^{(\ell)}}\Big(D_\ell(x)+6{\|\sigma\|_{\text{Lip}}^2}|z_{1}^{(\ell)}(x)|\Big)+\frac{3\|\sigma\|_{\text{Lip}}^2}{K^{(\ell)}}\Bigg)^{2p}\Bigg]\\
+\frac{5^{2p-1}C_W^{5p}}{2^{2p}n_{\ell-1}^{3p+1}}\|\sigma\|_{\text{Lip}}^{4p}\Big(\frac{2}{K^{(\ell)}}+1\Big)^{2p}\sum_{j=1}^{n_{\ell-1}}\mathbb{E}\Big[|W_{1,j}^{(\ell)}-{W'}_{1,j}^{(\ell)}|^{10p}|\sigma^{10p}(z_{j}^{(\ell-1)}(x))|\frac{|z_{1}^{(\ell)}(x)|^{2p}}{(K^{(\ell)})^{2p}}\Big]\\
+5^{2p-1}\mathbb{E}\Big[|A_{\ell}(x)|^{4p}\Big]^{1/2}\mathbb{E}\Big[\Big(\tilde{D}_\ell(x)\mathbb{E}\Big[|z_{1}^{(\ell)}(x)|^4|\mathcal{F}_{\ell-1}\Big]^{1/2}
+D_\ell(x)\Big)^{4p}\Big]^{1/2}\\
+\mathbb{E}\Big[(W_{1,1}^{(\ell)})^4\Big]^{p}\frac{5^{2p-1}{C_W^{2p}}}{2^{2p}{n_{\ell-1}^{p}}} \mathbb{E}\Big[\sigma^{8p}(z_{1}^{(\ell-1)}(x))\Big]^{1/2}\mathbb{E}\Big[\Big(\tilde{D}_\ell(x)\mathbb{E}\Big[|z_{1}^{(\ell)}(x)|^4|\mathcal{F}_{\ell-1}\Big]^{1/2}+D_\ell(x)\Big)^{4p}\Big]^{1/2}
\end{multline*}
\begin{multline*}
\le \frac{5^{2p-1}C_W^{3p}}{2^{2p}{n_{\ell-1}^{p}}}\mathbb{E}\Big[|W_{1,1}^{(\ell)}-{W'}_{1,1}^{(\ell)}|^{12p}\Big]^{1/2}\mathbb{E}\Big[|\sigma^{12p}(z_{1}^{(\ell-1)}(x))|\Big]^{1/2}\cdot\\
\cdot \mathbb{E}\Bigg[\Bigg(\frac{|z_{1}^{(\ell)}(x)|}{K^{(\ell)}}\Big(D_\ell(x)+\frac{4\|\sigma\|_{\text{Lip}}^2}{K^{(\ell)}} |z_{1}^{(\ell)}(x)|^2+2\|\sigma\|_{\text{Lip}}^2|z_{1}^{(\ell)}(x)|\Big)\\
+\sqrt{\frac{\pi}{2K^{(\ell)}}}\Big(\frac{|\sigma(0)|^2}{K^{(\ell)}}+\|\sigma\|_{\text{Lip}}^2+\frac{2\|\sigma\|_{\text{Lip}}^2}{K^{(\ell)}}\Big)
 +\frac{2\|\sigma\|_{\text{Lip}}}{K^{(\ell)}}\Big( |\sigma(0)|+\|\sigma\|_{\text{Lip}}|z_{1}^{(\ell)}(x)|\Big)\Bigg)^{4p}\Bigg]^{1/2}\\
+\frac{5^{2p-1}C_W^{4p}}{3^{2p}n_{\ell-1}^{2p}}\mathbb{E}\Big[|W_{1,1}^{(\ell)}-{W'}_{1,1}^{(\ell)}|^{16p}\Big]^{1/2} \mathbb{E}\Big[\sigma^{16p}(z_{1}^{(\ell-1)}(x))\Big]^{1/2}\cdot\\
\cdot\mathbb{E}\Bigg[\Bigg( \frac{|z_{1}^{(\ell)}(x)|}{K^{(\ell)}}\Big(D_\ell(x)+6{\|\sigma\|_{\text{Lip}}^2}|z_{1}^{(\ell)}(x)|\Big)+\frac{3\|\sigma\|_{\text{Lip}}^2}{K^{(\ell)}}\Bigg)^{4p}\Bigg]^{1/2}
+\frac{5^{2p-1}C_W^{5p}}{2^{2p}n_{\ell-1}^{3p}}\|\sigma\|_{\text{Lip}}^{4p}\Big(\frac{2}{K^{(\ell)}}+1\Big)^{2p}\cdot\\
\cdot\mathbb{E}\Big[|W_{1,1}^{(\ell)}-{W'}_{1,1}^{(\ell)}|^{20p}\Big]^{1/2}\mathbb{E}\Big[|\sigma^{20p}(z_{1}^{(\ell-1)}(x))|\Big]^{1/2}\frac{\mathbb{E}[|z_{1}^{(\ell)}(x)|^{4p}]^{1/2}}{(K^{(\ell)})^{2p}}\Big]\\
+5^{2p-1}\mathbb{E}\Big[|A_{\ell}(x)|^{4p}\Big]^{1/2}\mathbb{E}\Big[\Big(\tilde{D}_\ell(x)\mathbb{E}\Big[|z_{1}^{(\ell)}(x)|^4|\mathcal{F}_{\ell-1}\Big]^{1/2}
+D_\ell(x)\Big)^{4p}\Big]^{1/2}\\
+\mathbb{E}\Big[(W_{1,1}^{(\ell)})^4\Big]^{p}\frac{5^{2p-1}{C_W^{2p}}}{2^{2p}{n_{\ell-1}^{p}}} \mathbb{E}\Big[\sigma^{8p}(z_{1}^{(\ell-1)}(x))\Big]^{1/2}\mathbb{E}\Big[\Big(\tilde{D}_\ell(x)\mathbb{E}\Big[|z_{1}^{(\ell)}(x)|^4|\mathcal{F}_{\ell-1}\Big]^{1/2}+D_\ell(x)\Big)^{4p}\Big]^{1/2}
\end{multline*}

\begin{equation}
\le \frac{L_1^{(\ell)(p)}(x)}{n_{\ell-1}^p}+L_2^{(\ell)(p)}(x)\sqrt{Q_{4p}^{(\ell-1)}(x)},
\end{equation}
where
\begin{multline}
L_1^{(\ell)(p)}(x)= \frac{5^{2p-1}C_W^{3p}}{2^{2p}}\mathbb{E}\Big[|W_{1,1}^{(\ell)}-{W'}_{1,1}^{(\ell)}|^{12p}\Big]^{1/2}\mathbb{E}\Big[|\sigma^{12p}(z_{1}^{(\ell-1)}(x))|\Big]^{1/2}\cdot\\
\cdot \mathbb{E}\Bigg[\Bigg(\frac{|z_{1}^{(\ell)}(x)|}{K^{(\ell)}}\Big({D}_\ell(x)+\frac{4\|\sigma\|_{\text{Lip}}^2}{K^{(\ell)}} |z_{1}^{(\ell)}(x)|^2+2\|\sigma\|_{\text{Lip}}^2|z_{1}^{(\ell)}(x)|\Big)\\
+\sqrt{\frac{\pi}{2{K}^{(\ell)}}}\Big(\frac{|\sigma(0)|^2}{K^{(\ell)}}
+\|\sigma\|_{\text{Lip}}^2+\frac{2\|\sigma\|_{\text{Lip}}^2}{K^{(\ell)}}\Big)
 +\frac{2\|\sigma\|_{\text{Lip}}}{K^{(\ell)}}\Big( |\sigma(0)|+\|\sigma\|_{\text{Lip}}|z_{1}^{(\ell)}(x)|\Big)\Bigg)^{4p}\Bigg]^{1/2}\\
+\frac{5^{2p-1}C_W^{4p}}{2^{2p}}\mathbb{E}\Big[|W_{1,1}^{(\ell)}-{W'}_{1,1}^{(\ell)}|^{16p}\Big]^{1/2}\mathbb{E}\Big[\sigma^{16p}(z_{1}^{(\ell-1)}(x))\Big]^{1/2}\cdot\\
\cdot \mathbb{E}\Bigg[\Bigg( \frac{|z_{1}^{(\ell)}(x)|}{K^{(\ell)}}\Big(D+6{\|\sigma\|_{\text{Lip}}^2}|z_{1}^{(\ell)}(x)|\Big)+\frac{3\|\sigma\|_{\text{Lip}}^2}{K^{(\ell)}}\Bigg)^{4p}\Bigg]^{1/2}
+\frac{5^{2p-1}C_W^{5p}}{2^{2p}}\|\sigma\|_{\text{Lip}}^{4p}\Big(\frac{2}{K^{(\ell)}}+1\Big)^{2p}\cdot\\
\cdot \mathbb{E}\Big[|W_{1,1}^{(\ell)}-{W'}_{1,1}^{(\ell)}|^{20p}\Big]^{1/2}\mathbb{E}\Big[|\sigma^{20p}(z_{1}^{(\ell-1)}(x))|\Big]^{1/2}\frac{\mathbb{E}[|z_{1}^{(\ell)}(x)|^{4p}]^{1/2}}{(K^{(\ell)})^{2p}}\Big]\\
+\mathbb{E}\Big[(W_{1,1}^{(\ell)})^4\Big]^{p}
\frac{5^{2p-1}{C_W^{2p}}}{2^{2p}}  
\mathbb{E}\Big[\sigma^{8p}(z_{1}^{(\ell-1)}(x))\Big]^{1/2}
\mathbb{E}\Bigg[\Big(\tilde{D}_\ell(x)
\mathbb{E}\Big[|z_{1}^{(\ell)}(x)|^4
|\mathcal{F}_{\ell-1}\Big]^{1/2}+D_\ell(x)\Big)^{4p}\Bigg]^{1/2}
\end{multline}
and 
\begin{equation}
L_2^{(\ell)(p)}(x)=5^{2p-1}\mathbb{E}\Big[\Big(\tilde{D}_\ell(x)\mathbb{E}\Big[|z_{1}^{(\ell)}(x)|^4|\mathcal{F}_{\ell-1}\Big]^{1/2}
+D_\ell(x)\Big)^{4p}\Big]^{1/2},
\end{equation}
Recalling that, for any integer $p$,  
\[
Q_{2p}^{(\ell-1)}(x) = \mathbb{E}\bigl[|A_\ell(x)|^{2p}\bigr].
\]  
After some straightforward calculations, one obtains a bound on $L_1^{(\ell)(p)}$ that more clearly highlights the dependence on the main parameters:

\begin{multline*}
L_1^{(\ell)(p)}(x)\le 
{5^{2p-1}2^{10 p+2}} C_W^{2p}(C_W+1)^{3p}\Big(1+\mathbb{E}[\sigma^{20p}(z_1^{(\ell-1)})]^{1/2}\Big)\Big(1+\mathbb{E}[|z_{1}^{(\ell)}(x)|^{12p}]^{1/2}\Big)\cdot\\
\cdot \Big(1+\mathbb{E}[(W_{1,1}^{(\ell)})^{20p}]^{1/2}\Big)\Big(|\sigma(0)|^2+\|\sigma\|_{\text{Lip}}^2\Big)^{2p}\Bigg({12}+{2K^{(\ell)}}+\frac{20}{K^{(\ell)}}+\frac{15}{(K^{(\ell)})^2}
\Bigg)^{2p}
\end{multline*}

\section{Appendix B}\label{Appendix_B}
In this appendix, as before we denote by  
\[
\sigma(x) := \bigl(\sigma(x_1), \dots, \sigma(x_d)\bigr)
\]  
the component-wise evaluation of $\sigma$ on a vector $x = (x_1, \dots, x_d) \in \mathbb{R}^d$.

\subsection{Preliminary Lemmas}
Before turning to the proof of Lemma~\ref{somma_schifa}, we first establish several technical results concerning the properties of the solution to a particular Stein's equation.

\begin{lemma}\label{sol_st_mult}
Suppose that $x^{(i)}, x^{(j)} \in \mathbb{R}^{n_0}$ are such that $x^{(i)} \ne x^{(j)}$, and let $\sigma:\mathbb{R} \to \mathbb{R}$ be a Lipschitz continuous function. 

Denote by ${\operatorname{Hess}} f$ the Hessian matrix of a function $f$, and by 
\[
\langle M, N \rangle_{\mathrm{HS}} := \sum_{i,j=1}^d M_{i,j}N_{i,j}
\]
the Hilbert--Schmidt inner product between two matrices $M, N \in \mathbb{R}^{d \times d}$. 

Then, the solution $f^{(\ell)}_{x^{(i)},x^{(j)}}$ to the Stein's equation

\[
\langle K^{(\ell)},{\operatorname{Hess}} f^{(\ell)}_{x^{(i)},x^{(j)}}(y)\rangle_{\text{HS}}-\langle y,\nabla f^{(\ell)}_{x^{(i)},x^{(j)}}(y)\rangle=\sigma(y_i)\sigma(y_j)-\mathbb{E}\Big[\sigma(G_{1}^{(\ell)}(x^{(i)}))\sigma(G_1^{(\ell)}(x^{(j)}))\Big].
\]
is
\begin{equation}\label{def_sol_stein_mult}
f^{(\ell)}_{x^{(i)},x^{(j)}}(y)=\int_0^{\infty}\mathbb{E}[h(N)-h(e^{-t}y+\sqrt{1-e^{-2t}}N)]dt, \quad y\in\mathbb{R}^d,
\end{equation}
where $h(y) := \sigma(y_i)\sigma(y_j)$ and $N \sim \mathcal{N}_d(0, K^{(\ell)})$, with $G^{(\ell)}$ and $K^{(\ell)}$ as defined in Theorem~\ref{Hanin}.
\end{lemma}

\begin{proof}
The proof is an adaptation of Proposition~4.3.2 in \cite{NP12} to the setting where the function $h$ is the square of a Lipschitz function.
The function $f^{(\ell)}_{x^{(i)},x^{(j)}}$ is well defined since, for any $t > 0$,

\begin{multline*}
|h(N)-h(e^{-t}y+\sqrt{1-e^{-2t}}N)|\\
=\Big|\sigma(N_i)\sigma(N_j)-\sigma(e^{-t}y_i+\sqrt{1-e^{-2t}}N_i)\sigma(e^{-t}y_j+\sqrt{1-e^{-2t}}N_j)\Big|
\end{multline*}
\begin{multline*}
\le\Big| \sigma(N_i)\Big(\sigma(N_j)-\sigma(e^{-t}y_j+\sqrt{1-e^{-2t}}N_j)\Big)\Big|\\
+\Big|\sigma(e^{-t}y_j+\sqrt{1-e^{-2t}}N_j)\Big(\sigma(N_i)-\sigma(e^{-t}y_i+\sqrt{1-e^{-2t}}N_i)\Big|
\end{multline*}
\begin{multline*}
\le\|\sigma\|_{\text{Lip}} \Big(|\sigma(0)|+\|\sigma\|_{\text{Lip}}|N_i|\Big)\Big(e^{-t}|y_j|\\
+e^{-2t}|N_j|\Big)+\Big(|\sigma(0)|+\|\sigma\|_{\text{Lip}}(|y_j|+|N_j|)\Big)\Big(e^{-t}|y_i|+e^{-2t}|N_i|\Big),
\end{multline*}
{
and hence 
\[
\mathbb{E}\left[|h(N)-h(e^{-t}y+\sqrt{1-e^{-2t}}N)|\right]<\infty.
\]
}

Let us first assume that $K^{(\ell)}=I_{d\times d}$ is the identity matrix of dimension $d$. Then almost everywhere with respect to the Lebesgue measure
\begin{multline*}
\frac{\partial}{\partial y_k}\Big(h(N)-h(e^{-t}y+\sqrt{1-e^{-2t}}N)\Big)\\
=
\begin{cases}
0 &\text{if $k\ne i,j$}\\
-e^{-t}\sigma'(e^{-t}y_i+\sqrt{1-e^{-2t}}N_i)\sigma(e^{-t}y_j+\sqrt{1-e^{-2t}}N_j) &\text{if $k= i$}\\
-e^{-t}\sigma(e^{-t}y_i+\sqrt{1-e^{-2t}}N_i)\sigma'(e^{-t}y_j+\sqrt{1-e^{-2t}}N_j) &\text{if $k= j$}.
\end{cases}
\end{multline*}
As a consequence, it is possible to pass with the derivative under the sign of integral and to use integration by parts (Theorem \ref{int_gauss_uno}) to obtain that, denoting $w_1,\dots,w_d$ the coordinates in the argument of $h$, 
\[
\frac{\partial f^{(\ell)}_{x^{(i)},x^{(j)}}}{\partial y_k}(y)=-\int_0^{\infty}e^{-t}\mathbb{E}\Big[\frac{\partial h}{\partial w_k}(e^{-t}y+\sqrt{1-e^{-2t}}N)\Big]dt
\]
\[
=
\begin{cases}
-\int_0^{\infty}\frac{e^{-t}}{\sqrt{1-e^{-2t}}}\mathbb{E}\Big[N_k{h}(e^{-t}y+\sqrt{1-e^{-2t}}N)\Big]dt &\text{if $k=i,j$}\\
0 &\text{if $k\ne i,j$}.
\end{cases}
\]
Moreover, again almost everywhere with respect to the Lebesgue measure,
\begin{multline*}
\frac{\partial^2}{\partial y_i \partial y_j}\Big(h(N)-h(e^{-t}y+\sqrt{1-e^{-2t}}N)\Big)\\
=-e^{-2t}\sigma'(e^{-t}y_i+\sqrt{1-e^{-2t}}N_i)\sigma'(e^{-t}y_j+\sqrt{1-e^{-2t}}N_j).
\end{multline*}
Therefore it is possible again to pass with the derivative under the sign of integral to get that
\begin{equation}\label{der_mista}
\frac{\partial^2 f^{(\ell)}_{x^{(i)},x^{(j)}}}{\partial y_i \partial y_j}(y)=-\int_{0}^{\infty}\mathbb{E}\Big[e^{-2t}\sigma'(e^{-t}y_i+\sqrt{1-e^{-2t}}N_i)\sigma'(e^{-t}y_j+\sqrt{1-e^{-2t}}N_j)\Big]dt.
\end{equation}
As before, one can see that
\[
\frac{\partial^2 f^{(\ell)}_{x^{(i)},x^{(j)}}}{\partial y_i^2}(y)=-\int_0^{\infty}\frac{e^{-2t}}{\sqrt{1-e^{-2t}}}\mathbb{E}\Big[N_i\frac{\partial h}{\partial w_i}(e^{-t}y+\sqrt{1-e^{-2t}}N)\Big]dt
\]
and
\[
\frac{\partial^2 f^{(\ell)}_{x^{(i)},x^{(j)}}}{\partial y_j^2}(y)=-\int_0^{\infty}\frac{e^{-2t}}{\sqrt{1-e^{-2t}}}\mathbb{E}\Big[N_j\frac{\partial h}{\partial w_j}(e^{-t}y+\sqrt{1-e^{-2t}}N)\Big]dt,
\]
while all the others second derivatives are equal to zero.
Then, denoting with $\triangle$ the Laplacian,
\begin{multline*}
\triangle f^{(\ell)}_{x^{(i)},x^{(j)}}(y)-\langle y,\nabla f^{(\ell)}_{x^{(i)},x^{(j)}}(y)\rangle\\
=\sum_{r=i,j}\Bigg\{-\int_0^{\infty}\frac{e^{-2t}}{\sqrt{1-e^{-2t}}}\mathbb{E}\Big[N_r\frac{\partial h}{\partial w_r}(e^{-t}y\\
+\sqrt{1-e^{-2t}}N)\Big]dt+y_r\int_0^{\infty}e^{-t}\mathbb{E}\Big[\frac{\partial h}{\partial w_r}(e^{-t}y+\sqrt{1-e^{-2t}}N)\Big]dt\Bigg\}\\
=\sum_{r=i,j}\Bigg\{\int_0^{\infty}\mathbb{E}\Bigg[\Big(-\frac{e^{-2t}}{\sqrt{1-e^{-2t}}}N_r+y_r e^{-t}\Big)\frac{\partial h}{\partial w_r}(e^{-t}y+\sqrt{1-e^{-2t}}N)\Bigg]dt\Bigg\}
\end{multline*}
\[
=-\int_0^{\infty}\mathbb{E}\Big[\frac{\partial}{\partial t}h(e^{-t}y+\sqrt{1-e^{-2t}}N)\Big]dt=-\mathbb{E}\Big[\int_0^{\infty}\frac{\partial}{\partial t}h(e^{-t}y+\sqrt{1-e^{-2t}}N)dt\Big]
\]
\[
=-\mathbb{E}\Big[\lim_{t\to\infty}h(e^{-t}y+\sqrt{1-e^{-2t}}N)\Big]+h(y)=h(y)-\mathbb{E}[h(N)].
\]
If $K^{(\ell)}$ is a symmetric and non-negative matrix then the rest of the proof follows exactly as in Proposition 4.3.2 in \cite{NP12}.
\end{proof}
\begin{lemma}\label{norm_inf_hess}
If $x^{(i)},x^{(j)}\in\mathbb{R}^d$ are such that $x^{(i)}\ne x^{(j)}$, then for every $\ell=1,\dots,L+1$ and every $y\in\mathbb{R}^d$
\[
\Big|\frac{\partial^2f^{(\ell)}_{x^{(i)},x^{(j)}}}{\partial y_i\partial y_j}(y)\Big|\le\frac{ \|\sigma\|_{\text{Lip}}^2}{2}.
\]
\end{lemma}

\begin{proof}
The proof follows directly by (\ref{der_mista}), which is true also if $K^{(\ell)}$ is not the identity matrix.
\end{proof}

\begin{lemma}\label{lemma_hess}
For $x^{(i)}\ne x^{(j)}\in\mathbb{R}^d$ and for every $\ell=1,\dots,L+1$ one has that
\begin{multline}\label{hess_sol_stein_mult}
\operatorname{Hess} f^{(\ell)}_{x^{(i)},x^{(j)}}(y)\\
=-\int_0^{\infty}\frac{e^{-t}}{\sqrt{1-e^{-2t}}}(K^{(\ell)})^{-1}\mathbb{E}\Big[G_1^{(\ell)}(\mathcal{X})\Big(\nabla_y h(e^{-t}y+\sqrt{1-e^{-2t}}G_1^{(\ell)}(\mathcal{X}))\Big)^T\Big]dt,
\end{multline}
recalling that $h(y):=\sigma(y_i)\sigma(y_j)$ and denoting with $\nabla_y$ the gradient with respect to the variable $y$.
\end{lemma}
\begin{proof}
It is enough to observe that if $N\sim \mathcal{N}_d(0,I_{d\times d})$, then
\[
h(e^{-t}y+\sqrt{1-e^{-2t}}G_1^{(\ell)}(\mathcal{X}))\stackrel{law}{=}h(e^{-t}y+\sqrt{1-e^{-2t}}\sqrt{K^{(\ell)}}N)
\]
and 
\begin{multline*}
 \mathbb{E}\Big[\nabla_N h(e^{-t}y+\sqrt{1-e^{-2t}}\sqrt{K^{(\ell)}}N)\Big]\\
 =e^t\sqrt{1-e^{-2t}}\sqrt{K^{(\ell)}}\nabla_y \mathbb{E}\Big[h(e^{-t}y+\sqrt{1-e^{-2t}}G_1^{(\ell)}(\mathcal{X}))\Big],
\end{multline*}
but also, because of Theorem \ref{int_gauss_uno}, it holds that
\[
 \mathbb{E}\Big[\nabla_Nh(e^{-t}y+\sqrt{1-e^{-2t}}\sqrt{K^{(\ell)}}N)\Big]
=\mathbb{E}\Big[Nh(e^{-t}y+\sqrt{1-e^{-2t}}\sqrt{K^{(\ell)}}N)\Big].
\]
Therefore
\begin{multline*}
\nabla_y \mathbb{E}\Big[h(e^{-t}y+\sqrt{1-e^{-2t}}G_1^{(\ell)}(\mathcal{X}))\Big]\\
=\frac{e^{-t}}{\sqrt{1-e^{-2t}}}(K^{(\ell)})^{-1}\mathbb{E}\Big[G_1^{(\ell)}(\mathcal{X})h(e^{-t}y+\sqrt{1-e^{-2t}}G_1^{(\ell)}(\mathcal{X}))\Big].
\end{multline*}
Identity \eqref{hess_sol_stein_mult} immediately follows from \eqref{def_sol_stein_mult}.
\end{proof}
\begin{lemma}\label{mom_quart_der_sec}
\noindent
If $x^{(i)}\ne x^{(j)}\in\mathbb{R}^d$, then for every $\ell=1,\dots,L+1$ 
\begin{multline}\label{der_sec_sol_stein_div}
\mathbb{E}\Big[\Big(\frac{\partial^2f^{(\ell)}_{x^{(i)},x^{(j)}}}{\partial y_i^2}(z_1^{(\ell)}(\mathcal{X}))\Big)^4\Big]
\le
 8 {\|\sigma\|^4_{\text{Lip}}}\Big({\sum_{k=1}^d\big|((K^{(\ell)})^{-1})_{i,k}\big|^2}\Big)^{2}{\operatorname{tr}(K^{(\ell)})}^2\cdot\\
\cdot \Big(|\sigma(z_1^{(\ell)}(x_j))|+{\|\sigma\|_{\text{Lip}}}|z_1^{(\ell)}(x_j)|\Big)^4
+8\|\sigma\|^8_{\text{Lip}}\Big({\sum_{k=1}^d\big|((K^{(\ell)})^{-1})_{i,k}\big|^2}\Big)^{2}\Big(K^{(\ell)}_{j,j}\Big)^2.
\end{multline}

An analogous estimate holds for $\mathbb{E}\Big[\Big(\frac{\partial^2f^{(\ell)}_{x^{(i)},x^{(j)}}}{\partial y_j^2}(z_1^{(\ell)}(\mathcal{X}))\Big)^4\Big]$.
\end{lemma}
\begin{proof}
{
We only prove the upper bound in \eqref{der_sec_sol_stein_div}, as the second part of the statement follows by a similar argument.
}
From Lemma \ref{lemma_hess} we have that
\begin{multline*}
\Big|\frac{\partial^2f^{(\ell)}_{x^{(i)},x^{(j)}}}{\partial y_i^2}(y)\Big|
\le
\int_0^{\infty}\Big|\frac{e^{-2t}}{\sqrt{1-e^{-2t}}}\sum_{k=1}^{d}((K^{(\ell)})^{-1})_{i,k}\cdot\\
\cdot \mathbb{E}\Big[G_1^{(\ell)}(x^{(k)})\sigma'(e^{-t}y_i+\sqrt{1-e^{-2t}}G_1^{(\ell)}(x^{(i)}))\sigma(e^{-t}y_j+\sqrt{1-e^{-2t}}G_1^{(\ell)}(x^{(j)}))\Big]\Big|dt
\end{multline*}
\begin{multline*}
\le {\|\sigma\|_{\text{Lip}}}
\int_0^{\infty}\frac{e^{-2t}}{\sqrt{1-e^{-2t}}}dt\sum_{k=1}^{d}\Big|((K^{(\ell)})^{-1})_{i,k}\Big| \mathbb{E}\Big[\Big|G_1^{(\ell)}(x^{(k)})\Big|\Big]\Big(|\sigma(y_j)|+\|\sigma\|_{\text{Lip}}|y_j|\Big)\\
+{\|\sigma\|^2_{\text{Lip}}}\int_0^{\infty}\frac{e^{-2t}}{\sqrt{1-e^{-2t}}}dt\sum_{k=1}^{d}\Big|((K^{(\ell)})^{-1})_{i,k}\Big| \mathbb{E}\Big[\Big|G_1^{(\ell)}(x^{(k)})G_1^{(\ell)}(x^{(j)})\Big|\Big]
\end{multline*}
\begin{multline*}
\le {\|\sigma\|_{\text{Lip}}}\Big({\sum_{k=1}^d\big|((K^{(\ell)})^{-1})_{i,k}\big|^2}\Big)^{1/2}\mathbb{E}\Big[\|G_1^{(\ell)}(\mathcal{X})\|^2\Big]^{1/2}\Big(|\sigma(y_j)|+\|\sigma\|_{\text{Lip}}|y_j|\Big)\\
+{\|\sigma\|^2_{\text{Lip}}}\sum_{k=1}^{d}\Big|((K^{(\ell)})^{-1})_{i,k}\Big| \mathbb{E}\Big[\Big|G_1^{(\ell)}(x^{(k)})\Big|^2\Big]^{1/2}\mathbb{E}\Big[\Big|G_1^{(\ell)}(x^{(j)})\Big|^2\Big]^{1/2}
\end{multline*}
\begin{multline}\label{der_sec_simp}
\le  {\|\sigma\|_{\text{Lip}}}\Big({\sum_{k=1}^d\big|((K^{(\ell)})^{-1})_{i,k}\big|^2}\Big)^{1/2}\sqrt{\operatorname{tr}(K^{(\ell)})}\Big(|\sigma(y_j)|+\|\sigma\|_{\text{Lip}}|y_j|\Big)\\
+{\|\sigma\|^2_{\text{Lip}}}\Big({\sum_{k=1}^d\big|((K^{(\ell)})^{-1})_{i,k}\big|^2}\Big)^{1/2}\sqrt{K^{(\ell)}_{j,j}}
\end{multline}

Therefore
\begin{multline*}
\mathbb{E}\Big[\Big(\frac{\partial^2f^{(\ell)}_{x^{(i)},x^{(j)}}}{\partial y_i^2}(z_1^{(\ell)}(\mathcal{X}))\Big)^4\Big]\le 8 {\|\sigma\|^4_{\text{Lip}}}\Big({\sum_{k=1}^d\big|((K^{(\ell)})^{-1})_{i,k}\big|^2}\Big)^{2}{\operatorname{tr}(K^{(\ell)})}^2\cdot\\
\cdot \Big(|\sigma(z_1^{(\ell)}(x_j))|+{\|\sigma\|_{\text{Lip}}}|z_1^{(\ell)}(x_j)|\Big)^4
+8\|\sigma\|^8_{\text{Lip}}\Big({\sum_{k=1}^d\big|((K^{(\ell)})^{-1})_{i,k}\big|^2}\Big)^{2}\Big(K^{(\ell)}_{j,j}\Big)^2.
\end{multline*}

\end{proof}
\begin{lemma}\label{der_ter_lemma}
If $x^{(i)}\ne x^{(j)}\in\mathbb{R}^d$, then for every $\ell=1,\dots,L+1$ one has 
\begin{multline*}
\Big\|\operatorname{Hess}\Big(\frac{\partial f^{(\ell)}_{x^{(i)},x^{(j)}}}{\partial y_i}\Big)(y)\Big\|_{\text{HS}}\\
\le
\|\sigma\|^2_{\text{Lip}}\sqrt{\operatorname{tr}(K^{(\ell)})}\Bigg(\Big(\sum_{k=1}^d |((K^{(\ell)})^{-1})_{j,k}|^2\Big)^{1/2}+\Big(\sum_{k=1}^d |((K^{(\ell)})^{-1})_{i,k}|^2\Big)^{1/2}\Bigg)\\
+\|\sigma\|_{\text{Lip}}\|(K^{(\ell)})^{-1}\|_{\mathrm{op}}\|(K^{(\ell)})^{-1}\|_{\text{HS}}\,\sqrt{\operatorname{tr}K^{(\ell)}}\sqrt{K^{(\ell)}_{i,i}}\Big(\|\sigma\|_{\text{Lip}}|y_j|+|\sigma(y_j)|\Big)\\
+\sqrt{3}\|\sigma\|^2_{\text{Lip}}\|(K^{(\ell)})^{-1}\|_{\mathrm{op}}\|(K^{(\ell)})^{-1}\|_{\text{HS}}\sqrt{\operatorname{tr}K^{(\ell)}}\sqrt{K^{(\ell)}_{i,i}}\sqrt{K^{(\ell)}_{j,j}}\\
+\|\sigma\|_{\text{Lip}}\Big|((K^{(\ell)})^{-1})_{i,i}\Big|\Big(|\sigma(y_j)|+\|\sigma\|_{\text{Lip}}|y_j|+\|\sigma\|_{\text{Lip}}\sqrt{K^{(\ell)}_{j,j}}\Big)
\end{multline*}
and an analogous estimate holds for $\Big\|\operatorname{Hess}\Big(\frac{\partial f^{(\ell)}_{x^{(i)},x^{(j)}}}{\partial y_j}\Big)(y)\Big\|_{\text{HS}}$.
\end{lemma}
\begin{proof}
From Lemma \ref{lemma_hess} it follows that
\begin{multline*}
\frac{\partial^2 f^{(\ell)}_{x^{(i)},x^{(j)}}}{\partial y_i \partial y_j}(y)\\
=-\int_0^{\infty}\frac{e^{-t}}{\sqrt{1-e^{-2t}}}\sum_{k=1}^{d}((K^{(\ell)})^{-1})_{i,k}\mathbb{E}\Big[G_1^{(\ell)}(x^{(k)})\frac{\partial h}{\partial y_j}(e^{-t}y+\sqrt{1-e^{-2t}}G_1^{(\ell)}(\mathcal{X}))\Big]dt.
\end{multline*}
and, similarly,
\begin{multline*}
\frac{\partial^2 f^{(\ell)}_{x^{(i)},x^{(j)}}}{\partial y_i \partial y_j}(y)\\
=-\int_0^{\infty}\frac{e^{-t}}{\sqrt{1-e^{-2t}}}\sum_{k=1}^{d}((K^{(\ell)})^{-1})_{j,k}\mathbb{E}\Big[G_1^{(\ell)}(x^{k)})\frac{\partial h}{\partial y_i}(e^{-t}y+\sqrt{1-e^{-2t}}G_1^{(\ell)}(\mathcal{X}))\Big]dt.
\end{multline*}
As a consequence,
\begin{multline*}
\frac{\partial^3 f^{(\ell)}_{x^{(i)},x^{(j)}}}{\partial y_i^2 \partial y_j}(y)=-\int_0^{\infty}\frac{e^{-3t}}{\sqrt{1-e^{-2t}}}\sum_{k=1}^{d}((K^{(\ell)})^{-1})_{i,k}\cdot\\
\cdot \mathbb{E}\Big[G_1^{(\ell)}(x^{(k)})\sigma'(e^{-t}y_i+\sqrt{1-e^{-2t}}G_1^{(\ell)}(x^{(i)}))\sigma'(e^{-t}y_j+\sqrt{1-e^{-2t}}G_1^{(\ell)}(x^{(j)}))\Big]dt
\end{multline*}
and
\begin{multline*}
\frac{\partial^3 f^{(\ell)}_{x^{(i)},x^{(j)}}}{\partial y_i \partial y_j^2}(y)=-\int_0^{\infty}\frac{e^{-3t}}{\sqrt{1-e^{-2t}}}\sum_{k=1}^{d}((K^{(\ell)})^{-1})_{j,k}\cdot\\
\cdot \mathbb{E}\Big[G_1^{(\ell)}(x^{(k)})\sigma'(e^{-t}y_i+\sqrt{1-e^{-2t}}G_1^{(\ell)}(x^{(i)}))\sigma'(e^{-t}y_j+\sqrt{1-e^{-2t}}G_1^{(\ell)}(x^{(j)}))\Big]dt.
\end{multline*}
It follows that
\[
\Big|\frac{\partial^3 f^{(\ell)}_{x^{(i)},x^{(j)}}}{\partial y_i \partial y_j^2}(y)\Big|\le 
\|\sigma\|^2_{\text{Lip}}\Big(\sum_{k=1}^d |((K^{(\ell)})^{-1})_{j,k}|^2\Big)^{1/2}\sqrt{\operatorname{tr}(K^{(\ell)})}
\]
and that 
\[
\Big|\frac{\partial^3 f^{(\ell)}_{x^{(i)},x^{(j)}}}{\partial y_i^2 \partial y_j}(y)\Big|\le \|\sigma\|^2_{\text{Lip}}\Big(\sum_{k=1}^d |((K^{(\ell)})^{-1})_{i,k}|^2\Big)^{1/2}\sqrt{\operatorname{tr}(K^{(\ell)})}.
\]
{
Now let us observe that taking $N\sim\mathcal{N}_d(0,I_{d\times d})$, one can write $G_1^{(\ell)}(\mathcal{X})\overset{\text{law}}{=}\sqrt{K^{(\ell)}}N$ and obtain that
}
\begin{multline*}
\mathbb{E}\Big[\nabla_N\Big(\sqrt{K^{(\ell)}}Nh(e^{-t}y+\sqrt{1-e^{-2t}}\sqrt{K^{(\ell)}}N)\Big)\Big]\\
=\sqrt{K^{(\ell)}}\mathbb{E}\Big[h(e^{-t}y+\sqrt{1-e^{-2t}}G_1^{(\ell)}(\mathcal{X}))\Big]\\
+e^t\sqrt{1-e^{-2t}}\sqrt{K^{(\ell)}}\mathbb{E}\Big[G_1^{(\ell)}(\mathcal{X})\Big(\nabla_y h(e^{-t}y+\sqrt{1-e^{-2t}}G_1^{(\ell)}(\mathcal{X}))\Big)^T\Big],
\end{multline*}
hence, thanks to Theorem \ref{int_gauss_uno},
\begin{multline*}
\mathbb{E}\Big[G_1^{(\ell)}(\mathcal{X})\Big(\nabla_y h(e^{-t}y+\sqrt{1-e^{-2t}}G_1^{(\ell)}(\mathcal{X})\Big)^T\Big]\\
=\frac{e^{-t}}{\sqrt{1-e^{-2t}}}(\sqrt{K^{(\ell)}})^{-1}\Bigg(\mathbb{E}\Big[\nabla_N\Big(\sqrt{K^{(\ell)}}Nh(e^{-t}y+\sqrt{1-e^{-2t}}\sqrt{K^{(\ell)}}N)\Big)\Big]\\
-\sqrt{K^{(\ell)}}\mathbb{E}\Big[h(e^{-t}y+\sqrt{1-e^{-2t}}G_1^{(\ell)}(\mathcal{X}))\Big]\Bigg)
\end{multline*}
\begin{multline*}
=\frac{e^{-t}}{\sqrt{1-e^{-2t}}}(\sqrt{K^{(\ell)}})^{-1}\Bigg(\mathbb{E}\Big[N\Big(\sqrt{K^{(\ell)}}N\Big)^T h(e^{-t}y+\sqrt{1-e^{-2t}}\sqrt{K^{(\ell)}}N)\Big]\\
-\sqrt{K^{(\ell)}}\mathbb{E}\Big[h(e^{-t}y+\sqrt{1-e^{-2t}}G_1^{(\ell)}(\mathcal{X}))\Big]\Bigg)
\end{multline*}
\begin{multline*}
=\frac{e^{-t}}{\sqrt{1-e^{-2t}}}\Bigg((K^{(\ell)})^{-1}\mathbb{E}\Big[G_1^{(\ell)}(\mathcal{X})\Big(G_1^{(\ell)}(\mathcal{X}\Big)^T h(e^{-t}y+\sqrt{1-e^{-2t}}\sqrt{K^{(\ell)}}N)\Big]\\
-I_{d\times d}\mathbb{E}\Big[h(e^{-t}y+\sqrt{1-e^{-2t}}G_1^{(\ell)}(\mathcal{X}))\Big]\Bigg)
\end{multline*}
At this point, one can apply Lemma \ref{lemma_hess} to infer that
\begin{multline*}
\operatorname{Hess} f^{(\ell)}_{x^{(i)},x^{(j)}}(y)\\
=-\int_0^{\infty}\frac{e^{-t}}{\sqrt{1-e^{-2t}}}(K^{(\ell)})^{-1}\mathbb{E}\Big[G_1^{(\ell)}(\mathcal{X})\Big(\nabla_y h(e^{-t}y+\sqrt{1-e^{-2t}}G_1^{(\ell)}(\mathcal{X}))\Big)^T\Big]dt
\end{multline*}
\begin{multline*}
=-\int_0^{\infty}\frac{e^{-2t}}{{1-e^{-2t}}}(K^{(\ell)})^{-2}\mathbb{E}\Big[G_1^{(\ell)}(\mathcal{X})\Big(G_1^{(\ell)}(\mathcal{X})\Big)^T h(e^{-t}y+\sqrt{1-e^{-2t}}G_1^{(\ell)}(\mathcal{X}))\Big]dt\\
+\int_0^{\infty}\frac{e^{-2t}}{{1-e^{-2t}}}(K^{(\ell)})^{-1}\mathbb{E}\Big[h(e^{-t}y+\sqrt{1-e^{-2t}}G_1^{(\ell)}(\mathcal{X}))\Big]dt.
\end{multline*}
Then,
\begin{multline*}
\frac{\partial^2f^{(\ell)}_{x^{(i)},x^{(j)}}}{\partial y_i^2}(y)\\
=-\int_0^{\infty}\frac{e^{-2t}}{{1-e^{-2t}}}\sum_{k=1}^{d}((K^{(\ell)})^{-2})_{i,k}\mathbb{E}\Big[G_1^{(\ell)}(x_k)G_1^{(\ell)}(x_i) h(e^{-t}y+\sqrt{1-e^{-2t}}G_1^{(\ell)}(\mathcal{X}))\Big]dt\\
+\int_0^{\infty}\frac{e^{-2t}}{{1-e^{-2t}}}((K^{(\ell)})^{-1})_{i,i}\mathbb{E}\Big[h(e^{-t}y+\sqrt{1-e^{-2t}}G_1^{(\ell)}(\mathcal{X}))\Big]dt
\end{multline*}
and so
\begin{multline*}
\frac{\partial^3f^{(\ell)}_{x^{(i)},x^{(j)}}}{\partial y_i^3}(y)=-\int_0^{\infty}\frac{e^{-3t}}{{1-e^{-2t}}}\sum_{k=1}^{d}((K^{(\ell)})^{-2})_{i,k}\cdot\\
\cdot \mathbb{E}\Big[G_1^{(\ell)}(x_k)G_1^{(\ell)}(x_i) \sigma'(e^{-t}y_i+\sqrt{1-e^{-2t}}G_1^{(\ell)}(x_i))\sigma(e^{-t}y_j+\sqrt{1-e^{-2t}}G_1^{(\ell)}(x_j))\Big]dt\\
+\int_0^{\infty}\frac{e^{-3t}}{{1-e^{-2t}}}((K^{(\ell)})^{-1})_{i,i}\mathbb{E}\Big[\sigma'(e^{-t}y_i+\sqrt{1-e^{-2t}}G_1^{(\ell)}(x_i))\cdot\\
\cdot\sigma(e^{-t}y_j+\sqrt{1-e^{-2t}}G_1^{(\ell)}(x_j))\Big]dt.
\end{multline*}
Therefore,
\begin{multline*}
\Big|\frac{\partial^3f^{(\ell)}_{x^{(i)},x^{(j)}}}{\partial y_i^3}(y)\Big|\le\|\sigma\|_{\text{Lip}}\int_0^{\infty}\frac{e^{-3t}}{{1-e^{-2t}}}\sum_{k=1}^{d}\Big|((K^{(\ell)})^{-2})_{i,k}\Big|\cdot\\
\cdot \mathbb{E}\Big[\Big|G_1^{(\ell)}(x_k)G_1^{(\ell)}(x_i)\Big| \Big(\Big|\sigma(e^{-t}y_j+\sqrt{1-e^{-2t}}G_1^{(\ell)}(x_j))-\sigma(y_j)\Big|+|\sigma(y_j)|\Big)\Big]dt\\
+\|\sigma\|_{\text{Lip}}\int_0^{\infty}\frac{e^{-3t}}{{1-e^{-2t}}}\Big|((K^{(\ell)})^{-1})_{i,i}\Big|\cdot\\
\cdot \mathbb{E}\Big[\Big|\sigma(e^{-t}y_j+\sqrt{1-e^{-2t}}G_1^{(\ell)}(x_j))-\sigma(y_j)\Big|+|\sigma(y_j)|\Big]dt.
\end{multline*}
\begin{multline*}
\le\|\sigma\|_{\text{Lip}}\int_0^{\infty}\frac{e^{-3t}}{{1-e^{-2t}}}dt\sum_{k=1}^{d}\Big|((K^{(\ell)})^{-2})_{i,k}\Big|\mathbb{E}\Big[\Big|G_1^{(\ell)}(x_k)G_1^{(\ell)}(x_i)\Big|\Big]\Big(\|\sigma\|_{\text{Lip}}|y_j|+|\sigma(y_j)|\Big)\\
+\|\sigma\|^2_{\text{Lip}}\int_0^{\infty}\frac{e^{-3t}}{{1-e^{-2t}}}dt\sum_{k=1}^{d}\Big|((K^{(\ell)})^{-2})_{i,k}\Big|\mathbb{E}\Big[\Big|G_1^{(\ell)}(x_k)G_1^{(\ell)}(x_i)\Big|\Big|G_1^{(\ell)}(x_j)\Big|\Big]\\
+\|\sigma\|_{\text{Lip}}\int_0^{\infty}\frac{e^{-3t}}{{1-e^{-2t}}}dt\Big|((K^{(\ell)})^{-1})_{i,i}\Big|\Big(|\sigma(y_j)|+\|\sigma\|_{\text{Lip}}|y_j|+\|\sigma\|_{\text{Lip}}\mathbb{E}[|G_1^{(\ell)}(x_j)|]\Big)
\end{multline*}
\begin{multline*}
\le\|\sigma\|_{\text{Lip}}\int_0^{\infty}\frac{e^{-3t}}{{1-e^{-2t}}}dt\|(K^{(\ell)})^{-2}\|_{\text{HS}}\sqrt{\operatorname{tr}K^{(\ell)}}\sqrt{K^{(\ell)}_{i,i}}\Big(\|\sigma\|_{\text{Lip}}|y_j|+|\sigma(y_j)|\Big)\\
+\|\sigma\|^2_{\text{Lip}}\int_0^{\infty}\frac{e^{-3t}}{{1-e^{-2t}}}dt\|(K^{(\ell)})^{-2}\|_{\text{HS}}\sqrt{\operatorname{tr} K^{(\ell)}}\mathbb{E}[|G_1^{(\ell)}(x_i)|^2|G_1^{(\ell)}(x_j)|^2]^{1/2}\\
+\|\sigma\|_{\text{Lip}}\int_0^{\infty}\frac{e^{-3t}}{{1-e^{-2t}}}dt\Big|((K^{(\ell)})^{-1})_{i,i}\Big|\Big(|\sigma(y_j)|+\|\sigma\|_{\text{Lip}}|y_j|+\|\sigma\|_{\text{Lip}}\sqrt{K^{(\ell)}_{j,j}}\Big)
\end{multline*}

\begin{multline}\label{bound_terd}
\le \|\sigma\|_{\text{Lip}}\|(K^{(\ell)})^{-1}\|_{\mathrm{op}}\|(K^{(\ell)})^{-1}\|_{\text{HS}}\,\sqrt{\operatorname{tr}K^{(\ell)}}\sqrt{K^{(\ell)}_{i,i}}\Big(\|\sigma\|_{\text{Lip}}|y_j|+|\sigma(y_j)|\Big)\\
+\sqrt{3}\|\sigma\|^2_{\text{Lip}}\|(K^{(\ell)})^{-1}\|_{\mathrm{op}}\|(K^{(\ell)})^{-1}\|_{\text{HS}}\sqrt{\operatorname{tr}K^{(\ell)}}\sqrt{K^{(\ell)}_{i,i}}\sqrt{K^{(\ell)}_{j,j}}\\
+\|\sigma\|_{\text{Lip}}\Big|((K^{(\ell)})^{-1})_{i,i}\Big|\Big(|\sigma(y_j)|+\|\sigma\|_{\text{Lip}}|y_j|+\|\sigma\|_{\text{Lip}}\sqrt{K^{(\ell)}_{j,j}}\Big)
\end{multline}

An analogous bound as (\ref{bound_terd}) holds also for $\Big|\frac{\partial^3 f^{(\ell)}_{x^{(i)},x^{(j)}}}{\partial y_j^3}(y)\Big|$. 

We can now prove the Lemma: 
\[
\Big\|\operatorname{Hess}\Big(\frac{\partial f^{(\ell)}_{x^{(i)},x^{(j)}}}{\partial y_i}\Big)(y)\Big\|_{\text{HS}}\le \Bigg(\Big|\frac{\partial^3 f^{(\ell)}_{x^{(i)},x^{(j)}}}{\partial y_i^3}(y)\Big|+\Big|\frac{\partial^3 f^{(\ell)}_{x^{(i)},x^{(j)}}}{\partial y_i\partial y_j^2}(y)\Big|
+\Big|\frac{\partial^3 f^{(\ell)}_{x^{(i)},x^{(j)}}}{\partial y_i^2\partial y_j}\Big|\Bigg)
\]
\begin{multline*}
\le 
\|\sigma\|^2_{\text{Lip}}\sqrt{\operatorname{tr}(K^{(\ell)})}\Bigg(\Big(\sum_{k=1}^d |((K^{(\ell)})^{-1})_{j,k}|^2\Big)^{1/2}+\Big(\sum_{k=1}^d |((K^{(\ell)})^{-1})_{i,k}|^2\Big)^{1/2}\Bigg)\\
+\|\sigma\|_{\text{Lip}}\|(K^{(\ell)})^{-1}\|_{\mathrm{op}}\|(K^{(\ell)})^{-1}\|_{\text{HS}}\,\sqrt{\operatorname{tr}K^{(\ell)}}\sqrt{K^{(\ell)}_{i,i}}\Big(\|\sigma\|_{\text{Lip}}|y_j|+|\sigma(y_j)|\Big)\\
+\sqrt{3}\|\sigma\|^2_{\text{Lip}}\|(K^{(\ell)})^{-1}\|_{\mathrm{op}}\|(K^{(\ell)})^{-1}\|_{\text{HS}}\sqrt{\operatorname{tr}K^{(\ell)}}\sqrt{K^{(\ell)}_{i,i}}\sqrt{K^{(\ell)}_{j,j}}\\
+\|\sigma\|_{\text{Lip}}\Big|((K^{(\ell)})^{-1})_{i,i}\Big|\Big(|\sigma(y_j)|+\|\sigma\|_{\text{Lip}}|y_j|+\|\sigma\|_{\text{Lip}}\sqrt{K^{(\ell)}_{j,j}}\Big),
\end{multline*}
yielding the desired estimate.

\end{proof}
\subsection{Proof of Lemma \ref{somma_schifa}}
Let us consider the solution to the Stein's equation (with $x^{(i)}\ne x^{(j)}\in\mathbb{R}^d$)
\[
\langle K^{(\ell)},\operatorname{Hess} f^{(\ell)}_{x^{(i)},x^{(j)}}(y)\rangle_{\text{HS}}-\langle y,\nabla f^{(\ell)}_{x^{(i)},x^{(j)}}(y)\rangle=\sigma(y_i)\sigma(y_j)-\mathbb{E}\Big[\sigma(G_{1}^{(\ell)}(x^{(i)}))\sigma(G_1^{(\ell)}(x^{(j)}))\Big].
\]
Then
\begin{multline*}
B_{\ell}(x^{(i)},x^{(j)}):=\Big|\mathbb{E}\Big[\sigma(z_{1}^{(\ell)}(x^{(i)}))\sigma(z_1^{(\ell)}(x^{(j)}))\big|\mathcal{F}_{\ell-1}\Big]-\mathbb{E}\Big[\sigma(G_{1}^{(\ell)}(x^{(i)}))\sigma(G_1^{(\ell)}(x^{(j)}))\Big]\Big|\\
=\Big|\mathbb{E}\Big[\langle K^{(\ell)},\operatorname{Hess} f^{(\ell)}_{x^{(i)},x^{(j)}}(z_1^{(\ell)}(\mathcal{X}))\rangle_{\text{HS}}\big|\mathcal{F}_{\ell-1}\Big]-\mathbb{E}\Big[\langle z_1^{(\ell)}(\mathcal{X}),\nabla f^{(\ell)}_{x^{(i)},x^{(j)}}( z_1^{(\ell)}(\mathcal{X}))\rangle\big|\mathcal{F}_{\ell-1}\Big]\Big|.
\end{multline*}
Introduce ${W'}^{(\ell)}$ as a copy of $W^{(\ell)}$ independent of $\mathcal{F}_{\ell}$ and, for every $s=1,\dots, n_\ell$, define
\begin{equation}\label{z_1_l_s}
    z_1^{(\ell)(s)}(\mathcal{X})=b_1+\frac{\sqrt{C_W}}{{\sqrt{n_{\ell-1}}}}\sum_{k=1,k\ne s}^d W_{1,k}^{(\ell)}\sigma(z_k^{(\ell-1)}(\mathcal{X}))+\frac{\sqrt{C_W}}{\sqrt{n_{\ell-1}}} W_{1,s}^{'(\ell)}\sigma(z_s^{(\ell-1)}(\mathcal{X}))
\end{equation}
We can now use Lemma \ref{cov_ch} and the Gaussian integration by parts (Lemma \ref{Gauss_integ_multdim}) to obtain that
\begin{multline*}
\mathbb{E}\Big[\langle z_1^{(\ell)}(\mathcal{X}),\nabla f^{(\ell)}_{x^{(i)},x^{(j)}} (z_1^{(\ell)}(\mathcal{X}))\rangle\big|\mathcal{F}_{\ell-1}\Big]=\sum_{k=1}^{d}\mathbb{E}\Big[ z_1^{(\ell)}(x^{(k)})\frac{\partial f^{(\ell)}_{x^{(i)},x^{(j)}}}{\partial y_k}( z_1^{(\ell)}(\mathcal{X}))\big|\mathcal{F}_{\ell-1}\Big]\\
=\sum_{k=1}^{d}\mathbb{E}\Big[b_1^{(\ell)}\frac{\partial f^{(\ell)}_{x^{(i)},x^{(j)}}}{\partial y_k}( z_1^{(\ell)}(\mathcal{X}))\Big|\mathcal{F}_{\ell-1}\Big]
+ \sum_{k=1}^{d}\sum_{s=1}^{n_{\ell-1}}\frac{\sqrt{C_W}}{2\sqrt{n_{\ell-1}}}\mathbb{E}\Big[ \big(W_{1,s}^{(\ell)}-W^{'(\ell)}_{1,s}\big)\sigma(z_s^{(\ell-1)}(x^{(k)}))\cdot\\
\cdot\Big(\frac{\partial f^{(\ell)}_{x^{(i)},x^{(j)}}}{\partial y_k}( z_1^{(\ell)}(\mathcal{X}))-\frac{\partial f^{(\ell)}_{x^{(i)},x^{(j)}}}{\partial y_k}( z_1^{(\ell)(s)}(\mathcal{X}))\Big)\big|\mathcal{F}_{\ell-1}\Big]\\
=\sum_{k=1}^d\sum_{u=1}^d \mathbb{E}\Big[C_b\frac{\partial^2 f^{(\ell)}_{x^{(i)},x^{(j)}}}{\partial y_k\partial y_u}( z_1^{(\ell)}(\mathcal{X}))\Big|\mathcal{F}_{\ell-1}\Big]+\sum_{k=1}^{d}\sum_{s=1}^{n_{\ell-1}}\frac{\sqrt{C_W}}{2\sqrt{n_{\ell-1}}}\mathbb{E}\Big[ \big(W_{1,s}^{(\ell)}-W^{'(\ell)}_{1,s}\big)\cdot\\
\cdot\sigma(z_s^{(\ell-1)}(x^{(k)}))\Big\langle\nabla\frac{\partial f^{(\ell)}_{x^{(i)},x^{(j)}}}{\partial y_k}( z_1^{(\ell)}(\mathcal{X})),\frac{\sqrt{C_W}}{\sqrt{n_{\ell-1}}}\big(W_{1,s}^{(\ell)}-W^{'(\ell)}_{1,s}\big)\sigma(z_s^{(\ell-1)}(\mathcal{X}))\Big\rangle\big|\mathcal{F}_{\ell-1}\Big]\\
\end{multline*}
\begin{multline*}
=\sum_{k=1}^d\sum_{u=1}^d \mathbb{E}\Big[C_b\frac{\partial^2 f^{(\ell)}_{x^{(i)},x^{(j)}}}{\partial y_k\partial y_u}( z_1^{(\ell)}(\mathcal{X}))\Big|\mathcal{F}_{\ell-1}\Big]\\
+\sum_{k=1}^{d}\sum_{u=1}^{d}\sum_{s=1}^{n_{\ell-1}}\frac{{C_W}}{2{n_{\ell-1}}}\mathbb{E}\Big[ \big(W_{1,s}^{(\ell)}-W^{'(\ell)}_{1,s}\big)^2\sigma(z_s^{(\ell-1)}(x^{(k)}))\sigma(z_s^{(\ell-1)}(x^{(u)}))\cdot\\
\cdot \frac{\partial^2 f^{(\ell)}_{x^{(i)},x^{(j)}}}{\partial y_u\partial y_k}( z_1^{(\ell)}(\mathcal{X}))\big|\mathcal{F}_{\ell-1}\Big]
+\sum_{k=1}^{d}\sum_{s=1}^{n_{\ell-1}}\frac{\sqrt{C_W}}{2\sqrt{n_{\ell-1}}}\mathbb{E}\Big[ \big(W_{1,s}^{(\ell)}-W^{'(\ell)}_{1,s}\big)\sigma(z_s^{(\ell-1)}(x^{(k)}))\cdot\\
\cdot \Big(\frac{\partial f^{(\ell)}_{x^{(i)},x^{(j)}}}{\partial y_k}( z_1^{(\ell)}(\mathcal{X}))-\frac{\partial f^{(\ell)}_{x^{(i)},x^{(j)}}}{\partial y_k}( z_1^{(\ell)(s)}(\mathcal{X}))\\
-\Big\langle\nabla\frac{\partial f^{(\ell)}_{x^{(i)},x^{(j)}}}{\partial y_k}( z_1^{(\ell)}(\mathcal{X})),\frac{\sqrt{C_W}}{\sqrt{n_{\ell-1}}}\big(W_{1,s}^{(\ell)}
-W^{'(\ell)}_{1,s}\big)\sigma(z_s^{(\ell-1)}(\mathcal{X}))\Big\rangle\Big)\big|\mathcal{F}_{\ell-1}\Big]\\
\end{multline*}
Therefore
\begin{multline*}
B_{\ell}(x^{(i)},x^{(j)})\le\Big|\sum_{k=1}^{d}\sum_{u=1}^{d}\mathbb{E}\Big[\Big(K^{(\ell)}_{u,k}-C_b-\sum_{s=1}^{n_{\ell-1}}\frac{{C_W}}{2{n_{\ell-1}}} \big(W_{1,s}^{(\ell)}-W^{'(\ell)}_{1,s}\big)^2\sigma(z_s^{(\ell-1)}(x^{(k)}))\cdot\\
\cdot\sigma(z_s^{(\ell-1)}(x^{(u)}))\Big)\frac{\partial^2 f^{(\ell)}_{x^{(i)},x^{(j)}}}{\partial y_u\partial y_k}( z_1^{(\ell)}(\mathcal{X}))\big|\mathcal{F}_{\ell-1}\Big]\Big|
+\Big|\sum_{k=1}^{d}\sum_{s=1}^{n_{\ell-1}}\frac{\sqrt{C_W}}{2\sqrt{n_{\ell-1}}}\mathbb{E}\Big[ \big(W_{1,s}^{'(\ell)}-W^{(\ell)}_{1,s}\big)\cdot\\
\cdot\sigma(z_s^{(\ell-1)}(x^{(k)}))\Big(\frac{\partial f^{(\ell)}_{x^{(i)},x^{(j)}}}{\partial y_k}( z_1^{(\ell)}(\mathcal{X}))-\frac{\partial f^{(\ell)}_{x^{(i)},x^{(j)}}}{\partial y_k}( z_1^{(\ell)(s)}(\mathcal{X}))\\
+\Big\langle\nabla\frac{\partial f^{(\ell)}_{x^{(i)},x^{(j)}}}{\partial y_k}( z_1^{(\ell)}(\mathcal{X})),\frac{\sqrt{C_W}}{\sqrt{n_{\ell-1}}}\big(W_{1,s}^{'(\ell)}-W^{(\ell)}_{1,s}\big)\sigma(z_s^{(\ell-1)}(\mathcal{X}))\Big\rangle\Big)\big|\mathcal{F}_{\ell-1}\Big]\Big|\\
\end{multline*}
\begin{multline*}
\le \Big|\sum_{k=1}^{d}\sum_{u=1}^{d}\mathbb{E}\Big[\Big({K^{(\ell)}_{u,k}-C_b}-\sum_{s=1}^{n_{\ell-1}}\frac{{C_W}}{{n_{\ell-1}}}\sigma(z_s^{(\ell-1)}(x^{(k)}))\sigma(z_s^{(\ell-1)}(x^{(u)}))\Big)\cdot\\
\cdot\frac{\partial^2 f^{(\ell)}_{x^{(i)},x^{(j)}}}{\partial y_u\partial y_k}( z_1^{(\ell)}(\mathcal{X}))\big|\mathcal{F}_{\ell-1}\Big]\Big|
+\Big|\sum_{k=1}^{d}\sum_{u=1}^{d}\mathbb{E}\Big[\sum_{s=1}^{n_{\ell-1}}\frac{{C_W}}{2{n_{\ell-1}}}\Big( \big(W_{1,s}^{(\ell)}-W^{'(\ell)}_{1,s}\big)^2-2\Big) \sigma(z_s^{(\ell-1)}(x^{(k)}))\cdot\\
\cdot\sigma(z_s^{(\ell-1)}(x^{(u)}))\frac{\partial^2 f^{(\ell)}_{x^{(i)},x^{(j)}}}{\partial y_u\partial y_k}( z_1^{(\ell)}(\mathcal{X}))\big|\mathcal{F}_{\ell-1}\Big]\Big|
+\sum_{k=1}^{d}\sum_{s=1}^{n_{\ell-1}}\frac{\sqrt{C_W}}{4\sqrt{n_{\ell-1}}}\mathbb{E}\Big[ \big|W_{1,s}^{'(\ell)}-W^{(\ell)}_{1,s}\big|\cdot\\
\cdot|\sigma(z_s^{(\ell-1)}(x^{(k)}))|\Big\|\operatorname{Hess}\Big(\frac{\partial f^{(\ell)}_{x^{(i)},x^{(j)}}}{\partial y_k}\Big)(\eta^{(\ell,s)})\Big\|_{\mathrm{op}}\Big\|\frac{\sqrt{C_W}}{\sqrt{n_{\ell-1}}}\big(W_{1,s}^{'(\ell)}-W^{(\ell)}_{1,s}\big)\sigma(z_s^{(\ell-1)}(\mathcal{X}))\Big\|^2\big|\mathcal{F}_{\ell-1}\Big],
\end{multline*}
where $\eta^{(\ell,s)}:=t_0z_1^{(\ell)}(\mathcal{X})+(1-t_0)z_1^{(\ell)(s)}(\mathcal{X})$, for a certain $t_0\in [0,1]$. 

From Lemma \ref{sol_st_mult} it results that $\frac{\partial^2 f^{(\ell)}_{x^{(i)},x^{(j)}}}{\partial y_u\partial y_k}(y)=0$ if $k$ or $u$ are different from $i$ and $j$, therefore

\begin{multline*}
B_{\ell}(x^{(i)},x^{(j)})\\
\le 2 \Big|\mathbb{E}\Big[\Big({K^{(\ell)}_{i,j}-C_b}-\sum_{s=1}^{n_{\ell-1}}\frac{{C_W}}{{n_{\ell-1}}} \sigma(z_s^{(\ell-1)}(x^{(i)}))\sigma(z_s^{(\ell-1)}(x^{(j)}))\Big)\frac{\partial^2 f^{(\ell)}_{x^{(i)},x^{(j)}}}{\partial y_i\partial y_j}( z_1^{(\ell)}(\mathcal{X}))\big|\mathcal{F}_{\ell-1}\Big]\Big|\\
+\Big|\sum_{r=i,j}\mathbb{E}\Big[\Big({K^{(\ell)}_{r,r}-C_b}-\sum_{s=1}^{n_{\ell-1}}\frac{{C_W}}{{n_{\ell-1}}} \sigma^2(z_s^{(\ell-1)}(x^{(r)}))\Big)\frac{\partial^2 f^{(\ell)}_{x^{(i)},x^{(j)}}}{\partial y_r^2}( z_1^{(\ell)}(\mathcal{X}))\big|\mathcal{F}_{\ell-1}\Big]\Big|\\
+\Big|\sum_{r=i,j}\sum_{u=i,j}\mathbb{E}\Big[\sum_{s=1}^{n_{\ell-1}}\frac{{C_W}}{2{n_{\ell-1}}} \Big( \big(W_{1,s}^{(\ell)}-W^{'(\ell)}_{1,s}\big)^2-2\Big)\sigma(z_s^{(\ell-1)}(x^{(r)}))\sigma(z_s^{(\ell-1)}(x^{(u)}))\cdot\\
\cdot \frac{\partial^2 f^{(\ell)}_{x^{(i)},x^{(j)}}}{\partial y_r\partial y_u}( z_1^{(\ell)}(\mathcal{X}))\big|\mathcal{F}_{\ell-1}\Big]\Big|
+\sum_{r=i,j}\sum_{s=1}^{n_{\ell-1}}\frac{{C_W^{3/2}}}{4{n_{\ell-1}^{3/2}}}\mathbb{E}\Big[ \big|W_{1,s}^{'(\ell)}-W^{(\ell)}_{1,s}\big|^3|\sigma(z_s^{(\ell-1)}(x^{(r)}))|\cdot\\
\cdot \Big\|\operatorname{Hess}\Big(\frac{\partial f^{(\ell)}_{x^{(i)},x^{(j)}}}{\partial y_r}\Big)(\eta^{(\ell,s)})\Big\|_{\mathrm{op}}\big\|\sigma(z_s^{(\ell-1)}(\mathcal{X}))\big\|^2\big|\mathcal{F}_{\ell-1}\Big],
\end{multline*}

 Using the Cauchy-Schwartz inequality twice and Lemma \ref{norm_inf_hess} it follows that
\begin{multline*}
B_{\ell}(x^{(i)},x^{(j)})\\
\le 2\sum_{r=i,j} {\Big({K^{(\ell)}_{r,r}-C_b}-\sum_{s=1}^{n_{\ell-1}}\frac{{C_W}}{{n_{\ell-1}}} \sigma^2(z_s^{(\ell-1)}(x^{(r)}))\Big)}{\mathbb{E}\Big[\frac{\partial^2 f^{(\ell)}_{x^{(i)},x^{(j)}}}{\partial y_r^2}( z_1^{(\ell)}(\mathcal{X}))\big|\mathcal{F}_{\ell-1}\Big]}\\
+\|\sigma\|_{\text{Lip}}^2
\Big|{K^{(\ell)}_{i,j}-C_b}-\sum_{s=1}^{n_{\ell-1}}\frac{{C_W}}{{n_{\ell-1}}}\sigma(z_s^{(\ell-1)}(x^{(i)}))\sigma(z_s^{(\ell-1)}(x^{(j)}))\Big|\\
+\sum_{r=i,j}\sum_{u=i,j}\sqrt{\mathbb{E}\Big[\Big(\frac{\partial^2 f^{(\ell)}_{x^{(i)},x^{(j)}}}{\partial y_r \partial y_u}( z_1^{(\ell)}(\mathcal{X}))\Big)^2\big|\mathcal{F}_{\ell-1}\Big]}\cdot\\
\cdot \sqrt{\mathbb{E}\Big[\Big(\sum_{s=1}^{n_{\ell-1}}\frac{{C_W}}{2{n_{\ell-1}}}\Big( \big(W_{1,s}^{(\ell)}-W^{'(\ell)}_{1,s}\big)^2-2\Big) \sigma(z_s^{(\ell-1)}(x^{(r)}))\sigma(z_s^{(\ell-1)}(x^{(u)}))\Big)^2\big| \mathcal{F}_{\ell-1}\Big]}\\
+\sum_{r=i,j}\sum_{s=1}^{n_{\ell-1}}\frac{{C_W}^{3/2}}{4{n_{\ell-1}}^{3/2}}|\sigma(z_s^{(\ell-1)}(x^{(r)}))|\|\sigma(z_s^{(\ell-1)}(\mathcal{X}))\|^2\mathbb{E}\Big[ \big|W_{1,s}^{'(\ell)}-W^{(\ell)}_{1,s}\big|^3\cdot\\
\cdot
\Big\|\operatorname{Hess}\Big(\frac{\partial f^{(\ell)}_{x^{(i)},x^{(j)}}}{\partial y_r}\Big)(\eta^{(\ell,s)})\Big\|_{\text{HS}}\big|\mathcal{F}_{\ell-1}\Big]
\end{multline*}
\begin{multline*}
\le 2\sum_{r=i,j} {\mathbb{E}\Big[\frac{\partial^2 f^{(\ell)}_{x^{(i)},x^{(j)}}}{\partial y_r^2}( z_1^{(\ell)}(\mathcal{X}))\big|\mathcal{F}_{\ell-1}\Big]}
{\Big({K^{(\ell)}_{r,r}-C_b}-\sum_{s=1}^{n_{\ell-1}}\frac{{C_W}}{{n_{\ell-1}}} \sigma^2(z_s^{(\ell-1)}(x^{(r)}))\Big)}\\
+\|\sigma\|_{\text{Lip}}^2
\Big|{K^{(\ell)}_{i,j}-C_b}-\sum_{s=1}^{n_{\ell-1}}\frac{{C_W}}{{n_{\ell-1}}} \sigma(z_s^{(\ell-1)}(x^{(i)}))\sigma(z_s^{(\ell-1)}(x^{(j)}))\Big|
\\
+\sum_{r=i,j}\sum_{u=i,j}\sqrt{\mathbb{E}\Big[\Big(\frac{\partial^2 f^{(\ell)}_{x^{(i)},x^{(j)}}}{\partial y_r \partial y_u}( z_1^{(\ell)}(\mathcal{X}))\Big)^2\big|\mathcal{F}_{\ell-1}\Big]}\cdot\\
\cdot\sqrt{\sum_{s=1}^{n_{\ell-1}}\frac{{C_W^2}}{4{n_{\ell-1}^2}}\mathbb{E}\Big[\Big( \big(W_{1,s}^{(\ell)}-W^{'(\ell)}_{1,s}\big)^2-2\Big)^2\Big] \sigma^2(z_s^{(\ell-1)}(x^{(r)}))\sigma^2(z_s^{(\ell-1)}(x^{(u)}))}\\
+\sum_{r=i,j}\sum_{s=1}^{n_{\ell-1}}\frac{{C_W}^{3/2}}{4{n_{\ell-1}}^{3/2}}|\sigma(z_s^{(\ell-1)}(x^{(r)}))|\Big\|\sigma(z_s^{(\ell-1)}(\mathcal{X}))\Big\|^2\sqrt{\mathbb{E}\Big[ \big|W_{1,s}^{'(\ell)}-W^{(\ell)}_{1,s}\big|^6\big|\mathcal{F}_{\ell-1}\Big]}\cdot\\
\cdot\sqrt{\mathbb{E}\Big[
\Big\|\operatorname{Hess}\Big(\frac{\partial f^{(\ell)}_{x^{(i)},x^{(j)}}}{\partial y_r}\Big)(\eta^{(\ell,s)})\Big\|_{\text{HS}}^2\big|\mathcal{F}_{\ell-1}\Big]}.
\end{multline*}
As a consequence,
\begin{multline*}
\sum_{1\le i\ne j\le d} \mathbb{E}\Big[\Big(B_{\ell}(x^{(i)},x^{(j)})\Big)^2\Big]\\
\le 16  \sum_{1\le i\ne j\le d}\mathbb{E}\Bigg[\mathbb{E}\Big[\frac{\partial^2 f^{(\ell)}_{x^{(i)},x^{(i)}}}{\partial y_i^2}( z_1^{(\ell)}(\mathcal{X}))\big|\mathcal{F}_{\ell-1}\Big]^2
\Big({K^{(\ell)}_{i,i}-C_b}-\sum_{s=1}^{n_{\ell-1}}\frac{{C_W}}{{n_{\ell-1}}} \sigma^2(z_s^{(\ell-1)}(x^{(i)}))\Big)^2\Bigg]\\
+4\|\sigma\|_{\text{Lip}}^4\sum_{1\le i\ne j\le d}
 \mathbb{E}\Bigg[
\Big|{K^{(\ell)}_{i,j}-C_b}-\sum_{s=1}^{n_{\ell-1}}\frac{{C_W}}{{n_{\ell-1}}} \sigma(z_s^{(\ell-1)}(x^{(i)}))\sigma(z_s^{(\ell-1)}(x^{(j)}))\Big|^2\Bigg]\\
+32 \sum_{i=1}^d \sum_{j=1,j\ne i}^d \mathbb{E}\Bigg[{\mathbb{E}\Big[\Big(\frac{\partial^2 f^{(\ell)}_{x^{(i)},x^{(j)}}}{\partial y_i^2 }( z_1^{(\ell)}(\mathcal{X}))\Big)^2\big|\mathcal{F}_{\ell-1}\Big]}\cdot\\
\cdot \sum_{s=1}^{n_{\ell-1}}\frac{{C_W^2}}{4{n_{\ell-1}^2}}\mathbb{E}\Big[\Big( \big(W_{1,s}^{(\ell)}-W^{'(\ell)}_{1,s}\big)^2-2\Big)^2\Big] \sigma^4(z_s^{(\ell-1)}(x^{(i)}))\Bigg]\\
+32\sum_{1\le i\ne j\le d}\mathbb{E}\Bigg[{\mathbb{E}\Big[\Big(\frac{\partial^2 f^{(\ell)}_{x^{(i)},x^{(j)}}}{\partial y_i \partial y_j}( z_1^{(\ell)}(\mathcal{X}))\Big)^2\big|\mathcal{F}_{\ell-1}\Big]}\cdot\\
\cdot{\sum_{s=1}^{n_{\ell-1}}\frac{{C_W^2}}{4{n_{\ell-1}^2}}\mathbb{E}\Big[\Big( \big(W_{1,s}^{(\ell)}-W^{'(\ell)}_{1,s}\big)^2-2\Big)^2\Big] \sigma^2(z_s^{(\ell-1)}(x^{(i)}))\sigma^2(z_s^{(\ell-1)}(x^{(j)}))}\Bigg]\\
+32\sum_{1\le i\ne j\le d}\mathbb{E}\Bigg[\Bigg(\sum_{s=1}^{n_{\ell-1}}\frac{{C_W}^{3/2}}{4{n_{\ell-1}}^{3/2}}|\sigma(z_s^{(\ell-1)}(x^{(i)}))|\Big\|\sigma(z_s^{(\ell-1)}(\mathcal{X}))\Big\|^2\sqrt{\mathbb{E}\Big[ \big|W_{1,s}^{'(\ell)}-W^{(\ell)}_{1,s}\big|^6\Big]}\cdot\\
\cdot\sqrt{\mathbb{E}\Big[
\Big\|\operatorname{Hess}\Big(\frac{\partial f^{(\ell)}_{x^{(i)},x^{(j)}}}{\partial y_i}\Big)(\eta^{(\ell,s)})\Big\|_{\text{HS}}^2\big|\mathcal{F}_{\ell-1}\Big]}\Bigg)^2\Bigg]
\end{multline*}
Using the bound (\ref{der_sec_simp}) for the second derivative of $f^{(\ell)}_{x^{(i)},x^{(j)}}$ it follows that:
\,

\,

(I)
\begin{multline*}
\sum_{1\le i\ne j\le d}\mathbb{E}\Bigg[\mathbb{E}\Big[\frac{\partial^2 f^{(\ell)}_{x^{(i)},x^{(j)}}}{\partial y_i^2}( z_1^{(\ell)}(\mathcal{X}))\big|\mathcal{F}_{\ell-1}\Big]^2
\Big({K^{(\ell)}_{i,i}-C_b}-\sum_{s=1}^{n_{\ell-1}}\frac{{C_W}}{{n_{\ell-1}}} \sigma^2(z_s^{(\ell-1)}(x^{(i)}))\Big)^2\Bigg]\\
\le\Bigg[\Bigg(4\|\sigma\|_{\text{Lip}}^2\|(K^{(\ell)})^{-1}\|^2_{\text{HS}}\,\operatorname{tr}K^{(\ell)}\Big(\|\sigma(z^{(\ell)}_1(\mathcal{X}))\|^2+\|\sigma\|_{\text{Lip}}^2\|z^{(\ell)}_1(\mathcal{X})\|^2\Big)\\
+2\|\sigma\|_{\text{Lip}}^4{\|(K^{(\ell)})^{-1}\|^2}_{\text{HS}}\,\operatorname{tr}K^{(\ell)}\Bigg) \sum_{i=1}^d\Big({K^{(\ell)}_{i,i}-C_b}-\sum_{s=1}^{n_{\ell-1}}\frac{{C_W}}{{n_{\ell-1}}} \sigma^2(z_s^{(\ell-1)}(x^{(i)}))\Big)^2\Bigg]
\end{multline*}

\begin{multline*}
\le 2\|\sigma\|_{\text{Lip}}^2\|(K^{(\ell)})^{-1}\|^2_{\text{HS}}\,\operatorname{tr}K^{(\ell)}\Bigg(2\mathbb{E}\Big[\|\sigma(z^{(\ell)}_1(\mathcal{X}))\|^4\Big]^{1/2}+2\|\sigma\|_{\text{Lip}}^2\mathbb{E}\Big[\|z^{(\ell)}_1(\mathcal{X})\|^4\Big]^{1/2}\\
+\|\sigma\|_{\text{Lip}}^2\Bigg) \sum_{i=1}^d\mathbb{E}\Bigg[\Big({K^{(\ell)}_{i,i}-C_b}-\sum_{s=1}^{n_{\ell-1}}\frac{{C_W}}{{n_{\ell-1}}} \sigma^2(z_s^{(\ell-1)}(x^{(i)}))\Big)^4\Bigg]^{1/2},
\end{multline*}
\begin{multline*}
= 2\|\sigma\|_{\text{Lip}}^2\|(K^{(\ell)})^{-1}\|^2_{\text{HS}}\,\operatorname{tr}K^{(\ell)}\sum_{i=1}^d\sqrt{Q_4^{(\ell-1)}(x^{(i)})}\cdot\\
\cdot\Bigg(2\mathbb{E}\Big[\|\sigma(z^{(\ell)}_1(\mathcal{X}))\|^4\Big]^{1/2}+2\|\sigma\|_{\text{Lip}}^2\mathbb{E}\Big[\|z^{(\ell)}_1(\mathcal{X})\|^4\Big]^{1/2}
+\|\sigma\|_{\text{Lip}}^2\Bigg) ,
\end{multline*}
where $Q_4^{(\ell-1)}(x^{(i)})$ is defined in \eqref{Q_k};
\,

\,

(II)
\begin{multline*}
   \sum_{i=1}^d \sum_{j=1,j\ne i}^d \mathbb{E}\Bigg[{\mathbb{E}\Big[\Big(\frac{\partial^2 f^{(\ell)}_{x^{(i)},x^{(j)}}}{\partial y_i^2 }( z_1^{(\ell)}(\mathcal{X}))\Big)^2\big|\mathcal{F}_{\ell-1}\Big]}\cdot\\
\cdot \sum_{s=1}^{n_{\ell-1}}\frac{{C_W^2}}{4{n_{\ell-1}^2}}\mathbb{E}\Big[\Big( \big(W_{1,s}^{(\ell)}-W^{'(\ell)}_{1,s}\big)^2-2\Big)^2\Big] \sigma^4(z_s^{(\ell-1)}(x^{(i)}))\Bigg]
\end{multline*}
\begin{multline*}
\le \frac{{C_W^2}}{4{n_{\ell-1}^2}}\mathbb{E}\Bigg[\Bigg(4\|\sigma\|_{\text{Lip}}^2\|(K^{(\ell)})^{-1}\|^2_{\text{HS}}\,\operatorname{tr}K^{(\ell)}\Big(\|\sigma(z^{(\ell)}_1(\mathcal{X}))\|^2+\|\sigma\|_{\text{Lip}}^2\|z^{(\ell)}_1(\mathcal{X})\|^2\Big)\\
+2\|\sigma\|_{\text{Lip}}^4{\|(K^{(\ell)})^{-1}\|^2}_{\text{HS}}\,\operatorname{tr}K^{(\ell)}\Bigg)\sum_{s=1}^{n_{\ell-1}}\|\sigma(z^{(\ell-1)}_s(\mathcal{X}))\|^4\Bigg]\operatorname{Var}\Big((W_{1,1}^{(1)}-W_{1,1}^{'(1)})^2\Big)
\end{multline*}

\begin{multline*}
\le \frac{4{C_W^2}}{{n_{\ell-1}}}\|\sigma\|_{\text{Lip}}^2\|(K^{(\ell)})^{-1}\|^2_{\text{HS}}\,\operatorname{tr}K^{(\ell)}\mathbb{E}\Big[\|\sigma(z^{(\ell-1)}_1(\mathcal{X}))\|^8\Big]^{1/2}\mathbb{E}\Big[(W_{1,1}^{(1)})^4\Big]\cdot\\
\cdot \Bigg(2\mathbb{E}\Big[\|\sigma(z^{(\ell)}_1(\mathcal{X}))\|^4\Big]^{1/2}+2\|\sigma\|_{\text{Lip}}^2\mathbb{E}\Big[\|z^{(\ell)}_1(\mathcal{X})\|^4\Big]^{1/2}
+\|\sigma\|_{\text{Lip}}^2\Bigg).
\end{multline*}

Moreover, applying Lemma \ref{der_ter_lemma} and using that for every $s=1,\dots, n_{\ell-1}$ the random variable defined in \eqref{z_1_l_s}, $z_1^{(\ell)}(\mathcal{X})$, is equal in law to $z_1^{(\ell)(s)}(\mathcal{X})$, one can easily show that

\begin{multline*}
\sum_{1\le i\ne j\le d}\mathbb{E}\Bigg[\Bigg(\sum_{s=1}^{n_{\ell-1}}\frac{{C_W}^{3/2}}{4{n_{\ell-1}}^{3/2}}|\sigma(z_s^{(\ell-1)}(x^{(i)}))|\Big\|\sigma(z_s^{(\ell-1)}(\mathcal{X}))\Big\|^2\sqrt{\mathbb{E}\Big[ \big|W_{1,s}^{'(\ell)}-W^{(\ell)}_{1,s}\big|^6\Big]}\cdot\\
\cdot\sqrt{\mathbb{E}\Big[
\Big\|\operatorname{Hess}\Big(\frac{\partial f^{(\ell)}_{x^{(i)},x^{(j)}}}{\partial y_i}\Big)(\eta^{(\ell,s)})\Big\|_{\text{HS}}^2\big|\mathcal{F}_{\ell-1}\Big]}\Bigg)^2\Bigg]
\end{multline*}
\begin{multline*}
    \le\frac{16{C_W}^{3}\|\sigma\|_{\text{Lip}}^2}{{n_{\ell-1}}}(1+\|\sigma\|_{\text{Lip}}^2)(1+C_W)\|(K^{(\ell)})^{-1}\|^2_{\text{HS}}{\mathbb{E}\Big[ \big(W_{1,1}^{(1)}\big)^6\Big]^2}\cdot\\
\cdot\Big(\operatorname{tr}K^{(\ell)}+\sqrt{d}\mathbb{E}\Big[\|z_1^{(\ell)}(\mathcal{X})\|^4\Big]^{1/2}
+2\sqrt{d}\mathbb{E}\Big[\|\sigma(z_1^{(\ell-1)}(\mathcal{X})\|^4\Big]^{1/2}+2\sqrt{d}\mathbb{E}\Big[\|\sigma(z_1^{(\ell)}(\mathcal{X})\|^4\Big]^{1/2}\Big)\\
\cdot \mathbb{E}\Big[\|\sigma(z_1^{(\ell-1)}(\mathcal{X}))\|^{12}\Big]^{1/2}
\Big(4d
+5\|(K^{(\ell)})^{-1}\|^2_{op}\,{(\operatorname{tr}K^{(\ell)})^2}
+12\Big)
\end{multline*}

Hence, applying the previous bounds and Lemma \ref{norm_inf_hess},
\begin{multline*}
    \sum_{1\le i\ne j\le d} \mathbb{E}\Big[\Big(B_{\ell}(x^{(i)},x^{(j)})\Big)^2\Big]
\le  
32\|\sigma\|_{\text{Lip}}^2\|(K^{(\ell)})^{-1}\|^2_{\text{HS}}\,\operatorname{tr}K^{(\ell)}\sum_{i=1}^d\sqrt{Q_4^{(\ell-1)}(x^{(i)})}\cdot\\
\cdot\Bigg(2\mathbb{E}\Big[\|\sigma(z^{(\ell)}_1(\mathcal{X}))\|^4\Big]^{1/2}+2\|\sigma\|_{\text{Lip}}^2\mathbb{E}\Big[\|z^{(\ell)}_1(\mathcal{X})\|^4\Big]^{1/2}
+\|\sigma\|_{\text{Lip}}^2\Bigg)
\\
+8\|\sigma\|_{\text{Lip}}^4\frac{C_W^2}{n_{\ell-1}}\mathbb{E}\Big[\|\sigma(G_1^{(\ell-1)}(\mathcal{X})\|^4\Big]+8\|\sigma\|_{\text{Lip}}^4\frac{C_W^2}{n_{\ell-1}}\mathbb{E}\Big[\|\sigma(z_1^{(\ell-1)}(\mathcal{X})\|^4\Big]\\
+4C_W^2\|\sigma\|_{\text{Lip}}^4
\sum_{1\le i\ne j\le d} \mathbb{E}\Big[\Big(B_{\ell-1}(x^{(i)},x^{(j)})\Big)^2\Big]
\\
+ \frac{2^7{C_W^2}}{{n_{\ell-1}}}\|\sigma\|_{\text{Lip}}^2\|(K^{(\ell)})^{-1}\|^2_{\text{HS}}\,\operatorname{tr}K^{(\ell)}\mathbb{E}\Big[\|\sigma(z^{(\ell-1)}_1(\mathcal{X}))\|^8\Big]^{1/2}\mathbb{E}\Big[(W_{1,1}^{(1)})^4\Big]\cdot\\
\cdot \Bigg(2\mathbb{E}\Big[\|\sigma(z^{(\ell)}_1(\mathcal{X}))\|^4\Big]^{1/2}+2\|\sigma\|_{\text{Lip}}^2\mathbb{E}\Big[\|z^{(\ell)}_1(\mathcal{X})\|^4\Big]^{1/2}
+\|\sigma\|_{\text{Lip}}^2\Bigg)
\\
+16\|\sigma\|_{\text{Lip}}^4\frac{{C_W^2}}{{n_{\ell-1}}}\mathbb{E}\Big[( W_{1,1}^{(1)})^4\Big]\mathbb{E}\Big[ \|\sigma(z_1^{(\ell-1)}(\mathcal{X}))\|^4\Big]\\
+\frac{2^9{C_W}^{3}\|\sigma\|_{\text{Lip}}^2}{{n_{\ell-1}}}(1+\|\sigma\|_{\text{Lip}}^2)(1+C_W)\|(K^{(\ell)})^{-1}\|^2_{\text{HS}}{\mathbb{E}\Big[ \big(W_{1,1}^{(1)}\big)^6\Big]^2}\cdot\\
\cdot\Big(\operatorname{tr}K^{(\ell)}+\sqrt{d}\mathbb{E}\Big[\|z_1^{(\ell)}(\mathcal{X})\|^4\Big]^{1/2}
+2\sqrt{d}\mathbb{E}\Big[\|\sigma(z_1^{(\ell-1)}(\mathcal{X})\|^4\Big]^{1/2}+2\sqrt{d}\mathbb{E}\Big[\|\sigma(z_1^{(\ell)}(\mathcal{X})\|^4\Big]^{1/2}\Big)\\
\cdot \mathbb{E}\Big[\|\sigma(z_1^{(\ell-1)}(\mathcal{X}))\|^{12}\Big]^{1/2}
\Big(4d
+5\|(K^{(\ell)})^{-1}\|^2_{op}\,{(\operatorname{tr}K^{(\ell)})^2}
+12\Big).
\end{multline*}

After some computations,
\[
\sum_{1\le i\ne j\le d} \mathbb{E}\Big[\Big(B_{\ell}(x^{(i)},x^{(j)})\Big)^2\Big]\le V_1^{(\ell)(n)}(\mathcal{X})+V_2\sum_{1\le i\ne j\le d} \mathbb{E}\Big[\Big(B_{\ell-1}(x^{(i)},x^{(j)})\Big)^2\Big],
\]
where
\[
V_2:=4C_W^2\|\sigma\|_{\text{Lip}}^4
\]
and
\begin{multline*}
    V_1^{(\ell)(n)}(\mathcal{X}):=2^9\Big(1+\|\sigma\|_{\text{Lip}}^2\Big)^2(1+C_W)^4\mathbb{E}\Big[ \big(W_{1,1}^{(1)}\big)^6\Big]^2\Bigg(\operatorname{tr} K^{(\ell)}+2\mathbb{E}\Big[\|\sigma(z^{(\ell)}_1(\mathcal{X}))\|^4\Big]^{1/2}\\
   +2(\|\sigma\|_{\text{Lip}}^2+1)\mathbb{E}\Big[\|z^{(\ell)}_1(\mathcal{X})\|^4\Big]^{1/2}
   +\mathbb{E}\Big[\|\sigma(G_1^{(\ell-1)}(\mathcal{X})\|^4\Big]^{1/2}
+2\mathbb{E}\Big[\|\sigma(z^{(\ell-1)}_1(\mathcal{X}))\|^4\Big]^{1/2}+\|\sigma\|_{\text{Lip}}^2\Bigg)\cdot\\
\Bigg(1+\mathbb{E}\Big[\|\sigma(G_1^{(\ell-1)}(\mathcal{X})\|^4\Big]^{1/2}+\mathbb{E}\Big[\|\sigma(z_1^{(\ell-1)}(\mathcal{X}))\|^{12}\Big]^{1/2}\Bigg)\cdot\\
\cdot \Big(1+\|(K^{(\ell)})^{-1}\|^2_{\text{HS}}\Big)\Big(1+ {(\operatorname{tr}K^{(\ell)})^2}\Big)\Bigg\{\sum_{i=1}^d\sqrt{Q_4^{(\ell-1)}(x^{(i)})}
+\frac{2}{n_{\ell-1}}
\\
+\frac{1}{{n_{\ell-1}}}\Big(1+\sqrt{d}\Big)
\Big(4d+5\|(K^{(\ell)})^{-1}\|^2_{op}+12\Big)\Bigg\}.
\end{multline*}
Therefore
\begin{multline*}
\sum_{1\le i\ne j\le d} \mathbb{E}\Big[\Big(B_{\ell}(x^{(i)},x^{(j)})\Big)^2\Big]\\
\le \sum_{r=0}^{\ell-3}V_1^{(\ell-r)(n)}(\mathcal{X})(V_2)^r+(V_2)^{\ell-2}\sum_{1\le i\ne j\le d} \mathbb{E}\Big[\Big(B_{2}(x^{(i)},x^{(j)})\Big)^2\Big],
\end{multline*}
where
\[
\mathbb{E}[(B^{(2)})^2]\le V_1^{(2)},
\]
and so
\[
\sum_{1\le i\ne j\le d} \mathbb{E}\Big[\Big(B_{\ell}(x^{(i)},x^{(j)})\Big)^2\Big]\le \sum_{r=0}^{\ell-2}V_1^{(\ell-r)(n)}(\mathcal{X})(V_2)^r.
\]
\subsection{Proof of Theorem \ref{th_fin_mult_dim}}
From Theorem \ref{mod_KG}, one deduces that

\begin{multline*}
d_C(z_1^{(L+1)}(\mathcal{X}),G_1^{(L+1)}(\mathcal{X}))\le
541\cdot d^4 \sqrt{C_W}(1+C_W)\max\{1,\|(K^{(L+1)})^{-1}\|_{\mathrm{op}}^2\}\cdot\\
\cdot\mathbb{E}\Big[|W_{1,1}^{(1)}|^6\Big]^{1/2}\frac{1}{\sqrt{n_{L}}} \Bigg\{43\Big(1+\mathbb{E}\Big[\|\sigma(z_1^{(L)}(\mathcal{X})\|^6\Big]^{1/2}+\mathbb{E}\Big[\|\sigma(G_1^{(L)}(\mathcal{X})\|^4\Big]^{1/2}\Big)\\
+\sqrt{2}\sqrt{n_{L}}\Bigg(\sum_{j,k=1}^{d}\mathbb{E}\Big[(B_{L}(x^{(j)},x^{(k)}))^2\Big]\Bigg)^{1/2}\Bigg\}.
\end{multline*}
Suppose now $C_b,C_W\ne 0$ and use 
Proposition \ref{mom_z_prop} to bound the moments of the neural network and Lemma \ref{cost_schif}, Lemma \ref{somma_schifa} and Remark \ref{bound_V_1} to bound $\sum_{j=1}^d\sum_{k=1}^{d}\mathbb{E}\Big[(B_{L}(x^{(j)},x^{(k)}))^2\Big]$. 
After some computations and applying Remark \ref{HS_bondK} to bound the inverse of the minimum eigenvalue of $K^{(\ell)}$ for $\ell=1,\dots,L+1$, we obtain the desired bound.

\end{document}